\documentclass[a4paper,twoside,11pt]{book}

\usepackage[T1]{fontenc}
\usepackage[cp1250]{inputenc}

\usepackage{listings}
\usepackage{amsmath}
\usepackage{amstext,amssymb,amsthm,enumitem,geometry,float,tikz,pgfplots,mathtools}
\usepackage{graphicx}
\usepackage{setspace}
\tikzset{dots/.append style={ultra thick, fill=none}}
\usetikzlibrary{arrows}
\usepackage{array}
\newcolumntype{C}[1]{>{\centering\let\newline\\\arraybackslash\hspace{0pt}}m{#1}}

\usepackage[hidelinks]{hyperref}
\usepackage{nameref}
\hypersetup{colorlinks = true, urlcolor = blue, linkcolor = blue, citecolor = red}

\usepackage{pifont}
\newcommand{\cmark}{\ding{51}}
\newcommand{\xmark}{\ding{55}}

\numberwithin{equation}{section}

\newtheorem{theorem}[equation]{Theorem}
\newtheorem{corollary}[equation]{Corollary}
\newtheorem{proposition}[equation]{Proposition}
\newtheorem{observation}[equation]{Observation}
\newtheorem{fact}[equation]{Fact}
\newtheorem{lemma}[equation]{Lemma}
\newtheorem{example}[equation]{Example}
\newtheorem{remark}[equation]{Remark}
\theoremstyle{definition}

\usepackage{fancyhdr}
\fancyheadoffset{0.000\textwidth}
\allowdisplaybreaks

\newcommand{\subscript}[2]{$#1 #2$}

\newcommand{\NN}{\mathbb{N}}
\newcommand{\QQ}{\mathbb{Q}}
\newcommand{\ZZ}{\mathbb{Z}}
\newcommand{\RR}{\mathbb{R}}
\newcommand{\CCC}{\mathbb{C}}

\newcommand{\BB}{\mathfrak{B}}
\newcommand{\XX}{\mathfrak{X}}
\newcommand{\SSS}{\mathfrak{S}}
\newcommand{\TT}{\mathfrak{T}}
\newcommand{\UU}{\mathfrak{U}}
\newcommand{\VV}{\mathfrak{V}}
\newcommand{\WW}{\mathfrak{W}}
\newcommand{\YY}{\mathfrak{Y}}
\newcommand{\ZZZ}{\mathfrak{Z}}

\newcommand{\cc}{\mathbf{c}}
\newcommand{\CC}{\mathbf{C}}
\newcommand{\mm}{\mathbf{m}}

\newcommand{\AAA}{\mathcal{A}}
\newcommand{\MM}{\mathcal{M}^{\rm{c}}}

\newcommand{\MN}{\mathcal{M}}
\newcommand{\HH}{\mathcal{H}}
\newcommand{\SSSS}{\mathcal{S}}
\newcommand{\TTTT}{\mathcal{T}}

\newcommand{\LP}{L^p(\mathfrak{X})}
\newcommand{\LINF}{L^\infty(\mathfrak{X})}

\newcommand{\PSC}{P_{\rm s}^{\rm c}}
\newcommand{\PWC}{P_{\rm w}^{\rm c}}
\newcommand{\PRC}{P_{\rm r}^{\rm c}}
\newcommand{\PS}{P_{\rm s}}
\newcommand{\PW}{P_{\rm w}}
\newcommand{\PR}{P_{\rm r}}

\newcommand{\PKSC}{P_{\kappa, \rm s}^{\rm c}}
\newcommand{\PKWC}{P_{\kappa, \rm w}^{\rm c}}
\newcommand{\PKS}{P_{\kappa, \rm s}}
\newcommand{\PKW}{P_{\kappa, \rm w}}
	
\newcommand{\MMK}{\mathcal{M}_\kappa^{\rm{c}}}
\newcommand{\MNK}{\mathcal{M}_\kappa}
\newcommand{\MMKX}{\mathcal{M}_{\kappa, \mathfrak{X}}^{\rm{c}}}
\newcommand{\MNKX}{\mathcal{M}_{\kappa, \mathfrak{X}}}

\newcommand{\SXCKP}{\cc_{\rm s, \mathfrak{X}}^{\rm c}(\kappa, p)}
\newcommand{\WXCKP}{\cc_{\rm w, \mathfrak{X}}^{\rm c}(\kappa, p)}
\newcommand{\SXKP}{\cc_{\rm s, \mathfrak{X}}(\kappa, p)}
\newcommand{\WXKP}{\cc_{\rm w, \mathfrak{X}}(\kappa, p)}

\begin{document}

\frontmatter

\begin{titlepage}
	
	\thispagestyle{empty}
	
	\begin{center}
		\noindent {\large \textbf{WROC{\L}AW UNIVERSITY OF SCIENCE AND TECHNOLOGY}} \\
		\vspace*{0.3cm}
		\noindent \textbf{Faculty of Pure and Applied Mathematics} \\
		\vspace*{0.5cm}
		\noindent \Huge \textbf{DOCTORAL \ DISSERTATION} \\
		\vspace*{1.3cm}
		\noindent \LARGE Dariusz \textsc{Kosz} \\
		\vspace*{1.0cm}
		\noindent {\Huge \textbf{Maximal operators \\ in nondoubling \\
				metric measure spaces}} \\
		\vspace*{1.0cm}
		\noindent \Large Supervisor: Professor Krzysztof \textsc{Stempak} \\
		\vspace*{11.0cm}
		\noindent \Large {Wroc{\l}aw, May 2019} \\
		\small \smallskip {(revised version -- October 2021)}
	\end{center}
	
	\newpage
	\thispagestyle{empty}
	
	\newpage
	\thispagestyle{empty}
	
	\begin{center}
		\noindent {\large \textbf{POLITECHNIKA WROC{\L}AWSKA}} \\
		\vspace*{0.3cm}
		\noindent \textbf{Wydzia{\l} Matematyki} \\
		\vspace*{0.5cm}
		\noindent \Huge \textbf{ROZPRAWA \ DOKTORSKA} \\
		\vspace*{1.3cm}
		\noindent \LARGE Dariusz \textsc{Kosz} \\
		\vspace*{1.0cm}
		\noindent {\Huge \textbf{Operatory
				maksymalne \\ w niedubluj{\k a}cych przestrzeniach \\  metryczno-miarowych}} \\
		\vspace*{1.0cm}
		\noindent \Large Promotor rozprawy: prof. dr hab. Krzysztof \textsc{Stempak} \\
		\vspace*{11cm}
		\noindent \Large {Wroc{\l}aw, maj 2019} \\
		\small \smallskip {(wersja zaktualizowana -- pa{\' z}dziernik 2021)}
	\end{center}
	
	\newpage
	\thispagestyle{empty}
\end{titlepage}
\sloppy

\titlepage

\thispagestyle{empty}
\newgeometry{top=15cm}

\begin{flushright}
To my parents, \\
\vspace{0,10cm}
who have always supported me and believed in me.
\end{flushright}

\newpage
\thispagestyle{empty}

\newgeometry{top=17cm, left=3cm}

\textbf{Acknowledgments.}
First and foremost, I would like to express my deepest gratitude to my supervisor Professor Krzysztof Stempak. I cannot count how many hours he devoted to my education. His attitude was also a huge stimulus for writing this dissertation.

\vspace{0,05cm}
I would like to thank all the people with whom I have discussed dozens of mathematical problems in recent years. Among others, I would like to mention here all members of the seminar {\it Harmonic Analysis and Orthogonal Expansions} and all my PhD colleagues. I am also indebted to Professor Lech Maligranda for the discussion on Theorem~\ref{44-T2}. 

\vspace{0,05cm}
Finally, special thanks go to my family and those who are closest to me. It is thanks to you that I have the strength to go ahead and fulfill my dreams. With you I feel that I have something precious in my life.

\vline
\hrulefill

I gratefully acknowledge the support by the National Science Centre of Poland within the project PRELUDIUM 2016/21/N/ST1/01496.

\vspace{0,05cm}
The revised version of the dissertation was prepared during my stay at the Basque Center for Applied Mathematics as a postdoctoral fellow. I gratefully acknowledge the support by the Basque Government through the BERC 2018-2021 program and by the Spanish State Research Agency through BCAM Severo Ochoa excellence accreditation SEV-2017-2018.

\restoregeometry

\tableofcontents

\chapter[Streszczenie (Summary in Polish)]{Streszczenie}
\markboth{Streszczenie (Summary in Polish)}{Streszczenie (Summary in Polish)}
\setstretch{1.00}
Operatory maksymalne s{\k a} obiektami o~du{\. z}ym znaczeniu w~matematyce na czele z~analiz{\k a} harmoniczn{\k a}. M{\' o}wi{\k a}c zwi{\k e}{\' z}le, ich g{\l}{\' o}wn{\k a} rol{\k a} jest szacowanie z~g{\' o}ry warto{\' s}ci innych rozwa{\. z}anych operator{\' o}w. Standardowy spos{\' o}b ich u{\. z}ycia powinien zatem bazowa{\' c} na ich ograniczono{\' s}ci pomi{\k e}dzy pewnymi dwiema przestrzeniami funkcji. W~rzeczy samej istniej{\k a} dziesi{\k a}tki prac, w~kt{\' o}rych u{\. z}ywane s{\k a} r{\' o}{\. z}ne typy ograniczono{\' s}ci operator{\' o}w maksymalnych.

Spo{\' s}r{\' o}d ca{\l}ej rodziny wspomnianych obiekt{\' o}w k{\l}adziemy szczeg{\' o}lny nacisk na klasyczne operatory maksymalne Hardy'ego--Littlewooda, kt{\' o}re mo{\. z}na wprowadzi{\' c} w~kontek{\' s}cie dowolnej przestrzeni metryczno-miarowej $\XX$ w~dw{\' o}ch wersjach, scentrowanej $\MM$ oraz niescentrowanej $\MN$. Podstawow{\k a} w{\l}asno{\' s}ci{\k a} tych operator{\' o}w jest ich ograniczono{\' s}{\' c} na przestrzeni $L^\infty(\XX)$ ze sta{\l}{\k a} $1$. Aby m{\' o}c wskaza{\' c} inne interesuj{\k a}ce w{\l}asno{\' s}ci, powinni{\' s}my wiedzie{\' c} wi{\k e}cej o~strukturze przestrzeni $\XX$.

Na wst{\k e}pie po{\' s}wi{\k e}{\' c}my kilka s{\l}{\' o}w sytuacji, w~kt{\' o}rej $\XX$ to przestrze{\' n} $\RR^d$, $d \in \NN$, wyposa{\. z}ona w~miar{\k e} Lebesgue'a i~metryk{\k e} euklidesow{\k a}. Jednym z~najwa{\. z}niejszych wynik{\' o}w otrzymanych w~tym konkretnym przypadku jest, {\. z}e oba operatory, $\MM$ oraz $\MN$, s{\k a} s{\l}abego typu $(1,1)$, co oznacza, {\. z}e s{\k a} one ograniczone z~$L^1(\XX)$ do $L^{1, \infty}(\XX)$. Ten fakt ma kilka istotnych konsekwencji w{\l}{\k a}cznie z~twierdzeniem Lebesgue'a o~r{\' o}{\. z}niczkowaniu, s{\l}ynnym rezultatem z~analizy rzeczywistej. Poza tym, maj{\k a}c na uwadze, {\. z}e nasze operatory s{\k a} podliniowe, mo{\. z}emy zastosowa{\' c} twierdzenie interpolacyjne Marcinkiewicza w celu wykazania ich ograniczono{\' s}ci na przestrzeni $L^p(\XX)$ dla ka{\. z}dego $p \in (1, \infty)$.

Dalsze badania w~tym obszarze s{\k a} skupione mi{\k e}dzy innymi na wyznaczaniu optymalnych sta{\l}ych w~nier{\' o}wno{\' s}ciach zawieraj{\k a}cych funkcje maksymalne, z~nier{\' o}wno{\' s}ci{\k a} s{\l}abego typu $(1,1)$ na pierwszym miejscu \cite{Al1, Me1, Me2}. S{\k a} r{\' o}wnie{\. z} artyku{\l}y po{\' s}wi{\k e}cone w{\l}asno{\' s}ciom $\MM$ oraz $\MN$ w~kontek{\' s}cie pewnych przestrzeni, w~kt{\' o}rych mierzona jest regularno{\' s}{\' c} funkcji \cite{BCHP, HO, Ki, Ku, Ta}. Wreszcie, wa{\. z}ny kierunek bada{\' n} wyznaczaj{\k a} prace, przyjmuj{\k a}ce za cel analizowanie, co dzieje si{\k e} z~poszczeg{\' o}lnymi w{\l}asno{\' s}ciami operator{\' o}w maksymalnych, gdy przestrze{\' n} metryczno-miarowa przyjmuje r{\' o}{\. z}norakie formy.

Standardowymi narz{\k e}dziami, kt{\' o}rych u{\. z}ywa si{\k e} do pokazania oszacowa{\' n} s{\l}abego typu $(1,1)$, s{\k a} lematy pokryciowe. Na pierwszy rzut oka wydaje si{\k e}, {\. z}e mo{\. z}liwo{\' s}{\' c} ich u{\. z}ycia zale{\. z}y g{\l}{\' o}wnie od w{\l}asno{\' s}ci metrycznych danej przestrzeni. Aby to zilustrowa{\' c}, wspomnijmy, {\. z}e dla $\RR^d$ z~metryk{\k a} euklidesow{\k a} odpowiedni lemat pokryciowy zapewnia, {\. z}e $\MM$ jest s{\l}abego typu $(1,1)$ w~przypadku dowolnej ,,sensownej'' miary (mo{\. z}na tu na przyk{\l}ad pomy{\' s}le{\' c} o~dowolnej mierze Radona). Jednak{\. z}e, je{\' s}li tylko podstawimy $\MN$ w~miejsce $\MM$, to sytuacja zmieni si{\k e} diametralnie. Mianowicie, mo{\. z}liwe jest znalezienie miary na $\RR^d$, $d \in \NN \setminus \{1\}$, dla kt{\' o}rej stowarzyszony niescentrowany operator $\MN$ nie jest s{\l}abego typu $(1,1)$. W~istocie, Sj\"ogren \cite{Sj} pokaza{\l}, {\. z}e ma to miejsce w~przypadku ${\rm d} \mu(x,y) = {\rm exp}(-(x^2 + y^2)/2) \, {\rm d}x \, {\rm d}y$, czyli miary zwi{\k a}zanej ze standardowym dwuwymiarowym rozk{\l}adem Gaussa (warto zapozna{\' c} si{\k e} r{\' o}wnie{\. z} z~przyk{\l}adem Aldaza \cite{Al2}).

Ostatnia uwaga sugeruje, {\. z}e warunek narzucony na $\XX$, kt{\' o}ry zapewni{\l}by, {\. z}e wi{\k e}kszo{\' s}{\' c} klasycznej teorii dzia{\l}a, powinien uwzgl{\k e}dnia{\' c} oba aspekty: metryk{\k e} i~miar{\k e}. Rzeczywi{\' s}cie, w~kon-tek{\' s}cie r{\' o}{\. z}nych przestrzeni metryczno-miarowych tak zwany warunek dublowania jest w~lite-raturze bardzo intensywnie eksploatowany. Pokr{\' o}tce, zak{\l}ada on, {\. z}e miara dowolnej kuli $B$ jest w spos{\' o}b jednostajny por{\' o}wnywalna z~miar{\k a} $2B$, kuli o~tym samym {\' s}rodku i~dwukrotnie wi{\k e}kszym promieniu. Opr{\' o}cz wielu innych rezultat{\' o}w otrzymujemy, {\. z}e dla dowolnej przestrzeni dubluj{\k a}cej $\XX$ stowarzyszone operatory $\MM$ oraz $\MN$ s{\k a} por{\' o}wnywalne i~oba spe{\l}niaj{\k a} oszacowanie s{\l}abego typu $(1,1)$. W~literaturze rozwijane by{\l}o r{\' o}wnie{\. z} kilka koncepcji, maj{\k a}cych na celu zast{\k a}pienie warunku dublowania pewnymi s{\l}abszymi warunkami (jak na przyk{\l}ad w~\cite{H}) b{\k a}d{\' z} ca{\l}kowite z~niego zrezygnowanie.

Nazarov, Treil i~Volberg mieli znacz{\k a}cy wk{\l}ad w~rozwijanie analizy harmonicznej na dowolnych przestrzeniach metryczno-miarowych. W~ich prze{\l}omowej pracy \cite{NTV2} zawarte s{\k a} niezwykle cenne obserwacje, kt{\' o}re wskazuj{\k a}, jak mierzy{\' c} si{\k e} z~wa{\. z}nymi w~tej dziedzinie problemami w~sytuacji, w~kt{\' o}rej posi{\l}kowanie si{\k e} warunkiem dublowania nie jest mo{\. z}liwe. Dla nas szczeg{\' o}lnie interesuj{\k a}ce jest, {\. z}e przy tej okazji zosta{\l} wprowadzony zmodyfikowany scentrowany operator maksymalny $\MM_3$. U{\' s}ci{\' s}laj{\k a}c, modyfikacja polega na tym, {\. z}e w~{\' s}rednich miara $B$ jest zast{\k a}piona przez miar{\k e} $3B$. Kluczowy jest tutaj fakt, {\. z}e $\MM_3$ z~jednej strony mo{\. z}e cz{\k e}sto zast{\k e}powa{\' c} $\MM$ w~zastosowaniach, a z~drugiej ma znacznie lepsze w{\l}asno{\' s}ci w~og{\' o}lnym kontek{\' s}cie.

W p{\' o}{\' z}niejszych artyku{\l}ach \cite{Sa, St1, Te} badana by{\l}a nier{\' o}wno{\' s}{\' c} s{\l}abego typu $(1,1)$ dla rodzin zmodyfikowanych operator{\' o}w maksymalnych, $\{\MM_\kappa : \kappa \in [1, \infty)\}$ oraz $\{ \MN_\kappa : \kappa \in [1, \infty)\}$. W~rezultacie okaza{\l}o si{\k e}, {\. z}e w~przypadku dowolnej przestrzeni $\XX$, takiej {\. z}e miara ka{\. z}dej kuli jest sko{\' n}czona, stowarzyszone operatory $\MM_\kappa$ oraz $\MN_\kappa$ s{\k a} s{\l}abego typu $(1,1)$ odpowiednio dla $\kappa \in [2,\infty)$ oraz $\kappa \in [3, \infty)$. Ponadto, otrzymane zakresy s{\k a} ostre, jako {\. z}e zosta{\l}o r{\' o}wnie{\. z} pokazane, {\. z}e istniej{\k a} przestrzenie metryczno-miarowe, takie {\. z}e $\MM_\kappa$ (odpowiednio, $\MN_\kappa$) nie jest s{\l}abego typu $(1,1)$ dla ka{\. z}dego $\kappa \in [1,2)$ (odpowiednio, dla ka{\. z}dego $\kappa \in [1,3)$). Odpowiednie przyk{\l}ady mo{\. z}na znale{\' z}{\' c} w~pracach \cite{Sa, St2} (warto zapozna{\' c} si{\k e} r{\' o}wnie{\. z} z~artyku{\l}em \cite{SS}, w~kt{\' o}rym zosta{\l}y zawarte pewne detale uzasadniaj{\k a}ce poprawno{\' s}{\' c} konstrukcji opisanej w~\cite{Sa}).

Nieco inn{\k a} ga{\l}{\k a}{\' z} w~badaniu operator{\' o}w maksymalnych wytyczy{\l} wcze{\' s}niej wspomniany artyku{\l} Aldaza \cite{Al2}. Mianowicie, zainicjowa{\l} on program szukania przestrzeni, kt{\' o}re s{\k a} specyficzne z~punktu widzenia w{\l}asno{\' s}ci stowarzyszonych operator{\' o}w maksymalnych. H.-Q. Li napisa{\l} seri{\k e} prac \cite{Li1, Li2, Li3}, w~kt{\' o}rych w~tym w{\l}a{\' s}nie celu wprowadzone i~badane by{\l}y tak zwane przestrzenie kolczaste. Przyk{\l}adowo, w~\cite{Li2} pokazane zosta{\l}o, {\. z}e przy dowolnym ustalonym $p_0 \in (1, \infty)$ istnieje przestrze{\' n} $\XX$, dla kt{\' o}rej stowarzyszony operator $\MM$ jest mocnego typu $(p,p)$ wtedy i~tylko wtedy, gdy $p \in (p_0, \infty]$. Mimo {\. z}e H.-Q. Li bada{\l} g{\l}{\' o}wnie nier{\' o}wno{\' s}ci mocnego typu $(p,p)$, nie s{\k a} to jedyne interesuj{\k a}ce w~tym kontek{\' s}cie. Mi{\k e}dzy innymi, jako {\. z}e znajduje to uzasadnienie w~teorii interpolacji, nier{\' o}wno{\' s}ci s{\l}abego lub restrykcyjnie s{\l}abego typu $(p,p)$ (to znaczy, ograniczono{\' s}{\' c} z~$L^{p}(\XX)$ do $L^{p,\infty}(\XX)$ lub z~$L^{p,1}(\XX)$ do $L^{p,\infty}(\XX)$, odpowiednio) dla operator{\' o}w maksymalnych mog{\l}yby r{\' o}wnie{\. z} by{\' c} przedmiotem dyskusji.

Przypomnijmy, {\. z}e wcze{\' s}niej wymienione przestrzenie $L^p(\XX)$, $L^{p,\infty}(\XX)$ oraz $L^{p,1}(\XX)$ mo{\. z}na ulokowa{\' c} na skali przestrzeni Lorentza $L^{p,q}(\XX)$. Zatem naturalny kierunek, w~kt{\' o}rym mo{\. z}na by rozszerzy{\' c} teori{\k e}, wyznaczaj{\k a} pytania o~ograniczono{\' s}{\' c} operator{\' o}w maksymalnych pomi{\k e}dzy r{\' o}{\. z}nymi przestrzeniami Lorentza. W~przypadku $\RR^d$ i~klasycznych przestrzeni Lorentza pewne wyniki, pozwalaj{\k a}ce opisa{\' c} dzia{\l}anie operator{\' o}w w~spos{\' o}b ilo{\' s}ciowy, mo{\. z}na znale{\' z}{\' c} w~pracach \cite{AM, Sr}. Mimo to, uwzgl{\k e}dniaj{\k a}c obecny stan wiedzy autora, najprawdopodobniej nie ma jak dot{\k a}d w~literaturze przyk{\l}ad{\' o}w, kt{\' o}re ukazywa{\l}yby r{\' o}{\. z}ne szczeg{\' o}lne rodzaje zachowa{\' n} operator{\' o}w maksymalnych w~obr{\k e}bie tego zagadnienia.

Celem niniejszej rozprawy jest badanie w{\l}asno{\' s}ci operator{\' o}w maksymalnych stowarzyszonych z~przestrzeniami bez warunku dublowania, mi{\k e}dzy innymi z~uwzgl{\k e}dnieniem niekt{\' o}rych z~aspekt{\' o}w wspomnianych powy{\. z}ej. W~szczeg{\' o}lno{\' s}ci zauwa{\. z}amy brak wielu istotnych w{\l}asno{\' s}ci dobrze znanych z~przypadku dubluj{\k a}cego, a tak{\. z}e demonstrujemy r{\' o}{\. z}ne fenomeny, kt{\' o}re maj{\k a} szans{\k e} zaistnie{\' c} jedynie w~pewnych bardzo szczeg{\' o}lnych niedubluj{\k a}cych warunkach. Aby to zrobi{\' c}, wprowadzamy odpowiednie klasy przestrzeni, kt{\' o}re pozwol{\k a} wygenerowa{\' c} wiele interesuj{\k a}cych przyk{\l}ad{\' o}w.

Pierwszy rozdzia{\l} pracy stanowi wst{\k e}p. Zawarto{\' s}{\' c} kolejnych rozdzia{\l}{\' o}w przedstawiamy pokr{\' o}tce poni{\. z}ej.

W rozdziale drugim badamy nier{\' o}wno{\' s}ci mocnego, s{\l}abego oraz restrykcyjnie s{\l}abego typu $(p,p)$ dla operator{\' o}w maksymalnych scentrowanego i~niescentrowanego jednocze{\' s}nie. Naszym celem jest odniesienie si{\k e} do pytania 
\[
\emph{Dla jakich zakres{\' o}w $p$ operatory $\MM$ oraz $\MN$ zachowuj{\k a} wy{\. z}ej wspomniane typy nier{\' o}wno{\' s}ci?}
\]
poprzez podanie pe{\l}nej charakteryzacji mo{\. z}liwych sytuacji. Innymi s{\l}owy, rozwa{\. z}amy sze{\' s}{\' c} zbior{\' o}w warto{\' s}ci parametru, odpowiadaj{\k a}cych konkretnemu typowi nier{\' o}wno{\' s}ci dla konkretnego operatora, a nast{\k e}pnie opisujemy wszystkie mo{\. z}liwe ich konfiguracje. Ka{\. z}da dopuszczalna sytuacja jest zilustrowana odpowiednim przyk{\l}adem przestrzeni. W ten spos{\' o}b uzupe{\l}niamy i~wzmacniamy tutaj rezultaty otrzymane przez H.-Q. Li w~\cite{Li1, Li2, Li3}.

Rozdzia{\l} trzeci jest po{\' s}wi{\k e}cony badaniu nier{\' o}wno{\' s}ci mocnego oraz s{\l}abego typu $(p,p)$ dla zmodyfikowanych operator{\' o}w $\MM_\kappa$ oraz $\MN_\kappa$. Tym razem, przy zadanym $\kappa \in [1, \infty)$, mamy cztery zbiory warto{\' s}ci parametru $p$ oraz, podobnie jak poprzednio, analizowane s{\k a} wszystkie wyst{\k e}puj{\k a}ce mi{\k e}dzy nimi relacje. Jak mo{\. z}na si{\k e} spodziewa{\' c}, analiza rozbija si{\k e} na nast{\k e}puj{\k a}ce trzy przypadki: $\kappa \in [1,2)$, $\kappa \in [2,3)$ oraz $\kappa \in [3,\infty)$. W~ka{\. z}dym z~nich prezentujemy pe{\l}ne spektrum dopuszczalnych konfiguracji. Nast{\k e}pnie podejmujemy si{\k e} analizy bardziej z{\l}o{\. z}onego problemu, dotycz{\k a}cego opisu sytuacji, w~kt{\' o}rych parametr $\kappa$ jest zmienny. Nie podajemy tu wprawdzie twierdzenia, charakteryzuj{\k a}cego wszystkie mo{\. z}liwe relacje mi{\k e}dzy czterema rodzinami zbior{\' o}w, ale dajemy odpowied{\' z} na pokrewne, nieco prostsze pytanie. Przedstawiamy przyk{\l}ady ilustruj{\k a}ce wiele r{\' o}{\. z}nych sytuacji, wskazujemy pewne nieoczywiste fenomeny, a~w~ko{\' n}cu wyja{\' s}niamy, gdzie le{\. z}y trudno{\' s}{\' c}, stoj{\k a}ca na drodze do rozwi{\k a}zania problemu w og{\' o}lniejszej formie.

Rozdzia{\l} czwarty jest punktem kulminacyjnym rozprawy. W~tym miejscu badane s{\k a} w{\l}asno{\' s}ci operator{\' o}w maksymalnych w~kontek{\' s}cie ich dzia{\l}ania na przestrzeniach Lorentza $L^{p,q}(\XX)$. Wprowadzamy odpowiedni{\k a} klas{\k e} przestrzeni metryczno-miarowych w~celu pokazania, {\. z}e nadmienione w{\l}asno{\' s}ci mog{\k a} by{\' c} bardzo specyficzne. To z~kolei wymaga istotnego ulepszenia metod wypracowanych w~rozdzia{\l}ach drugim i~trzecim. W~rezultacie dostajemy szerok{\k a} gam{\k e} przyk{\l}ad{\' o}w ilustruj{\k a}cych wiele bardzo nietypowych sytuacji, dotycz{\k a}cych zachowania operator{\' o}w maksymalnych w~opisanym kontek{\' s}cie. Analiza przebiega w~trzech etapach, w~kt{\' o}rych rozwa{\. z}ane s{\k a} r{\' o}{\. z}ne zagadnienia o~rosn{\k a}cym stopniu trudno{\' s}ci. Dla przejrzysto{\' s}ci skupiamy si{\k e} wy{\l}{\k a}cznie na operatorze scentrowanym.

Rozdzia{\l} pi{\k a}ty jest pierwszym z~dw{\' o}ch rozdzia{\l}{\' o}w uzupe{\l}niaj{\k a}cych, kt{\' o}rych celem jest wzbogacenie uzyskanej wiedzy o~pewne dodatkowe obserwacje. Zbadamy tutaj rodzin{\k e} przestrzeni $\{ {\rm BMO}^p(\XX) : p \in [1, \infty)\}$, wprowadzon{\k a} w~kontek{\' s}cie niedubluj{\k a}cych przestrzeni metryczno-miarowych. Scharakteryzujemy wszystkie dopuszczalne relacje pomi{\k e}dzy tymi przestrzeniami rozumianymi jako zbiory funkcji. Ponownie, odpowiednio dobrana klasa przestrzeni metryczno-miarowych pozwoli na zilustrowanie wyszczeg{\' o}lnionych mo{\. z}liwo{\' s}ci. Pokusimy si{\k e} r{\' o}wnie{\. z} o~poczynienie kilku dalej id{\k a}cych uwag, kt{\' o}re b{\k e}d{\k a} zwi{\k a}zane z~nier{\' o}wno{\' s}ci{\k a} Johna--Nirenberga. Zaznaczmy w~tym miejscu, {\. z}e operatory $\MM$ oraz $\MN$ ani razu nie b{\k e}d{\k a} u{\. z}yte w~niniejszym rozdziale. Mimo to nasze rozwa{\. z}ania s{\k a} na miejscu, jako {\. z}e koncepcja przestrzeni ${\rm BMO}$ sama w~sobie jest bliska zagadnieniom, kt{\' o}re w~naturalny spos{\' o}b pojawiaj{\k a} si{\k e} przy operatorach maksymalnych.

W rozdziale sz{\' o}stym badamy w{\l}asno{\' s}ci dychotomii dla $\MM$ oraz $\MN$, kt{\' o}ra zosta{\l}a zauwa{\. z}ona przez Bennetta, DeVore'a i Sharpleya \cite{BDVS} w~kontek{\' s}cie przestrzeni euklidesowych, a nast{\k e}pnie badana w~bardziej og{\' o}lnych kontekstach w~pracach \cite{AK, FK}. U{\' s}ci{\' s}laj{\k a}c, zosta{\l}o pokazane, {\. z}e dla dowolnej przestrzeni dubluj{\k a}cej $\XX$ oraz $f \in L^1_{\rm loc}(\XX)$ zachodzi nast{\k e}puj{\k a}ca implikacja: je{\' s}li $\MM f(x) < \infty$ dla pewnego $x \in \XX$, to funkcja $\MM f$ jest sko{\' n}czona prawie wsz{\k e}dzie. Okazuje si{\k e}, {\. z}e jest to kolejny aspekt zwi{\k a}zany z~operatorami maksymalnymi, kt{\' o}ry zmienia si{\k e} diametralnie, gdy przechodzimy do analizy przestrzeni niedubluj{\k a}cych. W~rzeczy samej podajemy przyk{\l}ady przestrzeni, dla kt{\' o}rych wspomniana dychotomia nie zachodzi dla jednego wybranego b{\k a}d{\' z} obu operator{\' o}w. Przez wi{\k e}kszo{\' s}{\' c} tego rozdzia{\l}u ograniczamy nasz{\k a} uwag{\k e} do przestrzeni $\RR^d$ oraz $\ZZ^d$ wyposa{\. z}onych w~standardow{\k a} metryk{\k e} euklidesow{\k a} $d_{\rm e}$ lub metryk{\k e} supremum $d_\infty$ oraz r{\' o}{\. z}ne specyficzne miary niedubluj{\k a}ce.

Wreszcie, w~dodatku zamieszczonym po rozdziale sz{\' o}stym prezentujemy elementarny dow{\' o}d twierdzenia interpolacyjnego, kt{\' o}re pojawia si{\k e} w~rozdziale czwartym w~kontek{\' s}cie przestrzeni Lorentza $L^{p,q}(\XX)$. Mimo {\. z}e nie jest to nowy rezultat, ka{\. z}dy z~dotychczas znanych jego dowod{\' o}w, uwzgl{\k e}dniaj{\k a}c obecny stan wiedzy autora, wymaga g{\l}{\k e}bokiej znajomo{\' s}ci teorii interpolacji. Metoda opisana tutaj w~zasadzie nie wykracza daleko poza zastosowanie techniki dobrze znanej ze standardowego dowodu twierdzenia interpolacyjnego Marcinkiewicza.

Wszystkie nowe wyniki przedstawione w~rozdzia{\l}ach od drugiego do sz{\' o}stego mo{\. z}na znale{\' z}{\' c} w~artyku{\l}ach autora \cite{Ko1, Ko2, Ko3, Ko4, Ko5, Ko6, Ko7}. Opisane tu metody oraz konstrukcje w~wi{\k e}kszo{\' s}ci s{\k a} zaczerpni{\k e}te stamt{\k a}d i~nie zawieraj{\k a} {\. z}adnych istotnych zmian. Mimo to jest kilka cz{\k e}{\' s}ci, w~szczeg{\' o}lno{\' s}ci w~rozdzia{\l}ach drugim i~trzecim, kt{\' o}re prezentujemy inaczej, ni{\. z} by{\l}o to robione uprzednio. Zdecydowali{\' s}my si{\k e} na to, poniewa{\. z} z~obecnego punktu widzenia nowe podej{\' s}cie wygl{\k a}da bardziej naturalnie, a przy tym pozwala unikn{\k a}{\' c} wielu technicznych uci{\k a}{\. z}liwo{\' s}ci.  
\mainmatter
\chapter{Introduction}\label{chap1}
\markboth{Chapter 1. Introduction}{Introduction}
\setstretch{1.00}

Maximal operators are objects of great importance in mathematics, especially in harmonic analysis. In short, their main role is to estimate from above values of many other intensively studied operators. This means that the standard way of using them should be somehow related to the property that they are bounded from one function space to another. In fact, there are hundreds of works that use various types of boundedness of maximal operators.

Among the whole family of the aforementioned objects, particular attention is focused on the classical Hardy--Littlewood maximal operators which are introduced in the context of an arbitrary metric measure space $\XX$ and usually appear in the literature in two versions, centered $\MM$ and noncentered $\MN$. The first remark about these operators is that they are bounded on $L^\infty(\XX)$ with constant $1$. To indicate any other properties, one should know more about the structure of $\XX$.

At the beginning, let us say a few words about the classical situation in which $\XX$ is simply $\RR^d$, $d \in \NN$, equipped with Lebesgue measure and the Euclidean metric. One of the most important results obtained in this particular case is that both operators, $\MM$ and $\MN$, are of weak type $(1,1)$ which means that they are bounded from $L^1(\XX)$ to $L^{1, \infty}(\XX)$. This fact has several significant consequences including the Lebesgue differentiation theorem, a famous result in real analysis. Besides, keeping in mind that the operators are sublinear one can use the Marcinkiewicz interpolation theorem to prove their strong type $(p,p)$ estimate (that is, the boundedness on $L^p(\XX)$) for each $p \in (1, \infty)$. 

Further studies in this field are focused, among other things, on determining the best constants in certain inequalities with the maximal function, including the weak type $(1,1)$ inequality in the first place (see \cite{Al1, Me1, Me2}). Also some articles have been devoted to the boundedness properties of $\MM$ and $\MN$ in the context of some function spaces in which the regularity of functions is measured (see \cite{BCHP, HO, Ki, Ku, Ta}). Finally, an important direction of research is to analyze what happens with each particular property of maximal operators when the underlying metric measure space changes.

The standard tools used to show the weak type $(1,1)$ estimate for maximal operators are covering lemmas. At first glance, the possibility of using them depends mainly on the metric properties of a given space. To illustrate this let us mention that in the case of $\RR^d$ with the Euclidean metric a suitable covering argument provides that $\MM$ is of weak type $(1,1)$ in the case of any ``sensible'' measure (one can choose here an arbitrary Radon measure, for example). However, the situation changes significantly if only $\MM$ is replaced by $\MN$. Namely, it is possible to find a measure on $\RR^d$, $d \in \NN \setminus \{1\}$, for which the associated noncentered operator $\MN$ is not of weak type $(1,1)$. In fact, Sj\"ogren \cite{Sj} showed that this is the case for the two-dimensional Gaussian measure ${\rm d} \mu(x,y) = {\rm exp}(-(x^2 + y^2)/2) \, {\rm d} x \, {\rm d}y$ (see also an example given by Aldaz \cite{Al2}). 

The last fact suggests that a potential condition on $\XX$ ensuring that most of the classical theory works should rather take into account both the associated metric and measure. In fact, in the context of arbitrary metric measure spaces, the so-called doubling condition has been extensively used. Roughly speaking, it says that the measure of a given ball $B$ is comparable to the measure of $2B$, the ball concentric with $B$ and of radius two times that of $B$. In addition to many other results, it turned out that for any doubling space $\XX$ the associated operators $\MM$ and $\MN$ are comparable and both satisfy the weak type $(1,1)$ estimate. There were also a few concepts regarding the possibility of replacing the doubling condition with some weaker conditions (see \cite{H}, for example) or even eliminating it at all.

Nazarov, Treil, and Volberg made a great contribution to developing harmonic analysis on arbitrary metric measure spaces. Their famous work \cite{NTV2} contains valuable observations on how to deal with various important problems in this field without having the doubling condition in hand. It is particularly interesting for us that the modified centered maximal operator $\MM_3$ has been introduced there. To be precise, the modification is that the measure of the ball $3B$ instead of $B$ occurs in the averages in the definition. The key observation here is that $\MM_3$ can often be successfully used in place of $\MM$, while it has much better mapping properties in general.  

In the following years, several articles treating the weak type $(1,1)$ inequality appeared in the context of the families of modified maximal operators, $\{\MM_\kappa : \kappa \in [1, \infty) \}$ and $\{ \MN_\kappa : \kappa \in [1, \infty)\}$ (see \cite{Sa, St1, Te}). As a result, it turned out that for any $\XX$ such that the measure of each ball is finite the associated operators $\MM_\kappa$ and $\MN_\kappa$ are of weak type $(1,1)$ for $\kappa \in [2,\infty)$ and $\kappa \in [3, \infty)$, respectively. Moreover, these ranges are sharp as it has also been shown that there exist metric measure spaces such that $\MM_\kappa$ (respectively, $\MN_\kappa$) is not of weak type $(1,1)$ for each $\kappa \in [1,2)$ (respectively, for each $\kappa \in [1,3)$). The examples we mention are given in \cite{Sa, St2} (see also \cite{SS}, where certain details justifying the correctness of the construction described in \cite{Sa} are given).

A slightly different branch in the study of maximal operators was indicated by the previously mentioned work \cite{Al2}. Namely, this article initiated the program of searching spaces for which the mapping properties of the associated maximal operators are very specific. H.-Q. Li wrote a~series of papers (see \cite{Li1, Li2, Li3}) in which the so-called cusp spaces have been introduced for this purpose. For example, in \cite{Li2} it is shown that for each fixed $p_0 \in (1, \infty)$ there exists a~space $\XX$ for which the associated operator $\MM$ is of strong type $(p, p)$ if and only if $p \in (p_0,\infty]$. Although H.-Q. Li studied mostly strong type $(p,p)$ inequalities, they are not the only ones worth exploring here. For example, as it is justified by the possibility of interpolating, weak and restricted weak type $(p,p)$ inequalities (that is, the boundedness from $L^{p}(\XX)$ to $L^{p,\infty}(\XX)$ or from $L^{p,1}(\XX)$ to $L^{p,\infty}(\XX)$, respectively) for maximal operators could also be taken under consideration. 

Recall that the aforementioned spaces $L^p(\XX)$, $L^{p,\infty}(\XX)$, and $L^{p,1}(\XX)$ are located on the scale of Lorentz spaces $L^{p,q}(\XX)$. Thus, the natural way to extend the area of research described in the last paragraph is to study boundedness of maximal operators acting on Lorentz spaces. In the case of $\RR^d$ and the classical Lorentz spaces some results that allow one to describe the mapping properties of maximal operators in a more quantitative way has already been given (see \cite{AM, Sr}). However, to the author's best knowledge, there are no specific examples in the literature showing explicitly various peculiar behaviors of these operators in this context.

The aim of this dissertation is to investigate mapping properties of maximal operators associated with nondoubling spaces including, among others, most of the aspects mentioned above. In particular, we indicate the absence of many important properties which are well known in the doubling case and demonstrate various phenomena that arise only in very specific nondoubling settings. In order to do that we introduce some classes of spaces which provide the opportunity to generate a lot of interesting examples. 

The organization of the dissertation is as follows.

In Chapter~\ref{chap2} we study the strong, weak and restricted weak type $(p,p)$ inequalities for maximal operators, centered and noncentered, simultaneously. Our aim is to address the question
\[
\emph{For what ranges of $p$ the operators $\MM$ and $\MN$ satisfy the three studied types of inequalities?}
\]
by characterizing all cases that actually can happen. In other words, we consider six sets of parameters, each of them corresponding to the particular operator and type of inequality, and describe all possible configurations of them. Each admissible case is illustrated with a suitable example of a nondoubling space. Thus, we complement and strengthen the results obtained by H.-Q. Li in \cite{Li1, Li2, Li3}.

Chapter~\ref{chap3} is devoted to the study of strong and weak type $(p,p)$ inequalities for modified maximal operators $\MM_\kappa$ and $\MN_\kappa$. Now, given $\kappa \in [1, \infty)$, we have four sets of parameters and, just as before, the interrelations between them are investigated. As expected, the analysis breaks into three cases: $\kappa \in [1,2)$, $\kappa \in [2,3)$, and $\kappa \in [3,\infty)$. In each case, we present the full spectrum of possibilities. Next we deal with a much more complex issue regarding the situation of $\kappa$ varying. Although we do not characterize all possible configurations related to the whole family of sets taken into account, we give an answer to a slightly easier question. We provide examples illustrating many different situations, indicate some not obvious phenomena, and explain, more or less, what is the main obstacle making the problem in its most general form not resolved here.

Chapter~\ref{chap4} is the culmination of the dissertation. Here some mapping properties of maximal operators acting on Lorentz spaces $L^{p,q}(\XX)$ are studied. We introduce an appropriate class of metric measure spaces in order to show that these properties can be very peculiar. This, in turn, requires a significant improvement of the tools developed in Chapters~\ref{chap2}~and~\ref{chap3}. As a result, we get a wide range of examples illustrating many highly nontrivial situations regarding possible behaviors of maximal operators in this context. The analysis proceeds in three stages, in which certain increasingly difficult issues are considered. For clarity, we focus our attention only on the centered operator $\MM$. 

Chapter~\ref{chap5} is the first of two chapters that enrich the research described in previous paragraphs with some complementary observations. Here we study the family of spaces $\{ {\rm BMO}^p(\XX) : p \in [1, \infty)\}$, introduced in the context of nondoubling metric measure spaces $\XX$. We characterize all possible relations between these spaces considered as sets of functions. Again, we introduce an appropriate class of metric measure spaces which allows us to illustrate each of the admissible cases with a suitable example. Some further considerations related to the John--Nirenberg inequality are also included. It is worth noting that $\MM$ and $\MN$ do not appear in this chapter. However, the ${\rm BMO}$ concept itself is close to the issues concerning maximal operators.

In Chapter~\ref{chap6} we investigate a dichotomy property for $\MM$ and $\MN$ that was noticed by Bennett, DeVore, and Sharpley \cite{BDVS} in the context of Euclidean spaces, and then was studied more generally in \cite{AK, FK}. Precisely, it was shown that for any doubling space $\XX$ and $f \in L^1_{\rm loc}(\XX)$ the following holds: if $\MM f(x) < \infty$ for some $x \in \XX$, then $\MM f$ is finite almost everywhere. It turns out that this is another aspect related to the maximal functions which changes significantly if nondoubling spaces are considered instead. Indeed, we provide some examples where the dichotomy described above does not occur for each of the two operators. For most of this chapter we restrict our attention to the spaces $\RR^d$ and $\ZZ^d$ equipped with the standard Euclidean metric $d_{\rm e}$ or the supremum metric $d_\infty$ and several specific nondoubling measures. 

Finally, in \hyperref[Appendix]{Appendix} we present an elementary proof of certain interpolation theorem that appears in Chapter~\ref{chap4} in the context of Lorentz spaces $L^{p,q}(\XX)$. Although this result is not new, each of its proofs known so far, to the author's best knowledge, requires a deep understanding of the interpolation theory. The method described here does not go beyond the technique which is used in the standard proof of the Marcinkiewicz interpolation theorem.

All the results stated in the following chapters can be found in the author's articles \cite{Ko1, Ko2, Ko3, Ko4, Ko5, Ko6, Ko7}. Most of the methods and constructions are taken from there without making any significant changes. However, there are some parts, especially in Chapters~\ref{chap2}~and~\ref{chap3}, that we present in a different way than it was originally made. This is because the new approach seems much more natural and allows us to avoid tedious calculations in several places.  

\setstretch{1.05}
\section*{Basic notation} 
\addcontentsline{toc}{section}{Basic notation}

Throughout the thesis we consistently use the notation introduced here. First of all, by a metric measure space $\XX$ we mean a triple $(X, \rho, \mu)$, where $X$ is a nonempty set, $\rho$ is a metric on $X$, and $\mu$ is a nonnegative Borel measure on $X$. Further, $B(x,s) \coloneqq \{ y \in X : \rho(x,y) < s\}$ denotes the open ball in $X$ centered at $x \in X$ and of radius $s \in (0,\infty)$. As long as it is clear from the context which measure is considered, for measurable subsets $E \subset X$ we prefer to write shortly $|E|$ instead of $\mu(E)$. 
In a few places, usually at the beginning of each chapter, some additional assumptions on $\XX$ are specified.

While writing estimates, we use the notation $A_1 \lesssim A_2$ (equivalently, $A_2 \gtrsim A_1$) to indicate that $A_1 \leq CA_2$ with a positive constant $C$ independent of significant quantities (in particular, $A_1 = \infty$ implies that $A_2 = \infty$). We shall write $A_1 \simeq A_2$ if $A_1 \lesssim A_2$ and $A_2 \lesssim A_1$ hold simultaneously.  

For each $n \in \NN $ we use the symbol $[n]$ to denote the set of all positive integers which are not larger than $n$ (that is, $[n] \coloneqq \{1, \dots, n\}$). Occasionally, we also write $[0]$ for the empty set. 

Finally, we present a short list of other symbols that appear frequently in the dissertation:
\begin{align*}
\NN \ - \ & \textrm{the set of positive integers (we use the convection } \NN \coloneqq \{1, 2, \dots \} \textrm{)},  \\
\QQ \ - \ & \textrm{the set of rational numbers}, \\
\RR \ - \ & \textrm{the set of real numbers}, \\
\CCC \ - \ & \textrm{the set of complex numbers}, \\
| \cdot | \ - \ & \textrm{the absolute value function}, \\
\lfloor \, \cdot \, \rfloor \ - \ & \textrm{the floor function}, \\
\lceil \, \cdot \, \rceil \ - \ & \textrm{the ceiling function}, \\
\mathbf{1}_E \ - \ & \textrm{the indicator function of a measurable set } E \subset X, \\
L^p(\XX) \ - \ & \textrm{the Lebesgue space with a parameter } p \in [1,\infty], \\
L^{p,q}(\XX) \ - \ & \textrm{the Lorentz space with parameters } p,q \in [1,\infty], \\
L^1_{\rm loc}(\XX) \ - \ & \textrm{the space of functions integrable on every ball } B \subset X, \\
{\rm BMO}(\XX) \ - \ & \textrm{the space of functions of bounded mean oscillation}, \\
{\rm BMO}^p(\XX) \ - \ & \textrm{the space of functions of bounded mean $p$-oscillation with } p \in [1,\infty), \\
\MM \ - \ & \textrm{the centered Hardy--Litllewood maximal operator}, \\
\MN \ - \ & \textrm{the noncentered Hardy--Litllewood maximal operator}, \\
\MMK \ - \ & \textrm{the modified centered Hardy--Litllewood maximal operator}, \\
\MNK \ - \ & \textrm{the modified noncentered Hardy--Litllewood maximal operator}.
\end{align*}

\chapter{Strong, weak, and restricted weak type}\label{chap2}
\setstretch{1.15}

When dealing with some metric measure space it is usually an important issue to study mapping properties of the associated maximal operators. We know that $\MM$ and $\MN$ are always trivially bounded on $\LINF$. In addition, if the measure is doubling, then they are both of weak type $(1,1)$. The next very important fact is that the Marcinkiewicz interpolation theorem can be applied to these objects. Consequently, if $\MM$ (equivalently, $\MN$) is of weak or strong type $(p_0,p_0)$ for some $p_0 \in [1,\infty)$, then it is bounded on $\LP$ for every $p \in (p_0, \infty]$. Thus, for example, through the interpolation we can deduce that $\MM$ and $\MN$ are bounded on $\LP$ for every $p \in (1, \infty]$ as long as the doubling condition is satisfied. 

On the other hand, there are examples of spaces for which maximal operators are bounded on $\LP$ for every $p \in (1, \infty]$ while they are not of weak type $(1,1)$. 
It is even possible to find a space for which the associated operators $\MM$ and $\MN$ are not of weak type $(p,p)$ for every $p \in [1, \infty)$. All these observations prompt us to study the general question of existence of the weak or strong
type $(p, p)$ inequalities for $\MM$ and $\MN$ and of interrelations between these
properties.

The search for spaces with specific mapping properties of maximal operators
was greatly advanced by H.-Q. Li. In this context, in \cite{Li1, Li2, Li3} he considered a class of the cusp spaces. In \cite{Li1} H.-Q. Li showed that for any fixed $p_0 \in (1,\infty)$ there exists a space for which the associated operator $\MM$ is of strong type $(p,p)$ if and only if $p \in (p_0, \infty]$. Then, in \cite{Li2} examples of spaces were furnished for which $\MN$ is of strong type $(p,p)$ if and only if $p \in (p_0, \infty]$. Moreover, for every $\tau \in (1,2]$ there are examples of spaces for which $\MM$ is of weak type $(1,1)$, and $\MN$ is of strong type $(p,p)$ if and only if $p \in (\tau, \infty]$. Finally, in \cite{Li3} H.-Q. Li showed that there are spaces with exponential volume growth for which $\MM$ is of weak type $(1,1)$, while $\MN$ is of strong type $(p,p)$ for every $p  \in (1, \infty]$. 

Let us note that all previous works focused only on the estimates of weak or strong type. It is well known that the Marcinkiewicz interpolation theorem has a stronger version and to use interpolation one only needs to know that the maximal operator is
of restricted weak type $(p_0, p_0)$ for some $p_0 \in [1,\infty)$ (see \cite[Theorem 3.15, p. 197]{SW}, for example). Therefore, a natural way to go a step further is to take into account the three mentioned types of inequalities in order to relate them to each other. This is what we do in this chapter. 

\newpage

\setstretch{1.17}

\section{Preliminaries and results}\label{S2.1}

By a metric measure space $\XX$ we mean a triple $(X, \rho, \mu)$, where $X$ is a nonempty set, $\rho$ is a metric and $\mu$ is a nonnegative Borel measure. Unless otherwise stated, we additionally assume that the measure of each ball is finite and strictly positive. In this context we define the $\textit{Hardy--Littlewood}$ $\textit{maximal operators}$, centered $\MM$ and noncentered $\MN$, by
\begin{displaymath}
	\MM f(x) \coloneqq \MM_\XX f(x) \coloneqq \sup_{s \in (0,\infty)} \frac{1}{|B(x,s)|} \int_{B(x,s)} |f| \, {\rm d}\mu, \qquad x \in X,
\end{displaymath}
and
\begin{displaymath}
	\MN f(x) \coloneqq \MN_\XX f(x) \coloneqq \sup_{B \ni x} \frac{1}{|B|} \int_B |f| \, {\rm d}\mu , \qquad x \in X,
\end{displaymath}
respectively. Here $B$ refers to any open ball in $(X,\rho)$, while $B(x,s)$ stands for the open ball centered at $x \in X$ with radius $s \in (0,\infty)$. We also require $f$ to belong to the space $L^1_{\rm loc}(\XX)$ which means that $\int_B |f| \, {\rm d} \mu < \infty$ for every $B \subset X$. Finally, let us make it clear that in the case of arbitrary $\XX$ the balls $B$ such that $|B| = 0$ or $|B| = \infty$ are omitted in the definitions of $\MM$ and $\MN$ (in the extreme case we use the convention that the supremum of the empty set is $0$).

We introduce the notation $A_1 \lesssim A_2$ (equivalently, $A_2 \gtrsim A_1$) which means that $A_1 \leq CA_2$ with a positive constant $C$ independent of significant quantities (in particular, $A_1 = \infty$ implies that $A_2 = \infty$). We write $A_1 \simeq A_2$ if $A_1 \lesssim A_2$ and $A_2 \lesssim A_1$ hold simultaneously. 

For each $p \in [1, \infty)$ the space $L^p(\XX)$ consists of all measurable functions $f \colon X \rightarrow \CCC$ such that
\begin{equation*}
\|f \|_p \coloneqq \Big( \int_X |f|^p \, {\rm d}\mu \Big)^{1/p}
\end{equation*}
is finite.   
Similarly, we use the quantity
\begin{equation*}
\|f\|_{p, \infty} \coloneqq \sup_{\lambda \in (0, \infty)} \big\{ \lambda \cdot |E_\lambda(f)|^{1/p} \big\}
\end{equation*}
to introduce the space $L^{p, \infty}(\XX)$ for $p \in [1, \infty)$. Here $E_\lambda(f) \coloneqq \{x \in X : |f(x)| > \lambda\}$ is the level set of $f$. Finally, the space $L^\infty(\mathfrak{X})$ is defined analogously by using
\begin{equation*}
\|f \|_\infty \coloneqq \inf\{C \in [0,\infty) : |f| \leq C \textrm{ almost everywhere}\}.
\end{equation*} 

Accordingly, we say that an operator $\HH$ is \emph{of strong type} $(p,p)$ for some $p \in [1, \infty]$, if it is bounded on $L^p(\XX)$, that is, the inequality $ \| \HH f \|_p \lesssim \| f \|_p$ holds uniformly in $f \in L^p(\XX)$. Similarly, $\HH$ is \emph{of weak type} $(p,p)$ for some $p \in [1, \infty]$ if it is bounded from $L^p(\XX)$ to $L^{p,\infty}(\XX)$ which, in the case $p \in [1, \infty)$, means that the inequality
\begin{equation}\label{eq:2.1.1}
	\lambda \cdot  |E_\lambda(f)|^{1/p} \lesssim \| f \|_p
\end{equation}
holds uniformly in $f \in L^p(\XX)$ and $\lambda \in (0, \infty)$. For $p = \infty$ we use the convention $L^{\infty,\infty} (\XX) = L^\infty(\XX)$ and thus being of weak type $(\infty, \infty)$ is equivalent to being of strong type $(\infty, \infty)$. Finally, $\HH$ is \emph{of restricted weak type} $(p,p)$ for some $p \in [1, \infty)$ if it is bounded from $L^{p,1}(\XX)$ to $L^{p, \infty}(\XX)$. Since $L^{1,1}(\XX) = L^1(\XX)$, being of restricted weak type $(1,1)$ is the same as being of weak type $(1,1)$. In turn, if $p \in (1, \infty)$ and $\HH$ is sublinear and nonnegative (this is always the case for $\MM$ and $\MN$), then being of restricted weak type $(p,p)$ is equivalent to the statement that \eqref{eq:2.1.1} holds uniformly in $f = \mathbf{1}_E$, $|E| < \infty$, and $\lambda \in (0, \infty)$, where $E \subset X$ is measurable and $\mathbf{1}_E$ is its indicator function (see, for example, \cite[Theorem 5.3, p. 231]{BS}). For $p = \infty$ we use the convention that being of restricted weak type $(\infty,\infty)$ is equivalent to the boundedness on $L^\infty(\XX)$ (it would be very strange to consider operators acting on $L^{\infty, 1}(\XX)$, because the only element of $L^{\infty, 1}(\XX)$ is the zero function). It is easy to check that being of strong type $(p,p)$ implies being of weak type $(p,p)$ which in turn implies being of restricted weak type $(p,p)$. We do not give the definition of the space $L^{p,1}(\XX)$ in this chapter since it is not needed at this moment. 

For a fixed metric measure space $\XX$ we denote by $P_{\rm s}^{\rm c}(\XX)$, $P_{\rm w}^{\rm c}(\XX)$, and $P_{\rm r}^{\rm c}(\XX)$ the sets consisting of all parameters $p \in [1, \infty]$ for which $\MM_\XX$ is of strong, weak, or restricted weak type $(p,p)$, respectively. Similarly, let $P_{\rm s}(\XX)$, $P_{\rm w}(\XX)$, and $P_{\rm r}(\XX)$ consist of all parameters $p \in [1, \infty]$ for which $\MN_\XX$ is of strong, weak, or restricted weak type $(p,p)$, respectively. Clearly, the six introduced sets depend on the underlying space $\XX$. However, it is easy to see that there are some relations that must be satisfied by them no matter what the structure of $\XX$ is. Motivated by this, we drop the dependence on $\XX$ in the conditions listed below and write $P_{\rm s}^{\rm c}$, $P_{\rm w}^{\rm c}$, $P_{\rm r}^{\rm c}$, $P_{\rm s}$, $P_{\rm w}$, and $P_{\rm r}$ instead of $P_{\rm s}^{\rm c}(\XX)$, $P_{\rm w}^{\rm c}(\XX)$, $P_{\rm r}^{\rm c}(\XX)$, $P_{\rm s}(\XX)$, $P_{\rm w}(\XX)$, and $P_{\rm r}(\XX)$, respectively. 

\begin{observation}
	The following assertions hold for arbitrary metric measure space $\XX$:   
	
	\begin{enumerate}[label=\rm(\roman*)]
		\item \label{2i} Each of the sets $P_{\rm s}^{\rm c}$, $P_{\rm w}^{\rm c}$, $P_{\rm r}^{\rm c}$, $P_{\rm s}$, $P_{\rm w}$, and $P_{\rm r}$ is of the form $\{\infty\}$, $[p_0, \infty]$, or $(p_0,\infty]$ for some $p_0 \in [1, \infty)$.\smallskip
		
		\item \label{2ii} We have the following inclusions:
		\begin{displaymath}
		P_s \subset P_s^c, \quad P_w \subset P_w^c, \quad P_r \subset P_r^c, \quad P_s^c \subset P_w^c \subset  P_r^c \subset \overline{P_s^c}, \quad P_s \subset P_w \subset P_r \subset \overline{P_s},
		\end{displaymath}
		where $\overline{E}$ denotes the closure of $E$ in the usual topology of $\mathbb{R} \cup \{ \infty \}$.
		\item \label{2iii} We have the following implications:
		\begin{displaymath}
		P_r^c = [1, \infty] \implies P_w^c = [1, \infty], \quad P_r = [1, \infty] \implies P_w = [1, \infty].
		\end{displaymath}
	\end{enumerate}
\end{observation}

\noindent Indeed, the condition~\ref{2i} follows from the Marcinkiewicz interpolation theorem, while the condition~\ref{2ii} is a consequence of both the Marcinkiewicz interpolation theorem and the implications between different types of inequalities mentioned above. Finally, the condition~\ref{2iii} must be satisfied in view of the identity $L^{1,1}(\XX)= L^1(\XX)$.   

Our goal is to show that \ref{2i}--\ref{2iii} are the only conditions that the six sets considered above satisfy in general. Namely, we will prove the following theorem.

\begin{theorem}\label{thm:2.1.2}
	Let $P_{\rm s}^{\rm c}$, $P_{\rm w}^{\rm c}$, $P_{\rm r}^{\rm c}$, $P_{\rm s}$, $P_{\rm w}$, and $P_{\rm r}$ be arbitrary sets satisfying \ref{2i}--\ref{2iii}. Then there exists a (nondoubling) metric measure space $\ZZZ$ for which the associated Hardy--Littlewood maximal operators, centered $\MM_\ZZZ$ and noncentered $\MN_\ZZZ$, satisfy the following properties:
	\begin{itemize}
		\item $\MM_\ZZZ$ is of strong type $(p,p)$ if and only if $p \in P_{\rm s}^{\rm c}$,
		\item $\MM_\ZZZ$ is of weak type $(p,p)$ if and only if $p \in P_{\rm w}^{\rm c}$,
		\item $\MM_\ZZZ$ is of restricted weak type $(p,p)$ if and only if $p \in P_{\rm r}^{\rm c}$,
		\item $\MN_\ZZZ$ is of strong type $(p,p)$ if and only if $p \in P_{\rm s}$,
		\item $\MN_\ZZZ$ is of weak type $(p,p)$ if and only if $p \in P_{\rm w}$,
		\item $\MN_\ZZZ$ is of restricted weak type $(p,p)$ if and only if $p \in P_{\rm r}$.
	\end{itemize} 
\end{theorem}

The rest of this chapter is organized as follows. In Section~\ref{S2.2} we introduce the {\it space combining technique}. This is the main tool used to provide the desired examples of spaces. Section~\ref{S2.3} is devoted to the study of some simple structures, the so-called {\it first and second generation spaces}. Finally, the proof of Theorem~\ref{thm:2.1.2} is located in Section~\ref{S2.4}. Since $\MM_\XX f \equiv \MM_\XX |f|$, $\MN f \equiv \MN |f|$, and $\| f \|_p = \| |f| \|_p$ hold, from now on we shall deal mostly with nonnegative functions.    

\section{Space combining technique}\label{S2.2} 

As a starting point of our considerations we explain a specific technique of combining different metric measure spaces which will often be used later on. Let $\Lambda \neq \emptyset$ be a (finite or not) set of positive integers and for each $n \in \Lambda$ consider a metric measure space $\YY_n = (Y_n, \rho_n, \mu_n)$ such that $\mu_n(Y_n) < \infty$ and
\begin{displaymath}
{\rm diam}(Y_n) \coloneqq {\rm diam}_{\rho_n}(Y_n) \coloneqq \sup\{\rho_n(x,y) : x, y \in Y_n\}
\end{displaymath}
is finite. We introduce $\rho_n'$ and $\mu_n'$ by rescaling (if necessary) $\rho_n$ and $\mu_n$, respectively, in such a~way that ${\rm diam}_{\rho_n'}(Y_n) \leq 1$ and $\mu_n'(Y_n) \leq 2^{-n}$. Then, assuming that $Y_{n_1} \cap Y_{n_2} = \emptyset$ for any $n_1 \neq n_2$, we construct $\YY =  (Y, \rho, \mu)$ as follows. Set $Y \coloneqq \bigcup_{n \in \Lambda} Y_n$. Define the metric $\rho$ on $Y$ by
\begin{equation}\label{2.2.1}
\rho(x,y) \coloneqq \left\{ \begin{array}{rl}
\rho_n'(x,y) & \textrm{if }  \{x,y\} \subset Y_n \textrm{ for some } n \in \Lambda,   \\
2 & \textrm{otherwise,} \end{array} \right. 
\end{equation} 
and the measure $\mu$ on $Y$ by
\begin{displaymath}
\mu(E) \coloneqq \sum_{n \in \Lambda} \mu_n'(E \cap Y_n), \qquad E \subset Y.
\end{displaymath}

Next, for a given space $\XX$ and $p \in [1, \infty]$ let us denote by $\cc_{\rm s}^{\rm c}(p, \XX)$ the norm of $\MM_\XX$ considered as an operator on $L^p(\XX)$. Thus, in other words, $\cc_{\rm s}^{\rm c}(p, \XX)$ is the smallest constant $C$ such that 
\begin{displaymath}
\| \MM_\XX f \|_p \leq C \|f\|_p, \qquad f \in L^p(\XX),
\end{displaymath}
holds. If $\MM_\XX$ is not of strong type $(p,p)$, then we write $\cc_{\rm s}^{\rm c}(p, \XX) = \infty$. Similarly, let $\cc_{\rm w}^{\rm c}(p, \XX)$ and $\cc_{\rm r}^{\rm c}(p, \XX)$ be the best constants $C$ in the inequalities
\begin{displaymath}
\| \MM_\XX f \|_{p, \infty} \leq C \|f\|_p, \qquad f \in L^p(\XX),
\end{displaymath}
and 
\begin{displaymath}
\| \MM_\XX \mathbf{1}_E \|_{p, \infty} \leq C \| \mathbf{1}_E \|_p, \qquad E \subset X, \, |E| < \infty,
\end{displaymath}
respectively (with the aforementioned modifications for $p \in \{1,\infty\}$). Finally, we define $\cc_{\rm s}(p, \XX)$, $\cc_{\rm w}(p, \XX)$, and $\cc_{\rm r}(p, \XX)$ as before replacing $\MM_\XX$ with $\MN_\XX$.

In the following proposition we describe some relations between mapping properties of the maximal operators associated with $\YY$ and $\YY_n$, $n \in \Lambda$, in terms of the quantities defined above. 
\begin{proposition}\label{P2.2.1}
	Fix $\Lambda \subset \NN$, $\Lambda \neq \emptyset$, and for each $n \in \Lambda$ let $\YY_n = (Y_n, \rho_n, \mu_n)$ be a given metric measure space satisfying $\mu_n(Y_n) < \infty$ and ${\rm diam}(Y_n) < \infty$. Define $\YY$ as before. Then for each $p \in [1, \infty]$ we have the following estimates:
\begin{align*}
\cc_{\rm s}^{\rm c}(p, \YY) \simeq \sup_{n \in \Lambda} \cc_{\rm s}^{\rm c}(p, \YY_n), \quad
\cc_{\rm w}^{\rm c}(p, \YY) \simeq \sup_{n \in \Lambda} \cc_{\rm w}^{\rm c}(p, \YY_n), \quad
\cc_{\rm r}^{\rm c}(p, \YY) \simeq \sup_{n \in \Lambda} \cc_{\rm r}^{\rm c}(p, \YY_n), \\
\cc_{\rm s}(p, \YY) \simeq \sup_{n \in \Lambda} \cc_{\rm s}(p, \YY_n), \quad
\cc_{\rm w}(p, \YY) \simeq \sup_{n \in \Lambda} \cc_{\rm w}(p, \YY_n), \quad
\cc_{\rm r}(p, \YY) \simeq \sup_{n \in \Lambda} \cc_{\rm r}(p, \YY_n). 
\end{align*} 
\end{proposition}

\begin{proof}
	First, notice that the process of rescaling metrics and measures used in the construction of $\YY$ does not affect the studied mapping properties of the associated maximal operators $\MM_{\YY_n}$ and $\MN_{\YY_n}$, $n \in \Lambda$. Thus, without any loss of generality, we may simply assume that the spaces $\YY_n$ are the rescaled ones, that is, ${\rm diam}_{\rho_n}(Y_n) \leq 1$ and $\mu_n(Y_n) \leq 2^{-n}$.

	Fix $p \in [1, \infty)$ (we omit the trivial case $p = \infty$). For clarity, we shall prove only the first equivalence and the remaining ones may be verified similarly. Let $n \in \Lambda$ and take $f \in L^p(\YY_n)$. We extend $f$ to $F \in L^p(\YY)$ such that $F(y) = 0$ for $ y \in Y \setminus Y_n$. Notice that $\|F\|_p = \|f\|_p$ (here the symbol $\| \cdot \|_p$ refers to function spaces over different measure spaces) and $\MM_\YY F(y) = \MM_{\YY_n} f (y)$ for any $y \in Y_n$. Hence, $\|\MM_{\YY_n} f \|_p / \|f\|_p \leq \|\MM_\YY F\|_p / \|F\|_p$ and we conclude that 
	\begin{displaymath} 
	\cc_{\rm s}^{\rm c}(p, \YY) \geq \sup_{n \in \Lambda} \cc_{\rm s}^{\rm c}(p, \YY_n).
	\end{displaymath}
	
	Now we take $f \in L^p(\YY)$ and define $f_n \in L^p(\YY_n)$, $n \in \Lambda$, by restricting $f$ to $Y_n$. Then 
	\begin{displaymath}
	\MM_\YY f(y) = \max \{ \MM_{\YY_n} f_n(y), \|f\|_1 / \mu(Y)\}
	\end{displaymath}
	holds for each $y \in Y_n$. Thus, applying H\"older's inequality, we estimate $\|\MM_\YY f\|_p^p$ by
	\begin{align*}
	\sum_{n \in \Lambda} \|\MM_{\YY_n} f_n\|_p^p + \|f\|_1^p \cdot \mu(Y)^{1-p} \leq \sum_{n \in \Lambda} \cc_{\rm s}^{\rm c}(p, \YY_n)^p \|f_n\|_p^p + \|f\|_p^p \leq \big( \sup_{n \in \Lambda} \cc_{\rm s}^{\rm c}(p, \YY_n)^p + 1 \big) \|f\|_p^p.
	\end{align*}
	Moreover, by taking $g \coloneqq \mathbf{1}_{Y_n}$ one can easily show that $\cc_{\rm s}^{\rm c}(p, \YY_n)^p \geq 1$ for each $n \in \Lambda$. Hence,
	\begin{displaymath}
	\cc_{\rm s}^{\rm c}(p, \YY) \leq \big( \sup_{n \in \Lambda} \cc_{\rm s}^{\rm c}(p, \YY_n)^p + 1 \big)^{1/p} \leq 2^{1/p}  \sup_{n \in \Lambda} \cc_{\rm s}^{\rm c}(p, \YY_n) \leq 2 \sup_{n \in \Lambda} \cc_{\rm s}^{\rm c}(p, \YY_n)
	\end{displaymath}
	and, consequently, we get that $\cc_{\rm s}^{\rm c}(p, \YY) \simeq \sup_{n \in \Lambda} \cc_{\rm s}^{\rm c}(p, \YY_n)$.
\end{proof}

\begin{remark}\label{R2}
If at least one space from the family $\{\YY_n : n \in \Lambda\}$ is nondoubling or $\Lambda$ is infinite, then $\YY$ is nondoubling.	
\end{remark} 

\section{Test spaces}\label{S2.3}

In the following section we introduce and analyze auxiliary structures called first and second generation spaces. We emphasize here that each of these spaces can be viewed as the space $\YY$ from Section~\ref{S2.2} constructed with the aid of some family of spaces $\{\YY_n : n \in \NN\}$. Moreover, since $\YY$ satisfies $|Y| < \infty$ and ${\rm diam}(Y) < \infty$, any first or second generation space may also be used as a component space in Proposition~\ref{P2.2.1}.

\subsection{First generation spaces}

We begin with a description of the first generation spaces which will be denoted by $\SSS$. The common property of these spaces is that the associated operators $\MM_\SSS$ and $\MN_\SSS$ behave very similarly to each other. Namely, for each such space the identities $\PSC(\SSS) = \PS(\SSS)$, $\PWC(\SSS) = \PW(\SSS)$, and $\PRC(\SSS) = \PR(\SSS)$ hold. There will be three subtypes of spaces specified in this section. Now we present a construction which will be applied to the first two of them.  

Let $\tau$ be a fixed positive integer. Set $S \coloneqq S(\tau) \coloneqq \{x_0, x_1, \dots, x_\tau\}$,
where all elements are different (and located on the plane, say). We define the metric $\rho$ determining the distance between two different elements $x$ and $y$ by the formula
\begin{displaymath}
\rho(x,y) \coloneqq \rho_\tau(x,y) \coloneqq \left\{ \begin{array}{rl}
1 & \textrm{if } x_0 \in \{x,y\},  \\
2 & \textrm{otherwise.} \end{array} \right. 
\end{displaymath}
Figure~\ref{F2.1} shows a model of the space $(S, \rho)$. The solid line between two points indicates that the distance between them equals $1$. Otherwise the distance equals $2$.  
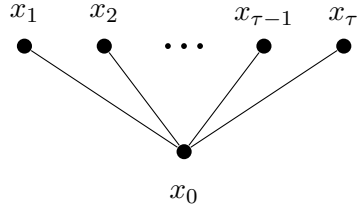
\begin{figure}[H]
	\begin{center}
	\begin{tikzpicture}
	[scale=.7,auto=left,every node/.style={circle,fill,inner sep=2pt}]
	
	\node[label={[yshift=-1cm]$x_0$}] (m0) at (8,1) {};
	\node[label=$x_{1}$] (m1) at (5,3)  {};
	\node[label=$x_{2}$] (m2) at (6.5,3)  {};
	\node[label={[yshift=-0.23cm]$x_{\tau-1}$}] (m3) at (9.5,3)  {};
	\node[label={[yshift=-0.04cm]$x_{\tau}$}] (m4) at (11,3)  {};
	\node[dots, scale=2] (m5) at (8,3)  {...};
	
	\foreach \from/\to in {m0/m1, m0/m2, m0/m3, m0/m4}
	\draw (\from) -- (\to);
	\end{tikzpicture}
	\caption{The first generation space $(S, \rho)$.}
	\label{F2.1}
	\end{center}
\end{figure}
\noindent Note that we can explicitly describe any ball:
\begin{displaymath}
B(x_0,s) = \left\{ \begin{array}{rl}
\{x_0\} & \textrm{for } 0 < s \leq 1, \\
S & \textrm{for } 1 < s, \end{array} \right.
\end{displaymath} 
\noindent and, for $i \in [\tau]$,
\begin{displaymath}
B(x_i,s) = \left\{ \begin{array}{rl}
\{x_i\} & \textrm{for } 0 < s \leq 1, \\
\{x_0, x_i\} & \textrm{for } 1 < s \leq 2,  \\
S & \textrm{for } 2 < s. \end{array} \right.
\end{displaymath} 
Next we define the measure $\mu$ on $S$ by letting 
\begin{displaymath}
\mu(\{x_{i}\}) \coloneqq \mu_{\tau, F}(\{x_{i}\}) \coloneqq 
\left\{ \begin{array}{rl}
1 & \textrm{if } i=0, \\
F(i) & \textrm{if } i \in [\tau], \end{array} \right.
\end{displaymath}
where $F > 0$ is a given function. 

Given $f \geq 0$ and $E \subset S$, $E \neq \emptyset$, we denote the average value of $f$ on $E$ by 
\begin{displaymath}
A_E(f) \coloneqq \frac{1}{|E|}\sum_{x \in E} f(x) |\{x\}| 
\end{displaymath}
(for arbitrary $\XX$ and $|E| \in (0, \infty)$ we denote similarly $A_E(f) \coloneqq \frac{1}{|E|} \int_E f \, {\rm d}\mu$). It is instructive to notice that the process of averaging does not increase the Lebesgue norm. For $p = \infty$ this is obvious, while for $p = [1, \infty)$ we use H\"older's inequality to see that $A_E(f)^p |E| \leq \sum_{x \in E} f(x)^p |\{x\}|$. 

We are ready to describe the first two subtypes of the first generation spaces. \\   

\noindent {\bf First subtype.} Now we construct and investigate a subclass of the first generation spaces $\SSS$ for which the identities $\PSC(\SSS) = \PS(\SSS)$, $\PWC(\SSS) = \PW(\SSS)$, and $\PRC(\SSS) = \PR(\SSS)$ hold, and, in addition, there are no differences between the incidences of strong, weak, or restricted weak type inequalities, by what we mean that $\PSC(\SSS) = \PWC(\SSS) = \PRC(\SSS)$ and $\PS(\SSS) = \PW(\SSS) = \PR(\SSS)$. Of course, combining all these identities, we obtain that for each such space all the six sets coincide. In the first step, for any fixed $p_0 \in [1, \infty] $ we construct a space denoted by ${\SSS}_{1,p_0}$ for which 
\begin{displaymath}
\PSC({\SSS}_{1,p_0}) = \PWC({\SSS}_{1,p_0}) = \PRC({\SSS}_{1,p_0}) = \PS({\SSS}_{1,p_0}) = \PW({\SSS}_{1,p_0}) = \PR({\SSS}_{1,p_0}) = [p_0, \infty]
\end{displaymath}
(here by $[\infty, \infty]$ we mean $\{\infty\}$). Then, after slight modifications, for any fixed $p_0 \in [1, \infty)$ we get a space ${\SSS}_{1,p_0}'$ for which 
\begin{displaymath}
\PSC({\SSS}_{1,p_0}') = \PWC({\SSS}_{1,p_0}') = \PRC({\SSS}_{1,p_0}') = \PS({\SSS}_{1,p_0}') = \PW({\SSS}_{1,p_0}') = \PR({\SSS}_{1,p_0}') = (p_0, \infty].
\end{displaymath}

Fix $p_0 \in [1, \infty]$ and for any $n \in \NN$ consider $\SSS_n = (S_n, \rho_n, \mu_n)$, where $S_n$, $\rho_n$, and $\mu_n$ are defined as before with the aid of $\tau_n =  \lfloor(n+1)^{p_0} /n  \rfloor$ in the case $p_0 \in [1, \infty)$, or $\tau_n = 2^n$ in the case $p_0 = \infty$, and $F_n(i) = n$ for each $i \in [\tau_n]$. We denote by ${\SSS}_{1,p_0}$ the space $\YY$ obtained by using Proposition~\ref{P2.2.1} for $\Lambda = \NN$ with $\YY_n = \SSS_n$ for each $n \in \NN$.

In the following lemma we describe the properties of the associated maximal operators. The key point here is that we have: in the case $p_0 \in (1, \infty]$,
\begin{displaymath}
\lim_{n \rightarrow \infty} \frac{n \tau_n}{(n+1)^p} = \infty, \qquad p \in [1, p_0), 
\end{displaymath} 
and, in the case $p_0 \in [1, \infty)$,
\begin{displaymath}
\frac{n \tau_n}{(n+1)^{p_0}} \leq 1, \qquad n \in \mathbb{N}.
\end{displaymath}

\begin{lemma}\label{L2.3.1}
	Fix $p_0 \in [1, \infty]$ and let ${\SSS}_{1,p_0}$ be the first generation space defined above. Then the associated maximal operators, centered $\MM_{{\SSS}_{1,p_0}}$ and noncentered $\MN_{{\SSS}_{1,p_0}}$, are not of restricted weak type $(p,p)$ for $p \in [1,p_0)$, but are of strong type $(p,p)$ for $p \in [p_0,\infty]$.
\end{lemma}
\begin{proof}
	It suffices to prove that $\MM_{{\SSS}_{1,p_0}}$ fails to be of restricted weak type $(p,p)$ for $p \in [1,p_0)$ and $\MN_{{\SSS}_{1,p_0}}$ is of strong type $(p_0,p_0)$. First we show that $\MM_{{\SSS}_{1,p_0}}$ is not of restricted weak type $(p,p)$ for $p \in [1,p_0)$. Fix $p_0 \in (1, \infty]$ (for $p_0 = 1$ the condition to check is empty) and let $p \in [1, p_0)$. We take $n \in \NN$ and restrict our attention to $\SSS_n$. Let $g \coloneqq \mathbf{1}_{\{x_0\}} \in L^p(\SSS_n)$. Then $\|g\|_p^p = 1$ and $\MM_{\SSS_n} g(x_{i}) \geq \frac{1}{n+1}$ for each $i \in [\tau_n]$. Thus, $|E_{1/(2(n+1))}(\MM_{\SSS_n} g)| \geq n \tau_n$ which implies
	\begin{displaymath}
	\cc_{\rm r}^{\rm c}(p, \SSS_n)^p \geq \frac{\|\MM_{\SSS_n} g \|_{p,\infty}^p}{\|g \|_p^p} \geq \frac{n \tau_n}{(2(n+1))^p}.
	\end{displaymath}
	Consequently, we obtain
	\begin{displaymath}
	\cc_{\rm r}^{\rm c}(p, {\SSS}_{1,p_0}) \gtrsim \limsup_{n \rightarrow \infty} \cc_{\rm r}^{\rm c}(p, \SSS_n) \geq \lim_{n \rightarrow \infty} \frac{(n \tau_n)^{1/p}}{2(n+1)} = \infty
	\end{displaymath}
	which means that $\MM_{{\SSS}_{1,p_0}}$ is not of restricted weak type $(p,p)$.
	
	In the next step we show that $\MN_{{\SSS}_{1,p_0}}$ is of strong type $(p_0,p_0)$. Fix $p_0 \in [1, \infty)$ (for $p_0 = \infty$ the condition to check is trivial). Again, we take $n \in \NN$ and restrict our attention to $\SSS_n$. Let $f \in L^{p_0}(\SSS_n)$. As mentioned at the end of Subsection~\ref{S2.1}, in view of $\| f \|_{p_0} = \| |f| \|_{p_0}$ and $\MN_{\SSS_n} f \equiv \MN_{\SSS_n} |f|$ we may assume that $f \geq 0$ (this assumption will often be made later on without any further explanation). Denote $\mathcal{D} \coloneqq \{\{x_0, x_{i}\} : i \in [\tau_n]\}$. We use the estimate
	\begin{displaymath}
	\|\MN_{\SSS_n} f\|_{p_0}^{p_0} \leq \sum_{B \subset S_n} \sum_{x \in B} A_B(f)^{p_0} |\{x\}| = \sum_{B \subset S_n} A_B(f)^{p_0} |B|. 
	\end{displaymath}
	Note that each $x \in S_n$ belongs to exactly two different balls which are not elements of $\mathcal{D}$, namely $\{x\}$ and $S_n$. Combining this observation with H\"older's inequality, we obtain
	\begin{displaymath}
	\sum_{B \notin \mathcal{D}}A_B(f)^{p_0} |B| \leq \sum_{B \notin \mathcal{D}} \sum_{x \in B} f(x)^{p_0} |\{x\}| \leq 2 \|f\|_{p_0}^{p_0}. 
	\end{displaymath}
	Therefore,
	\begin{equation} \label{2.3.1}
	\|\MN_{\SSS_n} f\|_{p_0}^{p_0} \leq 2 \|f\|_{p_0}^{p_0} + \sum_{i=1}^{\tau_n} \Big(\frac{f(x_0)+ n f(x_{i})}{n+1}\Big)^{p_0} |\{x_0, x_{i}\}|.
	\end{equation}
	By H\"older's inequality $(f(x_0)+nf(x_{i}))^{p_0} \leq 2^{p_0-1} \big( f(x_0)^{p_0} + (n f(x_{i}))^{p_0} \big)$ and combining this with $|\{x_0, x_{i}\}| \leq 2|\{x_i\}| = 2n |\{x_0\}|$ we see that the sum in \eqref{2.3.1} is controlled by
	\begin{displaymath}
	2^{p_0} \Big(\frac{n \tau_n }{(n+1)^{p_0}} \, f(x_0)^{p_0}  |\{x_0\}| +  \sum_{i=1}^{\tau_n} \Big(\frac{n}{n+1}\Big)^{p_0}  f(x_{i})^{p_0}  |\{x_{i}\}|   \Big) \leq 2^{p_0} \|f\|_{p_0}^{p_0}.
	\end{displaymath}
	Thus, we obtain $\cc_{\rm s}(p_0, \SSS_n)^{p_0} \leq 2 + 2^{p_0}$ and, as a result,
	$
	\cc_{\rm s}(p_0, {\SSS}_{1,p_0}) \lesssim \sup_{n \in \NN} \cc_{\rm s}(p_0, \SSS_n) < \infty
	$
	which means that $\MN_{{\SSS}_{1,p_0}}$ is of strong type $(p_0,p_0)$.
\end{proof}

A modification of the arguments from the proof of Lemma~\ref{L2.3.1} shows that for $p_0 \in [1, \infty)$, replacing the former $\tau_n$ by $\tau_n ' =  \lfloor  (\log(n)+1) (n+1)^{p_0} / n \rfloor $ leads to the space ${\SSS}_{1,p_0}'$ for which 
\begin{displaymath}
\PSC({\SSS}_{1,p_0}') = \PWC({\SSS}_{1,p_0}') = \PRC({\SSS}_{1,p_0}') = \PS({\SSS}_{1,p_0}') = \PW({\SSS}_{1,p_0}') = \PR({\SSS}_{1,p_0}') = (p_0, \infty].
\end{displaymath}
Moreover, it may be noted that only the properties $\lim_{n \rightarrow \infty}  \frac{n \tau_n '}{(n+1)^{p_0}} = \infty$ and $\sup_{n \in \mathbb{N}} \frac{n \tau_n '}{(n+1)^p} < \infty$ for $ p \in (p_0, \infty)$ are essential. \\

\noindent {\bf Second subtype.} In contrast to the former case, for the spaces $\SSS$ we now construct and study, the identities $\PSC(\SSS) = \PS(\SSS)$, $\PWC(\SSS) = \PW(\SSS)$ and $\PRC(\SSS) = \PR(\SSS)$ still hold, but there are differences between the incidences of weak and strong type inequalities. Namely, for any fixed $p_0 \in [1, \infty)$ we construct a space denoted by ${\SSS}_{2,p_0}$ for which $\PSC({\SSS}_{2,p_0}) = \PS({\SSS}_{2,p_0}) = (p_0, \infty]$ and $\PWC({\SSS}_{2,p_0}) = \PW({\SSS}_{2,p_0}) = [p_0, \infty]$. Notice that this implies $\PRC({\SSS}_{2,p_0}) = \PR({\SSS}_{2,p_0}) = [p_0, \infty]$. We begin with the case $p_0 = 1$ which is discussed separately because it is relatively simple and may be helpful to outline the core idea behind the more difficult case $p_0 \in (1, \infty)$. 
	
	For any $n \in \NN$ consider $\SSS_n = (S_n, \rho_n, \mu_n)$, where $S_n$, $\rho_n$, and $\mu_n$ are defined as before with the aid of $\tau_n = n$ and $F_n(i) = 2^i$ for each $i \in  [\tau_n]$. We denote by ${\SSS}_{2,1}$ the space $\YY$ obtained by using Proposition~\ref{P2.2.1} for $\Lambda = \NN$ with $\YY_n = \SSS_n$ for each $n \in \NN$. In the following lemma we describe the properties of the associated maximal operators.
	
	\begin{lemma}\label{L2.3.2}
		Let ${\SSS}_{2,1}$ be the first generation space defined above. Then the associated maximal operators, centered $\MM_{{\SSS}_{2,1}}$ and noncentered $\MN_{{\SSS}_{2,1}}$, are not of strong type $(1,1)$, but are of weak type $(1,1)$.
	\end{lemma}
	
	\begin{proof}	
		First, let us note that $\MM_{{\SSS}_{2,1}}$ fails to be of strong type $(1,1)$. Indeed, fix $n \in \NN$ and let $g \coloneqq \mathbf{1}_{\{x_0\}} \in L^1(\SSS_n)$. Then $\|g\|_1=1$ and for each $i \in [n]$ we have $\MM_{\SSS_n} g(x_{i}) \geq \frac{1}{1+2^i} > 2^{-i-1}$. Therefore,
		\begin{displaymath}
		\cc_{\rm s}^{\rm c}(1, \SSS_n) \geq \frac{\|\MM_{\SSS_n} g \|_{1}}{\|g \|_1} \geq \sum_{i=1}^n \frac{2^i}{2^{i+1}} = \frac{n}{2}
		\end{displaymath}
		and, as a consequence, $\cc_{\rm s}^{\rm c}(1, {\SSS}_{2,1}) \gtrsim \lim_{n \rightarrow \infty} \frac{n}{2} = \infty$. 
		
		In the next step we show that $\MN_{{\SSS}_{2,1}}$ is of weak type $(1,1)$. Fix $n \in \NN$ and estimate $\cc_{\rm w}(1, \SSS_n)$ from above. Let $f \in L^1(\SSS_n)$, $f \geq 0$, and consider $\lambda \in (0, \infty)$ such that $E_\lambda \coloneqq E_\lambda (\MN_{\SSS_n} f) \neq \emptyset$. If $\lambda \leq A_{S_n}(f)$, then $\lambda \cdot |E_\lambda| \leq A_{S_n}(f) |S_n| = \|f\|_1$ follows. Assume that $\lambda > A_{S_n}(f)$. If $E_\lambda = \{x_0\}$, then $f(x_0) > \lambda$ and $\lambda \cdot |E_\lambda| \leq \|f\|_1$ again follows. Otherwise, if $E_\lambda \subsetneq \{x_0\}$, then we denote 
		\[
		j \coloneqq \max\big\{i\in [n] : \MN_{\SSS_n} f(x_{i})> \lambda\big\}.
		\] 
		We have $f(x_{j}) > \lambda$ or 
		$
		\frac{f(x_0) + 2^j  f(x_{j})}{1 + 2^j} > \lambda.
		$
		In both cases the inequality $f(x_0) + 2^j  f(x_{j}) > 2^j  \lambda$ holds. Combining this with the fact that $|E_\lambda| \leq 2  |\{x_{j}\}| = 2^{j+1}$, we arrive at the estimate
		\begin{displaymath}
		\frac{\lambda \cdot |E_\lambda|}{\|f\|_1} \leq \frac{2^{j+1}  \lambda}{f(x_0) + 2^j f(x_{j})} \leq 2.
		\end{displaymath}
		Consequently, we have $\cc_{\rm w}(1, \SSS_n) \leq 2$ for any $n \in \NN$ which implies that $\cc_{\rm w}(1, {\SSS}_{2,1}) \lesssim 2 < \infty$.  
	\end{proof}
	
	Now fix $p_0 \in (1, \infty)$ and for any $n \in \NN$ consider $\SSS_n = (S_n, \rho_n, \mu_n)$, where $S_n$, $\rho_n$, and $\mu_n$ are introduced as before with the aid of $\tau_n = \tau_{n,p_0}$ and $F_n(i) = F_{n, p_0}(i)$ defined as follows. Let $c_n, e_n \in \NN$ be auxiliary parameters satisfying $c_n = \lfloor (n+1)^{p_0} / n \rfloor $ and 
	\begin{displaymath}
	e_n = \max \big\{k \in \mathbb{N} : 2^{k-1} \leq c_n \text{ and } 2^{k - 1 + p_0} \leq (1+n)^{p_0}\big\}.
	\end{displaymath}
	Observe that $\lim_{n \rightarrow \infty} e_n = \infty$. We also introduce $(m_{n,j})_{j=1}^{e_n}, (s_{n,j})_{j=1}^{e_n}$ satisfying
	\begin{displaymath}
	2^{1-j} \Big( \frac{1}{1+m_{n,j}} \Big)^{p_0}  = \Big( \frac{1}{1+n} \Big)^{p_0},
	\end{displaymath}
	and
	\begin{displaymath}
	s_{n,j} = \min \big\{k \in \mathbb{N} : km_{n,j} \geq 2^{2-j}nc_n \big\}.
	\end{displaymath}
	Since $2^{e_n - 1 + p_0} \leq (1+n)^{p_0}$ and $m_{n,j} < 2^{2-j}nc_n$, for each $j \in [e_n]$ we have
	\begin{displaymath}
	1 \leq m_{n,j} \leq n \quad {\rm and} \quad 2^{2-j}nc_n \leq s_{n,j} m_{n,j} \leq 2^{3-j}nc_n. 
	\end{displaymath}
	Finally, we put $\tau_n = \sum_{j=1}^{e_n} s_{n,j}$ and $F_n(i) = m_{n,j(n,i)}$ for each $i \in [\tau_n]$ with $j(n,i) \in \NN$ satisfying
	\begin{displaymath}
	i \in \big[s_{n,1} + \dots + s_{n,j(n,i)}\big] \setminus \big[s_{n,1} + \dots + s_{n,j(n,i)-1}\big].
	\end{displaymath}
	We denote by ${\SSS}_{2,p_0}$ the space $\YY$ obtained by using Proposition~\ref{P2.2.1} for $\Lambda = \NN$ with $\YY_n = \SSS_n$ for each $n \in \NN$. In the following lemma we describe the properties of $\MM_{{\SSS}_{2,p_0}}$ and $\MN_{{\SSS}_{2,p_0}}$.
	
	\begin{lemma}\label{L2.3.3}
		Fix $p_0 \in (1, \infty)$ and let ${\SSS}_{2,p_0}$ be the metric measure space defined above. Then the associated maximal operators, centered $\MM_{{\SSS}_{2,p_0}}$ and noncentered $\MN_{{\SSS}_{2,p_0}}$, are not of strong type $(p_0,p_0)$, but are of weak type $(p_0,p_0)$.
	\end{lemma}
	
	\begin{proof}
		First we note that $\MM_{{\SSS}_{2,p_0}}$ is not of strong type $(p_0,p_0)$. Indeed, fix $n \in \NN$ and let $g \coloneqq \mathbf{1}_{\{x_0\}} \in L^{p_0}(\SSS_n)$. Then $\|g\|_{p_0}^{p_0}=1$ and for each $i \in [\tau_n]$ we have $\MM_{\SSS_n} g(x_{i}) \geq \frac{1}{1+m_{n,j(n,i)}}$ which implies
		\begin{displaymath}
		\|\MM_{\SSS_n} g \|_{p_0}^{p_0} \geq \sum_{j=1}^{e_n} \sum_{k=1}^{s_{n,j}}\Big(\frac{1}{1+m_{n,j}}\Big)^{p_0} m_{n,j} = \sum_{j=1}^{e_n} \frac{s_{n,j}m_{n,j}}{(1+m_{n,j})^{p_0}} \geq \sum_{j=1}^{e_n} \frac{2^{2-j} n c_n}{(1+m_{n,j})^{p_0}} = \sum_{j=1}^{e_n}\frac{2nc_n}{(1+n)^{p_0}}.
		\end{displaymath}
		Thus, we get that $\cc_{\rm s}^{\rm c}(p_0, \SSS_n)^{p_0} \geq e_n \frac{2nc_n}{(1+n)^{p_0}}$. Since $\lim_{n \rightarrow \infty} e_n = \infty$ and $\lim_{n \rightarrow \infty} \frac{n c_n}{(1+n)^{p_0}} = 1$, we obtain $\cc_{\rm s}^{\rm c}(p_0, {\SSS}_{2,p_0}) = \infty$.
		
		In the next step we show that $\MN_{{\SSS}_{2,p_0}}$ is of weak type $(p_0,p_0)$. Fix $n \in \NN$ and estimate $\cc_{\rm w}(p_0, \SSS_n)$ from above. Let $f \in L^{p_0}(\SSS_n)$, $f \geq 0$, and take $\lambda \in (0,\infty)$. We write $f = f_1 + f_2$, where $f_1 \coloneqq f \cdot \mathbf{1}_{S_n \setminus \{x_0\}}$ and $f_2 \coloneqq f \cdot \mathbf{1}_{\{x_0\}}$. By the sublinearity of $\MN_{\SSS_n}$ we have  
		\begin{displaymath}
		\MN_{\SSS_n} f \leq \MN_{\SSS_n} f_1 + \MN_{\SSS_n} f_2
		\end{displaymath}
		which gives 
		\[
		E_{2\lambda} (\MN_{\SSS_n} f) \subset E_{\lambda} (\MN_{\SSS_n} f_1) \cup E_{\lambda} (\MN_{\SSS_n} f_2) \eqqcolon E_{\lambda,1} + E_{\lambda,2}.
		\] 
		By using H\"older's inequality we get
		\begin{displaymath}
		\lambda^{p_0} |E_{\lambda,1}| \leq \| \MN_{\SSS_n} f_1\|_{p_0}^{p_0} \leq \sum_{B \subset S_n} A_B(f_1)^{p_0} \, |B| \leq \sum_{B \subset S_n} \sum_{x \in B} f_1(x)^{p_0} |\{x\}| \leq 3 \|f_1\|_{p_0}^{p_0},
		\end{displaymath}
		where in the last inequality we use the fact that each $x \in S_n \setminus \{x_0\}$ belongs to at most three different balls $ B \subset S_n$. 
		Next, let $f_2(x_0) = \alpha \in (0, \infty)$ and assume that $E_{\lambda,2} \neq \emptyset$. Thus, we have $\alpha > \lambda$. If $E_{\lambda,2} = \{x_0\}$, then $\lambda^{p_0} |E_{\lambda,2}| < \|f_2\|_{p_0}^{p_0}$ follows. Otherwise, denote 
		\[
		r \coloneqq \min \Big\{j \in [e_n] : \frac{\alpha}{1+m_{n,j}} > \lambda \Big\}.
		\]
		Then we have
		\begin{displaymath}
		 \lambda^{p_0} |E_{\lambda,2}| \leq \frac{2 \alpha^{p_0} \sum_{j=r}^{e_n} s_{n,j} m_{n,j}}{(1+m_{n,r})^{p_0}}  \leq \frac{\alpha^{p_0} \sum_{j=r}^{e_n} 2^{4-j} n c_n}{(1+m_{n,r})^{p_0}}  \leq \frac{2^{5-r} \alpha^{p_0} n c_n }{(1+m_{n,r})^{p_0}} \leq \frac{16 \alpha^{p_0} n c_n}{(1+n)^{p_0}} \leq 16 \|f_2\|_{p_0}^{p_0}.
		\end{displaymath}
		Consequently,
		\begin{displaymath}
		(2\lambda)^{p_0} |E_{2\lambda} (\MN_{\SSS_n} f)| \leq 2^{p_0} \cdot 19 \|f\|_{p_0}^{p_0}.
		\end{displaymath}   
		Since $\cc_{\rm w}(p_0, \SSS_n) \leq 2 \cdot 19^{1/p_0}$ for any $n \in \NN$, we conclude that $\cc_{\rm w}(p_0, {\SSS}_{2,p_0}) \lesssim 2 \cdot 19^{1/p_0} < \infty$.
	\end{proof}
	
	Now we present a construction which will be applied to the third subtype of the first generation spaces. Fix $n_0 \in \NN$ and let ${\tau} = \tau_{n_0} = (\tau_{n_0,i})_{i=1}^{n_0}$ be a given system of positive integers satisfying $\frac{\tau_{n_0,i}}{2^{i-1}} \in \mathbb{N}$. We shall introduce several objects which clearly depend on $n_0 \in \NN$. It will be helpful to include this parameter in notation. Set 
	\begin{displaymath}
	\overline{S}_{n_0} \coloneqq \overline{S}_{n_0}(\tau) \coloneqq \Big\{x_{i,j}, \, x'_{i,k} : i \in [n_0], \, j \in [2^{i-1}], \, k \in [\tau_{n_0,i}] \Big\},
	\end{displaymath}
	where all elements $x_{i,j}, \, x'_{i,k}$ are different. We use auxiliary symbols for certain subsets of $\overline{S}_{n_0}$:
	\begin{displaymath}
	S_{n_0}' \coloneqq \Big\{x'_{i,k} : i \in [n_0], \, k \in [\tau_{n_0,i}] \Big\},
	\end{displaymath}
	for $i \in [n_0]$,
	\begin{align*}
	S_{n_0,i}  \coloneqq \Big\{x_{i,j} : j \in [2^{i-1}] \Big\}, \qquad 
	S_{n_0,i}'  \coloneqq \Big\{x'_{i,k} : k \in [\tau_{n_0,i}] \Big\},
	\end{align*}
	and, for $i, i' \in [n_0]$, $i \leq i'$, and $j \in [2^{i-1}]$,  
	\begin{displaymath}
	S'_{n_0,i',i,j} \coloneqq \Big\{x'_{i',k} : k \in \Big(\frac{j-1}{2^{i-1}} \tau_{n_0,i'}, \frac{j}{2^{i-1}}\tau_{n_0,i'}\Big]\Big\}.
	\end{displaymath}
	Observe that the family $\{S'_{n_0,i',i,j}\}_{j=1}^{2^{i-1}}$ consists of disjoint set, each of them containing exactly $\frac{\tau_{n_0,i'}}{2^{i-1}}$ elements. Moreover, we have $\bigcup_{j=1}^{2^{i-1}} S'_{n_0,i',i,j} = S'_{n_0,i'}$. Finally, if $1 \leq i_1 \leq i_2 \leq i' \leq n_0$ and $j_l \in [2^{i_l - 1}]$ for $l \in \{1,2\}$, then either $S'_{n_0,i',i_2,j_2} \subset S'_{n_0,i',i_1,j_1}$ or $S'_{n_0,i',i_1,j_1} \cap S'_{n_0,i',i_2,j_2} = \emptyset$.
	
	We define the metric $\overline{\rho}_{n_0}$ on $\overline{S}_{n_0}$ determining the distance between two different elements $x$ and $y$ by the formula
	\begin{displaymath}
	\overline{\rho}_{n_0}(x,y) \coloneqq \overline{\rho}_{n_0,\tau}(x,y) \coloneqq \left\{ \begin{array}{rl}
	1 & \textrm{if } \{x, y\} = \{x_{i,j},x'_{i',k}\} \textrm{ and } x'_{i',k} \in S'_{n_0,i',i,j},  \\
	2 & \textrm{otherwise.} \end{array} \right. 
	\end{displaymath}
	Figure~\ref{F2.2} shows a model of the space $(\overline{S}_{n_0}, \overline{\rho}_{n_0})$ with $n_0 = 2$. As before, the solid line between two points indicates that the distance between them equals $1$. Otherwise the distance equals $2$.  
	
	\begin{figure}[H]
		\begin{center}
		\begin{tikzpicture}
		[scale=.8,auto=left,every node/.style={circle,fill,inner sep=2pt}]
		
		\node[label={[yshift=-1cm]$x_{1,1}$}] (c0) at (6,1) {};
		\node[label={[yshift=-0.05cm]$x'_{1,1}$}] (c1) at (5,4)  {};
		\node[label={[yshift=-0.18cm]$x'_{1,\tau_{2,1}}$}] (c2) at (7,4)  {};
		\node[dots] (c3) at (6,4)  {...};
		
		\node[label={[yshift=-1cm]$x_{2,1}$}] (l0) at (10,1) {};
		\node[label={[yshift=-0.05cm]$x'_{2,1}$}] (l1) at (9,4)  {};
		\node[label={[xshift=-0.15cm, yshift=-0.3cm]$x'_{2,\tau_{2,2}/2}$}] (l2) at (11,4)  {};
		\node[dots] (l3) at (10,4)  {...};
		
		\node[label={[yshift=-1cm]$x_{2,2}$}] (r0) at (14,1) {};
		\node[label={[xshift=0.1cm, yshift=-0.5cm]$x'_{2,\tau_{2,2}/2+1}$}] (r1) at (13,4)  {};
		\node[label={[xshift=0.2cm, yshift=-0.2cm]$x'_{2,\tau_{2,2}}$}] (r2) at (15,4)  {};
		\node[dots] (r3) at (14,4)  {...};
		
		\foreach \from/\to in {l0/l1, l0/l2, c0/c1, c0/c2, r0/r1, r0/r2, l1/c0, l2/c0, r1/c0, r2/c0}
		\draw (\from) -- (\to);
		\end{tikzpicture}
		\caption{The first generation space $(\overline{S}_{n_0}, \overline{\rho}_{n_0})$ with $n_0 = 2$.}
		\label{F2.2}
		\end{center} 
	\end{figure}
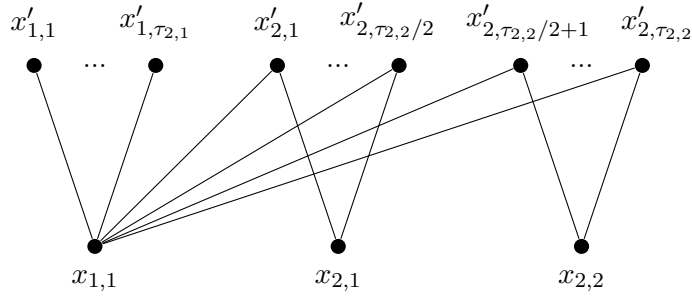
	\noindent Note that we can explicitly describe any ball: for $i \in [n_0], \, j \in [2^{i-1}]$,
	\begin{displaymath}
	B(x_{i,j},s) = \left\{ \begin{array}{rl}
	\{x_{i,j}\} & \textrm{for } 0 < s \leq 1, \\
	\{x_{i,j}\} \cup \bigcup_{i' = i}^{n_0} S_{n_0,i',i,j}& \textrm{for } 1 < s \leq 2,  \\
	\overline{S}_{n_0} & \textrm{for } 2 < s, \end{array} \right.
	\end{displaymath} 
	\noindent and for $i' \in [n_0], \, k \in [\tau_{n_0,i}]$,
	\begin{displaymath}
	B(x'_{i',k},s) = \left\{ \begin{array}{rl}
	\{x'_{i',k}\} & \textrm{for } 0 < s \leq 1, \\
	\{x'_{i',k}\} \cup \{x_{i, j} : x'_{i',k} \in S_{n_0,i',i,j}\}& \textrm{for } 1 < s \leq 2,  \\
	\overline{S}_{n_0} & \textrm{for } 2 < s. \end{array} \right.
	\end{displaymath}
	
	Finally, we define the measure $\overline{\mu}_{n_0}$ on $\overline{S}_{n_0}$ by letting 
	\begin{displaymath}
	\overline{\mu}_{n_0}(\{x\}) \coloneqq \overline{\mu}_{n_0,\tau, F, m}(\{x\}) \coloneqq \left\{ \begin{array}{rl}
	F(i) & \textrm{if } x = x_{i,j} \textrm{ for some } i \in [n_0], \, j \in [2^{i-1}], \\
	m i & \textrm{if } x = x'_{i,k} \textrm{ for some } i \in [n_0], \, k \in [\tau_{n_0,i}], \end{array} \right.
	\end{displaymath}
	where $0 <F \leq 1$ is a given function and $m$ is a positive number satisfying $m \geq 2^{n_0}$. 
	
	We are ready to describe the third subtype of the first generation spaces. \\
	
	\noindent {\bf Third subtype.} Now for any fixed $p_0 \in (1, \infty)$ we construct a space denoted by ${\SSS}_{3,p_0}$ for which $\PWC({{\SSS}_{3,p_0}}) = \PW({{\SSS}_{3,p_0}}) = (p_0, \infty]$ and $\PRC({{\SSS}_{3,p_0}}) = \PR({{\SSS}_{3,p_0}}) = [p_0, \infty]$. Note that these conditions imply $\PSC({{\SSS}_{3,p_0}}) = \PS({{\SSS}_{3,p_0}}) = (p_0, \infty]$.
	
	Fix $p_0 \in (1, \infty)$. For any $n \in \NN$ we choose $n_0 = n$ and consider $\SSS_n = (\overline{S}_n, \overline{\rho}_n, \overline{\mu}_n)$, where $\overline{S}_n$, $\overline{\rho}_n$, and $\overline{\mu}_n$ are introduced as before with the aid of $\tau_n = (\tau_{n,i})_{i=1}^n$, $F_n$, and $m_n$ defined as follows. Let $(a_i)_{i \in \mathbb{N}}$ satisfy $\sum_{i=1}^n a_i = n^{p_0}$ for each $n \in \NN$. For $ i \in  [n]$ set
	\begin{displaymath}
	F_n(i)=2^{(i-n)/(p_0-1)}, \qquad \tau_{n,i} =  \lfloor a_i \rfloor 2^{2n \lfloor p_0 \rfloor} n!/i, \qquad m_n = 2^{(2n \lfloor p_0  \rfloor -n) / (p_0-1)} (n!)^{1/(p_0-1)}.
	\end{displaymath}
	Observe that $\frac{\tau_{n,i}}{2^{i-1}} \in \mathbb{N}$ and $\tau_{n,i} \, i / m_n^{p_0-1} = 2^n \lfloor a_i \rfloor$. Moreover, for any $x \in S_n'$ we have $|\{x\}| \geq m_n \geq 2^n \geq |\overline{S}_n \setminus S_n'|$.
	We denote by ${\SSS}_{3,p_0}$ the space $\YY$ obtained by using Proposition~\ref{P2.2.1} for $\Lambda = \NN$ with $\YY_n = \SSS_n$ for each $n \in \NN$. The following lemma describes the properties of $\MM_{{\SSS}_{3,p_0}}$ and $\MN_{{\SSS}_{3,p_0}}$.
	
	\begin{lemma}
		Fix $p_0 \in (1, \infty)$ and let ${\SSS}_{3,p_0}$ be the metric measure space defined above. Then the associated maximal operators, centered $\MN_{{\SSS}_{3,p_0}}$ and noncentered $\MM_{{\SSS}_{3,p_0}}$, are not of weak type $(p_0,p_0)$, but are of restricted weak type $(p_0,p_0)$.
	\end{lemma}
	
	\begin{proof}
		First we show that $\MM_{{\SSS}_{3,p_0}}$ is not of weak type $(p_0,p_0)$. Indeed, fix $n \in \NN$ and let
		\begin{displaymath}
		g \coloneqq \sum_{i=1}^{n} 2^{(n-i)/(p_0-1)} \mathbf{1}_{S_{n,i}} \in L^{p_0}(\SSS_n).
		\end{displaymath}
		Then $\|g\|_{p_0}^{p_0} = \sum_{i=1}^{n} 2^{i-1} \, 2^{n-i} = 2^{n-1} n$ and
		\begin{displaymath}
		\MM_{\SSS_n} g(x'_{i',k}) \geq A_{B(x'_{i',k}, 3/2)}(g) \geq \frac{i'}{2 |\{x'_{i',k}\}|} = \frac{1}{2 m_n}
		\end{displaymath}
		for any $x'_{i',k} \in S_n'$ which implies that $|E_{1/(4m_n)}(\MM_{\SSS_n} g)| \geq |S_n'|$. Therefore,
		\begin{displaymath}
		 \frac{\|\MM_{\SSS_n} g \|_{p_0,\infty}^{p_0}}{\| g \|_{p_0}^{p_0}} \geq \frac{\sum_{i=1}^{n}\tau_{n,i} \, i \, m_n}{n 2^{n-1} (4m_n)^{p_0}} 
		= 2^{1-2p_0} \frac{\sum_{i=1}^{n}\tau_{n,i} \, i}{n \, m_n^{p_0-1} \, 2^n}
		= 2^{1-2p_0} \frac{\sum_{i=1}^{n} \lfloor a_i \rfloor}{n} = 2^{1-2p_0} n^{p_0-1}.
		\end{displaymath}
		Thus, we obtain $\cc_{\rm w}^{\rm c}(p_0, \SSS_n)^{p_0} \geq 2^{1-2p_0} n^{p_0-1}$ and, consequently, 
		\begin{displaymath}
		\cc_{\rm w}^{\rm c}(p_0, {\SSS}_{3,p_0}) \gtrsim \lim_{n \rightarrow \infty} n^{1-1/p_0} = \infty.
		\end{displaymath}
		
		In the next step we show that $\MN_{{\SSS}_{3,p_0}}$ is of restricted weak type $(p_0,p_0)$. Fix $n \in \NN$ and estimate $\cc_{\rm r}(p_0, \SSS_n)^{p_0}$ from above. Let 
		$U \subset \overline{S}_n$, $U \neq \emptyset$, and $\lambda \in (0, \infty)$. Our goal is to estimate
		\begin{equation}\label{2.3.2}
		\lambda^{p_0} | E_\lambda| / |U|,
		\end{equation}
		where $E_\lambda \coloneqq E_\lambda(\MN_{\SSS_n} \mathbf{1}_{U})$. Clearly, if $\lambda \geq 1$, then $E_\lambda = \emptyset$, so we can assume that $\lambda < 1$. We write $\mathbf{1}_{U} = \mathbf{1}_{U'} + \sum_{x \in U \setminus U'} \mathbf{1}_{\{x\}}$, where $U' = U \cap S_n'$. Note that for fixed $i \in [n]$ any ball $B$ with radius $s \leq 2$ contains at most one of the points $x_{i,1}, \dots, x_{i,2^{i-1}}$. Thus, for any $x \in \overline{S}_n$, the inequality
		\begin{displaymath}
		\MN_{\SSS_n} \mathbf{1}_{U}(x) \leq 2 \max\Big\{\MN_{\SSS_n} \mathbf{1}_{U'}(x), \, C \cdot \max_{y \in U \setminus U'}\big\{ \MN_{\SSS_n} \mathbf{1}_{\{y\}}(x)\big\} \Big\}
		\end{displaymath}
		is satisfied with $C \coloneqq \sum_{i=0}^{\infty} 2^{-i(p_0-1)}$. Moreover, if $\lambda \leq A_{\overline{S}_n} (\mathbf{1}_{U})$, then 
		\[
		\lambda^{p_0} |E_\lambda| \leq  \lambda |E_\lambda| \leq A_{\overline{S}_n}(\mathbf{1}_{U}) |\overline{S}_n| = |U|. 
		\]
		Consequently, we are reduced to finding a suitable bound for \eqref{2.3.2} in the case $A_{\overline{S}_n}(\mathbf{1}_{U}) < \lambda < 1$ and $U \subset S_n'$ or $U = \{x\} \subset \overline{S}_n \setminus S_n'$.
		
		First, assume that $U \subset S_n'$, $U \neq \emptyset$. Then $E_\lambda \cap S_n' \neq \emptyset$ and we have $|E_\lambda| \leq 2  |E_\lambda \cap S_n'|$. Consider two different balls, $B_{1}$ and $B_{2}$, and denote $B_{l}' = B_{l} \cap S_n'$ for $l \in [2]$. Then one of the three possibilities occurs: $B_{1}' \subset B_{2}'$, $B_{2}' \subset B_{1}'$, or $B_{1}' \cap B_{2}' = \emptyset$. Combining this with the assumption $U \subset S_n'$, we obtain $\lambda \leq A_{E_\lambda \cap S_n'}(\mathbf{1}_U)$. Consequently,
		\begin{displaymath}
		\lambda^{p_0} |E_\lambda| \leq \lambda |E_\lambda| \leq 2  A_{E_\lambda \cap S_n'}(\mathbf{1}_U) |E_\lambda \cap S_n'| = 2 |U|.
		\end{displaymath}
		
		Finally, assume that $U = \{x_{i,j}\} \subset \overline{S}_n \setminus S_n'$ and $A_{\overline{S}_n}(\mathbf{1}_{U}) < \lambda < 1$. If $E_\lambda \cap S_n' = \emptyset$, then $E_\lambda = U$ and $\lambda^{p_0} |E_\lambda| \leq |U|$ holds. On the other hand, if $E_\lambda \cap S_n' \neq \emptyset$, then 
		$
		|E_\lambda| \leq 2 |E_\lambda \cap S_n'|.
		$ 
		For $x \in S'_{n,i',i,j}$, $i' \geq i$, we see that
		\[
		\MN_{\SSS_n} \mathbf{1}_{U} (x) = A_{B(x, 3/2)}(\mathbf{1}_{U}) \leq 2^{(i-n)/(p_0-1)} / (m_n i'),
		\]
		while for $x \in S_n' \setminus \bigcup_{i'=i}^{n} S'_{n,i',i,j}$ we have
		$
		\MN_{\SSS_n} \mathbf{1}_{U}(x) = A_{\overline{S}_{n}}(\mathbf{1}_{U}) < \lambda.
		$
		Since
		\begin{align*}
		\frac{2^{(i-n)p_0/(p_0-1)}}{ (m_n \, i')^{p_0}} \sum_{l=i}^{i'} |S'_{n,l,i,j}| &= 
		\frac{2^{(i-n)p_0 / (p_0-1)}}{( m_n \, i')^{p_0}}    \sum_{l=i}^{i'} \frac{|S'_{n,l}|}{2^{i-1}} \\
		& = \frac{2^{(i-n)p_0 / (p_0-1)}}{(i')^{p_0}}  \sum_{l=i}^{i'} \frac{\tau_{n,l} \, l}{m_n^{p_0-1} \, 2^{i-1}} \\
		& = \frac{2^{(i-n)p_0 / (p_0-1)}}{(i')^{p_0}} \sum_{l=i}^{i'} 2^{n-i+1} \lfloor a_l \rfloor \\ & \leq 2^{((i-n)/(p_0-1)) + 1} \, \frac{\sum_{l=1}^{i'} a_l}{(i')^{p_0}} = 2 |U|,
		\end{align*}
		we conclude that $\lambda^{p_0} |E_\lambda| \leq 4 |U|$ holds. 
		
		Since for each $n \in \NN$ we have a bound of the form $\cc_{\rm r}(p_0, \SSS_n) \lesssim C(p_0)$ with the implicit constant independent of $n$, we conclude that $\cc_{\rm r}(p_0, {\SSS}_{3,p_0}) < \infty$.	
	\end{proof}

\subsection{Second generation spaces}

Now we construct and study metric measure spaces $\TT$ called by us the second generation spaces. The common attribute of these spaces is that the associated operators $\MM_\TT$ and $\MN_\TT$ behave significantly different, by what we mean that $\PSC(\TT) = \PWC(\TT) = \PRC(\TT) = [1, \infty]$ holds, while $\PS(\TT)$ (and possibly $\PW(\TT)$ and $\PR(\TT)$) is a proper subset of $[1, \infty]$. As before, we specify three subtypes here. The following construction will be applied to the first two of them.

Let $\tau$ be a fixed positive integer. Set $T \coloneqq T(\tau) \coloneqq \{ y_0, y_1, \dots, y_{\tau}, y_1', \dots, y_\tau'\}$, where all elements are different. We define the metric $\rho$ determining the distance between two different elements $x$ and $y$ by the formula
\begin{displaymath}
\rho(x,y) \coloneqq \rho_\tau(x,y) \coloneqq \left\{ \begin{array}{rl}
1 & \textrm{if } y_0 \in \{x,y\} \subset T \setminus T' \textrm{ or } \{x,y\} = T^{i}    \textrm{ for some } i \in [\tau], \\
2 & \textrm{otherwise,} \end{array} \right. 
\end{displaymath}
where $T' \coloneqq \{y_{1}', \dots , y_{\tau}'\}$ and $T^{i} \coloneqq \{y_{i}, y_{i}'\}$. Figure~\ref{F2.3} shows a model of the space $(T, \rho)$.

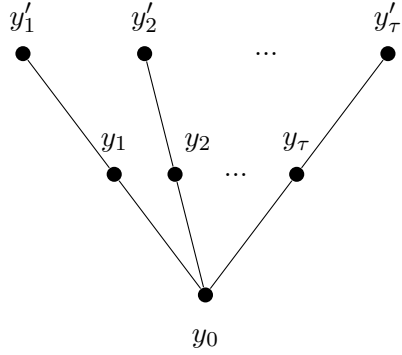
\begin{figure}[H]
	\begin{center}
	\begin{tikzpicture}
	[scale=.8,auto=left,every node/.style={circle,fill,inner sep = 2pt}]
	\node[label={[yshift=-1cm]$y_0$}] (n0) at (4,1) {};
	\node[label=$y_{1}$] (n1) at (2.5,3)  {};
	\node[label={[xshift=0.3cm]$y_{2}$}] (n2) at (3.5,3)  {};
	\node[label={[yshift=-0.00cm]$y_{\tau}$}] (n3) at (5.5,3)  {};
	\node[dots] (n4) at (4.5,3)  {...};
	\node[label=$y_{1}'$] (n5) at (1,5)  {};
	\node[label=$y_{2}'$] (n6) at (3,5)  {};
	\node[label={[yshift=-0.00cm]$y_{\tau}'$}] (n7) at (7,5)  {};
	\node[dots] (n8) at (5,5)  {...};
	
	\foreach \from/\to in {n0/n1, n0/n2, n0/n3, n1/n5, n2/n6, n3/n7}
	\draw (\from) -- (\to);
	\end{tikzpicture}
	\caption{The second generation space $(T, \rho)$.}
	\label{F2.3}
	\end{center}
\end{figure}
\noindent Note that we can explicitly describe any ball:
\begin{displaymath}
B(y_0,s) = \left\{ \begin{array}{rl}
\{y_0\} & \textrm{for } 0 < s \leq 1, \\
T \setminus T' & \textrm{for } 1 < s \leq 2,  \\
T & \textrm{for } 2 < s, \end{array} \right.
\end{displaymath} 
\noindent and, for $i \in [\tau]$,
\begin{displaymath}
B(y_{i},s) = \left\{ \begin{array}{rl}
\{y_{i}\} & \textrm{for } 0 < s \leq 1, \\
\{y_0\} \cup T^{i} & \textrm{for } 1 < s \leq 2,  \\
T & \textrm{for } 2 < s, \end{array} \right. \qquad
\end{displaymath} 
and
\begin{displaymath}
B(y_{i}',s) = \left\{ \begin{array}{rl}
\{y_{i}'\} & \textrm{for } 0 < s \leq 1, \\
T^{i} & \textrm{for } 1 < s \leq 2,  \\
T & \textrm{for } 2 < s. \end{array} \right. \qquad
\end{displaymath}
We define the measure $\mu$ by letting 
\begin{displaymath}
\mu(\{y\}) \coloneqq  \mu_{\tau, F}(\{y\}) \coloneqq \left\{ \begin{array}{rl}
1 & \textrm{if } y = y_0, \\
\frac{1}{\tau} & \textrm{if } y = y_i \textrm{ for some } i \in [\tau], \\
F(i) & \textrm{if } y = y'_i \textrm{ for some } i \in [\tau], \end{array} \right.
\end{displaymath}
where $F > 0$ is a given function.

We are ready to describe the first two subtypes of the second generation spaces. \\

\noindent {\bf First subtype.} Now we construct a subclass of the second generation spaces $\TT$ for which, apart from the basic property $\PSC(\TT) = \PWC(\TT) = \PRC(\TT) = [1, \infty]$, we also have $\PS(\TT) = \PW(\TT) = \PR(\TT)$. In the first step, for any fixed $p_0 \in (1, \infty]$ we construct a space denoted by ${\TT}_{1,p_0}$ for which $\PS({\TT}_{1,p_0}) = \PW({\TT}_{1,p_0}) = \PR({\TT}_{1,p_0}) = [p_0, \infty]$. Then, after slight modifications, for any fixed $p_0 \in [1, \infty)$ we get a space ${\TT}_{1,p_0}'$ for which $\PS({\TT}_{1,p_0}') = \PW({\TT}_{1,p_0}') = \PR({\TT}_{1,p_0}') = (p_0, \infty]$.

Fix $p_0 \in (1, \infty]$ and for any $n \in \NN$ consider $\TT_n = (T_n, \rho_n, \mu_n)$, where $T_n$, $\rho_n$, and $\mu_n$ are defined as before with the aid of $\tau_n = \lfloor (n+1)^{p_0} / n \rfloor$ in the case $p_0 \in (1, \infty)$, or $\tau_n = 2^n$ in the case $p_0 = \infty$, and $F_n(i) = n$ for each $i \in [\tau_n]$. We denote by ${\TT}_{1,p_0}$ the space $\YY$ obtained by using Proposition~\ref{P2.2.1} for $\Lambda = \NN$ with $\YY_n = \TT_n$ for each $n \in \NN$. In the following lemma we describe the properties of $\MM_{{\TT}_{1,p_0}}$ and $\MN_{{\TT}_{1,p_0}}$. 
 
\begin{lemma}\label{L2.3.5}
	Fix $p_0 \in (1, \infty]$ and let ${\TT}_{1,p_0}$ be the metric measure space defined above. Then the associated centered maximal operator $\MM_{{\TT}_{1,p_0}}$ is of strong type $(1,1)$, while the noncentered maximal operator $\MN_{{\TT}_{1,p_0}}$ is not of weak type $(p,p)$ for $p \in [1, p_0)$, but is of strong type $(p,p)$ for $p \in [p_0, \infty]$.
\end{lemma}

\begin{proof}
	First we show that $\MM_{{\TT}_{1,p_0}}$ is of strong type $(1,1)$. We fix $n \in \NN$ and restrict our attention to $\TT_n$. Let $f \in L^1(\TT_n)$, $f \geq 0$. Denote $\mathcal{G} \coloneqq \{\{y_0\} \cup T_n^{i} \colon i \in [\tau_n]\}$ and, for each $y \in T_n$, $\mathcal{B}_y \coloneqq \{B(y, \frac{1}{2}),B(y, \frac{3}{2}), B(y, \frac{5}{2})\}$. We have the estimate
	\begin{displaymath}
	\|\MM_{\TT_n} f\|_1 \leq \sum_{y \in T_n} \sum_{B \in \mathcal{B}_y} A_{B}(f)  |\{y\}|.
	\end{displaymath}
	Note that each $y \in T_n$ belongs to at most four different balls which are not elements of $\mathcal{G}$. Thus,
	\begin{displaymath}
	\sum_{y \in T_n} \sum_{B \in \mathcal{B}_y \setminus \mathcal{G}} A_{B}(f)  |\{y\}| \leq \sum_{B \notin \mathcal{G}} \sum_{y \in B} f(y)|\{y\}| \leq 4  \|f\|_1
	\end{displaymath}
	which implies
	\begin{displaymath}
	\|\MM_{\TT_n} f\|_1 \leq 4  \|f\|_1 + \sum_{i=1}^{\tau_n} A_{B(y_{i} \frac{3}{2})}(f)  |\{y_{i}\}|,
	\end{displaymath}
	and the sum above is estimated by
	\begin{displaymath}
	\tau_n f(y_0) |\{y_{1}\}| + \sum_{i=1}^{\tau_n} \Big( f(y_{i}) |\{y_{i}\}| + f(y_{i}') |\{y_{i}'\}| \Big) = \|f\|_1.
	\end{displaymath}
	Consequently, $\cc_{\rm s}^{\rm c}(1, \TT_n) \leq 5$ holds for each $n \in \NN$ which implies that $\cc_{\rm s}^{\rm c}(p, {\TT}_{1,p_0}) < \infty$.
	
	In the next step we show that $\MN_{{\TT}_{1,p_0}}$ is not of restricted weak type $(p,p)$ for $p \in [1, p_0)$. Indeed, fix $p \in [1, p_0)$ and $n \in \NN$, and consider $g \coloneqq \mathbf{1}_{\{y_0\}} \in L^{p}(\TT_n)$. We have $\|g\|_p = 1$ and $\MN_{\TT_n} g(y_{i}') \geq \frac{1}{n+1+(1/\tau_n)} \geq \frac{1}{n+2}$ for each $i \in [\tau_n]$ which gives $|E_{1/(2(n+2))}(\MN_{\TT_n} f_n)| \geq n \tau_n$. Thus,
	\begin{displaymath}
	\cc_{\rm r}(p, \TT_n) \geq  
	\frac{\|\MN_{\TT_n} g \|_{p,\infty}}{\| g \|_p} \geq \frac{(n \tau_n)^{1/p}}{2(n+2)}
	\end{displaymath} 
	and, as a consequence, $\cc_{\rm r}(p, {\TT}_{1,p_0}) \gtrsim \lim_{n \rightarrow \infty} \frac{(n \tau_n)^{1/p}}{2(n+2)} = \infty$. 
	
	To complete the proof it suffices to show that $\MN_{{\TT}_{1,p_0}}$ is of strong type $(p_0,p_0)$. Assume that $p_0 \in (1, \infty)$ (for $p_0 = \infty$ the condition to check is trivial). Fix $n \in \NN$ and estimate $\cc_{\rm s}(p_0, \TT_n)$ from above. Let $f \in L^{p_0}(\TT_n)$, $f \geq 0$. We have the inequality
	\begin{displaymath}
	\|\MN_{\TT_n} f\|_{p_0}^{p_0} \leq \sum_{B \subset T_n} \sum_{y \in B} A_B(f)^{p_0} |\{y\}| = \sum_{B \subset T_n} A_B(f)^{p_0} |B|. 
	\end{displaymath}
	As before, note that each $y \in T_n$ belongs to at most four different balls which are not elements of $\mathcal{G}$. Combining this observation with H\"older's inequality, we obtain
	\begin{displaymath}
	\sum_{B \notin \mathcal{G}} A_B(f)^{p_0} |B| \leq \sum_{B \notin \mathcal{G}} \sum_{y \in B} f(y)^{p_0} |\{y\}| \leq 4 \|f\|_{p_0}^{p_0}. 
	\end{displaymath}
	Therefore,
	\begin{equation} \label{2.3.3}
	\|\MN_{\TT_n} f\|_{p_0}^{p_0} \leq 4 \|f\|_{p_0}^{p_0} + \sum_{i=1}^{\tau_n} \Big(\frac{f(y_0)+ 1/{\tau_n}f(y_{i}) + n f(y_{i}')}{1 + 1/{\tau_n}+n}\Big)^{p_0} |\{y_0, y_{i}, y_{i}'\}|.
	\end{equation}
	By H\"older's inequality 
	\begin{displaymath}
	\big(f(y_0)+ 1/{\tau_n}f(y_{i}) + n f(y_{i}')\big)^{p_0} \leq 3^{p_0-1}  \Big( f(y_0)^{p_0} + \big(f(y_{i})/{\tau_n}\big)^{p_0} + \big(nf(y_{i}')\big)^{p_0} \Big)
	\end{displaymath}
	and combining this with
	\[|\{y_0, y_{i}, y_{i}'\}| \leq 3 |\{y_{i}'\}| = 3n |\{y_0\}|
	\]
	we see that the sum in \eqref{2.3.3} is controlled by
	\begin{displaymath}
	3^{p_0} \Big(\frac{n \tau_n f(y_0)^{p_0} }{(n+1)^{p_0}} |\{y_0\}| + \sum_{i=1}^{\tau_n} \frac{\big(f(y_{i})/{\tau_n}\big)^{p_0} + (nf(y_{i}')\big)^{p_0}}{(1+1/{\tau_n}+n)^{p_0}} |\{y_{i}'\}| \Big) \leq 3^{p_0} \|f\|_{p_0}^{p_0}.
	\end{displaymath}
	Thus, $\cc_{\rm s}(p_0, \TT_n)^{p_0} \leq 4 + 3^{p_0}$ for each $n \in \NN$ which clearly implies $\cc_{\rm s}(p_0, {\TT}_{1,p_0}) < \infty$. 
\end{proof}

Finally, let us note that, as in the previous subsection, replacing the former $\tau_n$ by $\tau_n ' = \lfloor (\log(n)+1) (n+1)^{p_0} / n \rfloor $, $p_0 \in [1, \infty)$, results in obtaining the space ${\TT}_{1,p_0}'$ for which $\PS({\TT}_{1,p_0}') = \PW({\TT}_{1,p_0}') = \PR({\TT}_{1,p_0}') = (p_0, \infty]$. \\

\noindent {\bf Second subtype.} In contrast to the former case the spaces $\TT$ we now construct, apart from the basic property $\PSC(\TT) = \PWC(\TT) = \PRC(\TT) = [1, \infty]$, satisfy $\PS(\TT) \varsubsetneq \PW(\TT)$. More precisely, for any fixed $p_0 \in [1, \infty)$ we construct a space ${\TT}_{2,p_0}$ for which $\PS({\TT}_{2,p_0}) = (p_0, \infty]$ and $\PW({\TT}_{2,p_0}) = [p_0, \infty]$. As previously, we consider the cases $p_0 = 1$ and $p_0 \in  (1, \infty)$ separately.

For any $n \in \NN$ consider $\TT_n = (T_n, \rho_n, \mu_n)$, where $T_n$, $\rho_n$, and $\mu_n$ are defined as before with the aid of $\tau_n = n$ and $F_n(i) = 2^i$ for each $i \in [\tau_n]$. We denote by ${\TT}_{2,1}$ the space $\YY$ obtained by using Proposition~\ref{P2.2.1} for $\Lambda = \NN$ with $\YY_n = \TT_n$ for each $n \in \NN$. In the following lemma we describe the properties of $\MM_{{\TT}_{2,1}}$ and $\MN_{{\TT}_{2,1}}$.

\begin{lemma}
	Let ${\TT}_{2,1}$ be the metric measure space defined above. Then the associated centered operator $\MM_{{\TT}_{2,1}}$ is of strong type $(1,1)$, while the noncentered operator $\MN_{{\TT}_{2,1}}$ is not of strong type $(1,1)$, but is of weak type $(1, 1)$.
\end{lemma}

\begin{proof}
	First, note that it is easy to verify that $\MM_{{\TT}_{2,1}}$ is of strong type $(1,1)$, by using exactly the same argument as in the proof of Lemma~\ref{L2.3.5}.
	
	In the next step we show that $\MN_{{\TT}_{2,1}}$ is not of strong type $(1,1)$. Indeed, fix $n \in \NN$ and let $g \coloneqq \mathbf{1}_{\{y_0\}} \in L^1(\TT_n)$. Then $\|g\|_1=1$ and we have $\MN_{\TT_n} g(y_{i}') \geq \frac{1}{1+ 1/n + 2^i} > 2^{-i-1}$ for each $i \in [n]$. Thus, we obtain $\|\MN_{\TT_n} g \|_1 \geq \sum_{i=1}^n 2^i \cdot 2^{-i-1} = \frac{n}{2}$ and, consequently, $\cc_{\rm s}(1, {\TT}_{2,1}) \gtrsim \lim_{n \rightarrow \infty} \frac{n}{2} = \infty$.
	
	To complete the proof it suffices to show that $\MN_{{\TT}_{2,1}}$ is of weak type $(1,1)$. Fix $n \in \NN$ and estimate $\cc_{\rm w}(1, \TT_n)$ from above. Let $f \in L^1(\TT_n)$, $f \geq 0$, and take $\lambda \in (0, \infty)$ such that $E_\lambda \coloneqq E_\lambda (\MN_{\TT_n} f) \neq \emptyset$. If $\lambda \leq A_{T_n}(f)$, then $\lambda \cdot |E_\lambda| \leq \|f\|_1$ follows. Thus, assume that $\lambda > A_{T_n}(f)$. First, consider the case $E_\lambda \cap T_n' = \emptyset$. If $y_0 \in E_\lambda$, then we obtain 
	\begin{displaymath}
	\sum_{y \in T_n \setminus T_n'} f(y) \, |\{y\}| > \lambda \cdot |\{y_0\}|
	\end{displaymath}
	and, since $|E_\lambda| \leq 2 |\{y_0\}|$, we deduce that
	$
	\lambda \cdot |E_\lambda| \leq 2  \|f\|_1.
	$	
	Otherwise, if $y_0 \notin E_\lambda$, then $f(y) > \lambda$ necessarily holds for every $y \in E_\lambda \cap T_n$ which implies that $\lambda \cdot |E_\lambda| \leq \|f\|_1$.
	Finally, in the case $E_\lambda \cap T_n' \neq \emptyset$, we denote 
	\[
	j \coloneqq \max\big\{i \in [n] : \MN_{\TT_n} f(y_{i}')>\lambda\big\}. 
	\]
	We have
	\begin{displaymath}
	\sum_{y \in T_n}f(y)|\{y\}| > \lambda \cdot |\{y_{j}'\}| 
	\end{displaymath}
	and combining this with the estimate $|E_\lambda| \leq 2  |\{y_{j}'\}|$, we conclude that $\lambda \cdot |E_\lambda| \leq 2  \|f\|_1$ again follows. Hence, we have $\cc_{\rm w}(1, \TT_n) \leq 2$ for any $n \in \NN$ which implies that $\cc_{\rm w}(1, {\TT}_{2,1}) \lesssim 2 < \infty$.  
\end{proof} 

Now fix $p_0 \in (1, \infty)$. For any $n \in \NN$ consider $\TT_n = (T_n, \rho_n, \mu_n)$, where $T_n$, $\rho_n$, and $\mu_n$ are introduced as before with the aid of $\tau_n = \tau_{n,p_0}$ and $F_n(i)= F_{n, p_0}(i)$ defined as in the case of ${\SSS}_{2,p_0}$, by using the auxiliary numbers $c_n$, $e_n$ and sequences $(m_{n,j})_{j=1}^{e_n}$, $(s_{n,j})_{j=1}^{e_n}$. We denote by ${\TT}_{2,p_0}$ the space $\YY$ obtained by using Proposition~\ref{P2.2.1} for $\Lambda = \NN$ with $\YY_n = \TT_n$ for each $n \in \NN$. In the following lemma we describe the properties of $\MM_{{\TT}_{2,p_0}}$ and $\MN_{{\TT}_{2,p_0}}$.

\begin{lemma}
	Fix $p_0 \in (1, \infty)$ and let ${\TT}_{2,p_0}$ be the metric measure space defined above. Then the associated centered maximal operator $\MM_{{\TT}_{2,p_0}}$ is of strong type $(1,1)$, while the noncentered $\MN_{{\TT}_{2,p_0}}$ is not of strong type $(p_0,p_0)$, but is of weak type $(p_0,p_0)$.
\end{lemma}

\begin{proof}
	First, note that it is easy to verify that $\MM_{{\TT}_{2,p_0}}$ is of strong type $(1,1)$, by using the same argument as in the proof of Lemma~\ref{L2.3.5}. Now we show that $\MN_{{\TT}_{2,p_0}}$ is not of strong type $(p_0,p_0)$. Indeed, fix $n \in \NN$ and let $g \coloneqq \mathbf{1}_{\{x_n\}}$. Then $\| g \|_{p_0}= 1$ and for each $i \in [\tau_n]$ we have
	\[\MN_{\TT_n} g (y_{i}') \geq \frac{1}{1+1/\tau_n+m_{n,j(n,i)}} \geq \frac{1}{2(1+m_{n,j(n,i)})}
	\]
	and hence
	\begin{displaymath}
	\|\MN_{\TT_n} g \|_{p_0}^{p_0} \geq \sum_{j=1}^{e_n} \sum_{k=1}^{s_{n,j}}\frac{2^{-p_0} m_{n,j}}{(1+m_{n,j})^{p_0}}  = \sum_{j=1}^{e_n} \frac{2^{-p_0} s_{n,j}m_{n,j}}{(1+m_{n,j})^{p_0}} \geq \sum_{j=1}^{e_n} \frac{2^{2-j-p_0} n c_n }{(1+m_{n,j})^{p_0}} = \sum_{j=1}^{e_n} \frac{2^{1-p_0} nc_n}{(1+n)^{p_0}}.
	\end{displaymath}
	Thus, $\cc_{\rm s}(p_0, \TT_n)^{p_0} \geq e_n \frac{2^{1-p_0} nc_n}{(1+n)^{p_0}}$. Since $\lim_{n \rightarrow \infty} e_n = \infty$ and $\lim_{n \rightarrow \infty} \frac{n c_n}{(1+n)^{p_0}} = 1$, we are done.
	
	To complete the proof it suffices to show that $\MN_{{\TT}_{2,p_0}}$ is of weak type $(p_0,p_0)$. Fix $n \in \NN$ and estimate $\cc_{\rm w}(p_0, \TT_n)$ from above. Let $f \in L^{p_0}(\TT_n)$, $f \geq 0$, and consider $\lambda \in (0, \infty)$ such that $E_{2\lambda} (\MN_{\TT_n} f) \neq \emptyset$. We write $f = f_1 + f_2$, where $f_1 \coloneqq f \cdot \mathbf{1}_{T_n \setminus \{y_0\}}$ and $f_2 \coloneqq f \cdot \mathbf{1}_{\{y_0\}}$. Note that 
	\begin{displaymath}
	E_{2\lambda} (\MN_{\TT_n} f) \subset E_{\lambda} (\MN_{\TT_n} f_1) \cup E_{\lambda} (\MN_{\TT_n} f_2) \eqqcolon E_{\lambda,1} + E_{\lambda,2}.
	\end{displaymath}
	Applying H\"older's inequality we deduce that
	\begin{displaymath}
	\lambda^{p_0} |E_{\lambda,1}| \leq \| \MN_{\TT_n} f_1\|_{p_0}^{p_0} \leq \sum_{B \subset T_n} A_B(f_1)^{p_0} \, |B| \leq \sum_{B \subset T_n} \sum_{y \in B} f_1(y)^{p_0} |\{y\}| \leq 5  \|f_1\|_{p_0}^{p_0},
	\end{displaymath}
	where in the last inequality we use the fact that each $y \in T_n \setminus \{y_0\}$ belongs to at most five different balls $ B \subset T_n$.  
	Next, let $f_2(y_0) = \alpha \in (0,\infty)$ and assume that $E_{\lambda,2} \neq \emptyset$. Thus, we have $\alpha > \lambda$. If $E_{\lambda,2} \cap T_n' = \emptyset$, then $\lambda^{p_0} |E_{\lambda,2}| < 2 \|f_2\|_{p_0}^{p_0}$ follows. Otherwise, denote 
	\[
	r \coloneqq \min\Big\{j \in [e_n] : \frac{\alpha}{1+1/ \tau_n + m_{n,j}} > \lambda\Big\}.
	\] 
	Then we have
	\begin{displaymath}
	\lambda^{p_0} |E_{\lambda,2}| \leq \frac{2 \alpha^{p_0} \sum_{j=r}^{e_n} s_{n,j} m_{n,j}}{(1+\frac{1}{\tau_n} + m_{n,r})^{p_0}}  \leq \frac{\alpha^{p_0} \sum_{j=r}^{e_n} 2^{4-j} n c_n}{(1+m_{n,r})^{p_0}}  \leq \frac{2^{5-r} \alpha^{p_0} n c_n }{(1+m_{n,r})^{p_0}} \leq \frac{16 \alpha^{p_0} n c_n}{(1+n)^{p_0}} \leq 16  \|f_2\|_{p_0}^{p_0}.
	\end{displaymath}
	Consequently,
	\begin{displaymath}
	(2\lambda)^{p_0} |E_{2\lambda} (\MN_{\TT_n} f)| \leq 2^{p_0} \cdot 21 \|f\|_{p_0}^{p_0}.
	\end{displaymath}   
	Since $\cc_{\rm w}(p_0, \TT_n) \leq 2 \cdot 21^{1/p_0}$ for any $n \in \NN$, we conclude that $\cc_{\rm w}(p_0, {\TT}_{2,p_0}) \lesssim 2 \cdot 21^{1/p_0} < \infty$.
\end{proof}

Now we present a construction which will be applied to the third subtype of the second generation spaces. Fix $n_0 \in \NN$ and let $\tau = \tau_{n_0} =  (\tau_{n_0,i})_{i=1}^{n_0}$ be a given system of positive integers satisfying $\frac{\tau_{n_0,i}}{2^{i-1}} \in \mathbb{N}$. As previously, we include the dependence on $n_0 \in \NN$ in notation. Set 
\begin{displaymath}
\overline{T}_{n_0} \coloneqq  \overline{T}_{n_0}(\tau) \coloneqq  \Big\{y_{i,j}, \, y^\circ_{i,k}, \, y'_{i,k} : i \in [n_0], \, j \in [2^{i-1}], \, k \in [\tau_{n_0,i}] \Big\},
\end{displaymath}
where all elements $y_{i,j}, \, y^\circ_{i,k}, \, y'_{i,k}$ are different. We use auxiliary symbols for some subsets of $\overline{T}_{n_0}$:
\begin{align*}
T_{n_0}^\circ & \coloneqq  \Big\{y^\circ_{i,k} : i \in [n_0], \, k \in [\tau_{n_0,i}]   \Big\}, \\ 
T_{n_0}' & \coloneqq   \Big\{y'_{i,k} : i \in [n_0], \, k \in [\tau_{n_0,i}]   \Big\},
\end{align*}
for $i \in [n_0]$,
\begin{align*}
T_{n_0,i}^\circ & \coloneqq  \Big\{y^\circ_{i,k} : k \in [\tau_{n_0,i}]  \Big\}, \\ T'_{n_0,i} & \coloneqq  \Big\{y'_{i,k} : k \in [\tau_{n_0,i}] \Big\},
\end{align*}
and, for $i, i' \in [n_0]$, $i \leq i'$, and $j \in [2^{i-1}]$,
\begin{align*}
T^\circ_{n_0,i',i,j} & \coloneqq \Big\{y^\circ_{i',k} : k \in \Big(\frac{j-1}{2^{i-1}} \tau_{n_0,i'}, \frac{j}{2^{i-1}}\tau_{n_0,i'}\Big]\Big\}, \\
T'_{n_0,i',i,j} & \coloneqq \Big\{y'_{i',k} : k \in \Big(\frac{j-1}{2^{i-1}} \tau_{n_0,i'}, \frac{j}{2^{i-1}}\tau_{n_0,i'}\Big]\Big\}.
\end{align*}

We define the metric $\overline{\rho}_{n_0}$ on $\overline{T}_{n_0}$ determining the distance between two different elements $x$ and $y$ by the formula
\begin{displaymath}
\overline{\rho}_{n_0}(x,y) \coloneqq \overline{\rho}_{n_0, \tau}(x,y) \coloneqq \left\{ \begin{array}{rl}
1 & \textrm{if } \{x, y\} = \{y_{i,j},y^\circ_{i',k}\} \textrm{ and } x^\circ_{i',k} \in T^\circ_{n_0,i',i,j}, \\
1 & \textrm{if } \{x, y\} \subset \overline{T}_{n_0} \setminus (T_{n_0}^\circ \cup T_{n_0}'), \\
1 & \textrm{if } \{x, y\} = \{y^\circ_{i', k}, y'_{i',k}\}, \\
2 & \textrm{otherwise.} \end{array} \right. 
\end{displaymath}

Figure~\ref{F2.4} shows a model of the space $(\overline{T}_{n_0}, \overline{\rho}_{n_0})$ with $n_0=2$.

\begin{figure}[H]
	\begin{center}
	\begin{tikzpicture}
	[scale=.8,auto=left,every node/.style={circle,fill,inner sep=2pt}]	
	
	\node[label={[yshift=-1cm]$y_{1,1}$}] (c0) at (6,1) {};
	
	\node[label={[xshift=-0.05cm]$y^\circ_{1,1}$}] (c1) at (5,4)  {};
	\node[label={[yshift=-0.18cm]$y^\circ_{1,\tau_{2,1}}$}] (c2) at (7,4)  {};
	\node[dots] (c3) at (6,4)  {...};
	
	\node[label={[yshift=-0.03cm]$y'_{1,1}$}] (c4) at (5,7)  {};
	\node[label={[yshift=-0.18cm]$y'_{1,\tau_{2,1}}$}] (c5) at (7,7)  {};
	\node[dots] (c6) at (6,7)  {...};

	\node[label={[yshift=-1cm]$y_{2,1}$}] (l0) at (10,1) {};
	
	\node[label={[xshift=-0.05cm]$y^\circ_{2,1}$}] (l1) at (9,4)  {};
	\node[label={[xshift=-0.15cm, yshift=-0.3cm]$y^\circ_{2,\tau_{2,2}/2}$}] (l2) at (11,4)  {};
	\node[dots] (l3) at (10,4)  {...};
	
	\node[label={[yshift=-0.03cm]$y'_{2,1}$}] (l4) at (9,7)  {};
	\node[label={[xshift=-0.15cm, yshift=-0.3cm]$y'_{2,\tau_{2,2}/2}$}] (l5) at (11,7)  {};
	\node[dots] (l6) at (10,7)  {...};

	\node[label={[yshift=-1cm]$y_{2,2}$}] (r0) at (14,1) {};
	
	\node[label={[xshift=0.1cm, yshift=-0.5cm]$y^\circ_{2,\tau_{2,2}/2+1}$}] (r1) at (13,4)  {};
	\node[label={[xshift=0.2cm, yshift=-0.2cm]$y^\circ_{2,\tau_{2,2}}$}] (r2) at (15,4)  {};
	\node[dots] (r3) at (14,4)  {...};
	
	\node[label={[xshift=0.1cm, yshift=-0.5cm]$y'_{2,\tau_{2,2}/2+1}$}] (r4) at (13,7)  {};
	\node[label={[xshift=0.2cm, yshift=-0.2cm]$y'_{2,\tau_{2,2}}$}] (r5) at (15,7)  {};
	\node[dots] (r6) at (14,7)  {...};

	\foreach \from/\to in {l0/l1, l0/l2, c0/c1, c0/c2, r0/r1, r0/r2, l1/c0, l2/c0, r1/c0, r2/c0, l0/c0, l0/r0}
	\draw (\from) -- (\to);
	
	\draw (5,7) -- (5,5); \draw (5,4.5) -- (5,4);
	\draw (7,7) -- (7,5); \draw (7,4.5) -- (7,4);
	\draw (9,7) -- (9,5); \draw (9,4.5) -- (9,4);
	\draw (11,7) -- (11,5); \draw (11,4.5) -- (11,4);
	\draw (13,7) -- (13,5); \draw (13,4.5) -- (13,4);
	\draw (15,7) -- (15,5); \draw (15,4.5) -- (15,4);
	\draw (6,1) arc (-113.75:-66.25:10);
	\end{tikzpicture}
	\caption{The second generation space $(\overline{T}_{n_0}, \overline{\rho}_{n_0})$ with $n_0 = 2$.}
	\label{F2.4}
	\end{center}
\end{figure}
\noindent Note that we can explicitly describe any ball: for $i \in [n_0], \, j \in [2^{i-1}]$,
\begin{displaymath}
B(y_{i,j},s) = \left\{ \begin{array}{rl}
\{y_{i,j}\} & \textrm{for } 0 < s \leq 1, \\
\big( \overline{T}_{n_0} \setminus (T_{n_0}^\circ \cup T_{n_0}') \big) \, \cup \, \bigcup_{i' = i}^{n_0} T^\circ_{n_0,i',i,j}& \textrm{for } 1 < s \leq 2,  \\
\overline{T}_{n_0} & \textrm{for } 2 < s, \end{array} \right.
\end{displaymath} 
\noindent and, for $i' \in [n_0]$ and $k \in [\tau_{n_0,i}]$,
\begin{displaymath}
B(y^\circ_{i',k},s) = \left\{ \begin{array}{rl}
\{y^\circ_{i',k}\} & \textrm{for } 0 < s \leq 1, \\
\{y^\circ_{i',k}, y'_{i',k}\} \cup \{y_{i, j} \colon y^\circ_{i',k} \in T^\circ_{n_0,i',i,j}\}& \textrm{for } 1 < s \leq 2,  \\
\overline{T}_{n_0} & \textrm{for } 2 < s, \end{array} \right.
\end{displaymath} 
\noindent and 
\begin{displaymath}
B(y'_{i',k},s) = \left\{ \begin{array}{rl}
\{y'_{i',k}\} & \textrm{for } 0 < s \leq 1, \\
\{y^\circ_{i',k}, y'_{i',k}\} & \textrm{for } 1 < s \leq 2,  \\
\overline{T}_{n_0} & \textrm{for } 2 < s. \end{array} \right.
\end{displaymath}

Finally, we define the measure $\overline{\mu}_{n_0}$ on $\overline{T}_{n_0}$ by letting 
\begin{displaymath}
\overline{\mu}_{n_0}(\{y\}) \coloneqq  \overline{\mu}_{n_0,\tau, F, G, m}(\{y\}) \coloneqq \left\{ \begin{array}{rl}
F(i) & \textrm{if } y = y_{i,j} \textrm{ for some } i \in [n_0], \, j \in [2^{i-1}], \\
G & \textrm{if } y = y^\circ_{i,k} \textrm{ for some } i \in [n_0], \, k \in [\tau_{n_0,i}], \\
m i & \textrm{if } y = y'_{i,k} \textrm{ for some } i \in [n_0], \, k \in [\tau_{n_0,i}], \end{array} \right.
\end{displaymath}
where $0 <F \leq 1$ is a given function and $G, m$ are positive numbers satisfying $G \leq 1 / \sum_{i=1}^{n_0} \tau_{n_0,i}$ and $m \geq 2^{n_0}$.

We are ready to describe the third subtype of the second generation spaces. \\

\noindent {\bf Third subtype.} Now for any fixed $p_0 \in (1, \infty)$ we construct a space denoted by ${\TT}_{3,p_0}$ for which $\PSC({\TT}_{3,p_0}) = \PWC({\TT}_{3,p_0}) = \PRC({\TT}_{3,p_0}) = [1, \infty]$, $\PW({\TT}_{3,p_0}) = (p_0, \infty]$ and $\PR({\TT}_{3,p_0}) = [p_0, \infty]$. Note that the last two conditions imply $\PS({\TT}_{3,p_0}) = (p_0, \infty]$.

Fix $p_0 \in (1, \infty)$. For any $n \in \NN$ we choose $n_0 = n$ and consider $\TT_{n} = (\overline{T}_n, \overline{\rho}_n, \overline{\mu}_n)$, where $\overline{T}_n$, $\overline{\rho}_n$, and $\overline{\mu}_n$ are introduced as before with the aid of $m_n$, $\tau_n = (\tau_{n,i})_{i=1}^n$, $F_n$ defined as in the case of ${\SSS}_{3,p_0}$ and 
$
G_n = 2^{(1-n)/(p_0-1)} / \sum_{i=1}^{n}\tau_{n,i}.
$
Observe that we have $|\{y\}| \geq m_n \geq 2^n \geq |\overline{T}_n \setminus T_n'|$ for any $y \in T_n'$ and $|\{y\}| \geq |T_n^\circ|$ for any $y \in \overline{T}_n \setminus T_n^\circ$. We denote by ${\TT}_{3,p_0}$ the space obtained by using Proposition~\ref{P2.2.1} for $\Lambda = \NN$ with $\YY_n = \TT_n$ for each $n \in \NN$. The following lemma describes the properties of $\MM_{{\TT}_{3,p_0}}$ and $\MN_{{\TT}_{3,p_0}}$.

\begin{lemma}
	Fix $p_0 \in (1, \infty)$ and let ${\TT}_{3,p_0}$ be the metric measure space defined above. Then the associated centered maximal operator $\MM_{{\TT}_{3,p_0}}$ is of strong type $(1,1)$ while the noncentered operator $\MN_{{\TT}_{3,p_0}}$ is not of weak type $(p_0,p_0)$, but is of restricted weak type $(p_0,p_0)$.
\end{lemma} 

\begin{proof}
	First we show that $\MM_{{\TT}_{3,p_0}}$ is of strong type $(1,1)$. We fix $n \in \NN$ and restrict our attention to $\TT_n$. Let $f \in L^1(\TT_n)$, $f \geq 0$. The following estimates hold: for $i \in [n]$ and $j \in [2^{i-1}]$,
	\begin{displaymath}
	\MM_{\TT_n} f(y_{i,j}) \leq f(y_{i,j}) + 2 A_{\overline{T}_n \setminus T_n'}(f) + A_{\overline{T}_n}(f),
	\end{displaymath}
	\noindent and, for $i' \in [n]$ and $k \in [\tau_{n,i}]$,
	\begin{align*}
	\MM_{\TT_n} f(y^\circ_{i',k}) & \leq f(y^\circ_{i',k}) + \sup_{y \in \overline{T}_n \setminus T_n^\circ}f(y) + A_{\overline{T}_n}(f), \\
	\MM_{\TT_n} f(y'_{i',k}) & \leq f(y'_{i',k}) + A_{\{y^\circ_{i', k}, y'_{i',k}\}}(f) + A_{\overline{T}_n}(f).
	\end{align*}
	Observe that 
	\begin{align*}
	2 A_{\overline{T}_n \setminus T_n'}(f) \cdot |\overline{T}_n \setminus (T_n^\circ \cup T_n')| & \leq 2 \|f\|_1
	\end{align*}
	and
	\begin{align*}
	\sum_{i'=1}^n \sum_{k=1}^{\tau_{n,i}} A_{\{y^\circ_{i', k}, y'_{i',k}\}} \cdot |\{y'_{i',k}\}| & \leq \|f\|_1.
	\end{align*}  
	Moreover, since $|\{y\}| \geq |T_n^\circ|$ for any $y \in \overline{T}_n \setminus T_n^\circ$, we have
	\begin{displaymath}
	\sum_{i'=1}^n \sum_{k=1}^{\tau_{n,i}} \sup_{y \in \overline{T}_n \setminus T_n^\circ}f(y) \cdot |\{y^\circ_{i',k}\}| \leq \sup_{y \in \overline{T}_n \setminus T_n^\circ} f(y) \cdot |\{y\}| \leq \|f\|_1,
	\end{displaymath}
	and hence $\|\MM_{\TT_n} f\|_1 \leq 6 \|f\|_1$. Thus, $\cc_{\rm s}^{\rm c}(1, \TT_n) \leq 6$ for each $n \in \NN$ which gives $\cc_{\rm s}^{\rm c}(1, {\TT}_{3,p_0}) < \infty$. 
	
	In the next step we show that $\MN_{{\TT}_{3,p_0}}$ is not of weak type $(p_0, p_0)$. Indeed, fix $n \in \NN$ and take
	\begin{displaymath}
	g \coloneqq \sum_{i=1}^{n} \sum_{j=1}^{2^{i-1}} 2^{(n-i)/(p_0-1)} \mathbf{1}_{\{y_{i,j}\}} \in L^{p_0}(\TT_n).
	\end{displaymath}
	Then $\|g\|_{p_0}^{p_0} = 2^{n-1} n$ and
	$
	\MN_{\TT_n} g(y'_{i',k}) \geq A_{B(y^\circ_{i',k}, 3/2)}(g) \geq \frac{i'}{2 |\{y'_{i',k}\}|} = \frac{1}{2 m_n}
	$
	holds for each $y'_{i',k} \in T_n'$ which implies that $|E_{1/(4m_n)}(\MN_{\TT_n} g)| \geq |T_n'|$. Therefore,
	\begin{displaymath}
	\frac{\|\MN_{\TT_n} g \|_{p_0,\infty}^{p_0}}{\| g \|_{p_0}^{p_0}}
	\geq \frac{\sum_{i=1}^{n}\tau_{n,i} \, i \, m_n}{n 2^{n-1} (4m_n)^{p_0}} 
	= 2^{1-2p_0} \frac{\sum_{i=1}^{n}\tau_{n,i} \, i}{n \, m_n^{p_0-1} 2^n}
	= 2^{1-2p_0} \frac{\sum_{i=1}^{n}  \lfloor a_i  \rfloor}{n} 
	= 2^{1-2p_0} n^{p_0-1}.
	\end{displaymath}
	Thus, we obtain $\cc_{\rm w}(p_0, \TT_n)^{p_0} \geq 2^{1-2p_0} n^{p_0-1}$ which gives
	$
	\cc_{\rm w}(p_0, {\TT}_{3,p_0}) \gtrsim \lim_{n \rightarrow \infty} n^{1-1/p_0} = \infty.
	$
	
	In the last step we show that $\MN_{{\TT}_{3,p_0}}$ is of restricted weak type $(p_0,p_0)$. Fix $n \in \NN$ and estimate $\cc_{\rm r}(p_0, \TT_n)$ from above. Let $U \subset \overline{T}_n$, $U \neq \emptyset$, and $\lambda \in (0, \infty)$. Our goal is to estimate
	\begin{equation}\label{2.3.5}
	\lambda^{p_0 } | E_\lambda| / |U|,
	\end{equation}
	where $E_\lambda \coloneqq E_\lambda(\MN_{\TT_n} \mathbf{1}_{U})$. Denote $U^\circ = U \cap T_n^\circ$ and $U' = U \cap T_n'$. For any $y \in \overline{T}_n$ we have
	\begin{displaymath}
	\MN_{\TT_n} \mathbf{1}_{U} (y) \leq 3  \max \big\{\MN_{\TT_n} \mathbf{1}_{U^\circ}(y), \ \MN_{\TT_n} \mathbf{1}_{U'}(y), \ \MN_{\TT_n}\mathbf{1}_{U \setminus (U^\circ \cup U')}(y) \big\}. 
	\end{displaymath}
	Thus, it suffices to find a bound for \eqref{2.3.5} with $U$ being a subset of $T_n^\circ$, $T_n'$, or $\overline{T}_n \setminus (T_n^\circ \cup T_n')$. Moreover, in each case we may assume that $A_{\overline{T}_n}(\mathbf{1}_{U}) < \lambda < 1$.
	
	First, consider $U \subset T_n'$ and assume that $E_\lambda \neq \emptyset$. Thus, we have $ |E_\lambda| \leq 2 |E_\lambda \cap T_n'|$. Observe that there is no ball $B \subsetneq \overline{T}_n$ containing two different points from $T_n'$. Therefore, if $y \in E_\lambda \cap T_n'$, then $y \in U$ and hence $\lambda^{p_0} | E_\lambda| \leq \lambda | E_\lambda| \leq  2\lambda |E_\lambda \cap T_n'| \leq 2  |U|$ follows.
	
	Now take $U \subset T_n^\circ$. First, assume that $E_\lambda \cap T_n' \neq \emptyset$. If $y'_{i',k} \in E_\lambda \cap T_n'$, then $y^\circ_{i',k} \in U$ and $\lambda < |\{y^\circ_{i',k}\}| / |\{y'_{i',k}\}|$. Consequently, 
	\begin{displaymath}
	\lambda^{p_0} | E_\lambda| \leq  \lambda | E_\lambda| \leq 2\lambda |E_\lambda \cap T_n'| \leq 2 |U|
	\end{displaymath}
	again follows.
	Next, assume that $E_\lambda \subset \overline{T}_n \setminus T_n'$ and $E_\lambda \not\subset T_n^\circ$. In this case we have $|E_\lambda| \leq 2  |E_\lambda \setminus (T_n^\circ \cup T_n')|$. Moreover, the volume of any ball $B$ such that $B \setminus (T_n^\circ \cup T_n') \neq \emptyset$ and $B \cap T_n^\circ \neq \emptyset$ is greater than $|\overline{T}_n \setminus (T_n^\circ \cup T_n')|$. Therefore, if there exists $y \in E_\lambda \setminus (T_n^\circ \cup T_n')$, then 
	\begin{displaymath}
	\lambda < \MN_{\TT_n} \mathbf{1}_{U}(y) \leq |U| / |\overline{T}_n \setminus (T_n^\circ \cup T_n')|
	\end{displaymath}
	which gives $\lambda^{p_0 } | E_\lambda| \leq \lambda |\overline{T}_n \setminus T_n'| \leq 2 |U|$. Assume the last case $E_\lambda \subset T_n^\circ$. Since there are no balls $B \subset T_n^\circ$ containing two different points from $T_n^\circ$, we have $E_\lambda = U$ which gives $\lambda^{p_0} |E_\lambda| \leq |U|$.
	
	Finally, take $U \subset \overline{T}_n \setminus (T_n^\circ \cup T_n')$ and assume that $E_\lambda \neq \emptyset$. First, consider the case 
	$E_\lambda \cap T_n' = \emptyset$. Then we have $|E_\lambda| \leq 2 |E_\lambda \setminus (T_n^\circ \cup T_n')|$. If $\lambda \leq A_{\overline{T}_n \setminus (T_n^\circ \cup T_n')}(\mathbf{1}_{U})$, then
	\begin{displaymath}
	\lambda^{p_0 } | E_\lambda| \leq \lambda | E_\lambda| \leq 2 A_{\overline{T}_n \setminus (T_n^\circ \cup T_n')}(\mathbf{1}_{U}) |E_\lambda \setminus (T_n^\circ \cup T_n')| \leq 2 |U|.
	\end{displaymath}
	Otherwise, assume that $\lambda > A_{\overline{T}_n \setminus (T_n^\circ \cup T_n')}(\mathbf{1}_{U})$. Suppose that there exists $y \in E_\lambda \setminus (T_n^\circ \cup T_n')$. Since the volume of each ball $B \ni y$ with radius $s > 1$ is greater than $|\overline{T}_n \setminus (T_n^\circ \cup T_n')|$, we deduce that $A_{B}(\mathbf{1}_{U}) \leq \lambda$. This means that $y \in U$. Consequently, 
	\begin{displaymath}
	\lambda^{p_0 } | E_\lambda| \leq \lambda | E_\lambda| \leq 2 |E_\lambda \setminus (T_n^\circ \cup T_n')| = 2 |U|.
	\end{displaymath}
	Now consider the case $E_\lambda \cap T_n' \neq \emptyset$. We have $|E_\lambda| \leq 2 |E_\lambda \cap T_n'|$. Moreover, if $B \subsetneq \overline{T}_n$ satisfies $B \cap T_n' \neq \emptyset$, then for each fixed $i \in [n]$ the ball $B$ contains at most one of the points $y_{i,1}, \dots, y_{i,2^{i-1}}$. Consequently, for each $y_0 \in T_n'$ we have $\MN_{\TT_n} \mathbf{1}_{U}(y_0) \leq C \cdot \max_{y \in U}\{ \MN_{\TT_n} \mathbf{1}_{\{y\}}(y_0)\}$ with $C \coloneqq \sum_{i=0}^{\infty} 2^{-i(p_0-1)}$. Thus,
	\begin{displaymath}
	E_\lambda \cap T_n' \subset \bigcup_{y \in U} E_{\lambda/C}(\MN_{\TT_n} \mathbf{1}_{\{y\}}) \cap T_n'
	\end{displaymath}
	and hence it suffices to estimate properly the quantity
	$
	\lambda^{p_0 } | E_\lambda \cap T_n'| / |U|
	$
	assuming $U = \{y_{i,j}\} \subset \overline{T}_n \setminus (T_n^\circ \cup T_n')$ and $\lambda > A_{\overline{T}_n}(\mathbf{1}_{U})$. In this case, for each $y \in T'_{n,i',i,j}$, $i' \geq i$, we obtain
	\[
	\MN_{\TT_n} \mathbf{1}_{U} (y) = A_{B(y, 3/2)}(\mathbf{1}_{U}) \leq 2^{(i-n)/(p_0-1)} / (m_n i'),
	\]
	while for $y \in T_n' \setminus \bigcup_{i'=i}^{n} T'_{n,i',i,j}$ we have $\MN_{\TT_n} \mathbf{1}_{U}(y) = A_{\overline{T}_n}(\mathbf{1}_{U}) < \lambda$. Since
	\begin{align*}
	\frac{2^{(i-n)p_0/(p_0-1)}}{ (m_n \, i')^{p_0}} \sum_{l=i}^{i'} |T'_{n,l,i,j}| &= 
	\frac{2^{(i-n)p_0 / (p_0-1)}}{( m_n \, i')^{p_0}}    \sum_{l=i}^{i'} \frac{|T'_{n,l}|}{2^{i-1}} \\
	& = \frac{2^{(i-n)p_0 / (p_0-1)}}{(i')^{p_0}}  \sum_{l=i}^{i'} \frac{\tau_{n,l} \, l}{m_n^{p_0-1} \, 2^{i-1}} \\
	& = \frac{2^{(i-n)p_0 / (p_0-1)}}{(i')^{p_0}} \sum_{l=i}^{i'} 2^{n-i+1} \lfloor a_l \rfloor \\ & \leq 2^{((i-n)/(p_0-1)) + 1} \, \frac{\sum_{l=1}^{i'} a_l}{(i')^{p_0}} = 2  |U|,
	\end{align*}
	we conclude that $\lambda^{p_0} |E_\lambda \cap T_n'| \leq 2  |U|$ holds. 
	
	We thus have $\cc_{\rm r}(p_0, \TT_n) \lesssim C(p_0)$ independently of $n \in \NN$ which gives $\cc_{\rm r}(p_0, {\TT}_{3,p_0}) < \infty$.	
\end{proof}

\section{Proof of the main result}\label{S2.4}
\begin{proof}[Proof of Theorem~\ref{thm:2.1.2}]
Let $P_{\rm s}^{\rm c}$, $P_{\rm w}^{\rm c}$, $P_{\rm r}^{\rm c}$, $P_{\rm s}$, $P_{\rm w}$, and $P_{\rm r}$ be such that the conditions \ref{2i}--\ref{2iii} hold. We consider three cases. If $\PSC = \PS$, $\PWC = \PW$, and $\PRC = \PR$ hold, then $\ZZZ$ may be chosen to be a first generation space. If, in turn, we have $\PSC = \PWC = \PRC = [1, \infty]$, but $\PS \neq [1, \infty]$, then $\ZZZ$ may be chosen to be a second generation space. Finally, if none of these cases occurs, then we can find spaces $\SSS$ and $\TT$ of the first and second generations, respectively, for which
\begin{itemize}
	\item $\PSC(\SSS) = \PS(\SSS) = \PSC$, $\PWC(\SSS) = \PW(\SSS) = \PWC$, and $\PRC(\SSS) = \PR(\SSS) = \PRC$,
	\item $\PSC(\TT) = \PWC(\TT) = \PRC(\TT) = [1, \infty]$, $\PS(\TT) = \PS$, $\PW(\TT) = \PW$, and $\PR(\TT) = \PR$.
\end{itemize}
We let $\ZZZ$ to be the space obtained by using Proposition~\ref{P2.2.1} for $\Lambda = \{1, 2\}$ with $\YY_1 = \SSS$ and $\YY_2 = \TT$. One can easily see that $\ZZZ$ has the following properties:
\begin{itemize}
	\item $\PSC(\ZZZ) = \PSC(\SSS) \cap \PSC(\TT) = \PSC \cap [1, \infty] = \PSC$,
	\item $\PWC(\ZZZ) = \PWC(\SSS) \cap \PWC(\TT) = \PWC \cap [1, \infty] = \PWC$,
	\item $\PRC(\ZZZ) = \PRC(\SSS) \cap \PRC(\TT) = \PRC \cap [1, \infty] = \PRC$,
	\item $\PS(\ZZZ) = \PS(\SSS) \cap \PS(\TT) = \PSC \cap \PS = \PS$,
	\item $\PW(\ZZZ) = \PW(\SSS) \cap \PW(\TT) = \PWC \cap \PW = \PW$,
	\item $\PR(\ZZZ) = \PR(\SSS) \cap \PR(\TT) = \PRC \cap \PR = \PR$,
\end{itemize} 
and therefore it may be chosen to be the expected space.

Finally, in view of Remark~\ref{R2} each space $\ZZZ$ specified above is nondoubling.
\end{proof}
\chapter{Modified maximal operators}\label{chap3}
\setstretch{1.0}

In the following chapter we investigate the strong and weak type $(p,p)$ inequalities for the modified maximal operators, centered $\MMK$ and noncentered $\MNK$. Here $\kappa \in [1, \infty)$ is a modification parameter and the difference between these operators and the classical ones is that the measure of the ball $\kappa B$ instead of $B$ occurs in the averages. Roughly speaking, for larger $\kappa$ the operators are smaller which makes them easier to be bounded between certain function spaces. On the other hand, the modification is so small that $\MMK$ and $\MNK$ can successfully play the role of $\MM$ and $\MN$ in many situations. This idea is due to Nazarov, Treil and Volberg \cite{NTV2}, who introduced and analyzed the centered operator $\MM_3$ in the context of arbitrary metric measure spaces.  

The major part of the research concerning the strong and weak type $(p,p)$ inequalities for $\MMK$ and $\MNK$ was devoted to the case $p=1$, especially to the weak type $(1,1)$ boundedness. In addition to the aforementioned paper \cite{NTV2}, there were several articles focused on the general description of all situations in which the weak type $(1,1)$ inequality must occur (see, for example, \cite{Sa, Te}). Finally, it was proven in \cite{St1} that $\MMK$ and $\MNK$ are of weak type $(1,1)$ for $\kappa \in [2, \infty)$ and $\kappa \in [3, \infty)$, respectively, in case of any metric measure space with a measure that is finite on bounded sets. Moreover, it is known that these ranges of the parameter $\kappa$ are sharp in the sense that for any $\kappa \in [1,2)$ (or $\kappa \in [1,3)$) one can find a metric measure space such that $\MMK$ (or $\MNK$) is not of weak type $(1,1)$. The examples we mention are given in \cite{Sa, St2} (see also \cite{SS}, where certain details justifying the correctness of the construction described in \cite{Sa} are given).

Our intention is to broaden the scope of research by taking into account both mentioned types of inequalities for the full range of the parameter $p$. More precisely, for a given space $\XX$ and each $\kappa \in [1, \infty)$ we introduce $\PKSC(\XX)$, $\PKWC(\XX)$, $\PKS(\XX)$, and $\PKW(\XX)$, the sets of all parameters $p \in [1,\infty]$ for which the associated operators, centered $\MMKX$ or noncentered $\MNKX$, are of strong or weak type $(p,p)$, respectively. Among others, we study the interrelations between $\PKSC(\XX)$, $\PKWC(\XX)$, $\PKS(\XX)$, and $\PKW(\XX)$, and illustrate many possible configurations of them by using structures similar to those occurring in Chapter~\ref{chap2}. It is worth noting at this point that our constructions are largely inspired by the examples given in \cite{St2}.     

The organization of this chapter is as follows. In Section~\ref{S3.1} we collect basic information about modified maximal operators. We also explain how to adapt the space combining technique described in Section~\ref{S2.2} to the current situation. Section~\ref{S3.2} is devoted to studying the case of fixed $\kappa$. We formulate the main result and prove it by using some variants of the spaces introduced in Chapter~\ref{chap2} and some new structures, the so-called {\it segment-type spaces}. In Section~\ref{S3.3} we investigate the case of varying $\kappa$ which turns out to be much more complex. In particular, the space combining technique is very extensively used here. Finally, in Section~\ref{S3.4} some further remarks and additional examples are given.

\setstretch{1.15}

\section{Preliminaries}\label{S3.1}

Let $\kappa \in [1, \infty)$. For a given metric measure space $\XX = (X, \rho, \mu)$ we define the associated {\it modified Hardy--Littlewood maximal operators}, centered $\MMK$ and noncentered $\MNK$, by
\begin{displaymath}
\MMK f(x) \coloneqq \MMKX f(x) \coloneqq \sup_{s \in (0,\infty)} \frac{1}{|B(x,\kappa s)|} \int_{B(x,s)} |f| \, {\rm d}\mu, \qquad x \in X,
\end{displaymath}
and
\begin{displaymath}
\MNK f(x) \coloneqq \MNKX f(x) \coloneqq \sup_{B \ni x} \frac{1}{|\kappa B|} \int_B |f| \, {\rm d} \mu , \qquad x \in X,
\end{displaymath}
respectively. Here $\kappa B$ refers to the ball concentric with $B$ and
of radius $\kappa$ times that of $B$. Note that, in general, neither the center nor the radius of a ball as a set are uniquely determined. Moreover, in the case $\kappa \in (1, \infty)$ it is possible that for some $x, y \in X$ and $s_1,s_2 \in (0,\infty)$ we have $B(x, s_1) = B(y, s_2)$, while $B(x, \kappa s_1) \neq B(y, \kappa s_2)$. If $\kappa=1$, then the modified operators coincide with the standard Hardy--Littlewood maximal operators, noncentered and centered, and hence we will write shortly $\MM$ or $\MN$ instead of $\mathcal{M}_{1}^{\rm c}$ or $\mathcal{M}_{1}$. As usual, the balls $B$ such that $|B| = 0$ or $|\kappa B| = \infty$ are omitted in the definitions of $\MMK$ and $\MNK$. However, unless otherwise stated, in this chapter we assume that the measure of each ball is finite and strictly positive.

Denote by $\SXCKP$ the best constant in the strong type $(p,p)$ inequality for the operator $\MMK$ (if $\MMK$ is not of strong type $(p,p)$, then we write $\SXCKP = \infty$). Analogously, we define $\WXCKP$, $\SXKP$, and $\WXKP$ (the meaning of each of these symbols should be clear to the reader). Below we present a variant of Proposition~\ref{P2.2.1} which allows us to use the space combining technique in an effective way when dealing with the modified operators. 

\begin{proposition}\label{P3.1.1}
	Let $\kappa_0 \in [1, \infty)$. Fix $\emptyset \neq \Lambda \subset \NN$ and for each $n \in \Lambda$ let $\YY_n = (Y_n, \rho_n, \mu_n)$ be a given metric measure space satisfying $\mu_n(Y_n) < \infty$ and ${\rm diam}(Y_n) < \infty$. Denote by $\YY$ the space constructed as in Section~\ref{S2.2} with the only modification that $\kappa_0 + 1$ instead of $2$ is used in \eqref{2.2.1}. Then for each $\kappa \in [1, \kappa_0]$ and $p \in [1, \infty]$ we have the following estimates:
	\begin{align*}
	\cc_{\rm s, \YY}^{\rm c}(\kappa, p) \simeq \sup_{n \in \Lambda} \cc_{\rm s, \YY_n}^{\rm c}(\kappa, p), \quad
	\cc_{\rm w, \YY}^{\rm c}(\kappa, p) \simeq \sup_{n \in \Lambda} \cc_{\rm w, \YY_n}^{\rm c}(\kappa, p), \\
	\cc_{\rm s, \YY}(\kappa, p) \simeq \sup_{n \in \Lambda} \cc_{\rm s, \YY_n}(\kappa, p), \quad
	\cc_{\rm w, \YY}(\kappa, p) \simeq \sup_{n \in \Lambda} \cc_{\rm w, \YY_n}(\kappa, p). 
	\end{align*} 
\end{proposition}

\begin{proof}
The proof is identical to the proof of Proposition~\ref{P2.2.1} and hence it is omitted.
\end{proof}

Two comments are in order. First, whenever we want to apply Proposition~\ref{P3.1.1} later on, we omit the details related to the proper indexing of the component spaces. We do not even specify $\Lambda$. The only important thing is that we always use at most countably many spaces. Second, we have the following analogue of Remark~\ref{R2}. 
\begin{remark}\label{R3}
If at least one space from the family $\{\YY_n : n \in \Lambda\}$ is nondoubling or $\Lambda$ is infinite, then $\YY$ is nondoubling.
\end{remark}  

\section{Results for fixed modification parameter}\label{S3.2}

In this section we assume that the parameter $\kappa \in [1, \infty)$ is fixed. For a given space $\XX$ we introduce $\PKSC(\XX)$ and $\PKWC(\XX)$, the sets consisting of all parameters $p \in [1, \infty]$ for which the associated centered operator $\MMKX$ is of strong or weak type $(p,p)$, respectively. Similarly, let $\PKS(\XX)$ and $\PKW(\XX)$ consist of all parameters $p \in [1, \infty]$ for which the associated noncentered operator $\MNKX$ is of strong or weak type $(p,p)$, respectively. If $\kappa = 1$, then we write shortly $\PSC(\XX)$ instead of $P_{1,\rm s}^{\rm c}(\XX)$ and so on. 

Below we list the conditions that the four sets must satisfy in general. As in Chapter~\ref{chap2}, we drop the dependence on $\XX$ for a moment and replace $\PKSC(\XX)$, $\PKWC(\XX)$, $\PKS(\XX)$, and $\PKW(\XX)$ with $\PKSC$, $\PKWC$, $\PKS$, and $\PKW$, respectively.

\begin{observation}
The following assertions hold for each metric measure space $\XX$ such that the associated measure is finite on bounded sets:
\begin{enumerate}[label={\rm (\roman*)}]
	\item \label{3i} Each of the sets $\PKSC$, $\PKWC$, $\PKS$, and $\PKW$ is of the form $\{\infty\}$, $[p_0, \infty]$ or $(p_0,\infty]$, for some $p_0 \in [1, \infty)$.
	
	\item \label{3ii} We have the following inclusions:
	\begin{displaymath}
	\PKS \subset \PKSC, \quad \PKW \subset \PKWC, \quad \PKSC \subset \PKWC \subset \overline{\PKSC}, \quad \PKS \subset \PKW \subset \overline{\PKS},
	\end{displaymath}
	where $\overline{E}$ denotes the closure of $E$ in the usual topology of $\mathbb{R} \cup \{ \infty \}$.
	
	\item \label{3iii} If $\kappa \in [2, \infty)$, then $\PKWC = [1, \infty]$.
	
	\item \label{3iv} If $\kappa \in [3, \infty)$, then $\PKW = [1, \infty]$.
\end{enumerate}
\end{observation}

\noindent Indeed, the condition \ref{3i} is a natural consequence of the $L^\infty$-boundedness of the considered operators and the Marcinkiewicz interpolation theorem, while the condition~\ref{2ii} is a consequence of both the Marcinkiewicz interpolation theorem and several obvious implications between different types of inequalities for different operators.
Finally, the conditions \ref{3iii} and \ref{3iv} must be satisfied in view of the results obtained in the literature (see \cite{NTV2, Sa, St1, Te}).

Our goal is to show that \ref{3i}--\ref{3iv} are the only conditions that the four sets considered above satisfy in general. Thus, the following theorem can be viewed as an analogue of Theorem~\ref{thm:2.1.2} stated for the modified operators.

\begin{theorem}\label{T3.2.2}
	Fix $\kappa \in [1, \infty)$. Let $\PKSC$, $\PKWC$, $\PKS$, and $\PKW$ be arbitrary sets satisfying \ref{3i}--\ref{3iv}. Then there exists a (nondoubling) metric measure space $\ZZZ$ for which the associated modified Hardy--Littlewood maximal operators, centered $\MM_{\kappa, \ZZZ}$ and noncentered $\MN_{\kappa, \ZZZ}$, satisfy the following properties:
	\begin{itemize}
		\item $\MM_{\kappa, \ZZZ}$ is of strong type $(p,p)$ if and only if $p \in \PKSC$,
		\item $\MM_{\kappa, \ZZZ}$ is of weak type $(p,p)$ if and only if $p \in \PKWC$,
		\item $\MN_{\kappa, \ZZZ}$ is of strong type $(p,p)$ if and only if $p \in \PKS$,
		\item $\MN_{\kappa, \ZZZ}$ is of weak type $(p,p)$ if and only if $p \in \PKW$.
	\end{itemize} 
\end{theorem}
We will prove Theorem~\ref{T3.2.2} in Subsection~\ref{S3.2.3}. To do that we need a few auxiliary lemmas which will be formulated in Subsections~\ref{S3.2.1}~and~\ref{S3.2.2}. The analysis will be made separately for each of the following three cases: $\kappa \in [1, 2)$, $\kappa \in [2, 3)$, and $\kappa \in [3, \infty)$. We also emphasize that the first two cases are the most interesting ones. Indeed, if $\kappa \in [3, \infty)$, then we have only three possibilities depending on whether $\MMK$ and $\MNK$ are of strong type $(1,1)$ or not.

\subsection{First and second generation spaces}\label{S3.2.1}

To prove Theorem~\ref{T3.2.2} we use some of the results obtained in Chapter~\ref{chap2}. Recall that the four types of structures have been introduced there, namely $\SSS = (S, \rho, \mu)$, $\overline{\SSS} = (\overline{S}, \overline{\rho}, \overline{\mu})$, $\TT = (T, \rho, \mu)$, and $\overline{\TT} = (\overline{T}, \overline{\rho}, \overline{\mu})$. Since the restricted weak type inequalities are not considered in this chapter, we may focus only on the spaces $\SSS$ and $\TT$ here. 

Note that for any space of type $\SSS$ the associated metric $\rho$ takes only two nonzero values, namely $1$ and $2$. Hence, in this case, for any $\kappa \in [1,2)$ the operators $\MM_{\kappa, \SSS}$ and $\MN_{\kappa, \SSS}$ coincide with $\MM_\SSS$ and $\MN_\SSS$, respectively. The key point here is that if $\kappa \in [1,2)$, then we can find $s>1$ such that $\kappa s \leq 2$. Moreover, the same is true if an arbitrary space of type $\TT$ is considered instead. Thus, in Proposition~\ref{P2.2.1} one can consider the modified operators $\MMK$ and $\MNK$ instead of $\MM$ and $\MN$, and the conclusion does not change. Namely, we have $\cc_{\rm s, \YY}^{\rm c}(\kappa, p) \simeq \sup_{n \in \Lambda} \cc_{\rm s, \YY_n}^{\rm c}(\kappa, p)$ and so on. Consequently, one clearly gets that for each space $\ZZZ$ obtained in Chapter~\ref{chap2} by applying Proposition~\ref{P2.2.1} to a certain family of spaces of types $\SSS$ or $\TT$, the following identities hold:
\begin{displaymath}
P_{\kappa,\rm s}^{\rm c}(\ZZZ)=P_{\rm s}^{\rm c}(\ZZZ), \qquad P_{\kappa,\rm s}(\ZZZ)=P_{\rm s}(\ZZZ), \qquad P_{\kappa,\rm w}^{\rm c}(\ZZZ)= P_{\rm w}^{\rm c}(\ZZZ), \qquad P_{\kappa,\rm w}(\ZZZ)=P_{\rm w}(\ZZZ).
\end{displaymath}

In the case $\kappa \in [2,3)$ the situation is a bit different. The first change is that this time we should use Proposition~\ref{P3.1.1} with $\kappa_0 = \kappa$ instead of Proposition~\ref{P2.2.1} in order to combine spaces in an effective way. The second change is more crucial. Namely, if $\kappa \in [2,3)$, then for any ball $B \subset S$ (or $B \subset T$) containing at least two points the ball $\kappa B$ coincides with the whole space. This fact makes both modified maximal operators trivially bounded on $L^1(\SSS)$ (or $L^1(\TT)$) with their norms not larger than $2$. However, a slight modification of the metric used in the construction of $\TT$ will allow us to obtain more subtle results.

Let $(T, \rho, \mu)$ be a given space of type $\TT$. We define the metric $\rho'$ determining the distance between two different elements $x, y \in T$ by the formula
\begin{displaymath}
\rho'(x,y) \coloneqq \left\{ \begin{array}{rl}
1 & \textrm{if } \rho(x,y)=1,  \\
2 & \textrm{if there exists } z \in T \textrm{ such that } \rho(x,z)= \rho(y,z)=1, \\
3 & \textrm{otherwise.} \end{array} \right. 
\end{displaymath}
We emphasize that $\rho'$ is well-defined. Indeed, it can easily be shown that there is no set $\{x, y, z\} \subset T$ satisfying
\begin{displaymath}
\rho(x,y) =  \rho(x, z) = \rho(y,z) = 1,
\end{displaymath} 
and thus the first two conditions in the definition of $\rho'$ cannot happen at the same time. We denote by $\TT'$ the space $(T, \rho', \mu)$. In the following lemma we describe what happens if one uses $\TT'$ instead of $\TT$ to obtain the corresponding variant of the second generation spaces.   
 
\begin{lemma}\label{L3.2.3}
	Fix $\kappa \in [2, 3)$ and consider a second generation space of either first or second subtype $\TT_0$. Let $\{\TT_n = (T_n, \rho_n, \mu_n) : n \in \NN\}$ be the family of component spaces used to define $\TT_0$. We denote by $\TT_0'$ the space obtained by using Proposition~\ref{P3.1.1} with $\kappa_0=\kappa$ for the family $\{\TT_n' = (T_n, \rho'_n, \mu_n) : n \in \NN\}$, where $\rho'_n$ is the modification of $\rho_n$ described above. Then we have $\PKSC(\TT_0') = \PKWC(\TT_0') = [1, \infty]$, while $\PKS(\TT_0') = \PKS(\TT_0)$ and $\PKW(\TT_0') = \PKW(\TT_0)$.
\end{lemma}

\begin{proof}
	Fix $n \in \NN$ and notice that $L^1(\TT_n)$ and $L^1(\TT_n')$ are equal as Banach spaces. Moreover, we claim that for any $f \in L^1(\TT_n)$ we have $\MM_{\kappa,\TT_n'} f \leq \MM_{\TT_n} f$ and $\MN_{\kappa,\TT_n'}(f) \leq \MN_{\TT_n}(f)$. Indeed, regarding the centered operators suppose that $f \geq 0$ and fix $y_0 \in T_n$. If $s \leq 2$, then we have the inclusion $B_{\rho'}(y_0,\kappa s) \supset B_{\rho'}(y_0,s) = B_{\rho}(y_0,s)$, which implies that
	\begin{displaymath}
	\frac{1}{|B_{\rho'}(y_0,\kappa s)|} \sum_{y \in B_{\rho'}(y_0,s)} f(y) \cdot |\{y\}|  \leq \frac{1}{|B_{\rho}(y_0,s)|} \sum_{y \in B_{\rho}(y_0,s)} f(y)  \cdot |\{y\}| \leq \MM_{\TT_n} f(y_0).
	\end{displaymath}
	On the other hand, if $s > 2$, then we have $B_{\rho'}(y_0,\kappa s) = T_n$, which implies that
	\begin{displaymath}
	\frac{1}{|B_{\rho'}(y_0,\kappa s)|} \sum_{y \in B_{\rho'}(y_0,s)} f(y) \cdot  |\{y\}|  \leq \frac{1}{|T_n|} \sum_{y \in T_n} f(y) \cdot |\{y\}| \leq \MM_{\TT_n} f(y_0).
	\end{displaymath}
	This gives $\MM_{\kappa,\TT_n'} f \leq \MM_{\TT_n} f$ and the second claimed estimate may be verified analogously. Consequently, we obtain the following identities and inclusions:
	\begin{displaymath}
	\PKSC(\TT_0') = \PKWC(\TT_0') = [1, \infty], \quad \PKS( \TT_0') \supset \PS(\TT_0), \quad \PKW(\TT_0') \supset \PW(\TT_0).
	\end{displaymath}
	
	Now it remains to show that if $\MN_{\TT_0}$ is not of strong (or weak) type $(p,p)$ for some $p \in [1, \infty)$, then $\MN_{\kappa, \TT_0'}$ fails to be of strong (or weak) type $(p,p)$ for the same $p$. To this end, we recall briefly the argument that was used in Chapter~\ref{chap2} to obtain a certain property of $\MN_{\TT_0}$ and then convince the reader that the situation is very similar in the context of $\MN_{\kappa, \TT_0'}$ instead. For the sake of brevity we describe only the case related to the strong type $(p,p)$ inequalities. 
	
	Notice that each time when it was shown that the noncentered operator associated with the second generation space $\TT$ is not of strong type $(p,p)$, the functions $g_n = \mathbf{1}_{\{y_0\}} \in L^p(\TT_n)$, $n \in \mathbb{N}$, were considered. Then, the maximal functions $\MN_{\TT_n} g_n$ were estimated from below by:
	\begin{itemize}
		\item the average value of $g_n$ on the ball $B$ centered at $y_{i}$ with $s = \frac{3}{2}$ (denoted by $A_{B_{\rho_n}(y_{i},\frac{3}{2})}(g_n)$) for the points $y_{i}'$ with $i \in [\tau_n]$,
		\item 0 for all other points,	
	\end{itemize}   
	and finally it turned out that 
	\begin{displaymath}
	\lim_{n \rightarrow \infty} \frac{\|\MN_{\TT_n}(g_n)\|_p^p}{\|g_n\|_p^p} \geq \lim_{n \rightarrow \infty} \frac{ \sum_{i=1}^{\tau_n} \big( A_{B_{\rho_n}(y_{i},\frac{3}{2})}(g_n) \big)^p \, |\{y_{i}'\}|
	}{\|g_n\|_p^p} = \infty.
	\end{displaymath}
	
	Let us assume that the estimate stated above holds for some $p \in [1, \infty)$. Take $s>1$ such that $\kappa s \leq 3$ and observe that $B_{\rho'_n}(y_{i},s) = B_{\rho_n}(y_{i},\frac{3}{2})$ and 
	\begin{displaymath}
	|B_{\rho_n'}(y_{i},\kappa s)| = |\{y_{i}'\} \cup (T_n \setminus T_n') | \leq 2  |B_{\rho_n} (y_{i}, 3/2 )|.
	\end{displaymath}
	This implies that $\MN_{\kappa,\TT_n'}(g_n)(y_{i}') \geq \frac{1}{2} A_{B_{\rho_n}(y_{i},\frac{3}{2})}(g_n)$ and hence
	\begin{displaymath}
	\lim_{n \rightarrow \infty} \frac{\|\MN_{\kappa,\TT_n'}(g_n)\|_p}{\|g_n\|_p} = \infty
	\end{displaymath}
	as well. In view of Proposition~\ref{P3.1.1}, we obtain that $\MN_{\kappa, \TT_0'}$ is not of strong type $(p,p)$.  
\end{proof}

\subsection{Segment-type spaces}\label{S3.2.2}

Now we turn our attention to certain specific situations in which $\MMK$ or $\MNK$ are not of strong type $(1,1)$ for some $\kappa \in [2, \infty)$ or $\kappa \in [3, \infty)$, respectively. We present a construction which allows us to introduce the segment-type spaces mentioned before. Then, we specify two subtypes of these spaces and prove auxiliary lemmas related to them.

Fix $n_0 \in \NN$ and let $d = d_{n_0} = (d_{n_0,i})_{i=1}^{n_0}$ be a given system of strictly positive numbers. Set $J_{n_0} \coloneqq \{x_{0}, \dots, x_{n_0} \}$, where all elements are different. We define the metric $\rho_{n_0}$ on $J_{n_0}$ by
\begin{displaymath}
\rho_{n_0}(x_j,x_k) \coloneqq \rho_{n_0,d}(x_j,x_k) \coloneqq \sum_{i=j+1}^k d_{n_0, i}, 
\end{displaymath}
where $j,k \in \{0\} \cup [n_0]$ with $j < k$.
Figure~\ref{F3.1} shows a model of the space $(J_{n_0}, \rho_{n_0})$ with $n_0 = 4$.

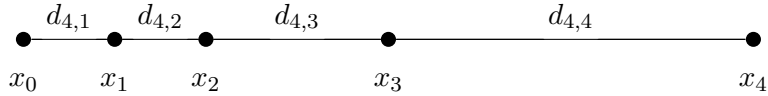
\begin{figure}[H]
	\begin{center}
	\begin{tikzpicture}
	[scale=.8,auto=left,every node/.style={circle,fill,inner sep=2pt}]
	
	\node[label={[yshift=-1cm]$x_{0}$}] (p0) at (2,2) {};
	\node[label={[yshift=-1cm]$x_{1}$}] (p1) at (3.5,2) {};
	\node[label={[yshift=-1cm]$x_{2}$}] (p2) at (5,2) {};
	\node[label={[yshift=-1cm]$x_{3}$}] (p3) at (8,2) {};
	\node[label={[yshift=-1cm]$x_{4}$}] (p4) at (14,2) {};
	
	\draw (p0) -- (p1) node [midway, fill=white, above=-5pt] {$d_{4,1}$};
	\draw (p1) -- (p2) node [midway, fill=white, above=-5pt] {$d_{4,2}$};
	\draw (p2) -- (p3) node [midway, fill=white, above=-5pt] {$d_{4,3}$};
	\draw (p3) -- (p4) node [midway, fill=white, above=-5pt] {$d_{4,4}$};
	
	\foreach \from/\to in {p0/p4}
	\draw (\from) -- (\to);
	
	\end{tikzpicture}
	\caption{The segment-type space $(J_{n_0}, \rho_{n_0})$ with $n_0 = 4$.}
	\label{F3.1}
\end{center}
\end{figure}
\noindent We define the measure $\mu_{n_0}$ on $J_{n_0}$ by letting
\[
\mu_{n_0}(\{x_i\}) \coloneqq \mu_{n_0, d, F}(\{x_i\}) \coloneqq F(i),
\] 
where $F>0$ is a given function. 

We are ready to describe two subtypes of segment-type spaces. \\

\noindent {\bf First subtype.} Now for any fixed $\kappa \in [2, \infty)$ we construct a space denoted by ${\mathcal{J}}_{1,\kappa}$  for which $\PKSC({\mathcal{J}}_{1,\kappa}) = \PKS({\mathcal{J}}_{1,\kappa}) = (1, \infty]$ and $\PKWC({\mathcal{J}}_{1,\kappa}) = \PKW({\mathcal{J}}_{1,\kappa}) = [1, \infty]$.

Fix $\kappa \in [2, \infty)$. For each $n \in \NN$ we choose $n_0 = n$ and consider $\mathcal{J}_n = (J_{n}, \rho_{n}, \mu_{n})$, where $J_{n}$, $\rho_{n}$, and $\mu_{n}$ are introduced as before with the aid of $d_{n,i} = (\kappa+1)^{i}$ for each $i \in [n]$ and $F_n(i) = 1$ for each $i \in \{0\} \cup [n]$.

We denote by ${\mathcal{J}}_{1,\kappa}$ the space $\YY$ obtained by Proposition 3.1.1 with $\kappa_0 = \kappa$ for the family $\{\mathcal{J}_n : n \in \NN  \}$. In the following lemma we describe the properties of $\MM_{\kappa, {\mathcal{J}}_{1,\kappa}}$ and $\MN_{\kappa, {\mathcal{J}}_{1,\kappa}}$.

\begin{lemma}\label{L3.2.4}
	Fix $\kappa \in [2, \infty)$ and let ${\mathcal{J}}_{1,\kappa}$ be the metric measure space defined above. Then the associated modified maximal operators, centered $\MM_{\kappa, {\mathcal{J}}_{1,\kappa}}$ and noncentered $\MN_{\kappa, {\mathcal{J}}_{1,\kappa}}$, are not of strong type $(1,1)$, but are of weak type $(1,1)$. 
\end{lemma}  

\begin{proof}
	First we show that $\MM_{\kappa, {\mathcal{J}}_{1,\kappa}}$ is not of strong type $(1,1)$. Fix $n \in \NN$ and let $g \coloneqq \mathbf{1}_{x_{0}} \in L^1(\mathcal{J}_n)$. Then $\|g\|_1 = 1$ holds. Observe that for each $j \in [n-1]$ we have
	\begin{displaymath}
	\kappa \, \sum_{i=1}^{j} d_{n,i} < d_{n,j+1}.
	\end{displaymath}
	Consequently, there exists $s_j \in (0,\infty)$ such that $B(x_{j},s_j) = B(x_{j}, \kappa s_j) = \{x_{0}, \dots, x_{j}\}$, which implies that $\MM_{\kappa, \mathcal{J}_n} g(x_j) \geq \frac{1}{j+1}$ holds. 
	Therefore, for each $n \in \NN$ we obtain
	\begin{displaymath}
	\cc_{\rm s, \mathcal{J}_n}^{\rm c}(\kappa, 1) \geq \frac{\|\MM_{\kappa, \mathcal{J}_n} g\|_1}{\|g\|_1} \geq  \sum_{j=1}^{n-1} \frac{1}{j+1}
	\end{displaymath}
	and, consequently, $\cc_{\rm s, {\mathcal{J}}_{1,\kappa}}^{\rm c}(\kappa, 1) = \infty$. 
	
	It remains to show that $\MN_{\kappa, {\mathcal{J}}_{1,\kappa}}$ is of weak type $(1,1)$. We fix $n \in\NN$ and estimate $\cc_{\rm w, \mathcal{J}_n}(\kappa, 1)$ from above. Let $f \in L^{1}(\mathcal{J}_n)$, $f \geq 0$, and consider $\lambda \in (0, \infty)$ such that $E_\lambda \coloneqq E_\lambda (\MN_{\kappa, \mathcal{J}_n} f) \neq \emptyset$. 
	Observe that, because of the linear structure of $J_n$, any ball $B \subset J_n$ is of the form $\{ x_{i}, \dots, x_{j} \}$ for some $i,j \in \{0\} \cup [n]$ with $i \leq j$. Define $\mathcal{B} \coloneqq \{B \subset J_n : \sum_{x \in B} f(x) / |\kappa B| > \lambda\}$ and observe that the elements of $\mathcal{B}$ form a cover of $E_\lambda$. By using the fact that each element of $\mathcal{B}$ has the form described above, we can find a subcover $\mathcal{B}'$ such that each $x \in E_\lambda$ belongs to at most two elements of $\mathcal{B}'$. Therefore,
	\begin{displaymath}
	\lambda \cdot |E_\lambda| \leq \sum_{B \in \mathcal{B}'} \lambda \cdot |B| \leq \sum_{B \in \mathcal{B}'} \Big( \sum_{x \in B} f(x) \Big) \frac{|B|}{|\kappa B|} \leq 2 \sum_{x \in E_\lambda} f(x) \leq 2 \|f\|_1.
	\end{displaymath}
	Consequently, for each $n \in \NN$ we have $\cc_{\rm w, \mathcal{J}_n}(\kappa, 1) \leq 2$, which implies that $\cc_{\rm w, {\mathcal{J}}_{1,\kappa}}(\kappa, 1) < \infty$. 
\end{proof}

\noindent {\bf Second subtype.} Now for any fixed $\kappa \in [3, \infty)$ we construct a space denoted by ${\mathcal{J}}_{2,\kappa}$  for which $\PKS({\mathcal{J}}_{2,\kappa}) = (1, \infty]$ and $\PKSC({\mathcal{J}}_{2,\kappa}) = \PKWC({\mathcal{J}}_{2,\kappa}) = \PKW({\mathcal{J}}_{2,\kappa}) = [1, \infty]$.

Fix $\kappa \in [3, \infty)$. For each $n \in \NN$ we choose $n_0 = n$ and consider $\mathcal{J}_n = (J_{n}, \rho_{n}, \mu_{n})$, where $J_{n}$, $\rho_{n}$, and $\mu_{n}$ are introduced as before with the aid of $d_{n,i} = (\kappa-\frac{1}{2})^{i}$ for each $i \in [n]$ and $F_n(i)$ chosen (uniquely) in such a way that $F_n(0) = 1$ and $F_n(i+1) = 2^{i+1} F_n(i)$ for each $i \in [n]$.

We denote by ${\mathcal{J}}_{2,\kappa}$ the space $\YY$ obtained by Proposition~\ref{P3.1.1} with $\kappa_0 = \kappa$ for the family $\{\mathcal{J}_n : n \in \NN  \}$. In the following lemma we describe the properties of $\MM_{\kappa, {\mathcal{J}}_{2,\kappa}}$ and $\MN_{\kappa, {\mathcal{J}}_{2,\kappa}}$.

\begin{lemma}\label{L3.2.5}
	Fix $\kappa \in [3, \infty)$ and let ${\mathcal{J}}_{2,\kappa}$ be the metric measure space defined above. Then the associated modified centered maximal operator $\MM_{\kappa, {\mathcal{J}}_{2,\kappa}}$ is of strong type $(1,1)$, while the modified noncentered operator $\MN_{\kappa, {\mathcal{J}}_{2,\kappa}}$ is not of strong type $(1,1)$, but is of weak type $(1,1)$. 
\end{lemma}

\begin{proof}
	Since $\kappa \in [3, \infty)$, the operator $\MN_{\kappa, {\mathcal{J}}_{2,\kappa}}$ is of weak type $(1,1)$. Hence, it suffices to show that $\MN_{\kappa, {\mathcal{J}}_{2,\kappa}}$ is not of strong type $(1,1)$, while $\MM_{\kappa, {\mathcal{J}}_{2,\kappa}}$ is of strong type $(1,1)$. Moreover, we notice that for each $n \in \NN$ and $j \in \{ 0\} \cup [n-1]$ we have the inequalities $\sum_{i=1}^{j} d_{n,i} < d_{n,j+1}$ and $\sum_{i=1}^{j} F_n(i) < F_n(j+1)$.
	
	First we show that $\MN_{\kappa, {\mathcal{J}}_{2,\kappa}}$ is not of strong type $(1,1)$. Fix $n \in \NN$ and let $g \coloneqq \mathbf{1}_{x_{0}} \in L^1(\mathcal{J}'_n)$. Then $\|g\|_1 = 1$ holds. Observe that for each $j \in [n-1]$ we have
	\begin{displaymath}
	\sum_{i=1}^{j-1} d_{n,i} < d_{n,j} < \frac{d_{n,j+1}}{\kappa -1}.
	\end{displaymath}
	Consequently, there exists $s_j \in (0, \infty)$ such that $B(x_{j-1},s_j) = B(x_{j-1}, \kappa s_j) = \{x_{0}, \dots, x_{j}\}$, which implies that $\MN_{\kappa, \mathcal{J}_n} g(x_{j}) \geq \frac{1}{2}|\{x_{j}\}|^{-1}$ holds. Therefore, for each $n \in \NN$ we obtain
	\begin{displaymath}
	\cc_{\rm s, \mathcal{J}'_n}(\kappa, 1) \geq \frac{\|\MM_{\kappa, \mathcal{J}_n} g\|_1}{\|g\|_1} \geq \frac{n-1}{2}
	\end{displaymath}
	and, consequently, $\cc_{\rm s, {\mathcal{J}}_{2,\kappa}}(\kappa, 1) = \infty$.
	
	It remains to show that $\MM_{\kappa, {\mathcal{J}}_{2,\kappa}}$ is of strong type $(1,1)$. Fix $n \in \NN$ and estimate $\cc_{\rm s, \mathcal{J}_n}^{\rm c}(\kappa, 1)$ from above. Let $f \in L^{1}(\mathcal{J}_n)$, $f \geq 0$. Observe that for each $j \in [n-1]$ we have $\kappa d_{n,j} > d_{n,j+1}$. Thus, if $B$ is centered at $x_{j}$ and $x_{j-1} \in B$, then $x_{j+1} \in \kappa B$. From this for each $j \in \{ 0\} \cup [n-1]$ we deduce the following estimate
	\begin{displaymath}
	\MM_{\kappa, \mathcal{J}_n} f(x_{j}) \leq \frac{ \sum_{i=0}^{j-1} f(x_{i}) |\{x_{i} \}|}{|\{x_{j+1}\}|} + \sum_{i=j}^{n} f(x_{i}),
	\end{displaymath}
	while for $j=n$ we obtain $\MM_{\kappa, \mathcal{J}_n} f(x_{n}) \leq \|f\|_1 / |\{x_n\}|$. Therefore,
	\begin{displaymath}
	\sum_{j=0}^{n} \MM_{\kappa, \mathcal{J}_n} f(x_j) |\{x_j \}| \leq \sum_{j=0}^{n} f(x_j) |\{x_j \}| \Big( 2 + \sum_{i=j+1}^{n-1} \frac{|\{x_i \}|}{|\{x_{i+1} \}|} + \sum_{i=0}^{j-1} \frac{|\{x_i \}|}{|\{x_j \}|} \Big) \leq 4  \|f\|_1.
	\end{displaymath}
	Since $\cc_{\rm s, \mathcal{J}_n}^{\rm c}(\kappa, 1) \leq 4$ holds for each $n \in \NN$, we conclude that $\cc_{\rm s, {\mathcal{J}}_{2,\kappa}}^{\rm c}(\kappa, 1) < \infty$.
\end{proof} 

\subsection{Proof of the main result}\label{S3.2.3}

\begin{proof}[Proof of Theorem~\ref{T3.2.2}]
First we note that if $P_{\kappa,\rm s}^{\rm c} = P_{\kappa,\rm s} = P_{\kappa,\rm w}^{\rm c} = P_{\kappa,\rm w} = [1, \infty]$, then one can find a first generation space $\ZZZ$ for which $P_{\rm s}^{\rm c}(\ZZZ) = P_{\rm s}(\ZZZ) = P_{\rm w}^{\rm c}(\ZZZ) = P_{\rm w}(\ZZZ) = [1, \infty]$, and hence we also have $P_{\kappa,\rm s}^{\rm c}(\ZZZ) = P_{\kappa,\rm s}(\ZZZ) = P_{\kappa,\rm w}^{\rm c}(\ZZZ) = P_{\kappa,\rm w}(\ZZZ) = [1, \infty]$ for every $\kappa \in [1, \infty)$. Therefore, from now on, assume that $P_{\kappa,\rm s}$ (and possibly some other sets) is a proper subset of $[1, \infty]$. We shall consider the following cases: $\kappa \in [1,2)$, $\kappa \in [2,3)$ and $\kappa \in [3, \infty)$.

Suppose that $\kappa \in [1,2)$. Then the sets $P_{\kappa,\rm s}^{\rm c}$, $P_{\kappa,\rm s}$, $P_{\kappa,\rm w}^{\rm c}$, and $P_{\kappa,\rm w}$ satisfy \ref{3i} and \ref{3ii}, while the remaining two conditions are empty. We can find two spaces $\SSS$ and $\TT$ of the first and second generations, respectively, for which
\begin{itemize}
	\item $\PSC(\SSS) = \PS(\SSS) = \PKSC$ and $\PWC(\SSS) = \PW(\SSS) = \PKWC$,
	\item $\PSC(\TT) = \PWC(\TT) = [1, \infty]$, $\PS(\TT) = \PKS$ and $\PW(\TT) = \PKW$,
\end{itemize}
and, in view of the observation made in Section~\ref{S3.2.1}, the same identities hold with $\PKSC(\SSS)$ in place of $\PSC(\SSS)$ and so on. Now we let $\ZZZ$ be the space obtained by using Proposition~\ref{P3.1.1} with $\kappa_0 = \kappa$ for $\SSS$ and $\TT$. One can easily see that $\ZZZ$ has the following properties:
\begin{itemize}
	\item $\PKSC(\ZZZ) = \PKSC(\SSS) \cap \PKSC(\TT) = \PKSC \cap [1, \infty] = \PKSC$,
	\item $\PKWC(\ZZZ) = \PKWC(\SSS) \cap \PKWC(\TT) = \PKWC \cap [1, \infty] = \PKWC$,
	\item $\PKS(\ZZZ) = \PKS(\SSS) \cap \PKS(\TT) = \PKSC \cap \PKS = \PKS$,
	\item $\PKW(\ZZZ) = \PKW(\SSS) \cap \PKW(\TT) = \PKWC \cap \PKW = \PKW$,
\end{itemize} 
and therefore it may be chosen to be the expected space.

Next, suppose that $\kappa \in [2,3)$. Then the sets $P_{\kappa,\rm s}^{\rm c}$, $P_{\kappa,\rm s}$, $P_{\kappa,\rm w}^{\rm c}$, and $P_{\kappa,\rm w}$ satisfy \ref{3i}--\ref{3iii}, while the condition \ref{3iv} is empty. We can find a second generation space $\TT$ for which $P_{\rm s}^{\rm c}(\TT) = P_{\rm w}^{\rm c}(\TT) = [1, \infty] = P_{\kappa,\rm w}^{\rm c}$, $P_{\rm s}(\TT) = P_{\kappa,\rm s}$, and $P_{\rm w}( \TT) = P_{\kappa,\rm w}$, and therefore we obtain the same identities with $P_{\rm s}^{\rm c}(\TT)$, $P_{\rm s}(\TT)$, $P_{\rm w}^{\rm c}(\TT)$, and $P_{\rm w}(\TT)$ replaced by $P_{\kappa,\rm s}^{\rm c}(\TT')$, $P_{\kappa,\rm s}(\TT')$, $P_{\kappa,\rm w}^{\rm c}(\TT')$, and $P_{\kappa,\rm w}(\TT')$, respectively, where $\TT'$ is the modification of $\TT$ considered in Lemma~\ref{L3.2.3}. If $P_{\kappa,\rm s}^{\rm c} = [1, \infty]$, then the expected space may be chosen to be just $\TT'$. Otherwise (that is, if $P_{\kappa,\rm s}^{\rm c} = (1, \infty]$) we use Proposition~\ref{P3.1.1} with $\kappa_0=\kappa$ for $\TT'$ and the space ${\mathcal{J}}_{1,\kappa}$ considered in Lemma~\ref{L3.2.4}. Then the obtained space $\YY$ satisfies $P_{\kappa,\rm s}^{\rm c}(\YY) = P_{\kappa,\rm s}^{\rm c}$, $P_{\kappa,\rm s}(\YY) = P_{\kappa,\rm s}$, $P_{\kappa,\rm w}^{\rm c}(\YY) = P_{\kappa,\rm w}^{\rm c}$, and $P_{\kappa,\rm w}(\YY) = P_{\kappa,\rm w}$.

Finally, suppose that $\kappa \in [3, \infty)$ and the sets $P_{\kappa,\rm s}^{\rm c}$, $P_{\kappa,\rm s}$, $P_{\kappa,\rm w}^{\rm c}$, and $P_{\kappa,\rm w}$ satisfy \ref{3i}--\ref{3iv}. If $P_{\kappa,\rm s}^{\rm c} = P_{\kappa,\rm s} = (1, \infty]$, then the expected space may be chosen to be the space ${\mathcal{J}}_{1,\kappa}$ considered in Lemma~\ref{L3.2.4}. Otherwise, if $P_{\kappa,\rm s}^{\rm c} = [1, \infty]$ and $P_{\kappa,\rm s} = (1, \infty]$, then the expected space may be chosen to be the space ${\mathcal{J}}_{2,\kappa}$ considered in Lemma~\ref{L3.2.5}.

Finally, in view of Remark~\ref{R3} each space $\ZZZ$ specified above is nondoubling.
\end{proof}

\section{Results for varying modification parameter}\label{S3.3}

This section is devoted to studying the case of varying parameter $\kappa \in [1, \infty)$. For the sake of clarity, we focus only on the weak type $(p,p)$ inequalities. For a given metric measure space $\XX$ let us define two auxiliary functions 
\begin{displaymath}
h^{\rm c}_\XX(\kappa) \coloneqq \inf\{p \in [1, \infty]: \cc^{\rm c}_{\rm w, \XX}(\kappa ,p) < \infty\} \quad \text{ and } \quad h_\XX(\kappa) \coloneqq \inf\{p \in [1, \infty]: \cc_{\rm w, \XX}(\kappa,p) < \infty\}.
\end{displaymath}
Since $\MM_{2, \XX}$ and $\MN_{3, \XX}$ are of weak type $(1,1)$, we can assume that the domains of $h^{\rm c}_\XX$ and $h_\XX$ are $[1,2]$ and $[1,3]$, respectively. The following assertions hold:
\begin{enumerate}[label={\rm (\roman*)}]
	\item \label{3.3i} We have $h^{\rm c}_\XX \colon [1,2] \rightarrow [1, \infty]$ and $h_\XX \colon [1,3] \rightarrow [1, \infty]$.
	\item \label{3.3ii} Both $h^{\rm c}_\XX$ and $h_\XX$ are nonincreasing.
	\item \label{3.3iii} For for $\kappa \in [1,2]$ we have $h(\kappa) \geq h^{\rm c}_\XX(\kappa)$.
	\item \label{3.3iv} We have $h^{\rm c}_\XX(2)=h_\XX(3) = 1$.
	\item \label{3.3v} For each $\kappa \in [1,2]$, if $h^{\rm c}_\XX(\kappa) = \infty$, then $\PKWC(\XX) = \{\infty\}$, and if $h^{\rm c}_\XX(\kappa) < \infty$, then either $\PKWC(\XX) = (h^{\rm c}_\XX(\kappa), \infty]$ or $\PKWC(\XX) = [h^{\rm c}_\XX(\kappa), \infty]$.
	\item \label{3.3vi} For each $\kappa \in [1,3]$, if $h_\XX(\kappa) = \infty$, then $\PKW(\XX) = \{\infty\}$, and if $h_\XX(\kappa) < \infty$, then either $\PKW(\XX) = (h_\XX(\kappa), \infty]$ or $\PKW(\XX) = [h_\XX(\kappa), \infty]$.
\end{enumerate}

Our principal motivation is to take arbitrary functions $h^{\rm c}$ and $h$ such that \ref{3.3i}--\ref{3.3iv} hold with $h^{\rm c}$ and $h$ in place of $h^{\rm c}_\XX$ and $h_\XX$, respectively, and to ask whether it is possible to find a~metric measure space $\ZZZ$ such that \ref{3.3v} and \ref{3.3vi} hold with $\PKWC(\XX)$, $\PKW(\XX)$, $h^{\rm c}_\XX$, and $h_\XX$ replaced by $\PKWC(\ZZZ)$, $\PKW(\ZZZ)$, $h^{\rm c}$, and $h$, respectively. It turns out that the answer is always positive. Namely, we have the following theorem.

\begin{theorem}\label{T3.3.1}
	Let $h^{\rm c}$ and $h$ be such that \ref{3.3i}--\ref{3.3iv} hold with $h^{\rm c}$ and $h$ in place of $h^{\rm c}_\XX$ and $h_\XX$, respectively. Then there exists a metric measure space $\ZZZ$ such that for each $\kappa \in [1, 2)$ the associated modified centered maximal operator $\MM_{\kappa, \ZZZ}$ is of weak type $(p,p)$ if and only if $p > h^{\rm c}(\kappa)$ or $p = \infty$, while for each $\kappa \in [1, 3)$ the modified noncentered maximal operator $\MN_{\kappa, \ZZZ}$ is of weak type $(p,p)$ if and only if $p > h(\kappa)$ or $p = \infty$. 
\end{theorem}

One comment is in order. Observe that the conditions \ref{3.3i}--\ref{3.3vi} usually do not cover complete information about the finiteness of $\cc^{\rm c}_{\rm w, \XX}(\kappa, p)$ and $\cc_{\rm w, \XX}(\kappa, p)$. Namely, having only the values of $h^{\rm c}_\XX$ and $h_\XX$ available, one is often unable to determine whether the values $\cc^{\rm c}_{\rm w, \XX}(\kappa, h^{\rm c}_\XX(\kappa))$ for $\kappa \in [1,2)$ and $\cc_{\rm w, \XX}(\kappa, h_\XX(\kappa))$ for $\kappa \in [1,3)$, are finite or not. Sometimes, there can be many possible cases depending on $\XX$ and the characterization of them is a difficult problem which will not be treated here. Nevertheless, the obtained results may be helpful to find a general principle related to this issue.

We prove Theorem~\ref{T3.3.1} in Subsection~\ref{S3.3.3}. Before that, in Subsections~\ref{S3.3.1}~and~\ref{S3.3.2} some auxiliary structures are considered. From now on we write shortly $\cc^{\rm c}_{\XX}(\kappa, p)$ and $\cc_{\XX}(\kappa, p)$ instead of $\cc^{\rm c}_{\rm w, \XX}(\kappa, p)$ and $\cc_{\rm w, \XX}(\kappa, p)$, respectively. 

\subsection{Basic spaces}\label{S3.3.1}

In this subsection we introduce and analyze certain simple structures which we call the {\it basic spaces} later on. Studying this class of structures allows us to produce many examples of spaces for which the associated modified maximal operators have very specific properties. We consider two types of basic spaces which are denoted by $\SSS$ and $\TT$ in order to indicate their similarity to the components of the first and second generation spaces, respectively. \\ 

\noindent {\bf First type.} Fix $\tau \in \mathbb{N}$, $d \in (1, 2]$, and $m \in [1, \infty)$. We introduce the basic space of the first type $\SSS = \SSS_{\tau, d, m} = (S, \rho, \mu)$ as follows. Set $S \coloneqq \{x_0, \dots, x_\tau\}$. Define $\rho$ by letting
\begin{displaymath}
\rho(x,y) \coloneqq \rho_d(x,y) \coloneqq \left\{ \begin{array}{rl}
1 & \textrm{if } x_0 \in \{x,y\},  \\
d & \textrm{otherwise,} \end{array} \right. 
\end{displaymath}
where $x$ and $y$ are two different elements of $S$. Finally, take $\mu$ defined by
\begin{displaymath}
\mu(\{x_i\}) \coloneqq \mu_m(\{x_i\}) \coloneqq \left\{ \begin{array}{rl}
1 & \textrm{if } i=0,  \\
m & \textrm{if } i \in [\tau]. \end{array} \right. 
\end{displaymath}
Figure~\ref{F3.2} shows a model of the basic space $\SSS$.

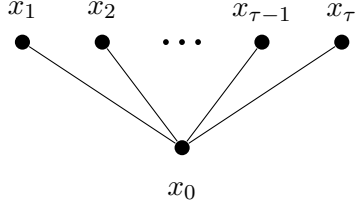
\begin{figure}[H]
	\begin{center}
	\begin{tikzpicture}
	[scale=.7,auto=left,every node/.style={circle,fill,inner sep=2pt}]
	
	\node[label={[yshift=-1cm]$x_0$}] (m0) at (8,1) {};
	\node[label=$x_{1}$] (m1) at (5,3)  {};
	\node[label=$x_{2}$] (m2) at (6.5,3)  {};
	\node[label={[yshift=-0.23cm]$x_{\tau-1}$}] (m3) at (9.5,3)  {};
	\node[label={[yshift=-0.04cm]$x_{\tau}$}] (m4) at (11,3)  {};
	\node[dots, scale=2] (m5) at (8,3)  {...};
	
	\foreach \from/\to in {m0/m1, m0/m2, m0/m3, m0/m4}
	\draw (\from) -- (\to);
	\end{tikzpicture}
	\caption{The basic space of the first type.}
	\label{F3.2}
	\end{center}
\end{figure}
\noindent We can explicitly describe any ball:
\begin{displaymath}
B(x_0,s) = \left\{ \begin{array}{rl}
\{x_0\} & \textrm{for } 0 < s \leq 1, \\
S & \textrm{for } 1 < s, \end{array} \right.
\end{displaymath} 
\noindent and, for $i \in [\tau]$,
\begin{displaymath}
B(x_i,s) = \left\{ \begin{array}{rl}
\{x_i\} & \textrm{for } 0 < s \leq 1, \\
\{x_0, x_i\} & \textrm{for } 1 < s \leq d,  \\
S & \textrm{for } d < s. \end{array} \right.
\end{displaymath}
In the following lemma we describe the properties of $\MM_{\kappa, \SSS}$ and $\MN_{\kappa, \SSS}$ for all $\kappa \in [1, \infty)$.  
\begin{lemma}\label{L3.3.2}
	Let $\SSS$ be the basic space of the first type defined above. Then
	\begin{displaymath}
	\cc_\SSS(\kappa,p) \simeq \cc^{\rm c}_\SSS(\kappa,p) \simeq \left\{ \begin{array}{rl}
	\max\{1, \tau^{1/p}  m^{1/p-1}\} & \textrm{if } \kappa \in [1,d) \textrm{ and } p \in [1, \infty),  \\
	1 & \textrm{if } \kappa \in [d, \infty) \textrm{ or } p = \infty.\end{array} \right. 
	\end{displaymath}
\end{lemma} 

\begin{proof}
	First we notice that, in view of the inequality $\cc_\mathfrak{S}(\kappa,p) \geq \cc^{\rm c}_\mathfrak{S}(\kappa,p)$, it suffices to estimate $\cc_\mathfrak{S}(\kappa,p)$ from above and $\cc^{\rm c}_\mathfrak{S}(\kappa,p)$ from below by the appropriate terms. 
	
	Let $f \colon S \to [0, \infty)$. Clearly, we have $\MM_{\kappa, \SSS} f \geq f$ and hence $\cc^{\rm c}_\mathfrak{S}(\kappa,p) \geq 1$ holds for any $\kappa \in [1, \infty)$ and $p \in [1, \infty]$. Next, if $\kappa \in [d, \infty)$ and $p \in [1, \infty)$, then for any ball $B$ containing at least two points, the ball $\kappa B$ coincides with $S$. Therefore, for each $x \in S$ we have
	\begin{displaymath}
	\MM_{\kappa, \SSS} f (x) \leq f(x) + A_S(f).
	\end{displaymath}
	Applying H\"older's inequality we obtain $\|\MM_{\kappa, \SSS} f\|_p^p \leq 2^{p-1} \|f\|_p^p$ which implies that $\cc_\mathfrak{S}(\kappa,p) \leq 2^{(p-1)/p} \lesssim 1$ holds. Obviously, we also have $\cc_\mathfrak{S}(\kappa, \infty) \leq 1$ for each $\kappa \in [1, \infty)$. Thus, it remains to analyze the case $\kappa \in [1, d)$ and $p \in [1, \infty)$. 
	
	Write $f = f_1 + f_2$ with $f_1 \coloneqq f \cdot \mathbf{1}_{\{x_0\}}$ and $f_2 \coloneqq f \cdot \mathbf{1}_{S \setminus \{x_0\}}$. Since $\MN_{\kappa, \SSS}$ is sublinear, we have 
	$
	\MN_{\kappa, \SSS} f \leq \MN_{\kappa, \SSS} f_1 + \MN_{\kappa, \SSS} f_2.
	$
	Observe that $\MN_{\kappa, \SSS} f_1(x_0) = f_1(x_0)$ and $\MN_{\kappa, \SSS} f_1(x_i) \leq  \frac{1}{m} f_1(x_0)$ for each $i \in [\tau]$. Thus, 
	$
	\|\MN_{\kappa, \SSS} f_1 \|_p^p \leq (1 + \tau m^{1-p}) \, \|f_1\|_p^p.
	$
	For $f_2$, in turn, we have the estimate $\MN_{\kappa, \SSS} f_2 \leq f_2 + A_S(f_2)$, which gives
	$
	\| \MN_{\kappa, \SSS} f_2 \|_p^p \leq 2^{p-1} \|f_2\|_p^p.
	$
	Therefore, $\|\MN_{\kappa, \SSS} f \|_p^p \leq 2^{p-1} (1 + \tau m^{1-p}+2^{p-1}) \|f\|_p^p $ holds and, consequently, we obtain
	\begin{displaymath}
	\cc_\mathfrak{S}(\kappa ,p) \leq \big( 2^{p-1} (1 + \tau m^{1-p}+2^{p-1} ) \big)^{1/p} \lesssim \max\{1, \tau^{1/p}  m^{1/p-1}\}.
	\end{displaymath}
	Finally, let $g \coloneqq \mathbf{1}_{\{x_0\}}$. Then $\|g\|_p = 1$ and $\MM_{\kappa, \SSS} f(x_i) = \frac{1}{m+1} > \frac{1}{2m}$ for each $i \in [\tau]$. Thus,
	\begin{displaymath}
	\cc^{\rm c}_\mathfrak{S}(\kappa,p) \geq \frac{1}{2m} \, \big|E_{\frac{1}{2m}}(g)\big|^{1/p} \gtrsim \tau^{1/p} m^{1/p-1} 
	\end{displaymath} 
	and, combining this with the inequality $\cc^{\rm c}_\mathfrak{S}(\kappa,p) \geq 1$, we obtain the desired estimate.
\end{proof}

\noindent {\bf Second type.} Fix $\tau \in \mathbb{N}$, $d \in (1, 3]$, and $m \in [1, \infty)$. We introduce the basic space of the second type $\mathfrak{T} = \mathfrak{T}_{\tau, d, m} = (T, \rho, \mu)$ as follows. Set $T \coloneqq \{y_0, y^\circ_1, \dots, y^\circ_\tau,  y_1', \dots, y_\tau' \}$. We use auxiliary symbols for certain subsets of $T$: $T^\circ \coloneqq \{y^\circ_1, \dots, y^\circ_\tau \}$, $T' \coloneqq \{y_1', \dots, y_\tau'\}$, and, for each $i \in [\tau]$, $T_i \coloneqq \{y^\circ_i, y_i'\}$. Define $\rho$ by the formula
\begin{displaymath}
\rho(x,y) \coloneqq \rho_d(x,y) \coloneqq \left\{ \begin{array}{rl}
1 & \textrm{if } y_0 \in \{x,y\} \subset T \setminus T' \textrm{ or } \{x,y\} = T_i \textrm{ for some } i \in [\tau], \\
\frac{d+1}{2} & \textrm{if } \{x,y\} \subset T^\circ \textrm{ or } \{x,y\} \subset T \setminus T^\circ,\\
d & \textrm{otherwise,} \end{array} \right. 
\end{displaymath}
where $x$ and $y$ are two different elements of $T$. Finally, take $\mu$ defined by
\begin{displaymath}
\mu(\{y\}) \coloneqq \mu_m(\{y\}) \coloneqq \left\{ \begin{array}{rl}
1 & \textrm{if } y=y_0,  \\
\frac{1}{\tau} & \textrm{if } y=y^\circ_i \textrm{ for some } i \in [\tau],  \\
m & \textrm{if } y=y_i' \textrm{ for some } i \in [\tau]. \end{array} \right. 
\end{displaymath}
Figure~\ref{F3.3} shows a model of the space $\mathfrak{T}$. Adding an imaginary point at the top makes $\rho$ easily readable as a minor modification of the geodesic distance on the graph.
\begin{figure}[H]
	\begin{center}
	\begin{tikzpicture}
	[scale=.7,auto=left,every node/.style={circle,fill,inner sep = 2pt}]
	\node[label={[yshift=-1cm]$y_0$}] (n0) at (8,1) {};
	\node[label=$y^\circ_{1}$] (n1) at (6,3)  {};
	\node[label={[xshift=0.16cm]$y^\circ_{2}$}] (n2) at (7,3)  {};
	\node[label={[xshift=-0.24cm, yshift=-0.2cm]$y^\circ_{\tau-1}$}] (n3) at (9,3)  {};
	\node[label={[yshift=-0.05cm]$y^\circ_{\tau}$}] (n4) at (10,3)  {};
	\node[dots, scale=2] (n9) at (8,3)  {...};
	\node[label=$y_{1}'$] (n5) at (4,5)  {};
	\node[label={[xshift=-0.05cm]$y_2'$}] (n6) at (6,5)  {};
	\node[label={[xshift=0.3cm, yshift=-0.24cm]$y_{\tau-1}'$}] (n7) at (10,5)  {};
	\node[label={[yshift=-0.06cm]$y_{\tau}'$}] (n8) at (12,5)  {};
	\node[dots, scale=2] (n10) at (8,5)  {...};
	\node[label=$ $] (im) at (8,6.65)  {};

	\foreach \from/\to in {n0/n1, n0/n2, n0/n3, n0/n4, n1/n5, n2/n6, n3/n7, n4/n8}
	\draw (\from) -- (\to);
	\draw (6,5) arc (155:103.5:3);
	\draw (10,5) arc (25:76.5:3);
	\draw (4,5) arc (125:100:10);
	\draw (12,5) arc (55:80:10);
	\end{tikzpicture}
	\caption{The basic space of the second type.}
	\label{F3.3}
	\end{center}
\end{figure}
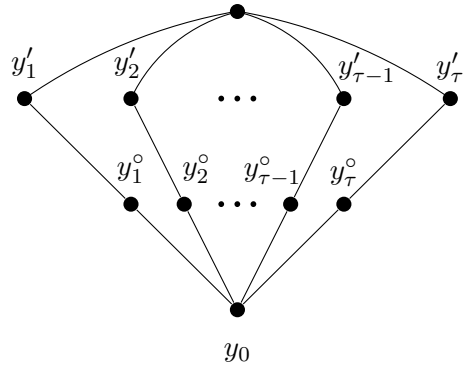

\noindent Once again we explicitly describe any ball:
\begin{displaymath}
B(y_0,s) = \left\{ \begin{array}{rl}
\{y_0\} & \textrm{for } 0 < s \leq 1, \\
T \setminus T' & \textrm{for } 1 < s \leq \frac{d+1}{2}, \\
T & \textrm{for } \frac{d+1}{2} < s, \end{array} \right.
\end{displaymath} 
\noindent and, for $i \in [\tau]$,
\begin{displaymath}
B(y^\circ_i,s) = \left\{ \begin{array}{rl}
\{y^\circ_i\} & \textrm{for } 0 < s \leq 1, \\
\{y_0\} \cup T_i & \textrm{for } 1 < s \leq \frac{d+1}{2},  \\
\{y_0, y_i'\} \cup T^\circ & \textrm{for } \frac{d+1}{2} < s \leq d,  \\
T & \textrm{for } d < s, \end{array} \right.
\end{displaymath}
and
\begin{displaymath}
B(y_i',s) = \left\{ \begin{array}{rl}
\{y_i'\} & \textrm{for } 0 < s \leq 1, \\
T_i & \textrm{for } 1 < s \leq \frac{d+1}{2},  \\
\{y_0, y^\circ_i\} \cup T'  & \textrm{for } \frac{d+1}{2} < s \leq d,  \\
T & \textrm{for } d < s. \end{array} \right.
\end{displaymath}
In the following lemma we describe the properties of $\MM_{\kappa, \TT}$ and $\MN_{\kappa, \TT}$, $\kappa \in [1, \infty)$.

\begin{lemma}\label{L3.3.3}
	Let $\mathfrak{T}$ be the basic space of the second type defined above. Then $\cc^{\rm c}_\mathfrak{T}(\kappa,p) \simeq 1$ for all $\kappa \in [1, \infty)$ and $p \in [1, \infty]$, while
	\begin{displaymath}
	\cc_\mathfrak{T}(\kappa,p) \simeq \left\{ \begin{array}{rl}
	\max\{1, \tau^{1/p}  m^{1/p-1}\} & \textrm{if } \kappa \in [1, d) \textrm{ and } p \in [1, \infty),  \\
	1 & \textrm{if } \kappa \in [d, \infty) \textrm{ or } p = \infty.\end{array} \right. 
	\end{displaymath}
\end{lemma}

\begin{proof}
	First, note that $\cc^{\rm c}_\mathfrak{T}(\kappa,p) \geq 1$ and $\cc_\mathfrak{T}(\kappa,p) \geq 1$ for all $\kappa \in [1, \infty)$ and $p \in [1, \infty]$. Moreover, both $\cc^{\rm c}_\mathfrak{T}(\kappa,p)$ and $\cc_\mathfrak{T}(\kappa,p)$ are nonincreasing as functions of $\kappa$. Therefore, to prove that $\cc^{\rm c}_\mathfrak{T}(\kappa,p) \simeq 1$, it remains to show that $\cc^{\rm c}_\mathfrak{T}(1,p) \lesssim 1$. 
	
	Let $f \colon T \to [0, \infty)$. Observe that $\max\{\MM_\TT f(y) : y \in T \} = \max\{ f(y) : y \in T \}$ which implies that $\cc^{\rm c}_\mathfrak{T}(1,\infty) = 1$. Now assume that $p \in [1, \infty)$. We have
	\begin{displaymath}
	\MM_\TT f(y_0) \leq \max\{f(y_0), A_{T \setminus T'}(f), A_T(f)\}.
	\end{displaymath}
	Moreover, since the estimate 
	$
	|\{y_0, y^\circ_i\} \cup T'| \geq |T'| \geq |T| / 3
	$
	holds for each $i \in [\tau]$, we obtain $A_{\{y_0, y^\circ_i\} \cup T'}(f) \leq 3 A_T(f)$ and, consequently,
	\begin{displaymath}
	\MM_\TT f(y_i') \leq \max\{f(y_i'), A_{T_i}(f), 3 A_T(f)\}.
	\end{displaymath}
	Finally, observe that
	\begin{displaymath}
	A_{\{y_0\} \cup T_i}(f) \leq \max\{f(y_0), A_{T_i}(f) \} \leq \max\{\MM_\TT f(y_0), \MM_\TT f(y_i') \}
	\end{displaymath}
	and
	\begin{displaymath}
	A_{\{y_0, y_i'\} \cup T^\circ }(f) \leq \max\{A_{T \setminus T'}(f), f(y_i')\} \leq \max\{\MM_\TT f(y_0), \MM_\TT f(y_i') \}.
	\end{displaymath}
	Consequently,
	\begin{displaymath}
	\MM_\TT f(y^\circ_i) \leq \max\{f(y^\circ_i), \MM_\TT f(y_0), \MM_\TT f(y_i'), A_T(f)\}.
	\end{displaymath}
	Therefore, since $|\{y^\circ_i\}| \leq |\{y_i'\}|$ and $\sum_{i=1}^{\tau} |\{y^\circ_i\}| = |\{y_0\}|$, we can estimate $\|\MM_\TT f\|_p^p$ by
	\begin{displaymath}
	2 \, \Big( \sum_{y \in T} f(y)^p  |\{y\}| + 3^p  A_T(f)^p  |T| + A_{T \setminus T'}(f)^p  |\{y_0\}| + \sum_{i=1}^{\tau} A_{T_i}(f)^p |\{y_i'\}| \Big).
	\end{displaymath} 
	Applying H\"older's inequality we obtain $\|\MM_\TT f\|_p^p \leq 2 (3^p + 3) \|f\|_p^p$ and hence $\cc^{\rm c}_\mathfrak{T}(1,p) \leq 12$.
	
	From now on we discuss only the noncentered case. It is easy to verify that $\cc_\mathfrak{T}(\kappa,p) \simeq 1$ if $\kappa \in [d, \infty)$ or $ p = \infty$, arguing as in the proof of Lemma~\ref{L3.3.2}. In the next step we prove that $\cc_\mathfrak{T}(\kappa,p) \lesssim \max\{1, \tau^{1/p}  m^{1/p-1}\}$ holds for $\kappa \in [1, d)$ and $p \in [1, \infty)$. In fact, it suffices to consider the case $\kappa = 1$. Take $f \geq 0$ and observe that we have
	\begin{displaymath}
	\MN_\TT f(y_0) \leq \max\{f(y_0), A_{T \setminus T'}(f), \MN_\TT f(y_1'), \dots, \MN_\TT f(y_\tau')\},
	\end{displaymath}
	since $\{y_0\}$ and $T \setminus T'$ are the only balls containing $y_0$ and disjoint with $T'$. Furthermore,
	\begin{displaymath}
	\MN_\TT f(y^\circ_i) \leq \max\{f(y^\circ_i), \MN_\TT f(y_0), \MN_\TT f(y_1'), \dots, \MN_\TT f(y_\tau')\}
	\end{displaymath}
	holds for each $i \in [\tau]$. Notice that if $y_i' \in B \subset T$, then either $B \subset \{y_0, y_i'\} \cup T^\circ$ or $|B| \geq \frac{1}{3} |T|$. Since $|\{y_i'\}| \geq \frac{1}{3} |\{y_0, y_i'\} \cup T^\circ|$, we obtain
	\begin{displaymath}
	\MN_\TT f(y_i') \leq 3 \max\{A_{\{y_0, y_i'\} \cup T^\circ}(f), A_T(f)\}.
	\end{displaymath}
	Since $|\{y_0, y_i^\circ, y_i'\}| \leq 3 |\{y_i' \}|$ holds for each $i \in [\tau]$, we arrive at the inequality
	\begin{displaymath}
	\|\MN_\TT f\|_p^p \leq 3 \Big( \sum_{y \in T \setminus T'} f(y)^p |\{y\}| + A_{T \setminus T'}(f)^p |\{y_0\}| + 3^p  A_T(f)^p |T'| +  3^p  \sum_{i=1}^{\tau} A_{ \{y_0, y_i'\} \cup T^\circ }(f)^p |\{y_i'\}| \Big)
	\end{displaymath}
	and, since $| \{y_0, y_i'\} \cup T^\circ| \geq |\{y'\}|$, by H\"older's inequality the last sum above is controlled by
	\begin{displaymath}
	2^{p-1} \sum_{i=1}^{\tau} \big( f(y_i')^p + \|f \cdot \mathbf{1}_{T \setminus T'}\|_1^p |\{y_i'\}|^{-p} \big) |\{y_i'\}| 
	\leq 2^{p-1} \Big( \sum_{i=1}^\tau f(y_i')^p |\{y_i'\}| + \tau m^{1-p} \|f \cdot \mathbf{1}_{T \setminus T'}\|_1^p \Big).
	\end{displaymath}
	Notice that $|T \setminus T'| = 2$. We apply H\"older's inequality again in order to get
	\begin{displaymath}
	\|\MN_\TT f\|_p^p \leq 3 \big( 2 + 3^p + 3 \cdot 6^{p-1} (1 + 2^{p-1}\tau m^{1-p}) \big) \|f\|_p^p,
	\end{displaymath}
	and, consequently, we obtain $c_\mathfrak{T}(1,p) \lesssim \max\{1, \tau^{1/p} m^{1/p-1}\}$.
	
	Finally, we estimate $\cc^{\rm c}_\mathfrak{T}(\kappa,p)$ from below for $\kappa \in [1, d)$ and $p \in [1, \infty)$. Take $g \coloneqq \mathbf{1}_{\{y_0\}}$. Then we have $\|g\|_p = 1$ and $\MM_{\kappa, \TT} g (y_i') = (1+ \frac{1}{\tau} + m)^{-1} > (3m)^{-1}$ for each $i \in [\tau]$. Thus, 
	\begin{displaymath}
	\cc^{\rm c}_\mathfrak{T}(\kappa,p) \geq \frac{1}{3m} \, \big| E_{\frac{1}{3m}}(g) \big|^{1/p}  \gtrsim  \tau^{1/p} m^{1/p-1} 
	\end{displaymath} 
	and, combining this with the inequality $\cc^{\rm c}_\mathfrak{T}(\kappa,p) \geq 1$, we obtain the desired estimate.
\end{proof}

\subsection{Composite spaces}\label{S3.3.2}

In the following subsection we use Proposition~\ref{P3.1.1} for certain properly chosen families of basic spaces. The structures obtained here can be regarded as intermediate objects between basic spaces and the spaces we want to construct in the proof of Theorem~\ref{T3.2.2}.   

\begin{lemma}\label{L3.3.4}
	Fix $\widetilde{\kappa} \in [1,2)$, $\widetilde{p} \in [1, \infty)$, $\epsilon \in (0, \frac{1}{4}]$, $\delta \in (0, 2-\widetilde{\kappa})$, and $ N \in \mathbb{N}$. For each $n \in \NN \setminus [N]$ let $\SSS_n = \mathfrak{S}_{\tau_n, d_n, m_n}$ be the basic space of the first type constructed with the aid of $\tau_n = N^{2 \widetilde{p}} {\lfloor}n^{\widetilde{p} (\widetilde{p}-1)/\epsilon}{\rfloor}$, $d_n = \widetilde{\kappa} + \frac{\delta}{n}$, and $m_n = n^{\widetilde{p}/\epsilon}$. Denote by $\widetilde{\SSS} = \widetilde{\SSS}_{\widetilde{\kappa}, \widetilde{p}, \epsilon, \delta, N}$ the space obtained by applying Proposition~\ref{P3.1.1} for $\kappa_0 = \widetilde{\kappa} + \delta$ and the family $\{ \SSS_n : n \in \NN \setminus [N]\}$. Then
	\begin{displaymath}
	\cc_{\widetilde{\SSS}}(\kappa,p) \simeq \cc^{\rm c}_{\widetilde{\SSS}}(\kappa,p)
	\end{displaymath}
	for all $\kappa \in [1, \infty)$ and $p \in [1, \infty]$. Moreover, the following assertions hold:
	\begin{enumerate}[label={\normalfont (\alph*)}]
		\item \label{3a} If $\kappa \in [\widetilde{\kappa} + \delta, \infty)$ or $p \in [\widetilde{p} + 4\epsilon, \infty]$, then $\cc_{\widetilde{\SSS}}(\kappa,p) \simeq 1$.
		\item \label{3b} If $\kappa \in  (\widetilde{\kappa}, \infty)$, then $\cc_{\widetilde{\SSS}}(\kappa,p) < \infty$.
		\item \label{3c} If $\kappa \in [1, \widetilde{\kappa}]$ and $p \in [1, \widetilde{p})$, then $\cc_{\widetilde{\SSS}}(\kappa,p) = \infty$.
		\item \label{3d} If $p \in [\widetilde{p}, \infty]$, then $\cc_{\widetilde{\SSS}}(\kappa,p) \lesssim N^2$.
		\item \label{3e} If $\kappa \in [1, \widetilde{\kappa}]$ and $p \in [\widetilde{p}, \widetilde{p}+\epsilon]$, then $\cc_{\widetilde{\SSS}}(\kappa,p) \gtrsim N^{1/2} $.  	
	\end{enumerate}	 
\end{lemma}

\noindent Figure~\ref{F3.4} describes the behavior of the function $\cc_{\widetilde{\SSS}}(\kappa,p)$ (and thus also of $\cc^{\rm c}_{\widetilde{\SSS}}(\kappa,p)$).

\begin{figure}[H]
	\begin{center}
	\begin{tikzpicture}[
	scale=0.9,
	axis/.style={very thick, ->, >=stealth'},
	important line/.style={thick},
	dashed line/.style={dashed, thin},
	pile/.style={thick, ->, >=stealth', shorten <=2pt, shorten
		>=2pt},
	every node/.style={color=black}
	]
	\draw[axis] (0,0)  -- (6,0) node(xline)[right]
	{$\kappa$};
	\draw[axis] (0,0) -- (0,6) node(yline)[above] {$p$};
	
	\draw[important line] (1,1) -- (6,1);
	\draw[important line] (1,1) -- (1,6);
	
	\draw[dashed line] (1,3) -- (4.2,3);
	\draw[dashed line] (1,3.8) -- (3,3.8);
	\draw[dashed line] (1,5) -- (4.2,5);
	
	\draw[dashed line] (3,1) -- (3,3.8);
	\draw[dashed line] (4.2,1) -- (4.2,5);		
	
	\draw (-0.1,1) node[left] {$1$} -- (0.1,1);
	\draw (-0.1,3) node[left] {$\widetilde{p}$} -- (0.1,3);
	\draw (-0.1,3.8) node[left] {$\widetilde{p}+\epsilon$} -- (0.1,3.8);
	\draw (-0.1,5) node[left] {$\widetilde{p}+4\epsilon$} -- (0.1,5);
	
	\draw (1,-0.1) node[below] {$1$} -- (1,0.1);
	\draw (3,-0.1) node[below] {$\widetilde{\kappa}$} -- (3,0.1);
	\draw (4.2,-0.1) node[below] {$\widetilde{\kappa}+\delta$} -- (4.2,0.1);
	
	\draw (2,2.3) node[below] {$= \infty$} -- (2,2.3);
	\draw (3.6,2.35) node[below] {$< \infty$} -- (3.6,2.35);
	
	\draw (2,3.8) node[below] {$ \gtrsim N^{1/2} $} -- (2,3.8);
	
	\draw (4.9,3.65) node[below] {$ \simeq 1 $} -- (4.9,3.65);
	
	\draw (2.8,4.75) node[below] {$ \lesssim N^2 $} -- (2.8,4.75);
	
	\draw (2.8,5.75) node[below] {$ \simeq 1 $} -- (2.8,5.75);
	
	\end{tikzpicture}
	\caption{The behavior of the function $\cc_{\widetilde{\SSS}}(\kappa,p)$.}
	\label{F3.4}
	\end{center}
\end{figure}

\begin{proof}
	First, observe that $\cc_{\widetilde{\SSS}}(\kappa,p) \simeq \cc^{\rm c}_{\widetilde{\SSS}}(\kappa,p)$ for $\kappa \in [1, \kappa_0]$ and $p \in [1, \infty]$. Indeed, in this case	
	\begin{displaymath}
	\cc_{\SSS_n}(\kappa,p) \simeq \cc^{\rm c}_{\SSS_n}(\kappa,p)
	\end{displaymath}
	holds for each $n \in \NN \setminus [N]$, and hence the same is true if we take the supremum over $n$. Moreover, we have $\kappa_0 \geq d_n$ for each $n$, which implies that $\cc_{\widetilde{\SSS}}(\kappa_0, p) \simeq \cc^{\rm c}_{\widetilde{\SSS}}(\kappa_0, p) \simeq 1$ for all $p \in [1, \infty]$. Combining this with the facts that $\cc_{\widetilde{\SSS}}(\kappa, p)$ and $\cc^{\rm c}_{\widetilde{\SSS}}(\kappa, p)$ are estimated from below by $1$ and nonincreasing as functions of $\kappa$, we conclude that $\cc_{\widetilde{\SSS}}(\kappa,p) \simeq \cc^{\rm c}_{\widetilde{\SSS}}(\kappa,p)$ holds for the full ranges of the parameters $\kappa$ and $p$.
	
	Now, to prove \ref{3a}, it suffices to show that $\cc_{\widetilde{\SSS}}(1,p) \simeq 1$ for $p \in [\widetilde{p} + 4 \epsilon, \infty)$. For each such $p$ and any $n \in \NN \setminus [N]$ we have the inequality
	\begin{align*}
	\cc_{\SSS_n}(1, p) & \lesssim 1 + N^{2\widetilde{p}/p} \cdot n^{\widetilde{p}(\widetilde{p}-1)/(\epsilon p)} \cdot n^{\widetilde{p}(1-p)/(\epsilon p)} \\
	& \lesssim 1 + N^{2\widetilde{p}/p} \cdot n^{\widetilde{p}(\widetilde{p}-p)/(\epsilon p)} \lesssim 1 + N^{2} n^{-2} \lesssim 1,
	\end{align*}
	since $1 \leq \widetilde{p} / p  \leq 2$ and $\widetilde{p}-p \leq - 4 \epsilon$. This implies that
	\begin{displaymath}
	\cc_{\widetilde{\SSS}}(1,p) \lesssim \sup_{n \in \NN \cap (N , \infty)} \cc_{\SSS_n}(1, p) \lesssim 1.
	\end{displaymath}
	The condition \ref{3b}, in turn, is a simple consequence of the fact that if $\kappa \in (\widetilde{\kappa}, \infty)$, then $d_n > \kappa$ holds only for finitely many values of $n$. Next, consider $\kappa \in [1, \widetilde{\kappa}]$ and $p \in [1, \widetilde{p})$ (of course, we can do this only if $\widetilde{p} \in (1, \infty)$). Then
	\begin{align*}
	\limsup_{n \rightarrow \infty} \cc_{\SSS_n}(\kappa, p) & \gtrsim \lim_{n \rightarrow \infty} N^{2\widetilde{p}/p}  \cdot n^{\widetilde{p}(\widetilde{p}-1)/(\epsilon p)} \cdot n^{\widetilde{p}(1-p)/(\epsilon p)} 
	\gtrsim \lim_{n \rightarrow \infty} N^{2\widetilde{p}/p} \cdot n^{\widetilde{p}(\widetilde{p}-p)/(\epsilon p)} = \infty
	\end{align*}
	and hence \ref{3c} holds. To prove \ref{3d} assume that $p \in [\widetilde{p}, \infty)$ (the case $p = \infty$ is trivial). For each $n \in \NN \setminus [N]$ we have
	\begin{align*}
	\cc_{\SSS_n}(\kappa, p) & \lesssim 1 + N^{2\widetilde{p}/p} \cdot n^{\widetilde{p}(\widetilde{p}-1)/(\epsilon p)} \cdot n^{\widetilde{p}(1-p)/(\epsilon p)} \\
	& \lesssim 1+ N^{2\widetilde{p}/p} \cdot n^{\widetilde{p}(\widetilde{p}-p)/(\epsilon p)} \leq 1 + N^{2} \cdot 1 \lesssim N^{2}
	\end{align*}
	which gives
	\begin{displaymath}
	\cc_{\widetilde{\SSS}}(\kappa, p) \lesssim \sup_{n \in \NN \cap (N , \infty)} \cc_{\SSS_n}(\kappa, p) \lesssim N^2.
	\end{displaymath}
	Finally, take $\kappa \in [1, \widetilde{\kappa}]$ and $p \in [\widetilde{p}, \widetilde{p}+\epsilon]$. Since $\frac{3}{4} \leq \widetilde{p} / p \leq 1$ and $-\epsilon \leq \widetilde{p}-p \leq 0$, we have
	\begin{align*}
	\cc_{\widetilde{\SSS}}(\kappa, p) \gtrsim \cc_{\SSS_{2N}}(\kappa, p) & \gtrsim N^{2\widetilde{p}/p} \cdot (2N)^{\widetilde{p}(\widetilde{p}-1)/(\epsilon p)} \cdot (2N)^{\widetilde{p}(1-p)/(\epsilon p)} \\
	& \gtrsim  N^{2\widetilde{p}/p} \cdot (2N)^{\widetilde{p}(\widetilde{p}-p)/(\epsilon p)} \gtrsim N^{3/2} \cdot N^{-1} = N^{1/2},
	\end{align*}
	which justifies \ref{3e} and completes the proof.	
\end{proof}

\begin{lemma}\label{L3.3.5}
	Fix $\widehat{\kappa} \in (1, 2]$ (respectively, $\widehat{\kappa} \in [1, 2)$). For each $n \in \NN$ let $\SSS_n = \SSS_{\tau_n, d_n, m_n}$ be the basic space of the first type constructed with the aid of $\tau_n = n$, $d_n = \widehat{\kappa}$ (respectively, $d_n = \widehat{\kappa} + \frac{2-\widehat{\kappa}}{n}$), and $m_n = 1$. Denote by $\widehat{\SSS} = \widehat{\SSS}_{\widehat{\kappa}}$ the space obtained by applying Proposition~\ref{P3.1.1} for $\kappa_0 = 2$ and the family $\{ \SSS_n : n \in \mathbb{N} \}$. Then $\cc_{\widehat{\SSS}}(\kappa,p) = \infty$ if and only if $\kappa \in [1,\widehat{\kappa})$ (respectively, $\kappa \in [1, \widehat{\kappa}]$) and $p \in [1, \infty)$, and the same is true with $\cc^{\rm c}_{\widehat{\SSS}}(\kappa,p)$ in place of $\cc_{\widehat{\SSS}}(\kappa,p)$.
\end{lemma}

\begin{proof}
	We prove only the first version of this lemma and the second one may be verified similarly. 
	
	Assume that $p \in [1,\infty)$ (the case $p = \infty$ is trivial). If $\kappa \in [1, \widehat{\kappa})$, then for any $n \in \mathbb{N}$ we have $\kappa < d_n$ and hence $\cc^{\rm c}_{\SSS_n}(\kappa,p) \simeq n^{1/p}$. Therefore, 
	\begin{displaymath}
	\cc_{\widehat{\SSS}}(\kappa,p) \geq \cc^{\rm c}_{\widehat{\SSS}}(\kappa,p) \gtrsim \limsup_{n \rightarrow \infty} \cc_{\SSS_n}(\kappa,p) \simeq \lim_{n \rightarrow \infty} n^{1/p} = \infty.
	\end{displaymath}
	Consider the remaining case $\kappa \in [\widehat{\kappa}, \infty)$. Since $\widehat{\kappa} = d_n$, we have $\cc_{\SSS_n}(\widehat{\kappa},p) \simeq 1$. Consequently, 
	\begin{displaymath}
	\cc^{\rm c}_{\widehat{\SSS}}(\kappa,p) \leq \cc_{\widehat{\SSS}}(\kappa,p) \leq \cc_{\widehat{\SSS}}(\widehat{\kappa},p) \lesssim 1 < \infty,
	\end{displaymath}
	which completes the proof.
\end{proof}

Finally, we notice (without furnishing the detailed proof) that using the adequate spaces $\mathfrak{T}_{\tau_n, d_n, m_n}$ instead of $\mathfrak{S}_{\tau_n, d_n, m_n}$ leads to the following counterparts of Lemmas~\ref{L3.3.4}~and~\ref{L3.3.5}.

 \begin{lemma}\label{L3.3.6}
 	Fix $\widetilde{\kappa} \in [1,3)$, $\widetilde{p} \in [1, \infty)$, $\epsilon \in (0, \frac{1}{4}]$, $\delta \in (0, 3-\widetilde{\kappa})$, and $N \in \mathbb{N}$. For each $n \in \NN \setminus [N]$ let $\TT_n = \TT_{\tau_n, d_n, m_n}$ be the basic space of the second type constructed with the aid of $\tau_n = N^{2 \widetilde{p}} {\lfloor}n^{\widetilde{p} (\widetilde{p}-1)/\epsilon}{\rfloor}$, $d_n = \widetilde{\kappa} + \frac{\delta}{n}$, and $m_n = n^{\widetilde{p}/\epsilon}$. Denote by $\widetilde{\TT} = \widetilde{\TT}_{\widetilde{\kappa},\widetilde{p}, \epsilon, \delta, N }$ the space obtained by applying Proposition~\ref{P3.1.1} for $\kappa_0 = \widetilde{\kappa} + \delta$ and the family $\{ \TT_n : n \in \NN \setminus [N]\}$. Then the following assertions hold:
 	\begin{itemize}
 		\item We have $\cc^{\rm c}_{\widetilde{\TT}}(\kappa,p) \simeq 1$ for all $\kappa \in [1, \infty)$ and $p \in [1, \infty]$.
 		\item The conditions \ref{3a}--\ref{3e} from Lemma~\ref{L3.3.4} hold with $\cc_{\widetilde{\TT}}(\kappa,p)$ in place of $\cc_{\widetilde{\SSS}}(\kappa,p)$.  	
 	\end{itemize}	 
 \end{lemma}
 
  \begin{lemma}\label{L3.3.7}
  	Fix $\widehat{\kappa} \in (1, 3]$ (respectively, $\widehat{\kappa} \in [1, 3)$). For each $n \in \NN$ let $\TT_n = \TT_{\tau_n, d_n, m_n}$ be the basic space of the second type constructed with the aid of $\tau_n = n$, $d_n = \widehat{\kappa}$ (respectively, $d_n = \widehat{\kappa} + \frac{3-\widehat{\kappa}}{n}$), and $m_n = 1$. Denote by $\widehat{\TT} = \widehat{\TT}_{\widehat{\kappa}}$ the space obtained by applying Proposition~\ref{P3.1.1} for $\kappa_0 = 3$ and the family $\{ \TT_n : n \in \mathbb{N} \}$. Then $\cc_{\widehat{\TT}}(\kappa,p) = \infty$ if and only if $\kappa \in [1,\widehat{\kappa})$ (respectively, $\kappa \in [1, \widehat{\kappa}]$) and $p \in [1, \infty)$, and the same is true with $\cc^{\rm c}_{\widehat{\TT}}(\kappa,p)$ in place of $\cc_{\widehat{\TT}}(\kappa,p)$.
  \end{lemma}

\subsection{Proof of the main result}\label{S3.3.3}

\begin{proof}[Proof of Theorem~\ref{T3.3.1}]
In the first step we construct a metric measure space $\ZZZ_1$ such that for each $\kappa \in [1, 2)$ the associated modified maximal operators $\MM_{\kappa, \ZZZ_1}$ and $\MN_{\kappa, \ZZZ_1}$ are of weak type $(p,p)$ if and only if $p > h^{\rm c}(\kappa)$ or $p = \infty$, while $\MN_{2, \ZZZ_1}$ is of weak type $(1,1)$. The last property can easily be verified, since only the basic spaces of the first type will be used to build $\ZZZ_1$. 

First, consider the case $h^{\rm c}(\kappa) < \infty$ for each $\kappa \in [1,2]$. Let us introduce a countable set $\Sigma = \Sigma_1 \cup \Sigma_2 = \{\kappa_1, \kappa_2, \dots\}$, where $\Sigma_1$ is the set of all $\kappa \in [1,2)$ for which $\lim_{\kappa' \rightarrow \kappa^+} h^{\rm c}(\kappa') < h^{\rm c}(\kappa)$ (the case $\Sigma_1 = \emptyset$ is possible) and $\Sigma_2$ is a dense subset of the interval $(1,2)$ that has no common points with $\Sigma_1$. By induction we will construct a family of metric measure spaces $\{ \SSS_{n,j} : n, j \in \mathbb{N}\}$ and then we will obtain $\ZZZ_1$ by applying Proposition~\ref{P3.1.1}.

Take $\kappa_1 \in \Sigma$ and let $\delta_1 \in (0, 2-\kappa_1)$ be such that
$h^{\rm c}(\kappa') \geq \lim_{\kappa \rightarrow \kappa_1^+} h^{\rm c}(\kappa) - 1$ for  $\kappa' \in [1,\kappa_1 + \delta_1]$.
For each $j \in \mathbb{N}$ we denote by $\SSS_{1,j}$ the space $\widetilde{\SSS}_{\widetilde{\kappa}, \widetilde{p}, \epsilon, \delta, N}$ from Lemma~\ref{L3.3.4} constructed with the aid of $\widetilde{\kappa}=\kappa_1$, $\widetilde{p}=h^{\rm c}(\kappa_1)$, $\epsilon = \frac{1}{4j}$, $\delta = \frac{\delta_1}{j}$, and $N=j$. Now we let $n \in \NN \setminus \{1\}$ and suppose that for each $i \in [n-1]$ and $j \in \mathbb{N}$ the space $\SSS_{i,j}$ has already been constructed. We choose $\delta_n \in (0, 2-\kappa_n)$ such that the following conditions are satisfied:
\begin{itemize}
	\item We have $h^{\rm c}(\kappa') \geq \lim_{\kappa \rightarrow \kappa_n^+} h^{\rm c}(\kappa) - \frac{1}{n}$ for each $\kappa' \in [1, \kappa_n + \delta_n]$.
	\item If $\kappa_i > \kappa_n$ for some $i \in [n-1]$, then $\kappa_n + \delta_n < \kappa_i$.
\end{itemize} 
For each $j \in \mathbb{N}$ we denote by $\SSS_{n,j}$ the space $\widetilde{\SSS}_{\widetilde{\kappa}, \widetilde{p}, \epsilon, \delta, N}$ from Lemma~\ref{L3.3.4} constructed with the aid to $\widetilde{\kappa}=\kappa_n$, $\widetilde{p}=h^{\rm c}(\kappa_n)$, $\epsilon = \frac{1}{4j}$, $\delta = \frac{\delta_n}{j}$, and $N=j$. Finally, denote by $\ZZZ_1$ the space obtained by applying Proposition~\ref{P3.1.1} for $\kappa_0 = 2$ and the family $\{ \SSS_{n,j} : n, j \in \mathbb{N}\}$. It suffices to show that for each $\kappa \in [1,2)$ we have $\cc^{\rm c}_{\ZZZ_1}(\kappa, h^{\rm c}(\kappa)) = \infty$, while $\cc_{\ZZZ_1}(\kappa, p) < \infty$ if $p > h^{\rm c}(\kappa)$. 

Fix $\kappa \in [1,2)$. If $\lim_{\kappa' \rightarrow \kappa^+} h^{\rm c}(\kappa') < h^{\rm c}(\kappa)$, then $\kappa = \kappa_{n} \in \Sigma$ for some $n \in \mathbb{N}$. Therefore, for each $j \in \NN$ we have $\cc^{\rm c}_{\SSS_{n,j}}(\kappa, h^{\rm c}(\kappa)) \gtrsim j^{1/2}$, which implies that
\begin{displaymath}
\cc^{\rm c}_{\ZZZ_1}(\kappa, h^{\rm c}(\kappa)) \gtrsim \sup_{j \in \NN} \cc^{\rm c}_{\SSS_{n,j}}(\kappa, h^{\rm c}(\kappa)) \gtrsim \lim_{j \rightarrow \infty} j^{1/2} = \infty.
\end{displaymath}
In turn, if $\lim_{\kappa' \rightarrow \kappa^+} h^{\rm c}(\kappa') = h^{\rm c}(\kappa)$, then for each $i \in \NN$ we can choose $\kappa_{n_i} \in \Sigma$ such that $\kappa_{n_i} > \kappa$ and $h^{\rm c}(\kappa_{n_i}) > h^{\rm c}(\kappa) - \frac{1}{4i}$. Hence $\cc^{\rm c}_{\ZZZ_1}(\kappa, h^{\rm c}(\kappa)) \gtrsim \cc^{\rm c}_{\SSS_{n_i,i}}(\kappa, h^{\rm c}(\kappa)) \gtrsim i^{1/2}$ and letting $i \rightarrow \infty$ we obtain $\cc^{\rm c}_{\ZZZ_1}(\kappa, h^{\rm c}(\kappa)) = \infty$.

Next, fix $\kappa \in [1,2)$ and let $p \in (h^{\rm c}(\kappa), h^{\rm c}(\kappa) + 1)$. It is obvious that $\cc_{\SSS_{n,j}}(\kappa,p) < \infty$ for any fixed $n$ and $j$. We will prove that
$
\sup_{n,j \in \NN} \cc_{\SSS_{n,j}}(\kappa,p) < \infty.
$
Let $n_0 \in \NN$ be such that
\begin{displaymath}
h^{\rm c}(\kappa) + \frac{1}{n_0+1} \leq p < h^{\rm c}(\kappa) + \frac{1}{n_0}.
\end{displaymath} 
Take $n \in \mathbb{N}$ such that $\kappa \notin [\kappa_n, \kappa_n + \delta_n)$. With this assumption we obtain $\cc_{\SSS_{n,j}}(\kappa,p) \simeq 1$ for $j \geq n_0+1$. In turn, if $j < n_0+1$, then $\cc_{\SSS_{n,j}}(\kappa,p) \lesssim j^2 \leq (n_0 + 1)^2$. Otherwise, let $n \in \mathbb{N}$ be such that $\kappa \in [\kappa_n, \kappa_n + \delta_n)$. There exists $j_0 = j_0(n)$ such that $\kappa \notin [\kappa_n, \kappa_n + \delta_{n,j_0})$ or $h^{\rm c}(\kappa_n) + \frac{1}{j_0} < p$. This implies that $\cc_{\SSS_{n,j}}(\kappa,p) \simeq 1$ for any $j \geq j_0$. Hence, we deduce that $\sup_{n,j \in \NN} \cc_{\SSS_{n,j}}(\kappa,p) < \infty$ holds provided that $\kappa \notin [\kappa_n, \kappa_n + \delta_{n})$ for all but finitely many values of $n$. Finally, suppose that this is not the case. Then we can choose $l \geq 2(n_0+1)$ such that $\kappa \in [\kappa_l, \kappa_l + \delta_{l})$. If $\kappa \in [\kappa_n, \kappa_n + \delta_{n})$ for some $n > l$, then
\begin{displaymath}
h^{\rm c}(\kappa) \geq \lim_{\kappa' \rightarrow \kappa_l^+} h^{\rm c}(\kappa') - \frac{1}{l} \geq h^{\rm c}(\kappa_n) - \frac{1}{2(n_0+1)},
\end{displaymath} 
since $\kappa_n \in (\kappa_l, \kappa_l + \delta_{l})$, which implies that
\begin{displaymath}
p \geq h^{\rm c}(\kappa) + \frac{1}{n_0+1} \geq h^{\rm c}(\kappa_n) + \frac{1}{2(n_0+1)}.
\end{displaymath}
Hence, for that $n$, if $j \geq 2(n_0+1)$, then $\cc_{\SSS_{n,j}}(\kappa,p) \simeq 1$. Since $\cc_{\SSS_{n,j}}(\kappa,p) \lesssim 4(n_0+1)^2$ for $j < 2(n_0+1)$, we conclude that $\sup_{n,j \in \NN} \cc_{\SSS_{n,j}}(\kappa,p) < \infty$ follows.

Now suppose that $h^{\rm c}$ takes the value $\infty$ and set $a \coloneqq \sup\{\kappa : h^{\rm c}(\kappa) = \infty\}$. If $a = 2$, then we use the appropriate version of Lemma~\ref{L3.3.5} with $\widehat{\kappa}=2$ to choose $\ZZZ_1$. Assume that $a \in [1, 2)$. If $\lim_{\kappa \rightarrow a^+} h^{\rm c}(\kappa) = \infty$, then $h^{\rm c}$ is continuous at $a$ and we just construct $\ZZZ_1$ in the same way as we did in the case $h^{\rm c} < \infty$, but now using $[a,2)$ and $(a,2)$ instead of $[1,2)$ and $(1,2)$, respectively. It is not hard to verify that $\ZZZ_1$ has all the expected properties. Otherwise, we introduce an auxiliary function $h'$ defined by the formula
\begin{displaymath}
h'(\kappa) \coloneqq \left\{ \begin{array}{rl}
h_0 & \textrm{if } \kappa \in [1, a],   \\
h^{\rm c}(\kappa) & \textrm{if } \kappa \in (a, 2], \end{array} \right. 
\end{displaymath} 
where $h_0 = h^{\rm c}(a)$ if $h^{\rm c}(a) < \infty$ or $h_0 = \lim_{\kappa \rightarrow a^+} h^{\rm c}(\kappa)$ otherwise. Let $\ZZZ_1'$ be the space constructed as before with $h'$ instead of $h^{\rm c}$. We use Proposition~\ref{P3.1.1} with $\kappa_0 = 2$ one more time and obtain $\ZZZ_1$ combining $\ZZZ_1'$ with the space from Lemma~\ref{L3.3.5} with $\widehat{\kappa} = a$ (we use the appropriate version of Lemma~\ref{L3.3.5} depending on whether $h^{\rm c}(a) < \infty$ or $h^{\rm c}(a) = \infty$).

In the second step we construct a metric measure space $\ZZZ_2$ such that for each $\kappa \in [1, 3)$ the associated modified maximal operator $\MN_{\kappa, \ZZZ_2}$ is of weak type $(p,p)$ if and only if $p > h(\kappa)$ or $p = \infty$, while for each $\kappa \in [1,2)$ the operator $\MM_{\kappa, \ZZZ_2}$ is of weak type $(p,p)$ for all $p \in [1, \infty]$. The method is similar to that which was used to construct $\ZZZ_1$. The key point is that this time Lemmas~\ref{L3.3.6}~and~\ref{L3.3.7} instead of Lemmas~\ref{L3.3.4}~and~\ref{L3.3.5} should be used. Moreover, Proposition~\ref{P3.1.1} should be applied with $\kappa_0 = 3$. We skip the technical details here.

Finally, we build the metric measure space $\ZZZ$ by applying Proposition~\ref{P3.1.1} with $\kappa_0 = 3$ for $\ZZZ_1$ and $\ZZZ_2$. It is not hard to see that we have: for each $\kappa \in [1,2)$,
\begin{displaymath}
\cc^{\rm c}_{\ZZZ}(\kappa,p) < \infty \iff  \max\{\cc^{\rm c}_{\ZZZ_1}(\kappa,p), \cc^{\rm c}_{\ZZZ_2}(\kappa,p)\} < \infty \iff p \in (h^{\rm c}(\kappa), \infty],
\end{displaymath}
and, for each $\kappa \in [1,3)$,
\begin{displaymath}
\cc_{\ZZZ}(\kappa,p) < \infty  \iff \max\{\cc_{\ZZZ_1}(\kappa,p), \cc_{\ZZZ_2}(\kappa,p)\} < \infty  \iff p \in (h(\kappa), \infty],
\end{displaymath}
where by $(\infty, \infty]$ we mean the singleton $\{\infty\}$. Thus, the space $\ZZZ$ satisfies all the expected conditions and the proof of Theorem~\ref{T3.3.1} is complete.
\end{proof}

\section{Further comments}\label{S3.4}

In the last part of this section we focus on the situation in which we want to find a space $\mathfrak{Z}$ such that the associated modified maximal operators $\MM_{\kappa, \ZZZ}$ and $\MN_{\kappa, \ZZZ}$ are of weak type $(p,p)$ if and only if $p \geq h^{\rm c}(\kappa)$ and $p \geq h(\kappa)$, respectively. In particular, we ask if there is a counterpart of Theorem~\ref{T3.3.1} with these inequalities instead of $p > h^{\rm c}(\kappa)$ and $p > h(\kappa)$. For simplicity, from now on we deal only with the centered operators.

The first example indicates that the answer is positive if $h^{\rm c}$ and $h$ are continuous. 

\begin{example}
	Let $h^{\rm c} \colon [1,2] \rightarrow [1, \infty]$ be a continuous nonincreasing function with $h^{\rm c}(2)=1$. Then there exists a metric measure space $\mathfrak{Z}$ such that for each $\kappa \in [1, 2)$ the associated modified centered maximal operator $\MM_{\kappa, \ZZZ}$ is of weak type $(p,p)$ if and only if $p \geq h^{\rm c}(\kappa)$. 
\end{example}

\begin{proof}
	If $h^{\rm c}(1) = 1$, then the result is trivial since $\mathfrak{Z}$ may be chosen to be $\{a\}$, the set of one point, equipped with the unique metric and counting measure. From now on assume that $h^{\rm c}(1) > 1$. Let us introduce the following auxiliary set
	\begin{displaymath}
	\Omega \coloneqq \{(\kappa,p) \in \big( [1,2] \cap \mathbb{Q} \big) \times \big( [1, \infty) \cap \mathbb{Q} \big) : p < h^{\rm c}(\kappa) \} \coloneqq \{(\kappa_n, p_n) : n \in \mathbb{N}\}.
	\end{displaymath}	
	For each $n \in \mathbb{N}$ we choose $\delta_n \in (0, 2-\kappa_n)$ such that $p_n < h^{\rm c}(\kappa_n + \delta_n)$. We denote by $\SSS_n$ the space $\widetilde{\SSS}_{\widetilde{\kappa}, \widetilde{p}, \epsilon, \delta, N}$ from Lemma~\ref{L3.3.4} constructed with the aid of $\widetilde{\kappa} = \kappa_n$, $\widetilde{p} = p_n$, $\epsilon = \frac{1}{4}$, $\delta = \delta_n$, and $N = 1$. Then it is easy to show that $\mathfrak{Z}$ may be chosen to be the space constructed by applying Proposition~\ref{P3.1.1} for $\kappa_0 = 2$ and the family $\{ \SSS_n : n \in \mathbb{N} \}$.
\end{proof}

The second example shows that there is no counterpart of Theorem~\ref{T3.3.1} for arbitrary functions $h^{\rm c}$ and $h$ satisfying \ref{3.3i}--\ref{3.3iv}. This example is more general in the sense that we take into account all metric measure spaces, not only those satisfying the assumptions specified at the beginning of this chapter. In the proof we will apply the estimates for the operator norm that can be obtained via interpolation (see, for example, \cite[Theorem VIII.9.1, p.~392]{DB}). Moreover, we will use the basic fact that for any metric measure space $\mathfrak{X}$ we have
\begin{displaymath}
\lim_{\kappa \rightarrow \kappa_0^+} \cc^{\rm c}_{\mathfrak{X}}(\kappa, p_0) = \cc^{\rm c}_{\mathfrak{X}}(\kappa_0, p_0), \qquad \kappa_0 \geq 1, \ p_0 \in [1, \infty].
\end{displaymath}

\begin{example}
	Let $(q_1, q_2, \dots)$ be an arbitrary enumeration of the set $\mathbb{Q} \cap (1,2)$. Define
	\begin{displaymath}
	h_0^{\rm c}(\kappa) \coloneqq 2 - \sum_{i \in \NN : q_i < \kappa} \frac{1}{2^i}, \qquad \kappa \in [1,2].
	\end{displaymath}
	Then there is no metric measure space $\XX$ such that for each $\kappa \in [1,2]$ the associated maximal operator $\MM_{\kappa, \XX}$ is of weak type $(p,p)$ if and only if $p \geq h_0^{\rm c}(\kappa)$. 
\end{example}

\begin{proof}
	Suppose by contradiction that $\XX$ is such a space. First we show that for any $1 \leq a < b \leq 2$ and $N \in \mathbb{N}$ we can find $ a \leq a' < b' \leq b$ such that
	$\cc_{\XX}^{\rm c}(\kappa, h_0^{\rm c}(\kappa)) \geq N$ for $\kappa \in [a', b']$.
	Indeed, take $q_i \in (a,b)$ and observe that
	\begin{equation}\label{3.4.1}
	\lim_{\kappa \rightarrow q_i^+} \cc^{\rm c}_{\XX}(\kappa, h_0^{\rm c}(q_i) - 2^{-i-1}) = \infty.
	\end{equation}
	Next, let $\epsilon \in (0, 2 - q_i)$. Then we have $h_0^{\rm c}(q_i)-2^{-i-1} - h_0^{\rm c}(q_i + \epsilon) \geq 2^{-i-1}$. Moreover, note that $q_i \notin (1, q_i)$, which implies that $1 \leq h_0^{\rm c}(q_i) - 2^{-i-1} \leq 2$. Thus, if $\cc^{\rm c}_{\XX}(q_i+\epsilon, h_0^{\rm c}(q_i+\epsilon)) \leq N$, then
	\begin{displaymath}
	\cc^{\rm c}_{\XX}(q_i+\epsilon, h_0^{\rm c}(q_i)-2^{-i-1}) \leq 2 \Big( \frac{h_0^{\rm c}(q_i)-2^{-i-1}}{h_0^{\rm c}(q_i)-2^{-i-1} - h_0^{\rm c}(q_i + \epsilon)} \Big)^{\frac{1}{h_0^{\rm c}(q_i) - 2^{-i-1}}} N^{1 - \frac{h_0^{\rm c}(q_i + \epsilon)}{h_0^{\rm c}(q_i) - 2^{-i-1}}} \leq 2^{i+3}N
	\end{displaymath}
	by interpolation. Of course, in view of \eqref{3.4.1}, such an estimate cannot occur for sufficiently small values of $\epsilon$. Therefore, we can choose an interval $[a', b'] \subset (q_i, b] \subset [a,b]$ with the expected property. The rest of the proof consists of constructing inductively a sequence of closed intervals $[1,2] \supset [a_1, b_1] \supset [a_2, b_2] \supset \dots$ in such a way that for each $n \in \mathbb{N}$ and $\kappa \in [a_n, b_n]$ we have $\cc_{\mathfrak{X}}^{\rm c}(\kappa, h_0^{\rm c}(\kappa)) \geq n$. Clearly, we have $\bigcap_{n=1}^\infty [a_n, b_n] \neq \emptyset$ and $\cc_{\mathfrak{X}}^{\rm c}(\overline{\kappa}, h_0^{\rm c}(\overline{\kappa})) = \infty$ holds for any $\overline{\kappa} \in \bigcap_{n=1}^\infty [a_n, b_n]$. This contradicts the assumption that $\MM_{\overline{\kappa}, \XX}$ is of weak type $(h_0^{\rm c}(\overline{\kappa}),h_0^{\rm c}(\overline{\kappa}))$.
\end{proof} 
\chapter{Boundedness from $L^{\lowercase{p},\lowercase{q}}$ to $L^{\lowercase{p},\lowercase{r}}$}\label{chap4}
\setstretch{1.1}
In the following chapter we return to the standard Hardy--Littlewood maximal operators, centered $\MM$ and noncentered $\MN$, and study their mapping properties in the context of Lorentz spaces $L^{p,q}$. In the case of $\RR^d$ and the classical Lorentz spaces some results may be found in \cite{AM, Sr}, for example. However, little is known in this field about maximal operators associated with general metric measure spaces. In particular, to the author's best knowledge, there are no examples in the literature showing explicitly various peculiar behaviors of $\MM$ and $\MN$ in this context. Here we introduce an appropriate class of spaces which provides the opportunity to generate a lot of such examples. For clarity, we deal only with the centered Hardy--Littlewood maximal operator $\MM$, but we emphasize that very similar analysis may be done also for $\MN$ instead. 

The aim of this part of the dissertation is to study inequalities of the form
\begin{equation}\label{4.0.1}
\|\MM_\XX f \|_{p,r} \leq c(p,q,r,\XX) \|f\|_{p,q}, \qquad f \in L^{p,q}(\XX),
\end{equation}
which, for various parameters $p$, $q$, and $r$, may or may not hold, depending on the structure of $\XX$. To be more precise, we are particularly interested in showing that the sets of parameters for which \eqref{4.0.1} occurs can vary in many different ways. In our approach, to avoid making the problem too complicated, we always assume that the parameter $p$ is fixed. Then the analysis is divided into the following three cases: 
\begin{itemize}
	\setlength\itemsep{0em}
	\item Case I: $q$ fixed and $r$ varying.
	\item Case II: $q$ varying and $r$ fixed.
	\item Case III: both $q$ and $r$ varying.
\end{itemize}
In each of these cases, we illustrate the situation with appropriately selected examples and the general rule is that the more difficult the problem is, the more complicated structures are used. 

The organization of this chapter is as follows. In Section~\ref{S4.1} we describe Lorentz spaces $L^{p,q}(\XX)$. We also present an improved version of the space combining technique introduced in Section~\ref{S2.2}. Sections~\ref{S4.2},~\ref{S4.3},~and~\ref{S4.4}, in turn, are devoted to the study of mapping properties of $\MM$ in the situations corresponding to the three cases specified above. Throughout this chapter, unless otherwise stated, we assume that $(X, \rho)$ is bounded and $|X| < \infty$. We also assume that the measure of each ball is strictly positive. 

\setstretch{1.175}
\section{Preliminaries}\label{S4.1}

We begin with basic information about Lorentz spaces $L^{p,q}(\XX)$.  
For any Borel function $f \colon X \rightarrow \mathbb{C}$ we define the distribution function $d_f \colon [0, \infty) \rightarrow [0, \infty)$ by
\begin{displaymath}
d_f(t) \coloneqq  |\{ x \in X : |f(x)| > t \}|, 
\end{displaymath}
and the decreasing rearrangement $f^* \colon [0, \infty) \rightarrow [0, \infty)$ by
\begin{displaymath}
f^*(t) \coloneqq \inf \{u \in (0,\infty) : d_f(u) \leq t \}.
\end{displaymath} 
Then for any $p \in [1, \infty)$ and $q \in [1, \infty]$ the space $L^{p,q}(\XX)$ consists of those functions $f$ for which the quasinorm $\|f\|_{p,q}$ is finite, where
\begin{displaymath}
\|f\|_{p,q} \coloneqq \left\{ \begin{array}{rl}
p^{1/q} \Big( \int_0^\infty \big( t  d_f(t)^{1/p} \big)^q \frac{{\rm d}t}{t}   \Big)^{1/q}& \textrm{if }  q \in [1, \infty),   \\[5pt]
\sup_{t \in (0,\infty)} t d_f(t)^{1/p} & \textrm{if }  q = \infty, \end{array} \right. 
\end{displaymath}
or, equivalently,
\begin{displaymath}
\|f\|_{p,q} \coloneqq \left\{ \begin{array}{rl}
\Big( \int_0^\infty \big( t^{1/p}  f^*(t) \big)^q \frac{{\rm d}t}{t}   \Big)^{1/q}& \textrm{if }  q \in [1, \infty),   \\[5pt]
\sup_{t \in (0,\infty)} t^{1/p} f^*(t) & \textrm{if }  q = \infty. \end{array} \right. 
\end{displaymath}
The second formula is valid also for $p = \infty$ (here we use the convention $t^{1/\infty} = 1$). However, it turns out that $L^{\infty,q}$ is nontrivial only if $q=\infty$, since in each of the remaining cases it contains only the zero-function. Let us also note that one could consider $L^{p,q}(\XX)$ even for the wider range $p,q \in (0, \infty]$, but this is not the case of our study. 

Many observations and details concerning Lorentz spaces are included in \cite{BS}, for example. For our purposes, it is instructive that one can estimate $\|f\|_{p,q}$ very precisely by calculating the values $d_f(2^k)$ for all $k \in \ZZ$. Furthermore, observe that for each $p \in [1, \infty]$ the space $L^{p,p}(\XX)$ coincides with the usual Lebesgue space $L^p(\XX)$ and hence we write shortly $\|f\|_p$ instead of $\|f\|_{p,p}$. Now we present several facts concerning $L^{p,q}(\XX)$ spaces. The metric measure space is arbitrary here, except for the condition $|X| \in (0,\infty)$ assumed in Fact~\ref{F4.1.2}.

\begin{fact}\label{F4.1.1}
	Let $p \in (1, \infty)$, $q \in [1, \infty]$, and $n_0 \in \NN$. Then there exists a numerical constant $\CC_\triangle(p,q)$ independent of $n_0$ and $\XX$ such that
	\begin{displaymath}
	\big\| \sum_{n=1}^{n_0} f_n \big\|_{p,q} \leq \CC_{\triangle}(p,q) \sum_{n=1}^{n_0} \|f_n \|_{p,q}, \qquad f_n \in L^{p,q}(\XX), \ n \in [n_0]. 
	\end{displaymath} 
\end{fact}

\begin{fact}\label{F4.1.2}
	Let $p \in (1, \infty)$ and $q \in [1, \infty]$, and assume that $|X| \in (0,\infty)$. Then there exists a~numerical constant $\CC_{\rm{avg}}(p,q)$ independent of $\XX$ such that
	\begin{displaymath}
	\| f_{\rm{avg}} \|_{p,q} \leq \CC_{\rm{avg}}(p,q) \|f \|_{p,q}, \qquad f \in L^{p,q}(\XX), 
	\end{displaymath}
	where $ f_{\rm{avg}} \equiv \|f\|_1 / |X|$ is constant. 
\end{fact}

\begin{fact}\label{F4.1.3}
	Let $p \in [1, \infty)$ and $q,r \in [1, \infty]$ with $q \leq r$. Then $L^{p,q}(\XX) \subset L^{p,r}(\XX)$ and there exists a numerical constant $\CC_{\hookrightarrow}(p,q,r)$ independent of $\XX$ such that
	\begin{displaymath}
	\|f \|_{p,r} \leq \CC_{\hookrightarrow}(p,q,r) \|f \|_{p,q}, \qquad f \in L^{p,q}(\XX). 
	\end{displaymath} 
\end{fact}

\noindent All these facts are well known. For the proof of Fact~\ref{F4.1.3} see \cite[Proposition 4.2]{BS}, for example. Facts~\ref{F4.1.1}~and~\ref{F4.1.2} are in turn easy consequences of \cite[Lemma~4.5 and Theorem~4.6]{BS}.

Now we formulate a lemma which will be very useful later on.

\begin{lemma}\label{L4.1.4}
	Let $\XX$ be an arbitrary metric measure space. Fix $p \in [1, \infty)$, $q \in [1, \infty]$, and $n_0 \in \NN$, and consider a finite sequence of functions $(f_n)_{n=1}^{n_0}$ with disjoint supports $A_n \subset X$. Assume that for each $n \in \NN \setminus\{1\}$ and $t \in (0,\infty)$ we have either $d_{f_n}(t) \geq |A_1 \cup \dots \cup A_{n-1}|$ or $d_{f_n}(t) = 0$. Then there exists a numerical constant $\CC_{\rm supp} = \CC_{\rm supp}(p,q)$ independent of $\XX$, $n_0$, and $(f_n)_{n=1}^{n_0}$ such that: if $q \in [1, \infty)$, then
	\begin{displaymath}
	\frac{1}{\CC_{\rm supp}} \, \Big(  \sum_{n=1}^{n_0} \| f_n\|_{p,q}^q \Big)^{1/q} \leq \big\| \sum_{n=1}^{n_0} f_n \big\|_{p,q} \leq \CC_{\rm supp} \, \Big(  \sum_{n=1}^{n_0} \| f_n\|_{p,q}^q \Big)^{1/q}, 
	\end{displaymath}
	and, if $q = \infty$, then
	\begin{displaymath}
	\frac{1}{\CC_{\rm supp}} \, \sup_{n \in [n_0]} \| f_n\|_{p,\infty} \leq \big\| \sum_{n=1}^{n_0} f_n \big\|_{p,\infty} \leq \CC_{\rm supp} \, \sup_{n \in [n_0]} \| f_n\|_{p,\infty}. 
	\end{displaymath}
\end{lemma}

\begin{proof}
	Let $f = \sum_{n=1}^{n_0} f_n$ and consider $q \in [1, \infty)$ (the case $q = \infty$ is very similar). The claim is an easy consequence of the fact that, under the specified assumptions, the quantities $d_{f}(t)^{1/p}$ and $( \sum_{n=1}^{n_0} d_{f_n}(t)^{q/p})^{1/q}$ are comparable with multiplicative constants independent of $t \in (0,\infty)$.    
\end{proof}

\noindent For our purposes, it will also be convenient to state the following variant of Lemma~\ref{L4.1.4}. 

\begin{lemma}\label{L4.1.5}
	Let $\XX$ be an arbitrary metric measure space. Fix $p \in (1, \infty)$ and $q \in [1, \infty]$, and consider a sequence of functions $(f_n)_{n=1}^{\infty}$ with disjoint supports $A_n \subset X$. Assume that for each $n \geq 1$ and $t \in (0,\infty)$ we have either $d_{f_n}(t) \geq |A_{n+1} \cup A_{n+2} \cup \cdots|$ or $d_{f_n}(t) = 0$. Then we have: for $q \in [1, \infty)$,
	\begin{displaymath}
	\frac{1}{\CC_{\rm supp}} \, \Big(  \sum_{n=1}^{\infty} \| f_n\|_{p,q}^q \Big)^{1/q} \leq \big\| \sum_{n=1}^{\infty} f_n \big\|_{p,q} \leq \CC_{\rm supp} \, \Big(  \sum_{n=1}^{\infty} \| f_n\|_{p,q}^q \Big)^{1/q}, 
	\end{displaymath}
	and, for $q = \infty$,
	\begin{displaymath}
	\frac{1}{\CC_{\rm supp}} \, \sup_{n \in \NN} \| f_n\|_{p,\infty} \leq \big\| \sum_{n=1}^{\infty} f_n \big\|_{p,\infty} \leq \CC_{\rm supp} \, \sup_{n \in \NN} \| f_n\|_{p,\infty}, 
	\end{displaymath}
	where $\CC_{\rm supp} = \CC_{\rm supp}(p,q)$ is the constant from Lemma~\ref{L4.1.4}.
\end{lemma}

\begin{proof}
	The proof is identical to the proof of Lemma~\ref{L4.1.4} and hence it is omitted.
\end{proof}

\noindent {\bf Range of parameters.} Recall that given a space $\XX$ we are interested in studying inequalities of the form \eqref{4.0.1} for various parameters $p$, $q$, and $r$. Now we indicate the exact range of parameters that will be taken into account later on. 

For each triple $(p,q,r)$ and each $\XX$ we denote by $\cc(p,q,r,\XX)$ the smallest constant $c(p,q,r,\XX)$ for which \eqref{4.0.1} holds 
(if there is no such constant, then we write $\cc(p,q,r,\XX) = \infty$). Let us mention here that for each fixed $p \in [1, \infty)$ the case $\cc(p,q,r,\XX) < \infty$ is easier to meet for smaller values of $q$ and bigger values of $r$. We say that a triple $(p,q,r)$ is {\it admissible} if one of the following conditions is satisfied: 
\begin{itemize}
	\setlength\itemsep{0em} 
	\item $p=q=1$ and $r \in [1, \infty]$,
	\item $p \in (1, \infty)$ and $q,r \in [1, \infty]$ with $q \leq r$. 
\end{itemize}

The range proposed above seems to be suitable for the following reasons. First, the considered problem is trivial if $p = \infty$. Next, if $r < q$, then we have $\cc(p,q,r,\XX) = \infty$ under very mild assumptions on $\XX$ (see Observation~\ref{O4.1.6}). The reason for this is that there are natural (usually proper) inclusions between Lorentz spaces and the maximal function $\MM_\XX f$ is usually not smaller than the initial function $f$. Finally, the case $p = 1$ and $q \in (1, \infty]$ also turns out to be outside our area of interest (see Observation~\ref{O4.1.7}). 

In Observations~\ref{O4.1.6}~and~\ref{O4.1.7} below we remove the restriction that the diameter of a given space is finite. Moreover, in Observation~\ref{O4.1.6} the condition $|X| < \infty$ is also skipped. Finally, by ${\rm supp}(\mu)$ we mean the \emph{support} of $\mu$, that is, the set $\{x \in X :  |B(x,s)| > 0 \text{ for all } s \in (0,\infty) \}$. 

\begin{observation}\label{O4.1.6}
	Let $\XX$ be such that $|X \setminus {\rm supp}(\mu)| = 0$. Assume that there exists an infinite family $\BB$ of pairwise disjoint balls $B$ satisfying $|B| \in (0,\infty)$. Then for each fixed $p \in (1, \infty)$ and $q,r \in [1, \infty]$ with $r < q$ we have $\cc(p,q,r,\XX) = \infty$.    
\end{observation}

\noindent Indeed, fix $p,q \in (1, \infty)$ and $r \in [1, \infty]$ with $r < q$ (the case $q = \infty$ can be considered very similarly). For any $n_0 \in \NN$ we can find a family of pairwise disjoint sets $\{E_n : n \in [n_0]\}$ with the following properties:
\begin{itemize}
	\setlength\itemsep{0em}
	\item Each $E_n$ is a union of finitely many elements from $\BB$.
	\item For each $n \in [n_0] \setminus \{1\}$ the estimate $|E_n| \geq |E_1 \cup \dots \cup E_{n-1}|$ holds. 
\end{itemize}
Consider $g_{n_0} \in L^{p,q}(\XX)$ defined by
\begin{displaymath}
g_{n_0} \coloneqq \sum_{n=1}^{n_0} n^{-2/(q+r)} |E_n|^{-1/p} \mathbf{1}_{E_n}.
\end{displaymath}
By Lemma~\ref{L4.1.4} the following estimates hold
\begin{displaymath}
\|g\|_{p,q} \leq \CC_{\rm supp}(p,q) \Big(  \frac{p}{q}\Big)^{1/q} \Big( \sum_{n=1}^{n_0} n^{-\frac{2q}{q+r}} \Big)^{1/q},
\qquad
\|g\|_{p,r} \geq \frac{1}{\CC_{\rm supp}(p,r)} \Big(  \frac{p}{r}\Big)^{1/r} \Big( \sum_{n=1}^{n_0} n^{-\frac{2q}{q+r}} \Big)^{1/r}.
\end{displaymath}
Observe that for each $x \in {\rm supp}(\mu)$ we have $\MM_\XX g(x) \geq g(x)$. Since $2r/(q+r) < 1 < 2q/(q+r)$, we obtain $\lim_{n_0 \to \infty} \frac{\|g\|_{p,r}}{\|g\|_{p,q}} = \infty$ and, consequently, $\cc(p,q,r, \XX) = \infty$.
\medskip 

One additional comment should be made here. Namely, if $\BB$ from Observation~\ref{O4.1.6} does not exist, then there are only finitely many points $x \in {\rm supp}(\mu)$ such that $|B(x,s_x)| < \infty$ for some $s_x \in (0, \infty)$. In this case $\MM_\XX$ is trivially bounded between any two Lorentz spaces, provided that $|X \setminus {\rm supp}(\mu)| = 0$ is satisfied. On the other hand, spaces for which $|X \setminus {\rm supp}(\mu)| \neq 0$ holds are rather exotic and will not be considered here.  

\begin{observation}\label{O4.1.7}
	Let $\XX$ be such that $|X| < \infty$. Assume that for any $\epsilon \in (0,\infty)$ there exists a Borel set $E$ with $|E| \in (0, \epsilon)$. Then for any $q \in (1, \infty]$ and $r \in [1, \infty]$ the associated maximal operator $\MM_\XX$ does not map $L^{1,q}(\XX)$ into $L^{1, r}(\XX)$. In particular, we have $\cc(1,q,r,\XX) = \infty$.	
\end{observation}	

\noindent Indeed, fix $q \in (1, \infty)$ and $r \in [1, \infty]$ (we omit the case $q = \infty$ since the thesis is the stronger the smaller $q$ is). Let $\{E_n : n \in \NN\}$ be a family of pairwise disjoint Borel subsets of $X$ such that 
\begin{displaymath}
2^{-l_n - 1} < |E_n| \leq 2^{-l_n},
\end{displaymath}
where $(l_n)_{n \in \NN}$ is an arbitrary sequence of positive integers satisfying $l_{n+1} \geq l_n + 2$. Define
\begin{displaymath}
g \coloneqq \sum_{k=1}^\infty \frac{\mathbf{1}_{E_n}}{n |E_n|}   
\end{displaymath}
and observe that, in view of Lemma~\ref{L4.1.5}, we have
\begin{displaymath}
\| g \|_{1, q} \leq \CC_{\rm supp}(1,q) \Big(  \frac{1}{q}\Big)^{1/q} 
\Big( \sum_{n=1}^{\infty} \frac{1}{n^q} \Big)^{1/q} < \infty.
\end{displaymath}
On the other hand, $|X| < \infty$ implies that for any $x \in X$ we have
\begin{displaymath}
\MM_\XX g(x) \geq \frac{\|g\|_1}{|X|} \geq \frac{1}{|X|} \, \sum_{k=1}^{\infty} \frac{1}{k} = \infty
\end{displaymath}
and hence $\MM_\XX g$ does not belong to $L^{1,r}(\XX)$. 

\medskip
The following remarks will be useful later on.

\begin{remark}\label{R4.1.8}
	Let $\XX = (X, \rho, \mu)$ be an arbitrary metric measure space. Define $\XX' = (X, \rho', \mu')$ by letting $\rho' = C_1 \rho$ and $\mu' = C_2 \mu$ for some numerical constants $C_1, C_2 \in (0, \infty)$. Then for each admissible triple $(p,q,r)$ we have $\cc(p,q,r,\XX) = \cc(p,q,r,\XX')$.
\end{remark}

\noindent Indeed, one can easily see that replacing $\rho$ with $\rho'$ does not change anything since for any $x \in X$ the families $\{ B_\rho(x,s) : s \in (0,\infty) \}$ and $\{ B_{\rho'}(x,s) : s \in (0,\infty) \}$ coincide. Moreover, replacing $\mu$ with $\mu'$ makes that both quasinorms in \eqref{4.0.1} are multiplied by $C_2^{1/p}$.

\begin{remark}\label{R4.1.9}
	Let $\XX = (X, \rho, \mu)$ be an arbitrary metric measure space. Fix an admissible triple $(p,1,r)$ with $p \in (1, \infty)$ and suppose that there exists $C = C(p, r, \XX) \in (0, \infty)$ such that
	\begin{displaymath}
	\|\MM_\XX \mathbf{1}_E \|_{p,r} \leq C \|\mathbf{1}_E\|_{p,1}
	\end{displaymath}
	holds for all measurable sets $E \subset X$ satisfying $|E| < \infty$. Then we have $\cc(p,1,r,\XX) \leq \CC_\mathbf{1} C$, where $\CC_\mathbf{1} = \CC_\mathbf{1}(p,r) \in (0,\infty)$ is some numerical constant independent of $\XX$.
\end{remark}

\noindent Indeed, the result for $r = \infty$ is well known and can be found in the literature (see \cite[Theorem~5.3, p.~231]{BS}). Moreover, careful analysis of the proof in \cite{BS} reveals that the claim follows also for $r \in [1, \infty)$.\\

\noindent {\bf Space combining technique.} At the end of this section we describe how to adapt the technique introduced in Section~\ref{S2.2} to the Lorentz setting. Suppose that for each $n \in \NN$ there is some space $\YY_n=(Y_n, \rho_n, \mu_n)$ for which the behavior of the function $\cc(p,q,r,\YY_n)$ is known. Our goal is to use the family $\{\YY_n : n \in \NN\}$ to create a new space, say $\YY = (Y, \rho, \mu)$, for which $\cc(p,q,r,\YY)$ is comparable to $\sup_{n \in \NN} \, \cc(p,q,r,\YY_n)$. It turns out that $\YY$ may be built in a very transparent way under the additional assumption that each of the spaces $\YY_n$ consists of finitely many elements.

\begin{proposition}\label{P4.1.10}
	Let $(\YY_n)_{n \in \NN}$ be a given sequence of spaces $\YY_n = (Y_n, \rho_n, \mu_n)$. Assume that each of them consists of finitely many elements and $\mu_n(Y_n) \in (0,\infty)$. Let $\YY = (Y, \rho, \mu)$ be the space constructed with the aid of $(\YY_n)_{n \in \NN}$ by using the method described below.

\smallskip
\noindent {\rm Step 1.} Introduce $\rho_n'$ and $\mu_n'$ by rescaling (if necessary) $\rho_n$ and $\mu_n$, respectively, in such a way that the following conditions are satisfied:

\begin{itemize}\setlength\itemsep{0em}
	\item The diameter of $Y_n$ with respect to $\rho_n'$ is strictly smaller than $1$.
	\item For every $y \in Y_n$ and $n \in \NN$ we have $0 < 2 \mu_{n+1}'(Y_{n+1}) \leq \mu_n'(\{y\})$.
\end{itemize} 
\noindent {\rm Step 2.} Denote $\YY_n' \coloneqq (Y_n, \rho_n', \mu_n')$ and notice that in view of Remark~\ref{R4.1.8} we have $\cc(p,q,r,\YY_n) = \cc(p,q,r,\YY_n')$ for each $n \in \NN$, $p \in [1, \infty)$, and $q, r \in [1, \infty]$. 

\smallskip
\noindent 
{\rm Step 3.} Set $Y \coloneqq \bigcup_{n \in \NN} Y_n$, assuming that $Y_{n_1} \cap Y_{n_2} = \emptyset$ for any $n_1 \neq n_2$. Finally, define the metric $\rho$ on $Y$ by
\begin{displaymath}
\rho(y_1,y_2) \coloneqq \left\{ \begin{array}{rl}
\rho_n'(y_1,y_2) & \textrm{if }  \{y_1,y_2\} \subset Y_n \textrm{ for some } n \in \NN,   \\
1 & \textrm{otherwise,} \end{array} \right. 
\end{displaymath} 
and the measure $\mu$ on $Y$ by
\begin{displaymath}
\mu(E) \coloneqq \sum_{n \in \NN} \mu_n'(E \cap Y_n), \qquad E \subset Y.
\end{displaymath}

\smallskip
\noindent
Then for each $p \in (1, \infty)$ and $q,r \in [1, \infty]$ with $q \leq r$ we have 
\begin{equation} \label{eq1}
\frac{1}{\CC} \, \sup_{n \in \NN} \, \cc(p,q,r,\YY_n) \leq \cc(p,q,r,\YY) \leq \CC \,\sup_{n \in \NN} \, \cc(p,q,r,\YY_n),
\end{equation}
where $\CC = \CC(p,q,r)$ is a numerical constant independent of $(\XX_n)_{n \in \NN}$.
\end{proposition}

\begin{proof}
	In the proof, it will be convenient to use the following local and global versions of $\MM_{\YY}$:
	\begin{displaymath}
	\MM_{\rm loc} f(y) \coloneqq \sup_{s \in (0,1]} \frac{1}{|B(y,s)|} \int_{B(y,s)} |f| \, {\rm d}\mu, \qquad 
	\MM_{\rm glob} f(y) \coloneqq \sup_{s \in (1, \infty)} \frac{1}{|B(y,s)|} \int_{B(y,s)} |f| \, {\rm d}\mu.
	\end{displaymath}
	
	First we show the inequality 
	\begin{displaymath}
	\sup_{n \in \NN} \, \cc(p,q,r,\YY_n) \leq \cc(p,q,r,\YY),
	\end{displaymath}
	assuming that $\cc(p,q,r,\YY) < \infty$ holds. For fixed $n \in \NN$ take $f \in L^{p,q}(\YY_n')$ and extend it to $F \in L^{p,q}(\YY)$ by setting $F(y)=0$ for all $ y \in Y \setminus Y_n$. Then $\|F\|_{p,q} = \|f\|_{p,q}$ (here the symbol $\| \cdot \|_{p,q}$ refers to function spaces over different measure spaces). Moreover, by the definition of $\rho$, for any $y \in Y_n$ we have 
	\[
	\MM_{\YY} F(y) \geq \MM_{\rm loc} F(y) = \MM_{\YY_n'} f(y).
	\]
	 Thus, if $\|\MM_{\YY} F\|_{p,r} \leq  \cc(p,q,r,\YY) \|F\|_{p,q}$, then also $\|\MM_{\YY_n'} f\|_{p,r} \leq  \cc(p,q,r,\YY) \|f\|_{p,q}$.
	
	Conversely, let us show
	\begin{displaymath}
	\cc(p,q,r,\YY) \leq \CC \, \sup_{n \in \NN} \, \cc(p,q,r,\YY_n).
	\end{displaymath}
	Assume that $r < \infty$ (the case $r=\infty$ is similar) and take $F \in L^{p,q}(\YY)$. By Fact~\ref{F4.1.1} we have
	\begin{displaymath}
	\| \MM_{\YY} F \|_{p,r} \leq \CC_{\triangle}(p,r) \big( \| \MM_{\rm loc} F \|_{p,r} + \| \MM_{\rm glob} F \|_{p,r} \big).
	\end{displaymath}
	For $n \in \NN$ define $f_n \in L^{p,q}(\YY_n')$ by restricting $F$ to $Y_n$. Using Lemma~\ref{L4.1.5}, together with the definitions of $\rho$ and $\mu$, we see that
	\begin{align*}
	\| \MM_{\rm loc} F \|_{p,r}^r & \leq \CC_{\rm supp}(p,r) \, \Big(  \sum_{n=1}^{\infty} \| \MM_{\rm loc} F \cdot \mathbf{1}_{Y_n} \|_{p,r}^r \Big)^{1/r} \\
	& = \CC_{\rm supp}(p,r) \, \Big(  \sum_{n=1}^{\infty} \| \MM_{\YY_n'} f_n \|_{p,r}^r \Big)^{1/r} \\
	& \leq \CC_{\rm supp}(p,r) \, \sup_{n \in \NN} \, \cc(p,q,r,\YY_n) \, \Big(  \sum_{n=1}^{\infty} \| f_n \|_{p,q}^r \Big)^{1/r}.
	\end{align*}
	Using Lemma~\ref{L4.1.5} again, we obtain 
	\begin{displaymath}
	\Big(  \sum_{n=1}^{\infty} \| f_n \|_{p,q}^r \Big)^{1/r} \leq \Big(  \sum_{n=1}^{\infty} \| f_n \|_{p,q}^q \Big)^{1/q} = \Big(  \sum_{n=1}^{\infty} \| F \cdot \mathbf{1}_{Y_n} \|_{p,q}^q \Big)^{1/q} \leq \CC_{\rm supp}(p,q) \| F \|_{p,q}.
	\end{displaymath}
	Let us now estimate $\| \MM_{\rm glob} F \|_{p,r}$. Note that $\MM_{\rm glob} F \equiv \| F \|_1 / \mu(Y)$ is constant by the definition of $\rho$. Thus, Facts~\ref{F4.1.2}~and~\ref{F4.1.3} imply
	\begin{displaymath}
	\| \MM_{\rm glob} F \|_{p,r} \leq \CC_{\rm{avg}}(p,r) \|F \|_{p,r} \leq \CC_{\rm{avg}}(p,q)  \CC_{\hookrightarrow}(p,q,r) \|F \|_{p,q}.
	\end{displaymath}
	Consequently,
	\begin{displaymath}
	\cc(p,q,r,\YY) \leq \CC_{\triangle}(p,r) \big( \CC_1(p,r) \CC_1(p,q) \, \sup_{n \in \NN} \, \cc(p,q,r,\YY_n) + \CC_{\rm{avg}}(p,q)  \CC_{\hookrightarrow}(p,q,r) \big). 
	\end{displaymath}
	Finally, it remains to notice that $\sup_{n \in \NN} \cc(p,q,r,\YY_n)$ cannot be arbitrarily small. Indeed, taking $g \coloneqq \mathbf{1}_{Y_1} \in L^{p,q}(\YY_1')$ we see that
	\begin{displaymath}
	\| \MM_{\YY_1'} g \|_{p,r} = \| g \|_{p,r} = p^{1/r - 1/q} r^{-1/r} q^{1/q} \|g \|_{p,q} 
	\end{displaymath}
	(here for $r=\infty$ we use the convention $\infty^{1/\infty} = \infty^{-1/\infty} = 1$). Hence,
	\begin{displaymath}
	\sup_{n \in \NN} \cc(p,q,r,\YY_n) \geq \cc(p,q,r,\YY_1) \geq p^{1/r - 1/q} r^{-1/r} q^{1/q}
	\end{displaymath}
	and the proof is complete.
\end{proof} 

\noindent As in the previous sections, we have the following remark.

\begin{remark}\label{R4.1.11}
	Each space $\YY$ obtained by using Proposition~\ref{P4.1.10} is nondoubling. 
\end{remark}

\noindent Indeed, fix $\epsilon \in (0,\infty)$ and let $n_0 = n_0(\epsilon) \in \NN$ be such that $\mu(Y_{n_0}) < \epsilon$. Then for any $y \in Y_{n_0}$ we have $B(y,\frac{3}{2}) = Y_{n_0}$ which implies that $\mu(B(y,\frac{3}{2})) < \epsilon$, while $\mu(B(y,3)) = \mu(Y)$.

\medskip
Several times in Sections~\ref{S4.2}~and~\ref{S4.3}, to avoid notational complications we write shortly $A_1 \lesssim A_2$ (equivalently, $A_2 \gtrsim A_1$) to indicate that $A_1 \leq CA_2$ with a positive constant $C$ independent of all significant quantities (in particular, $A_1 = \infty$ implies that $A_2 = \infty$). We shall write $A_1 \simeq A_2$ if $A_1 \lesssim A_2$ and $A_2 \lesssim A_1$ hold simultaneously.  
While studying the behavior of $\MM_\XX$ acting from $L^{p,q}(\XX)$ to $L^{p,r}(\XX)$, we allow the implicit constant to depend on the parameters $p$, $q$, and $r$, but not on any other factors, including the underlying metric measure space. 
We also use the convention that $[v, v) = (v, v] = \emptyset$ and $[v,v] = \{v\}$ for any $v \in \RR \cup \{\infty\}$. Finally, let us emphasize here that, in view of $\MM_\XX f  = \MM_\XX |f|$, each time we study the behavior of $\MM_\XX$ later on, we restrict our attention to functions $f \geq 0$.\\

\section{Results for $\lowercase{q}$ fixed and $\lowercase{r}$ varying}\label{S4.2}

In this section we describe the situation in which the maximal operator acts on a single Lorentz space $L^{p_0, q_0}(\XX)$. Our goal is to prove the following theorem.

\begin{theorem}\label{T4.2.1}
 	For each admissible triple $(p_0, q_0, r_0)$ the following statements are true:
 	\begin{itemize}
 		\item There exists a (nondoubling) metric measure space $\ZZZ$ such that $\cc(p_0,q_0,r, \ZZZ ) = \infty$ for $r \in [q_0, r_0]$, while $\cc(p_0,q_0,r,\ZZZ) < \infty$ for $r \in (r_0, \infty]$
 		\item There exists a (nondoubling) metric measure space $\ZZZ'$ such that $\cc(p_0,q_0,r,\ZZZ') = \infty$ for $r \in [q_0, r_0)$, while $\cc(p_0,q_0,r,\ZZZ') < \infty$ for $r \in [r_0, \infty]$.	
 	\end{itemize}
\end{theorem}

\noindent The proof of Theorem~\ref{T4.2.1} is located in Subsection~\ref{S4.2.2}.

\subsection{Test spaces of type~I}\label{S4.2.1}

Now we introduce and analyze auxiliary structures which we call the \emph{test spaces of type~{\rm I}}. Each such space is a system of finitely many points equipped with a metric measure structure. Hence, we can use it as a component space in Proposition~\ref{P4.1.10}. \\ 

\noindent {\bf Test spaces of type~I for $p \in (1,\infty)$.} Fix $l \in \NN$ and take a nondecreasing sequence $\mm = \mm(l) = (m_i)_{i=1}^l \subset \NN$. Let $(M_j)_{j=0}^l$ satisfy $M_j = \sum_{i=1}^j m_i$ for each $j \in \{0\} \cup [l]$. We introduce $\SSS = \SSS_{\mm} = (S, \rho, \mu)$, a test space of type~I, as follows. Set $S \coloneqq \{x_i : i \in \{0\} \cup [M_l]\}$, where all points are different. Define $\rho$ determining the distance between two different elements of $S$ by
\begin{displaymath}
\rho(x,y) \coloneqq \left\{ \begin{array}{rl}
1 & \textrm{if } x_0 \in \{x,y\}, \\
2 & \textrm{otherwise.} \end{array} \right.
\end{displaymath}
Finally, take $\mu$ defined by
\begin{displaymath}
\mu(\{x_i\}) \coloneqq \mu_{\mm}(\{x_i\}) \coloneqq \left\{ \begin{array}{rl}
1 & \textrm{if } i=0, \\
2^j & \textrm{if } i \in  [M_j] \setminus [M_{j-1}] \textrm{ for some } j \in [l]. \end{array} \right.
\end{displaymath}
Figure~\ref{F4.1} shows a model of the space $\SSS$.   

\begin{figure}[H]
	\begin{center}
	\begin{tikzpicture}
	[scale=.7,auto=left,every node/.style={circle,fill,inner sep=2pt}]
	
	\node[label={[yshift=-1cm]$x_0$}] (m0) at (8,1) {};
	\node[label=$x_{1}$] (m1) at (5,3)  {};
	\node[label=$x_{2}$] (m2) at (6.5,3)  {};
	\node[label={[yshift=-0.30cm]$x_{M_l-1}$}] (m3) at (9.5,3)  {};
	\node[label={[yshift=-0.14cm]$x_{M_l}$}] (m4) at (11,3)  {};
	\node[dots, scale=2] (m5) at (8,3)  {...};
	
	\foreach \from/\to in {m0/m1, m0/m2, m0/m3, m0/m4}
	\draw (\from) -- (\to);
	\end{tikzpicture}
	\caption{The test space of type I for $p \in (1, \infty)$.}
	\label{F4.1}
	\end{center}
\end{figure}
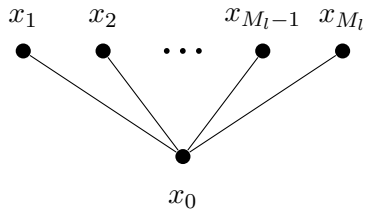
\noindent Note that we can explicitly describe any ball:
\begin{displaymath}
B(x_0,s) = \left\{ \begin{array}{rl}
\{x_0\} & \textrm{for } 0 < s \leq 1, \\
S & \textrm{for } 1 < s, \end{array} \right.
\end{displaymath} 
and, for $i \in [M_l]$,
\begin{displaymath}
B(x_i,s) = \left\{ \begin{array}{rl}
\{x_i\} & \textrm{for } 0 < s \leq 1, \\
\{x_0, x_i\} & \textrm{for } 1 < s \leq 2,  \\
S & \textrm{for } 2 < s. \end{array} \right.
\end{displaymath}
In the following lemma we express the behavior of $\cc(p,q,r, \SSS)$ in terms of $\mm$.

\begin{lemma}\label{L4.2.2}
	Let $\SSS$ be the test space of type~I defined above. Then for each admissible triple $(p,q,r)$ there is a numerical constant $\CC_1 = \CC_1(p, q, r)$ independent of $\mm$ such that: for $r \in [1, \infty)$,
	\begin{displaymath}
	\frac{1}{\CC_1} \, \Big( \sum_{j=1}^{l} 2^{jr(-1 + 1/p)} \, m_j^{r/p}    \Big)^{1/r} \leq \cc(p,q,r,\SSS) \leq \CC_1 \Big( \sum_{j=1}^{l} 2^{jr(-1 + 1/p)} \, m_j^{r/p}    \Big)^{1/r},
	\end{displaymath} 
	and, for $r = \infty$,
	\begin{displaymath}
	\frac{1}{\CC_1} \, \sup_{j \in [l]} \, 2^{j(-1 + 1/p)} m_j^{1/p}  \leq \cc(p,q,r,\SSS) \leq \CC_1 \sup_{j \in [l]} \, 2^{j(-1 + 1/p)} m_j^{1/p}.
	\end{displaymath} 
\end{lemma}

\begin{proof}
	Fix an admissible triple $(p,q,r)$. 
	First we estimate $\cc(p,q,r,\SSS)$ from above. It is worth noting here that if $r \in [1, \infty)$, then
	\begin{displaymath}
	\Big( \sum_{j=1}^{l} 2^{jr(-1 + 1/p)}  m_j^{r/p} \Big)^{1/r} \geq \sup_{j \in [l]} \, 2^{j(-1 + 1/p)}  m_j^{1/p} \geq \frac{1}{2}.
	\end{displaymath}
	Take $f \in L^{p,q}(\SSS)$ such that $\|f\|_{p,q} = 1$. One can easily check that
	\begin{displaymath}
	\MM_{\SSS} f \leq \max \{f, 2 \MM_0 f, f_{\rm avg}\},
	\end{displaymath}
	where $\MM_0 f(x_0) \coloneqq 0$ and $\MM_0 f(x_i) \coloneqq f(x_0) / 2^j$ for $i \in [M_j] \setminus [M_{j-1}]$, $j\in [l]$. By Fact~\ref{F4.1.1},
	\begin{displaymath}
	\|\MM_{\SSS} f\|_{p,r} \leq 2 \, \CC_\Delta(p,r) \, \big( \|f\|_{p,r} +  \| \MM_0 f\|_{p,r} + \| f_{\rm avg} \|_{p,r} \big)
	\end{displaymath}
	and then, by Facts~\ref{F4.1.2},~and~\ref{F4.1.3}, $\|f\|_{p,r} \leq \CC_{\hookrightarrow}(p,q,r)$ and $\| f_{\rm avg} \|_{p,r} \leq \CC_{\rm avg}(p,r) \, \CC_{\hookrightarrow}(p,q,r)$. Thus, it remains to estimate $\| \MM_0 f\|_{p,r}$. Note that $\|f\|_{p,q} = 1$ implies that $f(x_0) \leq (\frac{q}{p})^{1/q}$ if $q \in [1, \infty)$ and $f(x_0) \leq 1$ if $q=\infty$. We consider only the first case and the second one can be treated very similarly. Since $\mm$ is nondecreasing, we have 
	\begin{displaymath}
	d(f,j) \coloneqq d_{\MM_0 f} \Big( \Big(\frac{q}{p} \Big)^{1/q} \, 2^{-j-1} \Big) \leq \left\{ \begin{array}{rl}
	0 & \textrm{for } j \in \ZZ \setminus \NN, \\
	m_j 2^{j+1}  & \textrm{for } j \in  [l-1],  \\
	m_l 2^{l+1} & \textrm{for } j \in \NN \setminus [l-1], \end{array} \right.
	\end{displaymath} 
	which implies that, for $r \in [1,\infty)$,
	\begin{displaymath}
	\| \MM_0 f\|_{p,r}^r \lesssim \sum_{j \in \ZZ} d(f,j)^{r/p}  2^{-jr} 
	\lesssim \sum_{j=1}^{l} 2^{jr(-1 + 1/p)} \, m_j^{r/p}, 
	\end{displaymath}
	and, for $r = \infty$,
	\begin{displaymath}
	\| \MM_0 f\|_{p,r} \lesssim  \sup_{j \in \ZZ} \, d(f,j)^{1/p} 2^{-j} 
	\lesssim  \sup_{j \in [l]} \, 2^{j(-1 + 1/p)} m_j^{1/p}.
	\end{displaymath}
	Finally, to obtain the reverse inequality from the thesis it suffices to take $g \coloneqq \mathbf{1}_{\{x_0\}}$ and calculate $\| \MM_\SSS g \|_{p,r}$. We omit the details here.
\end{proof}

Before we go further let us look at the expression
\begin{equation}\label{4.2.1}
\Big( \sum_{j=1}^{l} 2^{jr(-1 + 1/p)} \, m_j^{r/p}    \Big)^{1/r}
\end{equation}
that appears in the thesis of Lemma~\ref{L4.2.2}. Observe that if $p \in (1, \infty)$, then the factor $ 2^{jr(-1 + 1/p)}$ tends rapidly to $0$ as $j$ tends to $\infty$. Thus, for example, given $r_0 \in [1, \infty)$, we can find a~nondecreasing sequence of positive integers $(m_j)_{j \in \NN}$ such that the series in \eqref{4.2.1} diverges if and only if $r \in [1, r_0]$. Unfortunately, this idea does not work for $p = 1$ and hence we consider this case separately. \\

\noindent {\bf Test spaces of type~I for $p=1$.}
Fix $l \in \NN$ and take a nondecreasing sequence of positive integers $\widetilde{\mm} = \widetilde{\mm}(l) = (\widetilde{m}_j)_{j=1}^l$ satisfying $\widetilde{m}_1 = 1$. Next, associate with $\widetilde{\mm}$ a strictly increasing sequence of positive integers $(h_j)_{j=1}^l$ such that
\begin{equation}\label{4.2.2}
\lfloor 2^{h_{j+1}} / \widetilde{m}_{j+1} \rfloor > 2^{h_j}
\end{equation}
for each $j \in [l-1]$. We introduce $\widetilde{\SSS} = \widetilde{\SSS}_{\widetilde{\mm}} = (S, \rho, \mu)$, a test space of type~I, as follows. Set 
\begin{displaymath}
S \coloneqq \{x_0\} \cup \big\{x_{j,k} : k \in [2^{h_j}], \ j \in [l]\big\}, 
\end{displaymath}
where all elements are different. We use auxiliary symbols for certain subsets of $S$. Namely, we set $\widetilde{S}_0 \coloneqq S_{l+1} \coloneqq \emptyset$ and for $j \in [l]$ denote 
\begin{displaymath}
S_j \coloneqq \big\{x_{j,k} : k \in [2^{h_j}]  \big\}, \qquad \widetilde{S}_j \coloneqq \big\{x_{j,k} : k \in [2^{h_j}] \setminus [ 2^{h_j}/ \widetilde{m}_j ] \big\},
\end{displaymath}
where we use the convention $[c] = [ \lfloor c \rfloor ]$ for noninteger $c \in (0,\infty)$ (notice that if $\widetilde{m}_j= 1$ for some $j$, then $\widetilde{S}_j = \emptyset$). Then we define the metric $\rho$ determining the distance between two different elements $x, y \in S$ by the formula
\begin{displaymath}
\rho(x,y) \coloneqq \left\{ \begin{array}{rl}
1 & \textrm{if } x_0 \in \{x,y\} \textrm{ or } \{x,y\} \in \widetilde{S}_{j-1} \cup S_j \textrm{ for some } j \in [l],  \\
2 & \textrm{otherwise.} \end{array} \right. 
\end{displaymath}
Finally, we let $\mu$ to be counting measure. Again we can explicitly describe any ball:
	\begin{displaymath}
	B(x_0,s) = \left\{ \begin{array}{rl}
	\{x_0\} & \textrm{for } 0 < s \leq 1, \\
	S & \textrm{for } 1 < s, \end{array} \right.
	\end{displaymath} 
	\noindent for $k \in [ 2^{h_j}/ \widetilde{m}_j ]$, $j \in [l]$,
	\begin{displaymath}
	B(x_{j,k},s) = \left\{ \begin{array}{rl}
	\{x_{j,k}\} & \textrm{for } 0 < s \leq 1, \\
	\{x_0\} \cup \widetilde{S}_{j-1} \cup S_j & \textrm{for } 1 < s \leq 2,  \\
	S & \textrm{for } 2 < s, \end{array} \right.
	\end{displaymath}
	\noindent and, for $k \in [2^{h_j}] \setminus [ 2^{h_j} / \widetilde{m}_j ]$, $j \in [l]$,
	\begin{displaymath}
	B(x_{j,k},s) = \left\{ \begin{array}{rl}
	\{x_{j,k}\} & \textrm{for } 0 < s \leq 1, \\
	\{x_0\} \cup \widetilde{S}_{j-1} \cup S_j \cup S_{j+1}& \textrm{for } 1 < s \leq 2,  \\
	S & \textrm{for } 2 < s. \end{array} \right.
	\end{displaymath}
	In the following lemma we express the behavior of $\cc(1,1,r, \widetilde{\SSS})$ in terms of $\widetilde{\mm}$. 
	
	\begin{lemma}\label{L4.2.3}
		Let $\widetilde{\SSS}$ be the test space of type~I defined above. Then for each $r \in [1, \infty]$ there is a numerical constant $\CC_1 = \CC_1(r)$ independent of $\widetilde{\mm}$ such that: for $r \in [1, \infty)$,
		\begin{displaymath}
		\frac{1}{\CC_1} \, \Big( \sum_{j=1}^{l-1} (\widetilde{m}_j)^{-r} \Big)^{1/r} \leq \cc(1,1,r,\widetilde{\SSS}) \leq \CC_1 \, \Big( \sum_{j=1}^{l-1} (\widetilde{m}_j)^{-r} \Big)^{1/r},
		\end{displaymath} 
		and, for $r = \infty$,
		\begin{displaymath}
		\frac{1}{\CC_1} \leq \cc(1,1,r,\widetilde{\SSS}) \leq \CC_1.
		\end{displaymath} 		
	\end{lemma}
	
	\begin{proof}
		Fix $r \in [1, \infty]$. First we estimate $\cc(1,1,r,\widetilde{\SSS})$ from above. 
		It is worth mentioning that if $r \in [1, \infty)$, then $\widetilde{m}_1 = 1$ implies that $\big( \sum_{j=1}^{l-1} (\widetilde{m}_j)^{-r} \big)^{1/r} \geq 1$. We take $f \in L^1(\widetilde{\SSS})$ with $\|f\|_1 = 1$. One can easily check that
		\begin{displaymath}
		\MM_{\widetilde{\SSS}} f \leq \max \{f, \MM_0 f, f_{\rm avg}\},
		\end{displaymath}
		where $\MM_0 f(x_0) \coloneqq 0$ and $\MM_0 f(x) \coloneqq |B(x,\frac{3}{2})|^{-1}$
		for $x \in S \setminus \{x_0\}$. Therefore,
		\begin{displaymath}
		\|\MM_{\widetilde{\SSS}} f\|_{1,r} \lesssim \|f\|_{1,r} + \| \MM_0 f\|_{1,r} + \| f_{\rm avg}\|_{1,r}
		\end{displaymath}
		(now Fact~\ref{F4.1.1} is not available but we still have the quasitriangle inequality since the number of summands is controlled uniformly). By Fact~\ref{F4.1.3} we have $\|f\|_{1,r} \leq \CC_{\hookrightarrow}(1,1,r)$ and $\| f_{\rm avg} \|_{1,r} \leq \CC_{\hookrightarrow}(1,1,r) \, \| f_{\rm avg} \|_{1} = \CC_{\hookrightarrow}(1,1,r)$. It remains to estimate $\| \MM_0 f\|_{1,r}$. By using \eqref{4.2.2} and the fact that $(h_j)_{j=1}^l$ is strictly increasing we obtain 
		\begin{displaymath}
		d(f,i) \coloneqq d_{\MM_0 f} \big( 2^{-i} \big) \lesssim \left\{ \begin{array}{rl}
		0 & \textrm{for } i \in (\ZZ \setminus \NN) \cup [h_1], \\
		2^{h_j} (\widetilde{m}_j)^{-1} & \textrm{for } i \in [h_{j+1}] \setminus [h_j], \ j \in [l-1],  \\
		2^{h_l} & \textrm{for } i \in \NN \setminus [h_l], \end{array} \right.
		\end{displaymath} 
		which implies that, for $r \in [1, \infty)$,
		\begin{displaymath}
		\| \MM_0 f\|_{1,r}^r 
		\lesssim \sum_{i \in \ZZ} d(f,i)^{r} 2^{-ir} 
		\lesssim \sum_{j=1}^{l-1} (\widetilde{m}_j)^{-r}
		\end{displaymath}
		and, for $r = \infty$,
		\begin{displaymath}
		\| \MM_0 f\|_{1,r} \lesssim \sup_{i \in \ZZ} d(f,i) 2^{-i} \lesssim 1
		\end{displaymath}
		Finally, to obtain the reverse inequality from the thesis it suffices to take $g \coloneqq \mathbf{1}_{\{x_0\}}$ and calculate $\| \MM_{\widetilde{\SSS}} g \|_{1,r}$. Again we omit the details.
	\end{proof}

\subsection{Proof of the main result}\label{S4.2.2}

\begin{proof}[Proof of Theorem~\ref{T4.2.1}]
Fix an admissible triple $(p_0, q_0, r_0)$. We consider four cases depending on the values of $p_0$ and $r_0$.

\smallskip \noindent \bf Case~\hypertarget{case1}{1}: \rm $p_0 \in (1, \infty)$ and $r_0 \in [1, \infty)$. 
First we obtain $\ZZZ$ such that $\cc(p_0,q_0,r, \ZZZ ) = \infty$ for $r \in [q_0, r_0]$, while $\cc(p_0,q_0,r,\ZZZ) < \infty$ for $r \in (r_0, \infty]$. Let $(a_i )_{i \in \NN}$ be given by the formula $a_i =  2^{i(p_0-1)} i^{-p_0/r_0}$ and let $i_0 \in \NN$ be the first index such that $a_{i_0+1} \geq 1$ and $(a_i)_{ i = i_0+1}^\infty$ is nondecreasing. Thus, the sequence $(\bar{a}_i)_ {i \in \NN}$ satisfying 
\begin{displaymath}
\bar{a}_i = \left\{ \begin{array}{rl}
1 & \textrm{for } i \in [i_0], \\
\lceil a_i \rceil & \textrm{for } i \in \NN \setminus [i_0], \end{array} \right.
\end{displaymath}
is also nondecreasing (here the symbol ${\lceil} \cdot {\rceil}$ refers to the ceiling function). Then, for any $n \in \NN$ let $\SSS_n = \SSS_{\mm_n(l_n)}$ be the test space of type~I constructed with the aid of $l_n = n$ and $\mm_n = (\bar{a}_1, \dots, \bar{a}_n )$. We denote by $\ZZZ$ the space $\YY$ obtained by applying Proposition~\ref{P4.1.10} with $\YY_n = \SSS_n$ for each $n \in \NN$. It is not hard to see that $\ZZZ$ satisfies the desired properties. Indeed, fix $r \in [1, \infty)$. By using Lemma~\ref{L4.2.2} we have  
\begin{displaymath}
\cc(p_0,q_0,r,\ZZZ) \simeq  \sup_{n \in \NN} \Big( \sum_{i=1}^{n} 2^{ir(-1 + 1/p_0)} \, \bar{a}_i^{r/p_0}    \Big)^{1/r}
\end{displaymath}
which gives
\begin{equation*}\label{4.2.3}
\cc(p_0,q_0,r,\ZZZ) \simeq \sup_{n \in \NN \setminus [i_0]} \Big( \sum_{i=1}^{i_0} 2^{jr(-1 + 1/p_0)} +  \sum_{i=i_0+1}^{n}  i^{-r/r_0} \Big)^{1/r}.
\end{equation*}
We can easily see that the second series above tends to $\infty$ as $n \rightarrow \infty$ if and only if $r \in [1, r_0]$. 

Finally, a slight modification of the argument above allows us to get a space $\ZZZ'$ such that $\cc(p_0,q_0,r,\ZZZ') = \infty$ for $r \in [q_0, r_0)$, while $\cc(p_0,q_0,r,\ZZZ') < \infty$ for $r \in [r_0, \infty]$. Namely, instead of $(a_i)_{ i \in \NN }$ we will use a family of sequences $ \{ ( a_i^{(n)} )_{ i \in \NN} : n \in \NN \}$, where $a_i^{(n)} = a_i \, \log(n+3)^{-p_0/r_0}$. Then for each $n \in \NN$ we build $(\bar{a}_i^{(n)})_{ i \in \NN}$ as before, this time using $(a_i^{(n)})_{ i \in \NN}$ and the critical index $i_0^{(n)}$. After all, we let $\SSS'_n = \SSS_{\mm_n(l_n)}$ be the test space of type~I constructed with the aid of $l_n = n$ and $\mm_n = ( \bar{a}_1^{(n)}, \dots, \bar{a}_n^{(n)})$. Then the space $\ZZZ'$ is constructed by applying Proposition~\ref{P4.1.10} with $\YY_n = \SSS'_n$ for each $n \in \NN$. It is clear that
\begin{displaymath}
\cc(p_0,q_0,r_0,\ZZZ') \lesssim \sup_{n \in \NN} \Big( \sum_{i=1}^{\infty} 2^{jr_0(-1 + 1/p_0)} +  \sum_{i=1}^{n}  i^{-1} \, \log(n+3)^{-1} \Big)^{1/r_0}
\end{displaymath}
and, since the quantity above is finite, we obtain $\cc(p_0,q_0,r,\ZZZ') < \infty$ for $r \in [r_0, \infty]$. Now let $r \in [q_0, r_0)$. Since for every $n \in \NN$ the sequence $ ( a_i^{(n)})_{ i = i_0 +1}^\infty$ is nondecreasing, for each $n \in \NN$ we have $ \bar{a}_i^{(n)} \geq a_i^{(n)}$ whenever $i > i_0$. Thus,
\begin{displaymath}
\cc(p_0,q_0,r,\ZZZ') \gtrsim \sup_{n \in \NN \setminus [i_0]} \Big(\sum_{i=i_0+1}^{n}  i^{-r/r_0} \, \log(n+3)^{-r/r_0} \Big)^{1/r_0}.
\end{displaymath}
Since the quantity above is infinite, the argument is complete.

\smallskip \noindent \bf Case~\hypertarget{case2}{2}: \rm $p_0 \in (1, \infty)$ and $r_0 = \infty$. 
Let $(b_i)_{ i \in \NN}$ be a nondecreasing sequence of positive integers such that $ \lim_{i \rightarrow \infty} 2^{i(-1+1/p_0)} b_i^{1/p_0} =  \infty$. Then for each $n \in \NN$ we let $\SSS_n = \SSS_{\mm_n(l_n)}$ be the test space of type~I constructed with the aid of $l_n = n$ and $\mm_n = (b_1, \dots, b_n )$. Finally, we let $\ZZZ$ to be the space $\YY$ obtained by applying Proposition~\ref{P4.1.10} with $\YY_n = \SSS_n$ for each $n \in \NN$. It is routine to check that $\cc(p_0,q_0,r,\ZZZ) = \infty$ holds for every $r \in [q_0, \infty]$.

In order to obtain $\ZZZ'$ such that $\cc(p_0,q_0,r,\ZZZ') = \infty$ for $r \in [q_0, \infty)$, while $\cc(p_0,q_0,\infty,\ZZZ') < \infty$ we use a proper variant of the diagonal argument. Namely, consider a sequence $(r^{(i)})_{i \in \NN}$ such that $r^{(i)} \in [1, \infty)$ for each $i \in \NN$ and $\lim_{i \rightarrow \infty} r^{(i)} = \infty$. Then for each $i \in \NN$ let $\{\SSS_n^i : n \in \NN\}$ be the family consisting of the test spaces used in Case~\hyperlink{case1}{1} to build $\ZZZ$ for $r_0 = r^{(i)}$. Now we construct $\ZZZ'$ by applying Proposition~\ref{P4.1.10} to the whole family $\{\SSS_n^i : n, i \in \NN\}$. For every $r \in [1, \infty)$ there is $i_0 \in \NN$ such that $r^{(i_0)} > r$, which implies that
\begin{displaymath}
\cc(p_0,q_0,r,\ZZZ') \gtrsim \sup_{n \in \NN} \cc(p_0,q_0,r,\SSS_n^{i_0})  = \infty.
\end{displaymath}
On the other hand, it is not hard to see that for each $n, i \in \NN$ we have $\cc(p_0,q_0,\infty,\SSS_n^{i}) \lesssim 1$, which implies $\cc(p_0,q_0,\infty,\ZZZ') < \infty$.  

\smallskip  \noindent \bf Case~\hypertarget{case3}{3}: \rm $p_0 = 1$ and $r_0 \in [1, \infty)$. 
First we obtain a space $\ZZZ$ such that $\cc(1,1,r, \ZZZ ) = \infty$ for $r \in [1, r_0]$, while $\cc(1,1,r,\ZZZ) < \infty$ for $r \in (r_0, \infty]$. For each $i \in \NN$ set $c_i = \lfloor i^{1/r_0} \rfloor$ and observe that $(c_i)_{ i \in \NN}$ is nondecreasing. For each $n \in \NN$ let $\widetilde{\SSS}_n = \widetilde{\SSS}_{\widetilde{\mm}_n(l_n)}$ be the test space of type~I constructed with the aid of $l_n = n$ and $\widetilde{\mm}_n = (c_1, \dots, c_n)$. We denote by $\ZZZ$ the space $\YY$ obtained by applying Proposition~\ref{P4.1.10} with $\YY_n = \widetilde{\SSS}_n$ for each $n \in \NN$. Again, it is not hard to see that $\ZZZ$ satisfies the desired properties. Indeed, fix $r \in [1, \infty)$. By using Lemma~\ref{L4.2.3} we have that $\cc(1,1,r,\ZZZ)$ is comparable to 
$\sup_{n \in \NN} \big( \sum_{i=1}^{n} i^{-r/r_0} \big)^{1/r}$ which is equal to $\infty$ if and only if $r \in [1, r_0]$. 

Now we build $\ZZZ'$ such that $\cc(1,1,r,\ZZZ') = \infty$ for $r \in [1, r_0)$, while $\cc(1,1,r,\ZZZ') < \infty$  for $r \in (r_0, \infty]$. For each $n \in \NN$ let $ ( c_i^{(n)})_{ i \in \NN}$ be defined by $ c_i^{(n)} = \lfloor i^{1/r_0} \log(n+3)^{1/r_0} \rfloor$. We let $\widetilde{\SSS}'_n = \widetilde{\SSS}_{\widetilde{\mm}_n(l_n)}$ be the test space of type~I constructed with the aid of $l_n = n$ and $\widetilde{\mm}_n = ( c_1^{(n)}, \dots, c_n^{(n)} )$ and construct $\ZZZ'$ by applying Proposition~\ref{P4.1.10} with $\YY_n = \widetilde{\SSS}'_n$ for each $n \in \NN$. Then, by using Lemma~\ref{L4.2.3}, for each fixed $r \in [1, \infty)$ we obtain that $\cc(1,1,r,\ZZZ')$ is comparable to 
$\sup_{n \in \NN} \big( \sum_{i=1}^{n} i^{-r/r_0} \, \log(n+3)^{-r/r_0} \big)^{1/r}$ which is equal to $\infty$ if and only if $r \in [1, r_0)$. 

\smallskip \noindent \bf Case~\hypertarget{case4}{4}: \rm $p_0 = 1$ and $r_0 = \infty$.
In order to obtain $\ZZZ$ such that $\cc(1,1,r,\ZZZ) = \infty$ for every $r \in [1, \infty]$ we proceed as in Case~\hyperlink{case2}{2}. Namely, we choose a nondecreasing sequence of positive integers $(d_i)_{ i \in \NN}$ such that $ \lim_{i \rightarrow \infty} d_i =  \infty$. Next, for any $n \in \NN$ we let $\SSS_n = \SSS_{\mm_n(l_n)}$ be the test space of type~I constructed with the aid of $l_n = n$ and $\mm_n = (d_1, \dots, d_n)$. Then we denote by $\ZZZ$ the space $\YY$ obtained by applying Proposition~\ref{P4.1.10} with $\YY_n = \SSS_n$ for each $n \in \NN$. By using Lemma~\ref{L4.2.2} we conclude that $\ZZZ$ satisfies the desired properties.

Finally, we can obtain $\ZZZ'$ such that $\cc(1,1,r,\ZZZ') = \infty$ for $r \in [1, \infty)$, while $\cc(1,1,\infty,\VV) < \infty$, by using a family of spaces $\{ \widetilde{\SSS}^i_n : n, i \in \NN\}$ introduced similarly to the family $\{ \SSS^i_n : n, i \in \NN\}$ considered in Case~\hyperlink{case2}{2}, but this time choosing appropriate test spaces $\widetilde{\SSS}_{\widetilde{\mm}}$ considered in Case~\hyperlink{case3}{3}. We skip the details here.
\end{proof} 

\section{Results for $\lowercase{q}$ varying and $\lowercase{r}$ fixed}\label{S4.3}

Our next goal is to construct spaces $\XX$ such that, given admissible triples $(p_0,q_0,r_0)$ and $(p_0,q_0',r_0)$ with $q_0' < q_0$, there exists a significant difference between the behaviors of $\MM_\XX$ acting from $L^{p_0,q_0}(\XX)$ to $L^{p_0, r_0}(\XX)$ and from $L^{p_0,q_0'}(\XX)$ to $L^{p_0, r_0}(\XX)$, respectively. The approach proposed in this section allows us to obtain such a result only if $q_0' = 1$. The reason for such a limitation is that the use of Remark~\ref{R4.1.9} is required here. Nevertheless, the following theorem is a good starter before the main course served in Section~\ref{S4.4}. 

\begin{theorem}\label{T4.3.1}
	Fix an admissible triple $(p_0, q_0, r_0)$ with $q_0 \in (1, \infty]$. Then there exists a~(nondoubling) metric measure space $\ZZZ$ such that $\cc(p_0,1,r_0, \ZZZ ) < \infty$ and $\cc(p_0,q_0,r_0, \ZZZ ) = \infty$. 
\end{theorem}

\noindent The proof of Theorem~\ref{T4.3.1} is located in Subsection~\ref{S4.3.2}.

\subsection{Test spaces of type~II}\label{S4.3.1}

Let us begin with the following observation. Each test space introduced in Subsection~\ref{S4.2.1} had one central point, namely $x_0$, and the function $\mathbf{1}_{\{x_0\}}$ played the main role in estimating the size of $\cc(p,q,r,\SSS)$ or $\cc(p,q,r,\widetilde{\SSS})$. Since the values $\|\mathbf{1}_{\{x_0\}} \|_{p_0,1}$ and $\|\mathbf{1}_{\{x_0\}} \|_{p_0,q_0}$ are comparable, we are now forced to change the strategy and introduce test spaces of another type, say $\TT$, for which the size of $\cc(p_0,q_0,r_0,\TT)$ will be calculated by testing the action of the associated maximal operator on more complicated functions. This can be done if we ensure that the new spaces will have more central points grouped into several different types. The detailed analysis will be made separately for the following two cases: $q_0, r_0 \in (1,\infty)$ with $q_0 \leq r_0$ or $q_0 \in (1,\infty)$ and $r_0 = \infty$. We omit the case $q_0 = r_0 = \infty$. \\

\noindent {\bf Test spaces of type~II for $r \in (1, \infty)$.} Fix $l \in \NN$ and an admissible triple $(p, q, r)$ with $q, r \in (1,\infty)$ (note that, in particular, $p \in (1, \infty)$ and $q \leq r$). Associate with the quadruple $(p, q, r, l)$ four sequences of positive integers, $(m_i)_{i=1}^l$, $(h_i)_{i=1}^l$, $(\alpha_i)_{i=1}^l$, and $(\beta_i)_{i=1}^l$, with the following properties:
\begin{enumerate}[label=(\roman*)]
	\setlength\itemsep{0em}
	\item \label{4i} $h_{i+1} / h_i \in \NN$,
	\item \label{4ii} $m_{i+1} \geq 2 m_i h_i$,
	\item \label{4iii} $1 \leq m_{i}^{1-p}h_i <2$,
	\item \label{4iv} $l^{\frac{p}{(p-1)r}} \alpha_1 \geq 2 m_l h_l$,
	\item \label{4v} $\alpha_{i+1} \geq 2 \alpha_i \beta_i$,
	\item \label{4vi} $1 \leq \alpha_{i}^{1-p} \beta_i h_i <2$.	
\end{enumerate}

\noindent The sequences introduced above will determine the structure of the test space constructed in this section. Let us emphasize that the properties \ref{4i}--\ref{4vi} can be met simultaneously. Indeed, let $h_1 = m_1 = 1$ and choose $m_2$ such that $m_2 \geq 2 m_1 h_1$ and the set $\{h \in \NN : 1 \leq m_{2}^{1-p}h < 2 \}$ contains at least $h_1$ elements. Thus, it is possible to take $h_2$ for which the conditions $h_2 / h_1 \in \NN$ and $1 \leq m_{2}^{1-p}h_2 <2$ are satisfied. Continue this way until the whole sequences $(m_i)_{i=1}^l$ and $(h_i)_{i=1}^l$ are chosen. Next, let $\alpha_1$ satisfies $l^{\frac{p}{(p-1)r}} \alpha_1 \geq 2 m_l h_l$ and $\alpha_1^{1-p} h_1 <2$. Take $\beta_1$ such that $1 \leq \alpha_{1}^{1-p} \beta_1 h_1 <2$. Then choose $\alpha_{2}$ such that $\alpha_2 \geq 2 \alpha_1 \beta_1$ and $\alpha_2^{1-p} h_2 < 2$ and take $\beta_2$ satisfying $1 \leq \alpha_{2}^{1-p} \beta_2 h_2 <2$. Continue this way until the whole sequences $(\alpha_i)_{i=1}^l$ and $(\beta_i)_{i=1}^l$ are chosen.

\medskip Now we formulate a few thoughts that one should keep in mind later on.
\begin{itemize}\setlength\itemsep{0em}
	\item The sequences $(m_i)_{i=1}^l$ and $(\alpha_i)_{i=1}^l$ are used to define the associated measure, while $(h_i)_{i=1}^l$ and $(\beta_i)_{i=1}^l$ help to describe the number of elements of a given type.
	\item The property \ref{4i} allows the set of points of a given type to be divisible into an appropriate number of equinumerous subsets.
	\item The properties \ref{4ii} and \ref{4v} say that the sequences $(m_i)_{i=1}^l$ and $(\alpha_i)_{i=1}^l$ grow very fast. This fact results in large differences between the masses of points of different types, which in turn simplifies many calculations regarding the distribution function.
	\item The properties \ref{4iii} and \ref{4vi} are of rather technical nature. They are responsible for the balance between the number of points of a given type and the mass of each one of them.
	\item The property \ref{4iv} says that the values $\alpha_1, \dots, \alpha_l$ are relatively large compared with $m_1, \dots, m_l$ and $h_1, \dots, h_l$. Thus, the points from the upper level (see Figure~\ref{F4.2}) will have much greater masses than those from the lower level.
	\item The property \ref{4iv} is the only one where the parameter $l$ is involved.
\end{itemize}

We construct $\TT = \TT_{p,q,r,l} = (T, \rho, \mu)$, a test space of type~II, as follows. Set
\begin{displaymath}
T \coloneqq \big\{x_{i,j}, \, x^\circ_{i,k} : i \in [l], \, j \in [h_i], \, k \in [h_i \beta_i]\big\},
\end{displaymath}
where all elements $x_{i,j}, \, x^\circ_{i,k}$ are different. We use auxiliary symbols for certain subsets of $T$: 
\begin{displaymath}
T^\circ \coloneqq \big\{x^\circ_{i,k} : i \in [l], \, k \in [h_i \beta_i]\big\},
\end{displaymath}	
for $i \in [l]$,
\begin{displaymath}
T_{i} \coloneqq \big\{x_{i,j} : j \in [h_i] \big\}, \qquad
T^\circ_{i} \coloneqq \big\{x^\circ_{i,k} : k \in [h_i \beta_i]\big\},
\end{displaymath}
and, for $i,i^* \in [l]$ with $i \leq i^*$ and $j \in [h_i]$,  
\begin{displaymath}
T^\circ_{i^*,i,j} \coloneqq \Big\{x^\circ_{i^*,k} : k \in \Big[ \frac{j}{h_i} h_{i^*} \beta_{i^*} \Big] \setminus \Big[ \frac{j-1}{h_i} h_{i^*} \beta_{i^*} \Big]\Big\}.
\end{displaymath}
Observe that the family $\{ T^\circ_{i^*,i,j} : j \in [h_i] \}$ consists of disjoint sets, each of them containing exactly $h_{i^*} \beta_{i^*}  / h_i$ elements (here the property \ref{4i} was used) and $\bigcup_{j=1}^{h_i} T^\circ_{i^*,i,j} = T^\circ_{i^*}$. 

We introduce $\mu$ by letting 
\begin{displaymath}
\mu(\{x\}) \coloneqq \left\{ \begin{array}{rl}
m_i & \textrm{if } x = x_{i,j} \textrm{ for some } i \in [l], \, j \in [h_i],  \\
l^{\frac{p}{(p-1)r}} \alpha_i & \textrm{if } x = x^\circ_{i,k} \textrm{ for some } i \in [l], \, k \in [h_i \beta_i]. \end{array} \right. 
\end{displaymath}
By using \ref{4ii}, \ref{4iv}, and \ref{4v}, we deduce that $\mu$ satisfies the following inequalities: for each $x \in T^\circ$,
\begin{displaymath}
|\{x\}| > |T \setminus T^\circ|,
\end{displaymath}
and for each $i,i^* \in [l]$ with $i \leq i^*$, $x_1 \in T_{i^*}$, and $x_2 \in T^\circ_{i^*}$,  
\begin{displaymath}
|\{x_1\}| > |T_1 \cup \dots \cup T_i|, \qquad  |\{x_2\}| > |T^\circ_1 \cup \dots \cup T^\circ_i|.
\end{displaymath}

Finally, we define the metric $\rho$ on $T$ determining the distance between two different elements $x, y \in T$ by the formula
\begin{displaymath}
\rho(x,y) \coloneqq \left\{ \begin{array}{rl}
1 & \textrm{if } \{x, y\} = \{x_{i,j},x^\circ_{i^*,k}\} \textrm{ and } x^\circ_{i^*,k} \in T^\circ_{i^*,i,j},  \\
2 & \textrm{otherwise.} \end{array} \right. 
\end{displaymath}
It is worth noting here that for each $i,i^* \in [l]$ with $i \leq i^*$ and $x \in T^\circ_{i^*}$ there is exactly one point $y \in T_i$ such that $\rho(x,y)=1$.

Figure~\ref{F4.2} shows a model of the space $(T, \rho)$ with $l = 2$, $h_1 = 1$, and $h_2 = 2$.

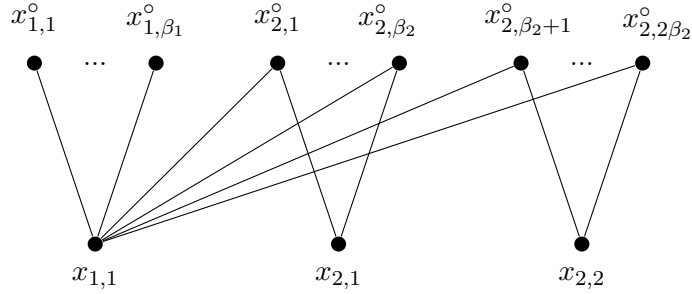
\begin{figure}[H]
	\begin{center}
	\begin{tikzpicture}
	[scale=.8,auto=left,every node/.style={circle,fill,inner sep=2pt}]
	
	\node[label={[yshift=-1cm]$x_{1,1}$}] (c0) at (6,1) {};
	\node[label=$x^\circ_{1,1}$] (c1) at (5,4)  {};
	\node[label={[yshift=-0.1cm]$x^\circ_{1,\beta_1}$}] (c2) at (7,4)  {};
	\node[dots] (c3) at (6,4)  {...};
	
	\node[label={[yshift=-1cm]$x_{2,1}$}] (l0) at (10,1) {};
	\node[label=$x^\circ_{2,1}$] (l1) at (9,4)  {};
	\node[label={[xshift=-0.15cm, yshift=-0.1cm]$x^\circ_{2,\beta_2}$}] (l2) at (11,4)  {};
	\node[dots] (l3) at (10,4)  {...};
	
	\node[label={[yshift=-1cm]$x_{2,2}$}] (r0) at (14,1) {};
	\node[label={[xshift=0.1cm, yshift=-0.25cm]$x^\circ_{2,\beta_2+1}$}] (r1) at (13,4)  {};
	\node[label={[xshift=0.2cm, yshift=-0.2cm]$x^\circ_{2,2\beta_2}$}] (r2) at (15,4)  {};
	\node[dots] (r3) at (14,4)  {...};
	
	\foreach \from/\to in {l0/l1, l0/l2, c0/c1, c0/c2, r0/r1, r0/r2, l1/c0, l2/c0, r1/c0, r2/c0}
	\draw (\from) -- (\to);
	\end{tikzpicture}
	\caption{The test space of type II for $r \in (1,\infty)$ with $l = 2$, $h_1 = 1$, and $h_2 = 2$.}
	\label{F4.2}
\end{center}
\end{figure}
\noindent As usual, we explicitly describe any ball: for $i \in [l], \, j \in [h_i]$,
\begin{displaymath}
B(x_{i,j},s) = \left\{ \begin{array}{rl}
\{x_{i,j}\} & \textrm{for } 0 < s \leq 1, \\
\{x_{i,j}\} \cup \bigcup_{i^* = i}^{l} T_{i^*,i,j}& \textrm{for } 1 < s \leq 2,  \\
T & \textrm{for } 2 < s, \end{array} \right.
\end{displaymath} 
\noindent and, for $i^* \in [l], \, k \in [h_i \beta_i]$,
\begin{displaymath}
B(x^\circ_{i^*,k},s) = \left\{ \begin{array}{rl}
\{x^\circ_{i^*,k}\} & \textrm{for } 0 < s \leq 1, \\
\{x^\circ_{i^*,k}\} \cup \{x_{i, j} : x^\circ_{i^*,k} \in T_{i^*,i,j}\}& \textrm{for } 1 < s \leq 2,  \\
T & \textrm{for } 2 < s. \end{array} \right.
\end{displaymath}

In the following lemma we describe the properties of $\cc(p,1,r,\TT)$ and $\cc(p,q,r,\TT)$.

\begin{lemma}\label{L4.3.2}
	Fix $l \in \NN$ and an admissible triple $(p, q, r)$ with $q, r \in (1,\infty)$. Let $\TT$ be the test space of type~II defined above. Then there is a numerical constant $\CC_2 = \CC_2(p,q,r)$ independent of $l$ such that
	\begin{displaymath}
	\cc(p,1,r,\TT) \leq \CC_2
	\end{displaymath} 
	and
	\begin{displaymath}
	\cc(p,q,r,\TT) \geq \frac{1}{\CC_2} \, l^{1-1/q}.
	\end{displaymath} 		
\end{lemma}

\begin{proof}
	First we estimate $\cc(p,1,r,\TT)$ from above. According to Remark~\ref{R4.1.9} it suffices to prove that $\| \MM_\TT \mathbf{1}_E \|_{p,r} \lesssim \|\mathbf{1}_E\|_{p,1}$ holds uniformly in $E \subset T$. Take $\emptyset \neq E \subset T$ and let $f = \mathbf{1}_E$. By using the sublinearity of $\MM_\TT$, we can assume that either $E \subset T^\circ$ or $E \subset T \setminus T^\circ$. Consider the case $E \subset T^\circ$. Then
	\begin{displaymath}
	\MM_{\TT} f \leq \max \{f, \mathbf{1}_{T \setminus T^\circ}, f_{\rm avg} \}
	\end{displaymath}
	and the desired estimate follows easily from the fact that $|T \setminus T^\circ| < |E|$. Now consider the case $E \subset T \setminus T^\circ$. We have
	\begin{displaymath}
	\MM_{\TT} f \leq \max \{f, \MM_0 f, f_{\rm avg}\},
	\end{displaymath}
	where
	\begin{displaymath}
	\MM_0 f(x) \coloneqq \mathbf{1}_{T^\circ}(x) \, \frac{|E \cup B(x, \frac{3}{2})|}{|B(x, \frac{3}{2})|}.  
	\end{displaymath}
	As previously, in view of Facts~\ref{F4.1.1},~\ref{F4.1.2},~and~\ref{F4.1.3} it remains to prove the estimate
	\begin{equation}\label{4}
	\|\MM_0 f\|_{p,r}^p \lesssim |E|.
	\end{equation} 
	Suppose for a moment that \eqref{4} holds for each $E$ such that $E \subset T_i$ for some $i \in [l]$. Then, for arbitrary $E \subset T \setminus T^\circ$, we let $E_i \coloneqq E \cap T_i$ and $f_i \coloneqq \mathbf{1}_{E_i}$ for each $i \in [l]$. By \ref{4ii} we have
	\begin{displaymath}
	\MM_0 f (x) \leq 2 \, \max_{i \in [l] } \MM_0 f_i(x)  
	\end{displaymath}
	for each $x \in T$. Therefore, given $t \in (0,\infty)$ we have 
	\[
	d_{\MM_0 f}(2t) \leq \sum_{i=1}^l d_{\MM_0 f_i}(t).
	\]
	Denoting $d(f,i,n) \coloneqq d_{\MM_0 f_i} (2^n)$ we thus obtain
	\begin{displaymath}
	\|\MM_0 f \|_{p,r}^p \lesssim \Big( \sum_{n \in \ZZ} 2^{nr} \Big( 
	\sum_{i=1}^l d(f,i,n) \Big)^{r/p}  
	\Big)^{p/r}.
	\end{displaymath}
	If $r \geq p$, then by Minkowski's inequality we have
	\begin{align*}
	\Big( \sum_{n \in \ZZ} 2^{nr} \Big( 
	\sum_{i=1}^l d(f,i,n) \Big)^{r/p}  
	\Big)^{p/r} \leq \sum_{i=1}^l  \Big( \sum_{n \in \ZZ} 2^{nr} d(f,i,n)^{r/p}  
	\Big)^{p/r} \lesssim \sum_{i=1}^l |E_i| = |E|.
	\end{align*}
	On the other hand, if $r < p$, then
	\begin{align*}
	\Big( \sum_{n \in \ZZ} 2^{nr} \Big( 
	\sum_{i=1}^l d(f,i,n) \Big)^{r/p}  
	\Big)^{p/r} \leq 
	\Big(  \sum_{i=1}^l \sum_{n \in \ZZ} 2^{nr} d(f,i,n)^{r/p} 
	\Big)^{p/r} 
	\lesssim \Big( \sum_{i=1}^l |E_i|^{r/p} \Big)^{p/r} \lesssim |E|
	\end{align*}
	where in the last estimate we again used \ref{4ii}.
	
	Now we return to the proof of \eqref{4} for $E \subset T_i$. Suppose that $E$ consists of $\gamma$ elements for some $\gamma \in [h_i]$. For each $x \in T^\circ_{i^*}$ with $i^* \in [i-1]$ we have $\MM_0 f(x) = 0$. On the other hand, for each $i^* \in [l] \setminus [i-1]$ we have precisely $\gamma h_{i^*} \beta_{i^*} / h_i$ elements $x \in T^\circ_{i^*}$ for which $\MM_0 f(x)$ is nonzero. Moreover, for each such $x$ we have 
	$
	\MM_0 f(x) \leq  l^{\frac{p}{(1-p)r}} m_i \alpha_{i^*}^{-1}.
	$
	Thus, we obtain
	\begin{align*}
	\| \MM_l f \|_{p,r}^r & \lesssim \sum_{i^* = i}^l \big( l^{\frac{p}{(1-p)r}} m_i \alpha_{i^*}^{-1}\big)^r \big( l^{\frac{p}{(p-1)r}} \alpha_{i^*} \gamma h_{i^*} \beta_{i^*} h_{i}^{-1} 
	\big)^{\frac{r}{p}} \lesssim \gamma^{r/p} m_i^r h_i^{-r/p} l^{-1} \sum_{i^* = i}^l \big( \alpha_{i^*}^{1-p} \beta_{i^*} h_{i^*} \big)^{\frac{r}{p}}   
	\end{align*}
	which is controlled by a constant multiple of $\gamma^{r/p} m_i^{r/p} = |E|^{r/p}$ in view of \ref{4iii} and \ref{4vi}, and the fact that the sum in the expression above has at most $l$ elements. 
	
	In the next step we estimate $\cc(p, q, r, \TT)$ from below. Take $g$ defined by
	\begin{displaymath}
	g \coloneqq \sum_{i=1}^l \frac{1}{m_i} \, \mathbf{1}_{T_i}.
	\end{displaymath}
	Then, by using \ref{4ii} we have
	\begin{align*}
	\|g\|_{p,q}^q \lesssim \sum_{i=1}^l m_i^{-q} |T_i|^{q/p} = \sum_{i=1}^l (m_i^{1-p} h_i)^{q/p}
	\end{align*}
	and then \ref{4iii} gives $\|g\|_{p,q} \lesssim l^{1/q}$.
	Let us now focus on $\MM_\TT g$. For each $i \in [l]$ and $x \in T^\circ_i$ we have 
	\begin{displaymath}
	\MM_\TT g(x) \geq \frac{1}{|B(x,\frac{3}{2})|} \, \sum_{y \in B(x,3/2)} g(y) \, |\{y\}|.
	\end{displaymath}
	Note that \ref{4iv} implies that $|B(x,\frac{3}{2})| \leq 2 l^{\frac{p}{(p-1)r}} \alpha_i$ and, as a result, we obtain
	\begin{displaymath}
	\MM_\TT g(x) \gtrsim l^{\frac{p}{(1-p)r}} \alpha_i^{-1} \, \sum_{i^*=1}^{i} \frac{1}{m_i} \, m_i = l^{\frac{p}{(1-p)r}} \alpha_i^{-1}.
	\end{displaymath}
	Next, by using \ref{4v} we deduce that 
	\begin{displaymath}
	\|\MM_\TT g\|_{p,r}^r \gtrsim \sum_{i=1}^l \big(  i \, l^{\frac{p}{(1-p)r}} \alpha_i^{-1} 
	\big)^{r} |T^\circ_i|^{r/p} = l^{-1} \, \sum_{i=1}^l i^r \, \big(\alpha_i^{1-p} \beta_i h_i\big)^{r/p} .
	\end{displaymath}
	In view of \ref{4vi} each element of the series above is bigger than $i^r$ and hence
	\begin{displaymath}
	\|\MM_\TT g\|_{p,r} \gtrsim l^{-1/r} \big( \sum_{i=1}^l i^r \big)^{1/r} \gtrsim l^{-1/r + (r+1)/r} = l.
	\end{displaymath}
	Thus, the estimates obtained for $\|g\|_{p,q}$ and $\|\MM_\TT g\|_{p,r}$ imply that $\cc(p,q,r,\TT) \gtrsim l^{1-1/q}$.					
\end{proof}

\noindent {\bf Test spaces of type~II for $r = \infty$.} The argument described above needs only a few minor modifications to cover the second case under consideration. Fix $l \in \NN$ and an admissible triple $(p,q,\infty)$ with $q \in (1, \infty)$ (note that, in particular, $p \in (1,\infty)$). Associate with $(p,q,r,l)$ a large constant $\alpha \in (0,\infty)$ and three sequences of positive integers, $(\widetilde{m}_i)_{i=1}^l$, $(\widetilde{h}_i)_{i=1}^l$, and $(\widetilde{\beta}_i)_{i=1}^l$, with the following properties:
\begin{enumerate}[label=(\subscript{\rm \roman*})] \setlength\itemsep{0em}
	\item \label{44i} $\widetilde{h}_{i+1} / \widetilde{h}_i \in \NN$,
	\item \label{44ii} $\widetilde{m}_{i+1} \geq 2 \widetilde{m}_i \widetilde{h}_i$,
	\item \label{44iii} $1 \leq (\widetilde{m}_i)^{1-p}\widetilde{h}_i < 2$,
	\item \label{44iv} $\alpha \geq 2 \widetilde{m}_l \widetilde{h}_l$,
	\item \label{44v} $i^{p-2} \leq \alpha^{1-p} \widetilde{\beta}_i \widetilde{h}_i \leq 2i^{p-2}$.
\end{enumerate}
\noindent As before, we notice that the properties \ref{44i}--\ref{44v} can be met simultaneously.

We construct $\widetilde{\TT} = \widetilde{\TT}_{p,q,l} = (T, \rho, \mu)$, a test space of type~II, as follows. The set $T$ and the metric $\rho$ are defined as before with the aid of $(\widetilde{h}_i)_{i=1}^l$ and $(\widetilde{\beta}_i)_{i=1}^l$ instead of $(h_i)_{i=1}^l$ and $(\beta_i)_{i=1}^l$, respectively. Then we introduce $\mu$ by letting
\begin{displaymath}
\mu(\{x\}) \coloneqq \left\{ \begin{array}{rl}
\widetilde{m}_i & \textrm{if } x = x_{i,j} \textrm{ for some } i \in [l], \, j \in [\widetilde{h}_i],  \\
i \alpha & \textrm{if } x = x^\circ_{i,k} \textrm{ for some } i \in [l], \, k \in [\widetilde{h}_i \widetilde{\beta}_i]. \end{array} \right. 
\end{displaymath}

The following lemma describes the properties of $\cc(p,1,\infty,\widetilde{\TT})$ and $\cc(p,q,\infty,\widetilde{\TT})$.

\begin{lemma}\label{L4.3.3}
	 Fix $l \in \NN$ and an admissible triple $(p,q,\infty)$ with $q \in (1, \infty)$. Let $\widetilde{\TT}$ be the text space of type~II defined above. Then there is a numerical constant $\widetilde{\CC}_2 = \widetilde{\CC}_2(p,q)$ independent of $l$ such that
	\begin{displaymath}
	\cc(p,1,\infty,\widetilde{\TT}) \leq \widetilde{\CC}_2
	\end{displaymath} 
	and
	\begin{displaymath}
	\cc(p,q,\infty,\widetilde{\TT}) \geq \frac{1}{\widetilde{\CC}_2} l^{1-1/q}.
	\end{displaymath} 		
\end{lemma}

\begin{proof}
	We present only a sketch of the proof. First we want to estimate $\cc(p,1,\infty,\widetilde{\TT})$ from above. The main step here, as in the proof of Lemma~\ref{L4.3.2}, is to prove that
	\begin{displaymath}
	\| \MM_0 \mathbf{1}_E \|_{p, \infty}^p \lesssim |E|
	\end{displaymath}
	holds uniformly in $E \subset T_i$ and $i \in [l]$, where
	\begin{displaymath}
	\MM_0 \mathbf{1}_E(x) \coloneqq \mathbf{1}_{T^\circ}(x) \, \frac{|E \cup B(x, \frac{3}{2})|}{|B(x,\frac{3}{2})|}.  
	\end{displaymath}
	Suppose that $E$ consists of $\gamma$ elements for some $\gamma \in [\widetilde{h}_i]$. 
	
	For each $x \in T^\circ_{i^*}$ with $i^* \in [i-1]$ we have $\MM_0 f(x) = 0$. On the other hand, for each $i^* \in [l] \setminus [i-1]$ we have precisely $\gamma \widetilde{h}_{i^*} \widetilde{\beta}_{i^*} / \widetilde{h}_i$ elements $x \in T^\circ_{i^*}$ for which $\MM_0 f(x)$ is nonzero. Moreover, for each such $x$ we have 
	$
	\MM_0 f(x) \leq  \widetilde{m}_i \, (i^* \alpha)^{-1}.
	$
	We observe that
	\begin{align*}
	\Big( \frac{\widetilde{m}_i}{i^* \alpha }\Big)^p \, \sum_{j=i}^{i^*} j \alpha \gamma \widetilde{\beta}_j \widetilde{h}_j \widetilde{h}_1^{-1} = \gamma \widetilde{m}_i^p \widetilde{h}_i^{-1} (i^*)^{-p}  \sum_{j=i}^{i^*} j \alpha^{1-p} \widetilde{\beta}_j \widetilde{h}_j 
	\end{align*}
	and, in view of \ref{44iii} and \ref{44v}, we have
	\begin{displaymath}
	\gamma \widetilde{m}_i^p \widetilde{h}_i^{-1} (i^*)^{-p}  \sum_{j=i}^{i^*} j \alpha^{1-p} \widetilde{\beta}_j \widetilde{h}_j 
	\lesssim \gamma \widetilde{m}_i (i^*)^{-p}  \sum_{j=i}^{i^*} j^{p-1} \lesssim \gamma \widetilde{m}_i = |E|.
	\end{displaymath}
	
	Now we estimate $\cc(p, q, \infty, \widetilde{\TT})$ from below. Take $g$ defined by
	\begin{displaymath}
	g \coloneqq \sum_{i=1}^l \frac{1}{\widetilde{m}_i} \, \mathbf{1}_{T_i}.
	\end{displaymath}
	Then, by \ref{44ii} and \ref{44iii}, we have
	\begin{align*}
	\|g\|_{p,q}^q \lesssim \sum_{i=1}^l \widetilde{m}_i^{-q} |T_i|^{q/p} = \sum_{i=1}^l (\widetilde{m}_i^{1-p} \widetilde{h}_i)^{q/p} \lesssim l.
	\end{align*}
	By \ref{44iv} for each $i \in[l]$ and $x \in T^\circ_i$ we have $\MM_{\widetilde{\TT}} g(x) \geq (2 \alpha)^{-1}$. Thus, in view of \ref{44v}, we obtain
	\begin{displaymath}
	\|\MM_{\widetilde{\TT}} g\|_{p,\infty}^p \gtrsim \alpha^{-p} \, |T^\circ| =  \sum_{i=1}^l \alpha^{1-p} j \widetilde{\beta}_j \widetilde{h}_j \geq \sum_{i=1}^l j^{p-1} \gtrsim l^p.
	\end{displaymath}
	Thus, the estimates obtained for $\|g\|_{p,q}$ and $\|\MM_{\widetilde{\TT}} g\|_{p,\infty}$ imply that $\cc(p,q,\infty,\widetilde{\TT}) \gtrsim l^{1-1/q}$.    
\end{proof}

\subsection{Proof of the main result}\label{S4.3.2}

Let $(p_0,q_0,r_0)$ be a fixed admissible triple with $q_0 \in (1, \infty)$ (we omit the case $q_0 = \infty$ since the thesis is the stronger the smaller $q_0$ is). Consider the case $r_0 \in [q_0, \infty)$. For each $n \in \NN$ let $\TT_n = \TT_{p_n,q_n,r_n,l_n}$ be the test space of type~II constructed with the aid of $(p_n,q_n,r_n) = (p_0,q_0,r_0)$ and $l_n = n$. We let $\ZZZ$ be the space $\YY$ obtained by applying Proposition~\ref{P4.1.10} with $\YY_n = \TT_n$ for each $n \in\NN$. By using Lemma~\ref{L4.3.2} we conclude that $\cc(p_0,1,r_0,\YY) < \infty$, while $\cc(p_0, q_0, r_0, \YY) = \infty$. On the other hand, if $r_0 = \infty$, then we let $\widetilde{\TT}_n = \widetilde{\TT}_{p_n,q_n,l_n}$ be the test space of type~II constructed with the aid of $(p_n,q_n) = (p_0,q_0)$ and $l_n = n$. Finally, we construct $\ZZZ$ by applying Proposition~\ref{P4.1.10} with $\YY_n = \widetilde{\TT}_n$ for each $n \in \NN$. By using Lemma~\ref{L4.3.3} we conclude that $\cc(p_0,1,\infty,\YY) < \infty$, while $\cc(p_0, q_0, \infty, \YY) = \infty$.

\section{Results for $\lowercase{q}$ and $\lowercase{r}$ varying}\label{S4.4}

The aim of this section is to complete the picture outlined in Sections~\ref{S4.2}~and~\ref{S4.3}. Namely, for each $p \in (1, \infty)$ we introduce a family of spaces $\XX$ for which we are able to control the behavior of $\cc(p,q,r,\XX)$ considered as a function of two variables, $q$ and $r$. As a result, we characterize all possible shapes of the sets 
$$
\Omega^p_{\rm HL}(\XX) \coloneqq \Big\{\Big(\frac{1}{q},\frac{1}{r}\Big) \in [0,1]^2 : \MM_\XX \ {\rm is} \ {\rm bounded} \ {\rm from} \ L^{p,q}(\XX) \ {\rm to} \ L^{p,r}(\XX)\Big\} \subset [0,1]^2
$$
(the shapes of these sets are described in terms of their topological boundaries and the underlying space is the square $[0,1]^2$ with its natural topology). The following theorem, the culmination point of this dissertation, can be viewed as an extension of Theorems~\ref{T4.2.1}~and~\ref{T4.3.1}.

\begin{theorem} \label{44-T0}
	Fix $p \in (1, \infty)$ and let $\XX$ be an arbitrary metric measure space such that $|X \setminus \rm{supp}(\mu)| = 0$. Then one of the following two possibilities holds:
	\begin{itemize}
		\item The boundary of $\Omega^p_{\rm HL}(\XX)$ is empty, that is, $\Omega^p_{\rm HL}(\XX) = \emptyset$ or $\Omega^p_{\rm HL}(\XX) = [0,1]^2$.
		\item The boundary of $\Omega^p_{\rm HL}(\XX)$ is of the form 
		$$\{ \delta \} \times [0, \lim_{u \rightarrow \delta} F(u)] \ \cup \ \{(u, F(u)) : u \in (\delta, 1] \},$$ 
		where $\delta \in [0,1]$ and $F \colon [\delta, 1] \rightarrow [0,1]$ is concave, nondecreasing, and satisfying $F(u) \leq u$.
	\end{itemize}
	Conversely, for each $F$ as above there exists $\XX$ such that 
	$\MM_\XX$ is bounded from $L^{p,q}(\XX)$ to $L^{p,r}(\XX)$
	if and only if the point $(\frac{1}{q}, \frac{1}{r})$ lies on or under the graph of $F$, that is, $\frac{1}{q} \geq \delta$ and $\frac{1}{r} \leq F\big(\frac{1}{q}\big)$.
\end{theorem}
\noindent Let us emphasize one more time that even though Theorem~\ref{44-T0} is stated for the centered operator $\MM_\XX$, it is possible to obtain its analogue for the noncentered operator $\MN_\XX$. 

To prove Theorem~\ref{44-T0} we should focus on two separate tasks. First we want to indicate some conditions that the sets $\Omega^p_{\rm HL}(\XX)$ must satisfy in general, in order to ensure that no situations other than those listed in Theorem~\ref{44-T0} are possible. This problem is treated in Subsection~\ref{44-S5} (see Remarks~\ref{44-R0}~and~\ref{44-R1}, and Theorem~\ref{44-T2}). The second goal, which turns out to be the harder one, is to introduce a special class of spaces for which we are able to control precisely the behavior of the maximal operator and, at the same time, this behavior is very peculiar. This problem is covered by Theorem~\ref{44-T1} stated below. We note that, in fact, Theorem~\ref{44-T1} is slightly more general and it consists of four similar results which have been collected together for the sake of completeness.

\begin{theorem} \label{44-T1}
	Fix $p \in (1, \infty)$ and $\delta \in [0, 1]$ (respectively, $\delta \in [0, 1)$). Let $F \colon [\delta, 1] \rightarrow [0,1]$ (respectively, $F \colon (\delta, 1] \rightarrow [0,1]$) be concave, nondecreasing and satisfying $F(u) \leq u$ for each $u \in [\delta, 1]$ (respectively, $u \in (\delta, 1]$). Then the following statements are true:
	\begin{itemize}
		\item There exists a (nondoubling) metric measure space $\ZZZ$ such that $\cc(p,q,r, \ZZZ) < \infty$ if and only if $\frac{1}{q} \geq \delta$ (respectively, $\frac{1}{q} > \delta$) and $\frac{1}{r} \leq F\big(\frac{1}{q}\big)$. 
		\item There exists a (nondoubling) metric measure space $\ZZZ'$ such that $\cc(p,q,r, \ZZZ') < \infty$ if and only if $\frac{1}{q} \geq \delta$ (respectively, $\frac{1}{q} > \delta$) and $\frac{1}{r} < F\big(\frac{1}{q}\big)$.
	\end{itemize} 
\end{theorem}

\noindent A short comment should be made here regarding the spaces $\ZZZ$ and $\ZZZ'$. Although the word ``exists'' is used in the formulation of Theorem~\ref{44-T1}, each space is constructed explicitly. Moreover, the construction process described later on originates in the idea of Stempak, who provided some interesting examples of spaces, when dealing with a certain related problem regarding modified maximal operators (see \cite{St2}).     

The rest of this section is organized as follows. In Subsections~\ref{44-S3}~and~\ref{44-S4} we study the behavior of the maximal operator in the context of two classes of very specific spaces, namely, the test spaces of type~III and their advanced cousins, the composite test spaces. The latter class is used in Subsection~\ref{44-S5} to prove Theorem~\ref{44-T1} and, as a consequence, the second part of Theorem~\ref{44-T0}. The rest of Subsection~\ref{44-S5} is devoted to indicating properties of $\Omega^p_{\rm HL}(\XX)$ which allow us to deduce the first part of Theorem~\ref{44-T0}. In particular, we formulate a suitable interpolation theorem for Lorentz spaces $L^{p,q}(\XX)$ with the first parameter fixed (see Theorem~\ref{44-T2}). This theorem, in fact, follows from a much more general result \cite[Theorem~5.3.1]{BL} using advanced interpolation methods. However, in \hyperref[Appendix]{Appendix} we give its elementary proof which, to the author's best knowledge, has never appeared in the literature so far. 

To avoid misunderstandings, we note that several times later on we identify $1 / \infty$ and $1/0$ with $0$ and $\infty$, respectively, when dealing with $q, r \in [1, \infty]$ and $u, F(u) \in [0,1]$. Also, for $\delta = 1$ the conventions $[\delta, 1] = \{1\}$, $(\delta,1] = \emptyset$, and $\lim_{u \rightarrow \delta} F(u) = F(1)$ are used.

\subsection{Test spaces of type~III} \label{44-S3}

Let $p \in (1, \infty)$ and $N, M, L \in \NN$. We associate with each quadruple $(p, N, M, L)$ four sequences of positive integers, $(m_i)_{i=1}^N$, $(h_i)_{i=1}^N$, $(\alpha_k)_{k=1}^M$, and $(\beta_k)_{k=1}^M$, satisfying the following assertions:
\begin{enumerate}[label=(\roman*)] \setlength\itemsep{0em}
	\item \label{44-i} $h_{N} / h_i \in \NN$,
	\item \label{44-ii} $m_{i+1} \geq 2 m_i h_i$,
	\item \label{44-iii} $1 \leq m_{i}^{1-p}h_i <2$,
	\item \label{44-iv} $\alpha_1 \geq 2 m_N h_N$,
	\item \label{44-v} $\alpha_{k+1} \geq 2 \alpha_k L \beta_k h_N$,
	\item \label{44-vi} $1 \leq \alpha_{k}^{1-p} \beta_k h_N <2$.	
\end{enumerate}
Let us check that the properties \ref{44-i}--\ref{44-vi} can be met simultaneously. Set $m_1 = h_1 = 1$. Then we specify $m_{i+1}$ and $h_{i+1}$ for some $i \in [N-1]$, assuming that the quantities $m_1, \dots, m_i,h_1, \dots, h_i$ have already been chosen. We take $m_{i+1} \geq 2 m_{i} h_{i}$ such that the set $\{h \in \NN : 1 \leq m_{i+1}^{1-p}h < 2 \}$ contains at least $h_i$ elements. Then we choose $h_{i+1}$ for which $h_{i+1} / h_i \in \NN$ and $1 \leq m_{i+1}^{1-p}h_{i+1} <2$ hold simultaneously. Thus, the properties \ref{44-i}--\ref{44-iii} are satisfied. Next we take $\alpha_1$ such that $\alpha_1 \geq 2 m_N h_N$ and $\alpha_1^{1-p} h_N < 2$ hold, and choose $\beta_1$ satisfying $1 \leq \alpha_{1}^{1-p} \beta_1 h_N <2$. Then we specify $\alpha_{k+1}$ and $\beta_{k+1}$ for some $k \in [M-1]$, assuming that the quantities $\alpha_1, \dots, \alpha_k,\beta_1, \dots, \beta_k$ have already been chosen. We take $\alpha_{k+1} \geq 2 \alpha_{k} L \beta_{k} h_N$. Since $\alpha_{k+1}^{1-p} h_N \leq \alpha_1^{1-p} h_N < 2$, we can choose $\beta_{k+1}$ satisfying $1 \leq \alpha_{k+1}^{1-p} \beta_{k+1} h_N <2$. Thus, the properties \ref{44-iv}--\ref{44-vi} are satisfied as well.

\medskip The four sequences will determine the structure of the test space of type~III constructed below. Here we formulate a few thoughts that one should keep in mind later on.
\begin{itemize}\setlength\itemsep{0em}
	\item Our space consists of two levels (lower and upper) concerning points of $N$ and $M$ types, respectively (see Figure~\ref{44-F1}).
	\item The sequences $(m_i)_{i=1}^N$ and $(\alpha_k)_{k=1}^M$ are used to define the associated measure, while $(h_k)_{k=1}^N$ and $(\beta_k)_{k=1}^M$ are responsible for the number of elements of a given type.
	\item The property \ref{44-i} makes the set of points of a given type divisible into an appropriate number of equinumerous subsets.
	\item The properties \ref{44-i} and \ref{44-v} say that the sequences $(m_i)_{i=1}^N$ and $(\alpha_k)_{i=k}^M$ grow very fast. The huge differences between the masses of points of different types allow one to use Lemma~\ref{L4.1.4} frequently.
	\item The property \ref{44-iv} says that the values $\alpha_i$ are large compared with $m_k$ and $h_k$. The points from the upper level have much greater masses than the ones from the lower level and Lemma~\ref{L4.1.4} can be applied also in this context.
	\item The properties \ref{44-iii} and \ref{44-vi} are of rather technical nature. They keep the balance between the number of points of a given type and the mass of each one of them.
	\item The property \ref{44-iv} is the only property involving the parameter $L$.
\end{itemize}

Let $K \in [1, \infty)$. We define $\UU = \UU_{p,N,M,K,L} = (U, \rho, \mu)$, the test space of type~III, as follows. Set
\begin{displaymath}
U \coloneqq \big\{x_{i,j}, \, x^\circ_{k,l} : i \in [N], \, j \in [h_i], \, k \in [M], \, l \in  [L \beta_k h_N] \big\},
\end{displaymath}
where all elements $x_{i,j}, x^\circ_{k,l}$ are different. We use auxiliary symbols for certain subsets of $U$: 
\begin{displaymath}
U^\circ \coloneqq \big\{x^\circ_{k,l} : k \in [M], \, l \in [L \beta_k h_N] \big\};
\end{displaymath}	
for $i \in [N]$ and $k \in [M]$,
\begin{displaymath}
U_{i} \coloneqq \big\{x_{i,j} : j \in [h_i] \big\}, \quad
U^\circ_{k} \coloneqq \big\{x^\circ_{k,l} : l \in [L \beta_k h_N] \big\};
\end{displaymath}
for $i \in [N]$, $j \in [h_i]$, and $k \in [M]$,  
\begin{displaymath}
U^\circ_{i,j,k} \coloneqq \Big\{x^\circ_{k,l} : l \in \Big[ \frac{j}{h_i} L \beta_k h_N \Big] \setminus \Big[ \frac{j-1}{h_i} L \beta_k h_N \Big] \Big\}.
\end{displaymath}
Observe that the sets $U^\circ_{i,j,k}$, $j \in [h_i]$, are disjoint and, in view of \ref{44-i}, each of them contains exactly $L \beta_{k} h_N  / h_i$ elements. Moreover, $\bigcup_{j=1}^{h_i} U^\circ_{i,j,k} = U^\circ_{k}$ holds for each $i \in [N]$. 

We introduce $\mu$ by letting 
\begin{displaymath}
\mu(\{x\}) \coloneqq \left\{ \begin{array}{rl}
m_i & \textrm{if } x = x_{i,j} \textrm{ for some } i \in [N], \, j \in [h_i],  \\
K \alpha_k & \textrm{if } x = x^\circ_{k,l} \textrm{ for some } k \in [M], \, l \in [L\beta_k h_N]. \end{array} \right. 
\end{displaymath}
Note that, in view of \ref{44-iv}, \ref{44-ii}, and \ref{44-v}, the following inequalities hold: for each $x \in U^\circ$,
\begin{displaymath}
|\{x\}| > |U \setminus U^\circ|,
\end{displaymath}
for each $i  \in [N] \setminus \{1\}$ and $x \in U_{i}$,   
\begin{displaymath}
|\{x\}| > |U_1 \cup \dots \cup U_{i-1}|,
\end{displaymath}
and for each $k \in [M] \setminus \{1\}$ and $x^\circ \in U^\circ_{k}$,
\begin{displaymath}
|\{x^\circ\}| > |U^\circ_1 \cup \dots \cup U^\circ_{k-1}|.
\end{displaymath}

Finally, we define $\rho$ by the formula
\begin{displaymath}
\rho(x,y) \coloneqq \left\{ \begin{array}{rl}
0 & \textrm{if } x=y, \\
1 & \textrm{if } \{x, y\} = \{x_{i,j},x^\circ_{k,l}\} \textrm{ and } x^\circ_{k,l} \in U^\circ_{i,j,k},  \\
2 & \textrm{otherwise.} \end{array} \right. 
\end{displaymath}
It is worth noting here that for each $i \in [N]$, $ k \in [M]$, and $x^\circ \in U^\circ_{k}$, there is exactly one point $x \in U_i$ such that $\rho(x,x^\circ)=1$. This point is denoted by $\Gamma_i(x^\circ)$ later on.

Figure~\ref{44-F1} shows a model of the space $(U, \rho)$ with $N=3$ and $M=2$. The solid line between two points indicates that the distance between them equals $1$. Otherwise the distance equals $2$.  

\begin{figure}[H]
	\begin{center}
	\begin{tikzpicture}
	[scale=.7,auto=left,every node/.style={circle,fill,inner sep=2pt}]
	
	\node[label={[yshift=-1cm]$x_{1,1}$}] (n1) at (1,1) {};
	
	\node[label={[yshift=-1cm]$x_{2,1}$}] (n2) at (5,1) {};
	\node[label={[yshift=-1cm]$x_{2,2}$}] (n3) at (7,1) {};
	
	\node[label={[yshift=-1cm]$x_{3,1}$}] (n4) at (11,1) {};
	\node[label={[yshift=-1cm]$x_{3,2}$}] (n5) at (12.5,1) {};
	\node[label={[yshift=-1cm]$x_{3,3}$}] (n6) at (14,1) {};
	\node[label={[yshift=-1cm]$x_{3,4}$}] (n7) at (15.5,1) {};
	
	\node[label={$x^\circ_{1,1}$}] (m1) at (2,5) {};
	\node[label={$x^\circ_{1,2}$}] (m2) at (3,5) {};
	\node[label={$x^\circ_{1,3}$}] (m3) at (4,5) {};
	\node[label={$x^\circ_{1,4}$}] (m4) at (5,5) {};
	
	\node[label={$x^\circ_{2,1}$}] (m5) at (7.5,5) {};
	\node[label={$x^\circ_{2,2}$}] (m6) at (8.5,5) {};
	\node[label={$x^\circ_{2,3}$}] (m7) at (9.5,5) {};
	\node[label={$x^\circ_{2,4}$}] (m8) at (10.5,5) {};
	\node[label={$x^\circ_{2,5}$}] (m9) at (11.5,5) {};
	\node[label={$x^\circ_{2,6}$}] (m10) at (12.5,5) {};
	\node[label={$x^\circ_{2,7}$}] (m11) at (13.5,5) {};
	\node[label={$x^\circ_{2,8}$}] (m12) at (14.5,5) {};
	
	\foreach \from/\to in {n1/m1, n1/m2, n1/m3, n1/m4, n1/m5, n1/m6, n1/m7, n1/m8, n1/m9, n1/m10, n1/m11, n1/m12}
	\draw (\from) -- (\to);
	
	\foreach \from/\to in {n2/m1, n2/m2, n3/m3, n3/m4, n2/m5, n2/m6, n2/m7, n2/m8, n3/m9, n3/m10, n3/m11, n3/m12}
	\draw (\from) -- (\to);
	
	\foreach \from/\to in {n4/m1, n5/m2, n6/m3, n7/m4, n4/m5, n4/m6, n5/m7, n5/m8, n6/m9, n6/m10, n7/m11, n7/m12}
	\draw (\from) -- (\to);
	
	\end{tikzpicture}
	\caption{The test space of type III with $N=3$ and $M=2$.}
	\label{44-F1}
	\end{center}
\end{figure}
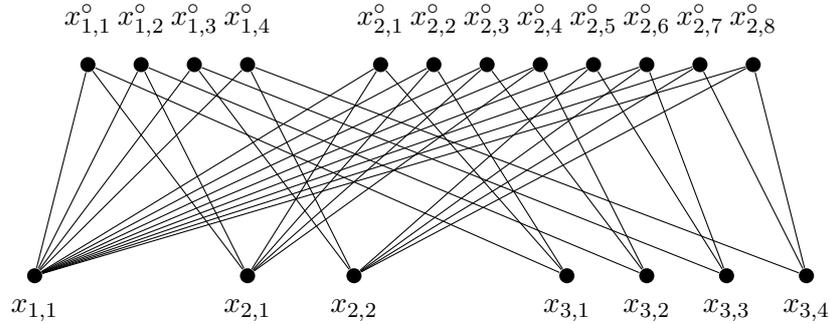

As usual, we explicitly describe any ball $B \subset U$. We thus have: for $i \in [N]$, $j \in [h_i]$, 
\begin{displaymath}
B(x_{i,j},s) = \left\{ \begin{array}{rl}
\{x_{i,j}\} & \textrm{for } 0 < s \leq 1, \\
\{x_{i,j}\} \cup \{x^\circ \in U^\circ : \Gamma_i(x^\circ) = x_{i,j} \} & \textrm{for } 1 < s \leq 2,  \\
U & \textrm{for } 2 < s, \end{array} \right.
\end{displaymath} 
\noindent and, for $k \in [M]$, $l \in [L\beta_k h_N]$,
\begin{displaymath}
B(x^\circ_{k,l},s) = \left\{ \begin{array}{rl}
\{x^\circ_{k,l}\} & \textrm{for } 0 < s \leq 1, \\
\{x^\circ_{k,l}\} \cup \{\Gamma_i(x^\circ_{k,l}) : i \in [N] \} & \textrm{for } 1 < s \leq 2,  \\
U & \textrm{for } 2 < s. \end{array} \right.
\end{displaymath}

Now, for each fixed $i \in [N]$ and $k \in [M]$, we introduce a linear operator $\AAA_{k,i} = \AAA_{k,i, \SSS}$ given by the formula
\begin{displaymath}
\AAA_{k,i}f(x) \coloneqq \left\{ \begin{array}{rl}
\frac{f(\Gamma_i(x)) \, |\{\Gamma_i(x)\}|}{|\{x\}| } & \textrm{if } x \in U^\circ_k,  \\
0 & \textrm{otherwise.} \end{array} \right. 
\end{displaymath}

In the following lemma we estimate the norm of $\AAA_{k,i}$ acting from $L^{p,q}(\UU)$ to $L^{p,r}(\UU)$.  

\begin{lemma} \label{44-L2}
	Let $\UU$ be the test space of type~III defined above. Fix $q,r \in [1, \infty]$ with $q \leq r$, $i \in [N]$, and $k \in [M]$, and consider the operator $\AAA_{k,i}$. Then there exists a numerical constant $\CC_{3,1} = \CC_{3,1}(p, q, r)$ independent of $N$, $M$, $K$, $L$, $i$, and $k$ such that
	\begin{displaymath}
	\frac{1}{\CC_{3,1}} \, K^{-1+1/p} L^{1/p} \leq \|\AAA_{k,i} \|_{L^{p,q}(\UU) \rightarrow L^{p,r}(\UU)} \leq \CC_{3,1} \, K^{-1+1/p} L^{1/p}.
	\end{displaymath}
\end{lemma}

\begin{proof}
	First we estimate $\|\AAA_{k,i} \|_{L^{p,q}(\UU) \rightarrow L^{p,r}(\UU)}$ from above. Take $f \in L^{p,q}(\UU)$. Since $\AAA_{k,i} f \equiv \AAA_{k,i} (f \cdot \mathbf{1}_{U_i})$, we may assume that the support of $f$ is contained in $U_i$. If this is the case, then for each $t \in (0,\infty)$ we have the equality
	\begin{displaymath}
	d_{\AAA_{k,i}f}(t) = \frac{KL\alpha_k \beta_k h_N}{m_i h_i} d_f(tK\alpha_k / m_i)
	\end{displaymath}  
	and simple calculations give
	\begin{displaymath}
	\|\AAA_{k,i}f \|_{p,r} = K^{-1 + 1/p} L^{1/p} m_i^{1-1/p} h_i^{-1/p} \alpha_k^{-1 + 1/p} \beta_k^{1/p} h_N^{1/p} \|f\|_{p,r}.
	\end{displaymath}
	Thus, in view of \ref{44-iii}, \ref{44-vi}, and Fact~\ref{F4.1.3}, we obtain
	\begin{displaymath}
	\|\AAA_{k,i}f \|_{p,r} \leq 4 \, \CC_{\hookrightarrow}(p,q,r) \, K^{-1 + 1/p} L^{1/p} \|f\|_{p,q}.
	\end{displaymath}
	
	Finally, consider $g \coloneqq \mathbf{1}_{U_i}$. Then we have $\AAA_{k,i} g = \frac{m_i}{K\alpha_k} \, \mathbf{1}_{U^\circ_k}$ and hence
	\begin{displaymath}
	\frac{ \|\AAA_{k,i}g\|_{p,r} }{\| g \|_{p,q} } = K^{-1 + 1/p} L^{1/p} m_i^{1-1/p} \alpha_k^{-1 + 1/p} \beta_k^{1/p} h_N^{1/p} h^{-1/p} \geq K^{-1 + 1/p} L^{1/p}, 
	\end{displaymath} 
	where in the last inequality we again used \ref{44-iii} and \ref{44-vi}. 
\end{proof}

Next we introduce a linear operator $\AAA = \AAA_\SSS$ given by the formula
\begin{displaymath}
\AAA f(x) \coloneqq \left\{ \begin{array}{rl}
\sum_{i=1}^{N} \AAA_{k,i}f(x) & \textrm{if } x \in U^\circ_k, \, k \in [M], \\
0 & \textrm{otherwise.} \end{array} \right. 
\end{displaymath}
As before, we estimate the norm of $\AAA$ acting from $L^{p,q}(\UU)$ to $L^{p,r}(\UU)$.

\begin{lemma} \label{44-L3}
	Let $\UU$ be the test space of type~III defined above. Fix $q,r \in [1, \infty]$ with $q \leq r$ and consider the operator $\AAA$. Then there exists a numerical constant $\CC_3 = \CC_{3,2}(p, q, r)$ independent of $N$, $M$, $K$, and $L$ such that
	\begin{displaymath}
	\frac{1}{\CC_{3,2}} \, K^{-1+1/p} L^{1/p} M^{1/r} N^{1 - 1/q} \leq \|\AAA \|_{L^{p,q}(\UU) \rightarrow L^{p,r}(\UU)} \leq \CC_{3,2} \, K^{-1+1/p} L^{1/p} M^{1/r} N^{1 - 1/q}.
	\end{displaymath}
\end{lemma}

\begin{proof}
	First we estimate $\|\AAA \|_{L^{p,q}(\UU) \rightarrow L^{p,r}(\UU)}$ from above. Take $f \in L^{p,q}(\UU)$. In view of $\AAA f \equiv \AAA(f \cdot \mathbf{1}_{U \setminus U^\circ})$, we may assume that the support of $f$ is contained in $U \setminus U^\circ$. We decompose $f = \sum_{i=1}^N f_i$, where $f_i \coloneqq f \cdot \mathbf{1}_{U_i}$. Then, by \ref{44-ii}, \ref{44-v}, and Lemma~\ref{L4.1.4}, we have
	\begin{displaymath}
	\|f\|_{p,q} \geq \frac{1}{\CC_{\rm supp}(p,q)} \, \Big(  \sum_{i=1}^N \| f_i\|_{p,q}^q \Big)^{1/q}
	\end{displaymath}
	and
	\begin{displaymath}
	\|\AAA f\|_{p,r} \leq \CC_{\rm supp}(p,r) \, \Big(  \sum_{k=1}^M \big\| \AAA f \cdot \mathbf{1}_{U^\circ_k}\big\|_{p,r}^r \Big)^{1/r}.
	\end{displaymath}
	Moreover, by using Fact~\ref{44-F1} and Lemma~\ref{44-L2}, we obtain the following estimate
	\begin{displaymath}
	\| \AAA f \cdot \mathbf{1}_{U^\circ_k}\|_{p,r} \leq \CC_\triangle(p,r) \sum_{i=1}^N \| \AAA_{k,i} f_i \|_{p,r} \leq \CC_\triangle(p,r) \CC_{3,1}(p,q,r) K^{-1+1/p} L^{1/p} \sum_{i=1}^N \| f_i \|_{p,q}
	\end{displaymath}
	for each $k \in [M]$. Therefore, 
	\begin{displaymath}
	\|\AAA f\|_{p,r} \leq \CC_{\rm supp}(p,r) \CC_\triangle(p,r) \CC_{3,1}(p,q,r) K^{-1+1/p} L^{1/p} M^{1/r} \sum_{i=1}^N \| f_i\|_{p,q}.
	\end{displaymath}
	On the other hand, an application of H\"older's inequality gives
	\begin{displaymath}
	\Big(  \sum_{i=1}^N \| f_i\|_{p,q}^q \Big)^{1/q} \geq N^{-1 + 1/q} \sum_{i=1}^N \| f_i\|_{p,q}.
	\end{displaymath}
	Combining the two estimates above we conclude that
	\begin{displaymath}
	\|\AAA f\|_{p,r} \leq \CC_{\rm supp}(p,q) \CC_{\rm supp}(p,r) \CC_\triangle(p,r) \CC_{3,1}(p,q,r) K^{-1+1/p} L^{1/p} M^{1/r} N^{1 - 1/q} \| f\|_{p,q}.
	\end{displaymath}
	
	Finally, consider $g \coloneqq \sum_{i=1}^N (h_i m_i)^{-1/p} \cdot \mathbf{1}_{U_i}$. Then, by using \ref{44-iii}, we have
	\begin{displaymath}
	\AAA g \geq \sum_{k=1}^M \frac{N}{2^{1/p} K \alpha_k} \cdot \mathbf{1}_{U^\circ_k}
	\end{displaymath}
	and thus
	\begin{displaymath}
	\frac{ \|\AAA g\|_{p,r} }{\| g \|_{p,q} } \geq \frac{ \Big(\sum_{k=1}^M \big(K^{-1+1/p} L^{1/p} N \alpha_k^{-1+1/p} \beta_k^{1/p} h_N^{1/p} \big)^r \Big)^{1/r}  }{2^{1/p} \CC_{\rm supp}(p,q) \CC_{\rm supp}(p,r) N^{1/q}} \geq \frac{ K^{-1+1/p} L^{1/p} M^{1/r} N^{1 - 1/q} }{2^{1/p} \CC_{\rm supp}(p,q) \CC_{\rm supp}(p,r)}, 
	\end{displaymath}
	where we used \ref{44-ii}, \ref{44-v}, and Lemma~\ref{L4.1.4} in the first inequality, and \ref{44-vi} in the second.
\end{proof}

In the following lemma we estimate the norm of the maximal operator $\MM_\UU$ acting from $L^{p,q}(\UU)$ to $L^{p,r}(\UU)$. This is the main result of this subsection. 

\begin{lemma} \label{44-L4}
	Let $\UU$ be the test space of type~III defined above. Fix $q,r \in [1, \infty]$ with $q \leq r$ and consider the associated operator $\MM_\UU$. Then there exists a numerical constant $\CC_3 = \CC_3(p, q, r)$ independent of $N$, $M$, $K$, and $L$ such that
	\begin{displaymath}
	\frac{1}{\CC_3} \, \Big( 1 + K^{-1+1/p} L^{1/p} M^{1/r} N^{1 - 1/q} \Big) \leq \cc(p,q,r,\UU) \leq \CC_3 \, \Big( 1 + K^{-1+1/p} L^{1/p} M^{1/r} N^{1 - 1/q} \Big).
	\end{displaymath}
\end{lemma}

\begin{proof}
	First we estimate $\cc(p,q,r,\UU)$ from above. Take $f \in L^{p,q}(\UU)$ such that $\|f\|_{p,q} = 1$. It is easy to check that
	\begin{displaymath}
	\MM_{\UU} f \leq \max \big\{f, 4 \AAA f, 2 \widetilde{\MM} f, f_{\rm{avg}} \big\},
	\end{displaymath}
	where $\widetilde{\MM} f \coloneqq \mathbf{1}_{U \setminus U^\circ} \cdot \max_{x^\circ \in U^\circ} f(x^\circ)$. Therefore, by Fact~\ref{F4.1.1}, we have
	\begin{displaymath}
	\|\MM_{\UU} f\|_{p,r} \leq 4 \, \CC_\triangle(p,r) \, \Big( \|f\|_{p,r} + \|\AAA f \|_{p,r} + \| \widetilde{\MM} f\|_{p,r} + \| f_{\rm{avg}} \|_{p,r} \Big).
	\end{displaymath}
	The inequalities $\| \widetilde{\MM} f\|_{p,r} \leq \|f\|_{p,r}$ and $\| f_{\rm{avg}} \|_{p,r} \leq \CC_{\rm{avg}}(p,r) \|f\|_{p,r}$ follows from \ref{44-iv} and Fact~\ref{F4.1.2}, respectively. Combining the estimates above with Lemma~\ref{44-L3} we conclude that
	\begin{align*}
	\|\MM_{\UU} f\|_{p,r} \leq 4 \, \CC_\triangle(p,r) \, \Big( \CC_{\hookrightarrow}(p,q,r) \big(2 + \CC_{\rm{avg}}(p,r) \big) + \CC_{3,2}(p,q,r) \big(K^{-1+1/p} L^{1/p} M^{1/r} N^{1 - 1/q} \big) \Big).
	\end{align*}
	
	Now we estimate $\cc(p,q,r,\UU)$ from below. First, arguing as in the proof of Proposition~\ref{P4.1.10}, we obtain $\cc(p,q,r,\UU) \geq p^{1/r - 1/q} r^{-1/r} q^{1/q}$. Finally, the inequality 
	\begin{displaymath}
	\cc(p,q,r,\UU) \geq \frac{1}{2 \CC_{3,2}(p,q,r)} \, K^{-1+1/p} L^{1/p} M^{1/r} N^{1 - 1/q}
	\end{displaymath}
	is a consequence of Lemma~\ref{44-L3} and the fact that $\MM_\UU f \geq \AAA f / 2$ for each $f \in L^{p,q}(\UU)$.  
\end{proof}

At the end of this subsection we reformulate the result of the previous lemma in a way that makes it easier to use later on.

\begin{corollary} \label{44-C1}
	Fix $p \in (1, \infty)$, $\lambda \in (0, \infty)$, and $a, b, \kappa \in \NN$. Let $\UU_{(p, \lambda, a, b, \kappa)}$ be the test space $\UU_{p,N,M,K,L}$ with $p$ as above, $N = \kappa^b$, $M = \kappa^a$, and some $K, L$ satisfying $K^{-1 + 1/p} L^{1/p} = \lambda \kappa^{-b}$. Then for each $q,r \in [1, \infty]$ with $q \leq r$ we have
	\begin{displaymath}
	\frac{1}{\CC_3} \, \Big( 1 + \lambda \kappa^{a/r - b/q}\Big) \leq \cc(p,q,r,\UU_{(p, \lambda, a, b, \kappa)}) \leq \CC_3 \, \Big( 1 + \lambda \kappa^{a/r - b/q} \Big),
	\end{displaymath}
	where $\CC_3 = \CC_3(p,q,r)$ is the constant from Lemma~\ref{44-L4}.
\end{corollary}

\subsection{Composite test spaces} \label{44-S4}
In the following two sections by a \textit{composite test space} we mean any metric measure space $\WW$ that arises as a result of applying Proposition~\ref{P4.1.10} to a certain family of test spaces introduced in Subsection~\ref{44-S3}. This is a bit imprecise, but one can think of composite test spaces as intermediate objects between test spaces and the spaces we want to obtain in Theorem~\ref{44-T1}. More precisely, these latter ones will be composite test spaces constructed with the aid of a sequence of simpler composite test spaces. We now briefly explain the details of such a construction.

\begin{proposition} \label{44-P2}
	Let $(\WW_n)_{n \in \NN}$ be a given sequence of composite test spaces. Then there exists a composite test space $\WW$ such that for each $p \in (1, \infty)$ and $q,r \in [1, \infty]$ with $q \leq r$ we have
	\begin{displaymath}
	\frac{1}{\CC^2} \, \sup_{n \in \NN} \, \cc(p,q,r,\WW_n) \leq \cc(p,q,r,\WW) \leq \CC^2 \,\sup_{n \in \NN} \, \cc(p,q,r,\WW_n),
	\end{displaymath}  
	where $\CC = \CC(p,q,r)$ is the constant from Proposition~\ref{P4.1.10}. 
\end{proposition}

\begin{proof}
	Note that each space $\WW_n$ is constructed with the aid of some sequence of test spaces, say $\{\UU_{n,m} : m \in \NN\}$. Let $\WW$ be the space obtained by applying Proposition~\ref{P4.1.10} to the family $\{ \UU_{n,m} : n,m \in \NN \}$. The thesis follows directly from Proposition~\ref{P4.1.10}.  	
\end{proof}

Now we show how to construct composite test spaces for which the associated maximal operators have very specific properties.

\begin{lemma} \label{44-L5}
	Let $p, \epsilon \in (1, \infty)$, $\gamma \in \RR$, and $a, b, R \in \NN$. Then there exists a composite test space $\WW = \WW_{p, \gamma, a, b, R, \epsilon}$ such that for each $q,r \in [1, \infty]$ with $q \leq r$ we have
	\begin{align*}
	\cc(p,q,r,\WW) = \infty \quad & \textrm{if } \quad a/r - b/q = \gamma,  \\
	\CC_4^{-1} R^{\epsilon d} \leq \cc(p,q,r, \WW) \leq \CC_4 \big( 1 + R^{2\epsilon d} \big) \quad & \textrm{if } \quad a/r - b/q \in (\gamma - 2 \epsilon d, \gamma - \epsilon d), \\
	\cc(p,q,r,\WW) \leq  \CC_4 \quad & \textrm{if } \quad a/r - b/q \leq \gamma - 3 \epsilon d,
	\end{align*} 
	where $d = \sqrt{a^2 + b^2}$ and $\CC_4 = \CC_4(p,q,r)$ is independent of $\gamma$, $a$, $b$, $R$, and $\epsilon$.	
\end{lemma}

\noindent Figure~\ref{44-F2} describes the behavior of the function $\cc(p, q, r, \WW)$. We notice that the parameter $d$ appears here only for purely aesthetic reasons (for example, the Euclidean distance between the lines $a/r - b/q = \gamma$ and $a/r - b/q = \gamma - \epsilon d$ equals $\epsilon$).

\begin{figure}[H]
	\begin{center}
	\begin{tikzpicture}[
	axis/.style={very thick, ->, >=stealth'},
	important line/.style={thick},
	dashed line/.style={dashed, thin},
	pile/.style={thick, ->, >=stealth', shorten <=2pt, shorten
		>=2pt},
	every node/.style={color=black}
	]
	
	\draw[axis] (0,0)  -- (5.5,0) node(xline)[right]
	{$1/q$};
	\draw[axis] (0,0) -- (0,5.5) node(yline)[left] {$1/r$};
	
	\draw (0.0,0.0) node[below] {$0$} -- (0.0,0.0);
	\draw (0,4.5) node[left] {$1$} -- (0.1,4.5);
	\draw (4.5,0) node[below] {$1$} -- (4.5,0.1);
	
	\draw[important line] (0,4.5) -- (4.5,4.5);
	\draw[important line] (4.5,0) -- (4.5,4.5);
	
	\draw[dashed line] (0,0) -- (4.5,4.5);
	
	\draw[dashed line] (1.5,0) -- (4.5,4);
	\draw[dashed line] (1.95,0) -- (4.5,3.4);
	\draw[dashed line] (2.4,0) -- (4.5,2.8);
	\draw[dashed line] (2.85,0) -- (4.5,2.2);
	
	\draw (4.5,4.5) node[right] {$ q=r $} -- (4.5,4.5);
	
	\draw (4.5,4) node[right] {$ a/r-b/q=\gamma $} -- (4.5,4);
	\draw (4.5,3.4) node[right] {$ a/r-b/q=\gamma - \epsilon d $} -- (4.5,3.4);
	\draw (4.5,2.8) node[right] {$ a/r-b/q=\gamma - 2 \epsilon d $} -- (4.5,2.8);
	\draw (4.5,2.2) node[right] {$ a/r-b/q=\gamma - 3 \epsilon d $} -- (4.5,2.2);
	
	\draw[important line] (3,2.5) -- (2,3.5);
	\draw (2,3.5) node[above] {$= \infty$} -- (2,3.5);
	
	\draw[important line] (3.3,1.45) -- (5,1);
	\draw (5,1) node[right] {$\CC_4^{-1} R^{\epsilon d} \leq \cdot \leq \CC_4 (1 + R^{2 \epsilon d})$} -- (5,1);
	
	\draw[important line] (4,0.5) -- (1,2);
	\draw (1,2) node[above] {$\leq \CC_4$} -- (1,2);

	\end{tikzpicture}
	\caption{The behavior of the function $\cc(p, q, r, \WW)$.}
	\label{44-F2}
	\end{center}
\end{figure}

\begin{proof}
	For each $n \in \NN$ let $\UU_n$ be the test space $\UU_{(p, \lambda, a, b, \kappa)}$ from Corollary~\ref{44-C1} with $p$, $a$, and $b$ as above, $\kappa = R^n$, and $\lambda = R^{- n \gamma + (n+2)\epsilon d}$. We let $\WW$ be the space obtained by applying Proposition~\ref{P4.1.10} to the family $\{ \UU_n : n \in \NN\}$. The following estimates for $\MM_\WW$ are satisfied: if
	$a/r - b/q = \gamma$, then
	\begin{displaymath}
	\cc(p,q,r,\WW) \geq \frac{1}{\CC \CC_3} \lim_{n \rightarrow \infty} R^{- n \gamma + (n+2)\epsilon d} R^{n \gamma} = \infty,
	\end{displaymath}  
	if $a/r - b/q \in (\gamma - 2 \epsilon d, \gamma - \epsilon d)$, then
	\begin{displaymath}
	\cc(p,q,r,\WW) \geq \frac{1}{\CC \CC_3} \sup_{n \in \NN} R^{- n \gamma + (n+2)\epsilon d} R^{n (\gamma - 2 \epsilon d)} = \frac{R^{\epsilon d}}{\CC \CC_3}
	\end{displaymath}
	and
	\begin{displaymath}
	\cc(p,q,r,\WW) \leq \CC \CC_3 \sup_{n \in \NN} \Big( 1 + R^{- n \gamma + (n+2)\epsilon d} R^{n (\gamma - \epsilon d)} \Big) \leq \CC \CC_3 \Big( 1 + R^{2\epsilon d} \Big),
	\end{displaymath}
	and, if $a/r - b/q \leq \gamma - 3 \epsilon d$, then
	\begin{displaymath}
	\cc(p,q,r,\WW) \leq \CC \CC_3 \sup_{n \in \NN} \Big( 1 + R^{- n \gamma + (n+2)\epsilon d} R^{n (\gamma - 3 \epsilon d)} \Big) = 2 \CC \CC_3.
	\end{displaymath}
	Therefore, $\WW$ satisfies the desired properties. 
\end{proof}

At the end of this section we present another result for composite test spaces, which is particularly helpful if the domain of $F$ in Theorem~\ref{44-T1} is of the form $(\delta, 1]$, or if the domain is of the form $[\delta,1]$, but either $\delta = 1$ or $F$ is not continuous at $\delta$. 

\begin{lemma} \label{44-L6}
	Let $p \in (1, \infty)$, $\delta \in [0,1]$, and $\omega \in [0,\delta]$. Then the following statements are true.
	\begin{itemize}\setlength\itemsep{0em}
		\item There exists a composite test space $\WW^\leq = \WW^\leq_{p, \delta, \omega}$ such that $\cc(p,q,r, \WW^\leq) < \infty$ if and only if $1/q > \delta$, $r \geq q$ or $1/q = \delta$, $1/r \leq \omega$.
		\item There exists a composite test space $\WW^< = \WW^<_{p, \delta, \omega}$ such that $\cc(p,q,r, \WW^<) < \infty$ if and only if $1/q > \delta$, $r \geq q$ or $1/q = \delta$, $1/r < \omega$.
	\end{itemize}	
\end{lemma}

\begin{proof}
	Fix $p \in (1, \infty)$, $\delta \in [0,1]$, and $\omega \in [0,\delta]$. First we construct $\WW^\leq$. For each $n$ take $a_n = n$, $b_n = n^2$, and $\gamma_n$ satisfying $a_n \omega - b_n \delta = \gamma_n - 3 d_n \epsilon_n$, where $d_n = \sqrt{a_n^2 + b_n^2}$ and $\epsilon_n = 1/(3n)$. Let $\WW_n$ be the composite test space from Lemma~\ref{44-L5} with $p$ as above, $\gamma = \gamma_n$, $a = a_n$, $b = b_n$, $R=n^n$, and $\epsilon = \epsilon_n$. Since $\lim_{n \rightarrow \infty} b_n/a_n = \infty$ and $a_n (\omega+1/n) - b_n \delta > \gamma_n - 2 d \epsilon_n$, it is easy to check that $\WW^\leq$ may be chosen to be the space obtained by applying Proposition~\ref{P4.1.10} to the family $\{ \WW_n : n \in \NN\}$ (to obtain $\cc(p,q,r, \WW^\leq) = \infty$ for $1/q > \delta$, $r < q$ we use Remark~\ref{44-R1}, see Subsection~\ref{44-S5}). Finally, in order to construct $\WW^<$ we take $a_n = n$, $b_n = n^2$, and $\gamma_n$ satisfying $a_n (\omega-1/n) - b_n \delta = \gamma_n - 3 d_n \epsilon_n$ and $a_n \omega - b_n \delta \in (\gamma_n - 2 d_n \epsilon_n, \gamma_n - d_n \epsilon_n)$, where $d_n$ and $\epsilon_n$ are as before, and then we repeat the previous procedure.
\end{proof}

\noindent We note that Lemma~\ref{44-L6} may also be used to construct $\ZZZ$ such that $\Omega^p_{\rm HL}(\ZZZ) = \emptyset$. Indeed, it suffices to take $\WW^<$ with $p$ as above, $\delta = 1$, and $\omega = 0$. 

\subsection{Proof of the main result} \label{44-S5}

This section is devoted to proving Theorem~\ref{44-T0}. The proof consists of two parts. First of them relies on showing that the conditions imposed on $F$ in Theorem~\ref{44-T0} are necessary. The second one consists of furnishing a bunch of examples of spaces so that each of the scenarios specified in the thesis can be illustrated with some nondoubling space. This part will be ensured by Theorem~\ref{44-T1}.

\subsubsection*{Necessary conditions}

Here we briefly discuss why there are no alternatives for the shape of $\Omega^p_{\rm HL}(\XX)$ other than those mentioned in Theorem~\ref{44-T0}. We begin with the following simple observation.

\begin{remark} \label{44-R0}
	Fix $p \in (1, \infty)$ and let $\XX$ be an arbitrary metric measure space. If $(u,w) \in \Omega^p_{\rm HL}(\XX)$, then $[u,1] \times [w,1] \subset \Omega^p_{\rm HL}(\XX)$.
\end{remark}

\noindent Indeed, this follows by the fact that the Lorentz spaces $L^{p,q}(\XX)$ increase as the parameter $q$ increases.

\smallskip By Remark~\ref{44-R0} we know that either $\Omega^p_{\rm HL}(\XX)$ is empty or it consists of points lying under the graph of some nondecreasing function, say $F$, and the domain of $F$ is of the form $[\delta, 1]$ or $(\delta, 1]$ for some $\delta \in [0,1]$ or $\delta \in [0,1)$, respectively. More precisely, for each $u$ from the domain of $F$ we have $(u,w) \in \Omega^p_{\rm HL}(\XX)$ for $w < F(u)$ and $(u,w) \notin \Omega^p_{\rm HL}(\XX)$ for $w > F(u)$ (here we do not focus on whether $(u,F(u))$ belongs to $\Omega^p_{\rm HL}(\XX)$ or not, except for the case $F(u) = 0$, which forces that the first option actually takes place).   

Remark~\ref{44-R1} below, in turn, explains why the assumption $F(u) \leq u$ is needed. 

\begin{remark} \label{44-R1}
	Let $\XX$ be a metric measure space such that $|X \setminus \rm{supp}(\mu)| = 0$. Assume that there exists an infinite family $\BB$ of pairwise disjoint balls $B$ satisfying $|B| \in (0,\infty)$. Then for each $p \in (1, \infty)$ we have $\Omega^p_{\rm HL}(\XX) \subset \{(u,w) \in [0,1]^2 : u \leq w \}$.    
\end{remark}

\noindent Indeed, this is just a reformulation of Observation~\ref{O4.1.6}. 

\smallskip
Finally, the fact that $\Omega^p_{\rm HL}(\XX)$ is convex, and hence $F$ must be concave, is justified by the following interpolation argument. 

\begin{theorem} \label{44-T2}
	Fix $p \in [1, \infty)$ and $q_0,q_1,r_0,r_1 \in [1,\infty]$ with $q_0 \leq q_1$, $q_0 \leq r_0$, and $q_1 \leq r_1$. Let $\XX$ be an arbitrary metric measure space and assume that the associated maximal operator $\MM_\XX$ is bounded from $L^{p, q_i}(\XX)$ to $L^{p, r_i}(\XX)$ for $i \in \{0, 1\}$. Then for each $\theta \in (0,1)$ the operator $\MM_\XX$ is bounded from $L^{p, q_\theta}(\XX)$ to $L^{p, r_\theta}(\XX)$, where
	\begin{displaymath}
	\frac{1}{q_\theta} = \frac{1-\theta}{q_0} + \frac{\theta}{q_1}, \qquad \frac{1}{r_\theta} = \frac{1-\theta}{r_0} + \frac{\theta}{r_1}.
	\end{displaymath}
\end{theorem}

\noindent We explain briefly how Theorem~\ref{44-T2} can be inferred from the general theory of interpolation. We begin with the comment that Lorentz spaces in this context were considered for the first time by Hunt in \cite{Hu}. However, the theorem formulated there does not cover Theorem~\ref{44-T2}. Hence, we are forced to refer to the literature where some more advanced interpolation methods are developed. The appropriate variant of Theorem~\ref{44-T2} for linear operators can be directly deduced from \cite[Theorem 5.3.1]{BL} (see also \cite{Ma}, where the $K$-functional for the couple $(L^{p,q_0}, L^{p,q_1})$ is computed). Then, a suitable linearization argument (see \cite{Ja}, for example) allows us to extend this result to the class of sublinear operators and thus the maximal operator $\MM_\XX$ is also included.   

Although there are several ways to deduce Theorem~\ref{44-T2} from the theorems that appear in the literature, each of them, to the author's best knowledge, requires a deep understanding of the interpolation theory. As the author found an elegant, elementary proof of Theorem~\ref{44-T2}, he decided to present it in \hyperref[Appendix]{Appendix}.

\subsubsection*{Proof of Theorem~\ref{44-T1}}

\begin{proof}[Proof of Theorem~\ref{44-T1}]
We consider three cases depending on the properties of $F$.
	
\smallskip \noindent \bf Case~\hypertarget{44case1}{1}: \rm $F \colon [\delta, 1] \rightarrow [0,1]$, $F$ continuous at $\delta$. Fix $p \in (1, \infty)$ and $\delta \in [0, 1]$, and take $F \colon [\delta, 1] \rightarrow [0,1]$ concave, nondecreasing, continuous at $\delta$, and such that $F(u) \leq u$ for each $u \in [\delta, 1]$. 

First we construct $\ZZZ$. We can assume that $\delta \in [0,1)$, since the case $\delta = 1$ is covered by Lemma~\ref{44-L6}. Consider the following countable set 
\begin{displaymath}
\Bigg\{ \Big(\frac{1}{q}, \frac{1}{r}\Big) \in \big( [0,1] \cap \QQ \big)^2 : \Bigg( {\frac{1}{q}} \geq  \delta \ \wedge \ \frac{1}{r} > F\Big(\frac{1}{q}\Big) \Bigg) \ \vee \ \Bigg( \frac{1}{q} < \delta \Bigg) \Bigg\} 
\end{displaymath}
and enumerate it to obtain a sequence $( P_1, P_2, \dots )$. Fix $n \in \NN$ and let $P_n = \big( \frac{1}{q_n}, \frac{1}{r_n} \big)$. Since $F$ is concave and nondecreasing, we can choose $\gamma_n \in \RR$, $a_n, b_n \in \NN$, and $\epsilon_n \in (0,\infty)$ such that
\begin{itemize}\setlength\itemsep{0em}
	\item $a_n / r_n - b_n / q_n = \gamma_n$,
	\item if $a_n / r - b_n / q > \gamma_n - 3 \epsilon_n d_n$, then $\frac{1}{q} \geq  \delta$, $\frac{1}{r} > F\big(\frac{1}{q}\big)$ or $\frac{1}{q} < \delta$, where $d_n = \sqrt{a_n^2 + b_n^2}$.
\end{itemize}
Let $\WW_n$ be the composite test space from Lemma~\ref{44-L5} with $p$ as above, $\gamma = \gamma_n$, $a = a_n$, $b = b_n$, $R = 1$, and $\epsilon = \epsilon_n$. It is easy to check that $\YY$ may be chosen to be the space obtained by applying Proposition~\ref{44-P2} to the family $\{ \WW_n : n \in \NN \}$.

Now we construct $\ZZZ'$. Again we assume that $\delta \in [0,1)$, since the case $\delta = 1$ is covered by Lemma~\ref{44-L6}. For each $n \in \NN$ and $u \in [\delta, 1]$ we choose $\gamma_{n,u} \in \RR$ and $a_{n,u}, b_{n,u} \in \NN$ such that
\begin{itemize}\setlength\itemsep{0em}
	\item $\gamma_{n,u} - 2 d_{n,u} / n < a_{n,u} u - b_{n,u} F(u) < \gamma_{n,u} - d_{n,u} / n$, where $d_{n,u} = \sqrt{a_{n,u}^2 + b_{n,u}^2}$,
	\item if $a_{n,u} / r - b_{n,u} / q \geq \gamma_{n,u} - d_{n,u} / n$, then $\frac{1}{q} \geq \delta$, $\frac{1}{r} > F\big(\frac{1}{q}\big)$ or $\frac{1}{q} < \delta$. 
\end{itemize}
Let $\WW_{n,u}$ be the composite test space from Lemma~\ref{44-L5} with $p$ as above, $\gamma = \gamma_{n,u}$, $a = a_{n,u}$, $b = b_{n,u}$, $R = n^n$, and $\epsilon = 1/n$. Fix $n \in \NN$ and observe that for each $u \in [\delta, 1]$ the set
\begin{displaymath}
E_{n,u} = \Big\{ v \in [\delta, 1] : \gamma_{n,u} - 2 d_n / n < av - bF(v) < \gamma_{n,u} - d_n/n \Big \}
\end{displaymath} 
is open in $[\delta, 1]$ with its natural topology. Thus $\{ E_{n,u} : u \in [\delta, 1] \}$ is an open cover of $[\delta, 1]$ and we can find a finite subset $U_n \subset [\delta, 1]$ such that $\bigcup_{u \in U_n} E_{n,u} = [\delta, 1]$. Finally, we let $\ZZZ'$ be the space obtained by applying Proposition~\ref{44-P2} to the family $\{ \WW_{n,u} : n \in \NN, u \in U_n\}$. We will show that $\ZZZ'$ satisfies the desired properties. Fix $u_0 \in [\delta, 1]$ and observe that for each $n \in \NN$ there exists $u_n \in U_n$ such that $u_0 \in E_{n, u_n}$. Therefore, in view of Lemma~\ref{44-L5},
\begin{displaymath}
\cc(p, 1/u_0, 1/ F(u_0), \ZZZ') \geq \frac{1}{\CC^2} \cc(p, 1/u_0, 1/F(u_0), \WW_{n,u_n}) \geq \frac{1}{\CC^2 \CC_4} n^{d_{n,u}}.
\end{displaymath}
Since $n$ is arbitrary and $d_{n,u} \geq 1$, we conclude that  $\cc(p, 1/u_0, 1/ F(u_0), \ZZZ') = \infty$ and, as a~result, we obtain $ \cc(p, q, r, \ZZZ') = \infty$ if $\frac{1}{q} \geq \delta$, $\frac{1}{r} \geq F\big(\frac{1}{q}\big)$ or $\frac{1}{q} < \delta$. Now let us consider a~pair $(q,r)$ satisfying $\frac{1}{q} \geq \delta$, $\frac{1}{r} < F\big(\frac{1}{q}\big)$. Then we have
\begin{displaymath}
d(q,r,F) \coloneqq \min \Bigg\{ d_{\rm e}\Bigg( \Big( \frac{1}{q}, \frac{1}{r} \Big), \Big(u, F(u) \Big) \Bigg) : u \in [\delta, 1] \Bigg\} > 0,
\end{displaymath}
where $d_{\rm e}$ is the standard Euclidean metric on the plane. Observe that for each $n \in \NN$ and $u \in U_n$ we have the following implication
\begin{displaymath}
a_{n,u} / r - b_{n,u} / q > \gamma_{n,u} - 3 d_{n,u} / n \implies d(q,r,F) \leq 2/n.
\end{displaymath}
Hence if $n > 2 / d(q,r,F)$, then for each $u \in U_n$ we have $a_{n,u} / r - b_{n,u} / q \leq \gamma_{n,u} - 3 d_{n,u} / n$, which implies $\cc(p,q,r,\WW_{n,u}) \leq \CC_4$. Finally, since for each of the finitely many pairs $(n,u)$ satisfying $n \leq 2 / d(q,r,F)$ and $u \in U_n$ there is $\cc(p,q,r,\WW_{n,t}) < \infty$, we conclude that $\cc(p,q,r, \ZZZ') < \infty$.

\smallskip \noindent \bf Case~\hypertarget{44case2}{2}: \rm $F \colon [\delta, 1] \rightarrow [0,1]$, $F$ not continuous at $\delta$. Fix $p \in (1, \infty)$ and $\delta \in (0, 1)$, and take $F \colon [\delta, 1] \rightarrow [0,1]$ concave, nondecreasing, satisfying $F(\delta) = \omega < \lim_{u \rightarrow \delta} F(u)$ for some $\omega \in [0,\delta)$, and such that $F(u) \leq u$ for each $u \in [\delta, 1]$. Let $\widetilde{F}$ be the continuous modification of $F$, that is, $\widetilde{F}(u)=F(u)$ for $u \in (\delta, 1)$ and $\widetilde{F}(\delta) = \lim_{u \rightarrow \delta} F(u)$. Then $\widetilde{F}$ satisfies the conditions specified in Case~\hyperlink{44case1}{1}. Let $\widetilde{\ZZZ}$ and $\widetilde{\ZZZ}'$ be the spaces obtained in Case~\hyperlink{44case1}{1} for $\widetilde{F}$. We also let $\widetilde{\WW}$ be the composite test space $\WW^\leq$ (respectively, $\WW^<$) from Lemma~\ref{44-L6} with $p$, $\delta$, and $\omega$ as above. It is easy to check that $\ZZZ$ (respectively, $\ZZZ'$) may be chosen to be the space obtained by using Proposition~\ref{44-P2} to $\widetilde{\ZZZ}$ (respectively, $\widetilde{\ZZZ}'$) and countably many copies of $\widetilde{\WW}$.

\smallskip \noindent \bf Case~\hypertarget{44case3}{3}: \rm $F \colon (\delta, 1] \rightarrow [0,1]$. Fix $p \in (1, \infty)$ and $\delta \in [0,1)$, and take $F \colon (\delta, 1] \rightarrow [0,1]$ concave, nondecreasing and such that $F(u) \leq u$ for each $u \in (\delta,1]$. We extend $F$ to $\widetilde{F} \colon [\delta, 1] \rightarrow [0,1]$, setting $\widetilde{F}(\delta) = \lim_{u \rightarrow \delta} F(u)$. Then $\widetilde{F}$ satisfies the conditions specified in Case~\hyperlink{44case1}{1}. Let $\widetilde{\ZZZ}$ and $\widetilde{\ZZZ}'$ be the spaces obtained in Case~\hyperlink{44case1}{1} for $\widetilde{F}$. We also let $\widetilde{\WW}$ be the composite test space $\WW^<$ from Lemma~\ref{44-L6} with $p$ and $\delta$ as above, and $\omega = 0$. It is easy to check that $\ZZZ$ (respectively, $\ZZZ'$) may be chosen to be the space obtained by applying Proposition~\ref{44-P2} to $\widetilde{\ZZZ}$ (respectively, $\widetilde{\ZZZ}'$) and countably many copies of $\widetilde{\WW}$.
\end{proof}

\subsubsection*{Proof of Theorem~\ref{44-T0}}

We are ready to prove the main result of this chapter.

\begin{proof}[Proof of Theorem~\ref{44-T0}]
The first part of the theorem follows from Remarks~\ref{44-R0}~and~\ref{44-R1}, and Theorem~\ref{44-T2}, and the second part follows from Theorem~\ref{44-T1}.	
\end{proof} 

The last issue we would like to mention in this chapter is the boundary problem. Denote be $\bar{\partial}\Omega^p_{\rm HL}(\XX)$ the upper part of the boundary of $\Omega^p_{\rm HL}(\XX)$, that is, the set $\{(u, F(u)) : u \in \rm{Dom}(F)\}$, where $\rm{Dom}(F)$ is the domain of $F$. According to this, for each space constructed in Theorem~\ref{44-T1} one of the following two possibilities holds 
$$
\bar{\partial}\Omega^p_{\rm HL}(\XX) \subset \Omega^p_{\rm HL}(\XX) 
\quad {\rm or} \quad
\bar{\partial}\Omega^p_{\rm HL}(\XX) \cap \Omega^p_{\rm HL}(\XX) = \emptyset.
$$
In fact, Proposition~\ref{44-P2} combined with Lemmas~\ref{44-L5}~and~\ref{44-L6} can provide a wide range of other cases. For example, if $F$ is strictly concave, then for a~given set $E \subset \rm{Dom}(F)$ such that $\overline{E}$ is countable we can find $\XX$ such that $\MM_\XX$ is bounded from $L^{p, 1/u}(\XX)$ to $L^{p, 1/F(u)}(\XX)$ if and only if $u \notin E$. Nevertheless, it is probably very difficult to describe precisely all forms that the intersections $\bar{\partial}\Omega^p_{\rm HL}(\XX) \cap \Omega^p_{\rm HL}(\XX)$ can take.
\newpage
\thispagestyle{empty}
\chapter{{\rm BMO} spaces}\label{chap5}
\setstretch{1.0}
{\rm BMO} traditionally occurs in the literature as a function space associated with the space $\mathbb{R}^d$, $d \in \NN$, equipped with the Euclidean metric and Lebesgue measure. Roughly speaking, it contains functions whose mean oscillation over a given cube $Q \subset \mathbb{R}^d$ is bounded uniformly with respect to the choice of that cube. Although {\rm BMO} was introduced by John and Nirenberg \cite{JN} in the context of partial differential equations, it is also a very useful tool in harmonic analysis. One reason is that many of the operators considered there turn out to be bounded from $L^\infty$ to {\rm BMO} even though they are not always bounded on $L^\infty$. This, in turn, can often be used to prove the boundedness of such operators on $L^p$ for some $p \in (1, \infty)$ by using the interpolation theorem obtained by Fefferman and Stein \cite{FS}. Another interesting fact is that {\rm BMO} is dual to the Hardy space, $H^1$, which is of great use in harmonic analysis. This result was shown by Fefferman \cite{F}. Finally, {\rm BMO} functions are in close relation with other objects appearing in this field such as Carleson measures, paraproducts or commutator operators (for further consideration see \cite{C, ChS, Ch, G}, for example). 

It is well known that most of the theory mentioned above can be developed in a more general context including all doubling metric measure spaces. However, the situation changes significantly if the space we deal with is nondoubling. We have examples showing that some of the classical theorems fail to occur in certain nondoubling situations (see \cite{Al2, Sj} for studying the weak type $(1,1)$ boundedness of the Hardy--Littlewood maximal operator), while, in contrast, some theorems can be proved for wider classes of spaces, usually requiring more complicated methods (see \cite{NTV1, T}, where the boundedness of the Cauchy integral operator was considered). 

{\rm BMO} spaces for nondoubling spaces were quite successfully studied by Mateu, Mattila, Nicolau and Orobitg \cite{MMNO}. In particular, the authors have shown that for many Borel measures on $\mathbb{R}^d$, not necessarily doubling, it is possible to define {\rm BMO} spaces in such a way as to be able to use an interpolation argument analogous to that obtained in \cite{FS}. On the other hand, a~somewhat surprising fact shown in \cite{MMNO} is that there exist measures on $\mathbb{R}^2$ for which the associated spaces {\rm BMO} and ${\rm BMO}_{\rm b}$ defined with the aid of cubes and balls, respectively, do not coincide. Another result, which will be described later on, is related to some untypical behavior of the family of spaces $\{ {\rm BMO}^p_{\rm b} : p \in [1, \infty) \}$, which occurs under certain conditions. In summary, there are many examples in \cite{MMNO} which illustrate that in some specific situations {\rm BMO} spaces may have very unusual properties. This idea also accompanies the following chapter. 

Our main motivation here is to study the spaces ${\rm BMO}^p_{\rm b}$ with $p \in [1, \infty)$ considered as sets of functions, in order to describe whether the natural inclusions between them are proper or not. Since we deal with arbitrary metric measure spaces, balls determined by metrics are used to define ${\rm BMO}^p_{\rm b}$. From now on we omit the subscript and write ${\rm BMO}^p$ instead of ${\rm BMO}^p_{\rm b}$.

\setstretch{1.15} 

\section{Preliminaries and results}

Let $\XX = (X, \rho, \mu)$ be a given metric measure space. 
For a locally integrable function $f \colon X \rightarrow \CCC$ and an open ball $B \subset X$ such that $|B| \in (0, \infty)$ we denote the average value of $f$ on $B$ by
\begin{displaymath}
f_B \coloneqq \frac{1}{|B|} \int_B f(x) \, {\rm d}\mu(x).
\end{displaymath}
Then, given $p \in [1, \infty)$, we let ${\rm BMO}^p(\XX)$ be the space consisting of all functions $f$ for which
\begin{displaymath}
\|f\|_{\ast,p} \coloneqq \sup_{B \subset X} \Big( \frac{1}{|B|} \int_B |f(x)-f_B|^p \, {\rm d}\mu(x) \Big)^{1/p}
\end{displaymath}
is finite (the supremum is taken over all balls $B$ contained in $X$ and such that $|B| \in (0, \infty)$). We keep to the rule that two functions are identified if they differ by a constant. With this convention $\| \cdot \|_{\ast,p}$ satisfies the norm properties and thus ${\rm BMO}^p(\XX)$ can be viewed as a Banach space (it is a mathematical folklore that ${\rm BMO}^p(\XX)$ is complete in any setting). If $p=1$, then we will usually write shortly ${\rm BMO}(\XX)$ and $\|f\|_\ast$ instead of ${\rm BMO}^1(\XX)$ and $\|f\|_{\ast, 1}$, respectively.

Recall that if $p_1, p_2 \in [1, \infty)$ with $p_1 < p_2$, then by using H\"older's inequality we always have the inequality $\| f \|_{\ast, p_1} \leq \| f \|_{\ast, p_2}$ and, as a consequence, the inclusion ${\rm BMO}^{p_2}(\XX) \subset {\rm BMO}^{p_1}(\XX)$. Moreover, if ${\rm BMO}^{p_1}(\XX)$ and ${\rm BMO}^{p_2}(\XX)$ coincide as sets, then the corresponding norms are equivalent. In fact, this is always the case if $\mu$ is doubling. Indeed, one can obtain that the spaces ${\rm BMO}^p(\XX)$ with $p \in [1, \infty)$ coincide by using the John--Nirenberg inequality which is true for spaces satisfying the doubling condition (see \cite[Theorem A, p.~563]{MMNO}, for example). However, the John--Nirenberg inequality fails to occur in general. Moreover, in \cite{MMNO} the authors construct a~nondoubling space $\XX$ for which there exists $f \in {\rm BMO}(\XX)$ such that $f \notin {\rm BMO}^p(\XX)$ for all $p\in(1, \infty)$. Here we go further and describe precisely which types of relations between the spaces ${\rm BMO}^p(\XX)$ with $p \in [1, \infty)$ actually can happen. Namely, we prove the following theorem.

\begin{theorem}\label{T5.1.1}
	For a given space $\XX$ we have one of the following three possibilities:
	\begin{enumerate}[label={\rm (\Alph*)}]
		\item \label{5A} The spaces ${\rm BMO}^p(\XX)$ with $p \in [1, \infty)$ all coincide.
		\item \label{5B} There exists $p_0 \in (1, \infty)$ such that for two distinct parameters $p_1,p_2 \in [1, \infty)$ the spaces ${\rm BMO}^{p_1}(\XX)$ and ${\rm BMO}^{p_2}(\XX)$ coincide if and only if $\max\{p_1, p_2 \} < p_0$. 	
		\item \label{5C} There exists $p_0 \in [1, \infty)$ such that for two distinct parameters $p_1,p_2 \in [1, \infty)$ the spaces ${\rm BMO}^{p_1}(\XX)$ and ${\rm BMO}^{p_2}(\XX)$ coincide if and only if $\max\{p_1, p_2 \} \leq p_0$.
	\end{enumerate}
	Conversely, for each of the possibilities described above and for any permissible choice of $p_0$ in the cases \ref{5B} and \ref{5C} we can construct $\XX$ for which the associated spaces ${\rm BMO}^p(\XX)$ with $p \in [1, \infty)$ realize the desired properties.
\end{theorem}

\noindent The rest of this chapter is organized as follows. In Section~\ref{S5.2} we present a short proof of the main theorem based on certain results of a rather technical nature which are proved later on. Sections~\ref{S5.3}~and~\ref{S5.4} are devoted to the study of these technical issues. Finally, in Section~\ref{S5.5} some additional remarks concerning the John--Nirenberg inequality are given. 

\section{Proof of the main result}\label{S5.2}

In this section we prove Theorem~\ref{T5.1.1}. To do this we use two ingredients which we formulate here and prove in Sections~\ref{S5.3}~and~\ref{S5.4}, respectively. The first one is the following.

\begin{lemma}\label{L5.2.1}
Let $\XX$ be a given metric measure space. If ${\rm BMO}^{p_1}(\XX) \subsetneq {\rm BMO}^{p_2}(\XX)$ for some $p_1,p_2 \in [1, \infty)$ with $p_1 < p_2$, then for any $\alpha \in (1,\infty)$ we also have ${\rm BMO}^{\alpha p_1}(\XX) \subsetneq {\rm BMO}^{\alpha p_2}(\XX)$.
\end{lemma}

The second goal we need is to find a suitable family of spaces for which some specific relations between the associated {\rm BMO} spaces occur. The process of constructing such spaces is the most technical part of this chapter. We will obtain two complementary results stated below.

\begin{proposition}\label{P5.2.2}
	For each $p_0 \in (1, \infty)$ there exists a space $\XX$ such that ${\rm BMO}^p(\XX)$ coincides with ${\rm BMO}(\XX)$ if and only if $p \in [1,p_0)$. 
\end{proposition}

\begin{proposition}\label{P5.2.3}
	For each $p_0 \in [1, \infty)$ there exists a space $\XX$ such that ${\rm BMO}^p(\XX)$ coincides with ${\rm BMO}(\XX)$ if and only if $p \in [1, p_0]$. 
\end{proposition}

Now we show that Theorem~\ref{T5.1.1} follows easily from the results mentioned above.

\begin{proof}[Proof of Theorem~\ref{T5.1.1}]
	For a given space $\XX$ define 
	\begin{displaymath}
	p_0 \coloneqq p_0(\XX) \coloneqq \sup \big\{p \in [1, \infty) : {\rm BMO}^p(\XX) = {\rm BMO}(\XX) \big\}.
	\end{displaymath}	
	The case $p_0 = \infty$ corresponds to \ref{5A}. On the other hand, if $p_0 \in [1, \infty)$, then we have two possibilities: ${\rm BMO}^{p_0}(\XX)$ coincides with ${\rm BMO}(\XX)$ or not. We analyze only the first option which corresponds to \ref{5C}, and the second one which corresponds to \ref{5B} can be treated similarly. Obviously, we have that ${\rm BMO}^{p}(\XX)$ coincides with ${\rm BMO}(\XX)$ for each $p \in [1, p_0]$. Let us now consider two distinct parameters $p_1, p_2 \in [1, \infty)$ with $p_2 > \max\{p_0,p_1\}$. If $p_1 \leq p_0$, then by the definition of $p_0$ we have ${\rm BMO}^{p_1}(\XX) \subsetneq {\rm BMO}^{p_2}(\XX)$. On the other hand, if $p_1 > p_0$, then there exists $\alpha \in (1, \infty)$ such that $p_1 / \alpha \leq p_0 < p_2/\alpha$. Hence, we have ${\rm BMO}^{p_1/ \alpha}(\XX) \subsetneq {\rm BMO}^{p_2/\alpha}(\XX)$ and by using Lemma~\ref{L5.2.1} we conclude that ${\rm BMO}^{p_1}(\XX) \subsetneq {\rm BMO}^{p_2}(\XX)$.
	
	Finally, the last part of Theorem~\ref{T5.1.1} can be deduced from Propositions~\ref{P5.2.2}~and~\ref{P5.2.3}. Indeed, the spaces obtained there cover all specified cases corresponding to \ref{5B} and \ref{5C}. Since the scenario described in \ref{5A} can be realized by any doubling space, the proof is complete.
\end{proof}

\section{Proof of the key lemma}\label{S5.3}

This section is entirely devoted to proving Lemma~\ref{L5.2.1}. It is worth mentioning here that it is possible to formulate Lemma~\ref{L5.2.1} in a more general form than the one presented before. Indeed, the proof does not rely on the fact that balls were used to define the spaces ${\rm BMO}^p(\XX)$. Thus, the conclusion remains true if one considers the spaces ${\rm BMO}^p(\XX)$ introduced with the aid of an arbitrary base (that is, a fixed family of subsets of $X$) instead.

\begin{proof}[Proof of Lemma~\ref{L5.2.1}]
	Suppose that ${\rm BMO}^{p_1}(\XX) \subsetneq {\rm BMO}^{p_2}(\XX)$ for some $p_1,p_2 \in [1, \infty)$ with $p_1 < p_2$ and fix $\alpha \in (1,\infty)$. We begin with a simple observation that it suffices to find a sequence $(g_N)_{N=1}^\infty$ satisfying $\|g_N\|_{\ast, \alpha p_1} \leq C$ uniformly in $N$ and $\lim_{N \rightarrow \infty} \|g_N\|_{\ast, \alpha p_2} = \infty$. 
	
	Take $f \in {\rm BMO}^{p_1}(\XX) \setminus {\rm BMO}^{p_2}(\XX)$ and write $f = f_1 + i f_2$, where both $f_1$ and $f_2$ are real-valued. Observe that at least one of the functions $f_1, f_2$ also lies in ${\rm BMO}^{p_1}(\XX) \setminus {\rm BMO}^{p_2}(\XX)$. Therefore, we can assume $f$ to be real-valued. 
	
	Fix $N \in \mathbb{N}$ and choose a ball $B_N \subset X$ such that
	\begin{equation}\label{5.3.1}
	\frac{1}{|B_N|} \int_{B_N} |f - f_{B_N}|^{p_2} \, {\rm d} \mu \geq N.
	\end{equation}
	Take $f_N \coloneqq f - f_{B_N}$ and introduce $g_N$ defined by
	\begin{displaymath}
	g_N(x) \coloneqq {\rm sgn} (f_N(x)) \cdot |f_N(x)|^{1/\alpha}.
	\end{displaymath}
	
	Our first goal is to show that $\|g_N\|_{\ast, \alpha p_1} \leq C$ holds independently of $N$. Notice that
	\begin{align}\label{5.3.2}
	\begin{split}
	\frac{1}{|B|} \int_B |h - h_B|^p \, {\rm d}\mu \leq \frac{1}{|B|^2} \int_B \int_B |h(x) - h(y)|^p \, {\rm d}\mu(x) \, {\rm d}\mu(y) 
	\leq \frac{2^p}{|B|} \int_B |h - h_B|^p \, {\rm d}\mu
	\end{split}
	\end{align}
	holds for any $p \in [1, \infty)$, $B \subset X$, and locally integrable $h$. In particular, \eqref{5.3.2} implies that
	\begin{equation}\label{5.3.3}
	\frac{1}{|B|^2} \int_B \int_B |f_N(x) - f_N(y)|^{p_1} \, {\rm d}\mu(x) \, {\rm d}\mu(y) \leq 2^{p_1} \|f_N\|_{\ast, p_1}^{p_1} = 2^{p_1} \|f\|_{\ast, p_1}^{p_1}
	\end{equation}
	holds for each ball $B \subset X$.  
	We would like to obtain a similar estimate for $g_N$ and $\alpha p_1$ instead of $f_N$ and $p_1$. Take any two points $x, y \in B$. If $g_N(x)$ and $g_N(y)$ are of the same sign, then
	\begin{equation*}
	|g_N(x) - g_N(y)|^{\alpha p_1} = \big| \, |f_N(x)|^{1/ \alpha} - |f_N(y)|^{1/ \alpha} \, \big|^{\alpha p_1} \leq    |f_N(x) - f_N(y)|^{p_1}.
	\end{equation*} 
	On the other hand, if $g_N(x) > 0$ and $g_N(y) \leq 0$, then we obtain
	\begin{align*}
	\begin{split}
	|g_N(x) - g_N(y)|^{\alpha p_1} \leq 2^{\alpha p_1} (g_N(x)^{\alpha p_1} + (-g_N(y))^{\alpha p_1}) & = 2^{\alpha p_1} (f_N(x)^{p_1} + (-f_N(y))^{p_1}) 
	\\ & \leq 2^{\alpha p_1} |f_N(x) - f_N(y)|^{p_1}.
	\end{split}
	\end{align*} 
	Combining \eqref{5.3.3} with the last two estimates gives
	\begin{displaymath}
	\frac{1}{|B|^2} \int_B \int_B |g_N(x) - g_N(y)|^{\alpha p_1} \, {\rm d}\mu(x) \, {\rm d}\mu(y) \leq 2^{(1 + \alpha)p_1} \|f\|_{\ast, p_1}^{p_1}
	\end{displaymath} 
	which, by using \eqref{5.3.2} again, results in the desired inequality $\|g_N\|_{\ast, \alpha p_1} \leq 2^{1+ \alpha} \|f\|_{\ast, p_1}$.
	
	It remains to estimate $\|g_N\|_{\ast, \alpha p_2}$ from below. For $M \in (0, \infty)$ we take $N \in \NN$ satisfying
	\begin{equation}\label{5.3.4}
	2^{-\alpha p_2} N - 2^{\alpha p_2} (M+1)^{\alpha p_2} \geq M
	\end{equation}  
	and show that
	\begin{equation}\label{5.3.5}
	\frac{1}{|B_N|} \int_{B_N} |g_N - (g_N)_{B_N}|^{\alpha p_2} \, {\rm d}\mu \geq M.
	\end{equation}
	We consider the following two cases: $|(g_N)_{B_N}| \leq M+1$ and $(g_N)_{B_N} < -M-1$ (for the remaining case $(g_N)_{B_N} > M+1$ one can replace $f_N$ and $g_N$ by $-f_N$ and $-g_N$, respectively). If $|(g_N)_{B_N}| \leq M+1$, then we use the following estimates: for $x \in B_N$ such that $|g_N(x)| > 2(M+1)$,
	\begin{equation}\label{5.3.6}
	|g_N(x) - (g_N)_{B_N}|^{\alpha p_2} \geq 2^{-\alpha p_2} |g_N(x)|^{\alpha p_2} = 2^{-\alpha p_2} |f_N(x)|^{p_2},
	\end{equation}
	and, for $x \in B_N$ such that $|g_N(x)| \leq 2(M+1)$,
	\begin{equation}\label{5.3.7}
	|g_N(x) - (g_N)_{B_N}|^{\alpha p_2} \geq 0 \geq |g_N(x)|^{\alpha p_2} - 2^{\alpha p_2} (M+1)^{\alpha p_2} = |f_N(x)|^{p_2} - 2^{\alpha p_2} (M+1)^{\alpha p_2}.
	\end{equation}
	Applying first \eqref{5.3.6} and \eqref{5.3.7}, and then \eqref{5.3.1} and \eqref{5.3.4}, we obtain
	\begin{align}\label{5.3.8}
	\begin{split}
	\int_{B_N} |g_N - (g_N)_{B_N}|^{\alpha p_2} \, {\rm d}\mu & \geq
	2^{- \alpha p_2} \int_{B_N} |f_N|^{p_2} \, {\rm d}\mu - 2^{\alpha p_2} (M+1)^{\alpha p_2} |B_N| \\
	& \geq \Big( 2^{- \alpha p_2} N - 2^{\alpha p_2} (M+1)^{\alpha p_2} \Big) |B_N| \geq M |B_N|.
	\end{split}
	\end{align} 
	On the other hand, if $(g_N)_{B_N} < -M-1$, then
	\begin{displaymath}
	\int_{B_N} (f_N - g_N) \, {\rm d}\mu > (M+1) |B_N|.
	\end{displaymath}
	Let $U_N \coloneqq \{x \in B_N : g_N(x) \geq 1\}$. Observe that we have $f_N(y) - g_N(y) \leq 1$ for any $y \in B_N \setminus U_N$ (the definition of $g_N$ is involved here) and hence
	\begin{equation}\label{5.3.9}
	\int_{U_N} (f_N - g_N) \, {\rm d}\mu > M |B_N|.
	\end{equation}
	Therefore, by using the definition of $U_N$, the fact that $(g_N)_{B_N} < 0$, and \eqref{5.3.9} we obtain
	\begin{align}\label{5.3.10}
	\begin{split}
	\int_{B_N} |g_N - (g_N)_{B_N}|^{\alpha p_2} \, {\rm d}\mu \geq \int_{U_N} g_N^{\alpha p_2} \, {\rm d}\mu = \int_{U_N} f_N^{p_2} \, {\rm d}\mu & \geq \int_{U_N} (f_N - g_N) \, {\rm d}\mu > M |B_N|.
	\end{split}
	\end{align} 
	Finally, \eqref{5.3.5} follows from \eqref{5.3.8} and \eqref{5.3.10}. Thus, the proof is complete.
\end{proof}

\section{Test spaces}\label{S5.4}

In this section we present a simple method of constructing metric measure spaces $\XX = (X, \rho, | \, \cdot \, |)$ with specific properties of the associated spaces ${\rm BMO}^p(\XX)$ with $p \in [1, \infty)$. Here $| \, \cdot \, |$ is counting measure on $X$, and this is the only measure that will be considered in Sections~\ref{S5.4}~and~\ref{S5.5}. 

Throughout this chapter the term \textit{test space} will be used for each space $\XX$ built in the following way. Let $M = \{m_{n,i}  : i \in [n] ,\ n \in \mathbb{N}\}$ be a fixed triangular matrix of positive integers with $m_{1,1}=1$. Define 
\begin{displaymath}
X \coloneqq X_M \coloneqq \Big\{ x_{n,i,j} : j \in \{0\} \cup [m_{n,i}], \ i \in [n], \ n \in \mathbb{N} \Big\} ,
\end{displaymath}
where all elements $x_{n,i,j}$ are different. By $S_{n,i}$ we denote the branch 
\begin{displaymath}
S_{n,i} \coloneqq \big\{ x_{n,i,0}, x_{n,i,1}, \dots , x_{n,i,m_{n,i}} \big\}. 
\end{displaymath}
Later on we also use auxiliary symbols $S_n \coloneqq \cup_{i=1}^{n} S_{n,i}$ and $T_n \coloneqq \cup_{k=1}^{n} S_k$, and the function $\vee \colon X \times X \rightarrow \mathbb{N}$ defined by $\vee(x,y) \coloneqq \min \{n \in \mathbb{N} : \{x,y\} \subset T_n \}$ (in other words, if $x \in S_{n_1}$ and $y \in S_{n_2}$, then $\vee(x,y) = \max\{ n_1, n_2\}$). We introduce $\rho$ determining the distance between two different elements $x, y \in X$ by the formula
\begin{displaymath}
\rho(x,y) \coloneqq \left\{ \begin{array}{rl}
n+\frac{1}{2} & \textrm{if } \{x,y\} = \{x_{n,n,0}, x_{n+1,1,0}\} \textrm{ for some } n \in \mathbb{N},  \\

n-\frac{1}{2i+1} & \textrm{if } x_{n,i,0} \in \{x,y\} \subset S_{n,i} \textrm{ for some } 1 \leq i \leq n, \, n \in \mathbb{N},  \\

n-\frac{1}{2i+2} & \textrm{if } \{x,y\} = \{x_{n,i,0}, x_{n,i+1,0}\} \textrm{ for some } 1 \leq i \leq n-1, \, n \in \mathbb{N},  \\

\vee(x,y) & \textrm{otherwise. }  \end{array} \right. 
\end{displaymath}
At first glance, such a metric may look a bit strange. However, its main advantage lies in the arrangement of balls containing exactly two points which we call \textit{pair of neighbors} later on. Moreover, any ball that cannot be covered by at least one of the sets
\[
\mathcal{N}_x \coloneqq \{x\} \cup \{y \in X : y \text{ is a neighbor of } x\}, \qquad x \in X,
\] 
must be of the form $T_n$ or $T_n \cup \{x_{n+1, 1, 0}\}$ for some $n \geq 2$. These two properties make the associated ${\rm BMO}^p(\XX)$ spaces easy to deal with. Figure~\ref{F5.1} shows a model of the space $(X, \rho)$ (neighboring points are connected by a solid line).

	\begin{figure}[H]
		\begin{center}
	\begin{tikzpicture}
	[scale=.8,auto=left]
	\node[style={circle,fill,inner sep=2pt}, label={[yshift=-1cm]$x_{1,1,0}$}] (k0) at (1.5,1) {};
	\node[style={circle,fill,inner sep=2pt}, label={$x_{1,1,1}$}] (k1) at (0.75,2.25)  {};
	\node[style={circle,fill,inner sep=2pt}, label={[yshift=-0.05cm]$x_{1,1,m_{1,1}}$}] (k2) at (2.25,2.25)  {};
	\node[dots] (k4) at (1.5,2.25)  {...};
	
	\node[style={circle,fill,inner sep=2pt}, label={[yshift=-1cm]$x_{2,1,0}$}] (l0) at (4,1) {};
	\node[style={circle,fill,inner sep=2pt}, label=$x_{2,1,1}$] (l1) at (3,3.25)  {};
	\node[style={circle,fill,inner sep=2pt}, label={[yshift=-0.05cm]$x_{2,1,m_{2,1}}$}] (l2) at (5,3.25)  {};
	\node[dots] (l4) at (4,3.25)  {...};
	
	\node[style={circle,fill,inner sep=2pt}, label={[yshift=-1cm]$x_{2,2,0}$}] (m0) at (8,1) {};
	\node[style={circle,fill,inner sep=2pt}, label=$x_{2,2,1}$] (m1) at (6.5,4.5)  {};
	\node[style={circle,fill,inner sep=2pt}, label={[yshift=-0.05cm]$x_{2,2,m_{2,2}}$}] (m2) at (9.5,4.5)  {};
	\node[dots] (m4) at (8,4.5)  {...};
	
	\node[style={circle,fill,inner sep=2pt}, label={[yshift=-1cm]$x_{3,1,0}$}] (n0) at (13,1) {};
	\node[style={circle,fill,inner sep=2pt}, label=$x_{3,1,1}$] (n1) at (11,6)  {};
	\node[style={circle,fill,inner sep=2pt}, label={[yshift=-0.05cm]$x_{3,1,m_{3,1}}$}] (n2) at (15,6)  {};
	\node[dots] (n4) at (13,6)  {...};
	
	\node[dots] (o) at (17,1)  {...};
	
	\draw (k0) -- (l0) node [midway, fill=white, above=-22.5pt] {$\frac{3}{2}$};
	\draw (l0) -- (m0) node [midway, fill=white, above=-22.5pt] {$\frac{7}{4}$};
	\draw (m0) -- (n0) node [midway, fill=white, above=-22.5pt] {$\frac{5}{2}$};
	
	\draw (k0) -- (k2) node [midway, fill=white, left=10pt, above=-10pt] {$\frac{2}{3}$};
	\draw (l0) -- (l2) node [midway, fill=white, left=12.5pt, above=-10pt] {$\frac{5}{3}$};
	\draw (m0) -- (m1) node [midway, fill=white, left=10pt, above=-10pt] {$\frac{9}{5}$};
	\draw (m0) -- (m2) node [midway, fill=white, left=-10pt, above=-10pt] {$\frac{9}{5}$};
	\draw (n0) -- (n1) node [midway, fill=white, left=12.5pt, above=-10pt] {$\frac{8}{3}$};
	\draw (n0) -- (n2) node [midway, fill=white, left=-12.5pt, above=-10pt] {$\frac{8}{3}$};
	
	\foreach \from/\to in {k0/k1, k0/k2, l0/l1, l0/l2, n0/n1, n0/n2, m0/m1, m0/m2, k0/l0, l0/m0, m0/n0, n0/o}
	\draw (\from) -- (\to);
	
	\end{tikzpicture}
	\caption{The test space $(X, \rho)$.}
	\label{F5.1}
	\end{center}
\end{figure}
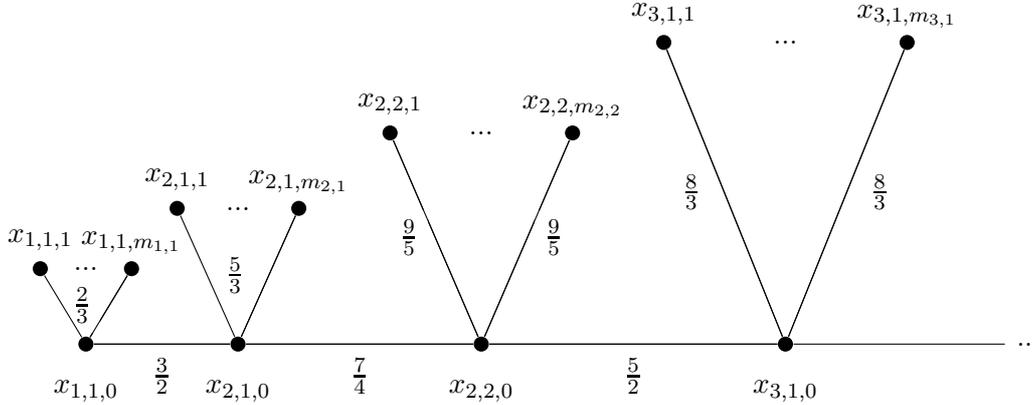

Fix $p_0 \in (1,\infty)$. Our intention is to choose the matrix $M$ in such a way as to obtain that ${\rm BMO}^p(\XX) = {\rm BMO}(\XX)$ if and only if $p \in [1, p_0)$. We construct $M$ inductively. Namely, for each $n \in \NN \setminus \{1\}$, supposing that the values $m_{k,i}$ with $i \in [k]$ and $k \in [n-1]$ have already been chosen, we take
\begin{equation} \tag{C1} \label{b1}
m_{n, i} = \Big\lfloor \frac{b_{n} }{(n - i + 1)^{p_0}} - \frac{b_{n} }{(n - i + 2)^{p_0}} \Big\rfloor 
\end{equation} 
for $i \in [n]$, where $b_n$ is an even positive integer so large that
\begin{equation} \tag{C2} \label{b2}
|T_{n-1}| \leq \min \Big\{ \Big\lfloor \frac{b_{n} }{(n + 1)^{p_0}} - \frac{b_{n} }{(n + 2)^{p_0}} \Big\rfloor, \frac{b_{n}}{n^{2p_0}}\Big\}. 
\end{equation}
We need a few auxiliary estimates. First, observe that from \eqref{b1}, \eqref{b2}, and the fact that $b_n$ is even it follows that $b_{n} / 2 \leq |T_{n}| \leq 2  b_{n}$. Moreover, for each $i \in [n]$ we have
\begin{equation}\label{b3}
\frac{|S_{n, i}|}{|T_{n}|} \leq \frac{4m_{n,i}}{b_{n}} \leq 4 \, \Big( \frac{1}{(n - i + 1)^{p_0}} - \frac{1}{(n - i + 2)^{p_0}} \Big) 
\end{equation}
and
\begin{equation}\label{b4}
\frac{|S_{n, i}|}{|T_{n}|} \geq \frac{m_{n,i}}{2b_{n}} \geq \frac{1}{4} \, \Big( \frac{1}{(n - i + 1)^{p_0}} - \frac{1}{(n - i + 2)^{p_0}} \Big). 
\end{equation}
We are now ready to prove Proposition~\ref{P5.2.2}.

\begin{proof}[Proof of Proposition~\ref{P5.2.2}]
	For fixed $p_0 \in (1, \infty)$ let $\XX = (X, \rho, | \, \cdot \, |)$ be the test space with $M$ defined with the aid of \eqref{b1} and \eqref{b2}. 
	
	First we show that for each $p \in (1,p_0)$ there exists $C_p \in (0,\infty)$ such that $\|f\|_{\ast, p} \leq C_p \|f\|_\ast$ holds for every $f \in {\rm BMO}(\XX)$. Without any loss of generality we can assume that $\|f\|_\ast = 1$. Observe that then we have $|f(x) - f(y)| \leq 2$ whenever $x$ and $y$ are neighbors. Hence, for each $B \subset X$ at least one of the following two possibilities holds:
	\begin{enumerate}[label=(\alph*)]
		\item \label{5a} We have $B \subset \mathcal{N}_x$ for some $x \in X$ and thus 
		$
		\max \{ |f(y) - f(z)| : y, z \in B\} \leq 4.
		$
		\item \label{5b} We have $B=T_n$ or $B = T_n \cup \{x_{n+1,1,0}\}$ for some $n \in \NN \setminus \{1\}$.	
	\end{enumerate}
	If \ref{5a} holds, then we obtain the trivial bound
	\begin{equation}\label{b5}
	\frac{1}{|B|} \sum_{x \in B} |f(x) - f_B|^p \leq 4^p.
	\end{equation}
	To analyze the case \ref{5b}, fix $n \in \NN \setminus \{1\}$ and assume that $B = T_{n}$ or $B = T_{n}\cup \{x_{n+1,1,0}\}$. For each $l \in \mathbb{N}$ set $E'_{l} \coloneqq \{x \in B : |f(x) - f(x_{n, n, 0)}| > l\}$ . In each of the two cases, $\{x_{n+1,1,0}\} \in B$ or $\{x_{n+1,1,0}\} \notin B$, by using \eqref{b2} and \eqref{b3} we obtain the following estimates: for $l \in [n-1]$,
	\begin{equation}\label{b6}
	\frac{|E'_{2l}|}{|B|} \leq \frac{| T_{n-1} \cup \bigcup_{i=1}^{n-l} S_{n,i} |}{|T_{n}|}  \leq \frac{4}{(l+1)^{p_0}}, 
	\end{equation}
	for $l \in [n^2] \setminus [n-1]$,
	\begin{equation}\label{b7}
	\frac{|E'_{2l}|}{|B|} \leq \frac{| T_{n-1}|}{|T_{n}|}  \leq \frac{2}{n^{2p_0}},
	\end{equation}
	and, finally, for $l \in \NN \setminus [n^2]$,
	\begin{equation}\label{b8}
	|E'_{2l}| = 0.
	\end{equation}
	Moreover, recall the basic fact that for any $a \in \mathbb{C}$ we have
	\begin{equation}\label{b9}
	\sum_{x \in B} |f(x) - f_B|^p \leq 2^p \sum_{x \in B} |f(x) - a|^p.
	\end{equation}
	Thus, applying first (\ref{b9}) with $a = f(x_{n,n,0})$, and then (\ref{b6}), (\ref{b7}), and (\ref{b8}), we obtain
	\begin{align}\label{b10}
	\begin{split}
	\frac{1}{|B|} \sum_{x \in B} |f(x) - f_B|^p & \leq \frac{2^p}{|B|} \sum_{x \in B} |f(x) - f(x_{n,n,0})|^p \\ &
	= \frac{2^p}{|B|} \int_0^\infty p \, \lambda^{p-1} \big| \big\{x \in B : |f(x) - f(x_{n,n,0})| > \lambda \big\} \big| \, {\rm d} \lambda \\ &
	\leq \frac{p \, 2^{p+1} }{|B|} \sum_{l=0}^\infty (2l+2)^{p-1} |E'_{2l}| \\ &
	\leq p \, 4^p \Big( 1 + \sum_{l=1}^{n-1} \frac{4 \cdot (l+1)^{p-1}}{(l+1)^{p_0}} + n^2 \cdot \frac{2(2n^2)^{p-1}}{n^{2p_0}} \Big) \\& 
	\leq  p \, 4^p \big( 1 + 4 \sum_{l=1}^\infty l^{p-p_0-1} + 2^p \big).
	\end{split}
	\end{align}
	Combining (\ref{b5}) and (\ref{b10}) shows that for each $f$ satisfying $\|f\|_\ast = 1$ and $B \subset X$ we have
	$
	\sup_{B \subset X} \frac{1}{|B|} \sum_{x \in B} |f(x) - f_B|^p \leq C_p^p
	$
	with $C_p$ independent of $B$ and $f$.
	
	Now we prove that there exists $g \in {\rm BMO}(\XX) \setminus {\rm BMO}^{p_0}(\XX)$. We begin with the following simple remark. Given $f$ such that $|f(x) - f(y)| \leq 2$ for any neighboring points $x,y$, and $B$ of the form $T_{n}$ or $T_{n} \cup \{x_{n+1, 1, 0}\}$ for $n \in \NN \setminus \{1\}$, the average value of $f$ over $B$ does not differ too much from $f(x_{n, n, 0})$. Indeed, by using (\ref{b2}), the estimate $|B| \geq b_{n}/2$, and \eqref{b3}, we obtain
	\begin{align}\label{b11}
	\begin{split}
	|f_B - f(x_{n, n, 0})| & \leq 2 + \frac{2}{|B|} \sum_{l=1}^{\infty} \big| \big\{x \in B \colon |f(x) - f(x_{n, n, 0})| > 2l \big\} \big| \\ &
	\leq 2 + \frac{2}{|T_{n}|} \Big( \sum_{l=1}^{n-1} \big|T_{n-1} \cup \bigcup_{i=1}^{n-l} S_{n,i}\big| + (n-1)^2 |T_{n-1}| \Big) \\ &
	\leq 2 + 2\sum_{l=1}^{n-1} \frac{|\bigcup_{i=1}^{n-l} S_{n,i}|}{|T_{n}|} + 2n^2 \frac{|T_{n-1}|}{|T_{n}|} \\&
	\leq 6 + 2\sum_{l=1}^{n-1} \frac{4}{(n-l)^{p_0}} \leq 6 + 8 \sum_{l=1}^{\infty} l^{-p_0} \leq N
	\end{split}
	\end{align}
	for some fixed integer $N = N(p_0)$. 
	Let us now take $g$ defined by the formula
	\begin{displaymath}
	g(x_{n,i,j}) \coloneqq i + \sum_{k=1}^{n-1} k , \qquad j \in \{0\} \cup [m_{n,i}], \ i \in [n], \ n \in \mathbb{N}.
	\end{displaymath}
	It is easy to check that $g \in {\rm BMO}(\XX)$ since for each $B \subset X$ at least one of the estimates \eqref{b5} and \eqref{b10} holds with $p$ and $f$ replaced by $1$ and $g$, respectively. Indeed, to obtain these inequalities for $f$ before we used only the information that $|f(x) - f(y)| \leq 2$ for any neighboring points $x$ and $y$. Our function $g$ satisfies this condition as well. In addition, \eqref{b11} remains true if we put $g$ in place of $f$. Let $n \in \NN \setminus \{1\}$ and take $B = T_{n}$. Observe that
	\begin{equation}\label{b12}
	|g(x) - g_B| \geq n - i - N, \qquad x \in S_{n, i}, \ i \in [n].
	\end{equation}
	Therefore, if $n \geq 4N$, then by using (\ref{b12}) and (\ref{b4}) we obtain
	\begin{align}\label{b13}
	\begin{split}
	\frac{1}{|B|} \sum_{x \in B} |g(x) - g_B|^{p_0} & \geq \frac{1}{|B|} \sum_{l=1}^{\infty} p_0  (l-1)^{p_0-1} \, \big| \big\{x \in B : |g(x) - g_B| > l \big\} \big| \\ & \geq \frac{1}{|T_{n}|} \sum_{l=2}^{n-N-1} p_0 (l-1)^{p_0-1} \, \Big|\bigcup_{i=1}^{n-N-l} S_{n,i} \Big| \\ 
	& \geq \frac{p_0}{4}  \sum_{l=2}^{n-N-1} (l-1)^{p_0-1} \Big( \frac{1}{(N+l-1)^{p_0}} - \frac{1}{(n+1)^{p_0}} \Big) \\ &
	\geq \frac{p_0}{8}  \sum_{l=2}^{\lfloor n/2+3/2-N \rfloor} \frac{(l-1)^{p_0-1}}{(N+l-1)^{p_0}} \\ &
	\geq \frac{p_0}{2^{p_0+3}}  \sum_{l=N+1}^{\lfloor n/2+3/2-N \rfloor} (l-1)^{-1},
	\end{split}
	\end{align}
	since $(N+l-1)^{-p_0} \geq 2 (n+1)^{-p_0}$ for $l \leq \lfloor n/2+3/2-N \rfloor$ and $N+l-1 \leq 2(l-1)$ for $l \geq N+1$. Letting $n \rightarrow \infty$ we conclude that $g \notin {\rm BMO}^{p_0}(\XX)$.
\end{proof}

At the end of this section we will be interested in test spaces $\XX$ for which ${\rm BMO}^p(\XX)$ coincides with ${\rm BMO}(\XX)$ if and only if $p \in [1, p_0]$, where $p_0 \in [1, \infty)$ is fixed. We can easily get such spaces slightly modifying the previous construction of $M$. Namely, instead of using \eqref{b1} and \eqref{b2}, we define $m_{n, i}$ for $n \in \NN \setminus \{1\}$ and $i \in [n]$ by
\begin{equation}
m_{n, i} = \Big\lfloor \frac{1}{\log(n)+1} \ \Big( \frac{b_{n} }{(n - i + 1)^{p_0}} - \frac{b_{n} }{(n - i + 2)^{p_0}} \Big) \Big\rfloor, \tag{C1'}\label{b1'}
\end{equation}
where $b_{n}$ is an even integer so large that
\begin{equation} \tag{C2'} \label{b2'}
|T_{n-1}| \leq \min \Big( \Big\lfloor \frac{1}{\log(n)+1} \ \Big( \frac{b_{n} }{(n + 1)^{p_0}} - \frac{b_{n} }{(n + 2)^{p_0}} \Big) \Big\rfloor, \frac{b_{n}}{n^{2p_0}}\Big). 
\end{equation}
Now we present a sketch of the proof of Proposition~\ref{P5.2.3}.

\begin{proof}[Proof of Proposition~\ref{P5.2.3}]
	For fixed $p_0 \in [1, \infty)$ let $\XX = (X, \rho, | \, \cdot \, |)$ be the test space with $M$ defined with the aid of \eqref{b1'} and \eqref{b2'}. 
	We show that for each $p \in [1,p_0]$ there exists $C_p \in (0,\infty)$ such that $\|f\|_{\ast, p} \leq C_p \|f\|_\ast$ holds for every $f \in {\rm BMO}(\XX)$. To this end, observe that
	\begin{displaymath}
	p \, 4^p \Big( 1 + \frac{1}{\log(n)+1} \sum_{l=1}^{n-1} \frac{4 \cdot (l+1)^{p-1}}{(l+1)^{p_0}} + n^2 \cdot \frac{2(2n^2)^{p-1}}{n^{2p_0}} \Big),
	\end{displaymath}
	is bounded uniformly in $n$ if $p \in [1,p_0]$. This allows us to get a proper variant of the estimate $(\ref{b10})$ for each such $p$. 
	
	Now we prove that for $g \in {\rm BMO}(\XX)$ defined as in the proof of Proposition~\ref{P5.2.2} we have $g \notin {\rm BMO}^{p}(\XX)$ for all $p \in  (p_0, \infty)$. To see this note that if $p \in  (p_0, \infty)$, then the estimates analogous to \eqref{b11} and \eqref{b13} remain true. Namely, for $B = T_{n}$ one can obtain
	\begin{displaymath}
	|g_B - g(x_{n, n, 0})| \leq N,
	\end{displaymath}
	where $N$ is some positive integer independent of $n$, and
	\begin{displaymath}
	\frac{1}{|B|} \sum_{x \in B} |g(x) - g_B|^{p}
	\geq \frac{p \, (\log(n)+1)^{-1}}{2^{p_0+3}}   \sum_{l=N+1}^{\lfloor n/2+3/2-N \rfloor} (l-1)^{p-p_0-1}.
	\end{displaymath}
	Finally, the right-hand side of the inequality above tends to infinity with $n$. 
\end{proof}

\section{Further constructions and comments}\label{S5.5}

In the last part of this chapter we consider several variants of the construction process described in Section~\ref{S5.4}, in order to obtain test spaces with another interesting properties. 

Our first goal is to show that if the entries of the matrix $M$ grow fast enough, then the John--Nirenberg inequality holds for all $f \in {\rm BMO}(\XX)$. This result may be a little surprising at first, since we know that the John--Nirenberg inequality holds for doubling spaces. Keeping that in mind, one may suppose that $\XX$ should have rather little chance of preserving this property if we force the terms $m_{n,i}$ to grow rapidly. However, observe that in Section~\ref{S5.4} the ratios between the values $m_{n,1}, \dots, m_{n,n}$ played a crucial role in estimating the mean oscillation of the studied functions and the estimates we obtained were stronger if the values $m_{n,i} / m_{n,n}$ for $i \in [n-1]$ were smaller.

To formulate our result in a more readable way it is convenient to identify the matrix $M$ with the sequence $M' = (m'_1, m'_2, \dots)$ formed by writing the entries of $M$ row by row, that is $M' \coloneqq (m_{1,1}, m_{2,1}, m_{2,2}, m_{3,1}, \dots)$. In what follows, for simplicity, we use $M$ based on the geometric sequence $( 2^{k-1})_{k=1}^\infty$. Nevertheless, it will be clear that the presented proof also works for any lacunary sequence $(m'_k)_{k=1}^\infty$, that is, a sequence satisfying $m'_{k+1}/m'_k \geq c$ for all $k \in \mathbb{N}$, where $c$ is some fixed constant strictly greater than $1$.

\begin{proposition}
	Let $\XX = (X, \rho, | \, \cdot \, |)$ be the test space with $M$ identified with the geometric sequence $( 2^{k-1})_{k=1}^\infty$. Then for the space ${\rm BMO}(\XX)$ the John--Nirenberg inequality 
	\begin{equation}\label{b14}
	\frac{|\{x \in B : |f(x) - f_B| > \lambda \}|}{|B|} \leq c_1 \exp(-c_2 \lambda / \|f\|_\ast),
	\end{equation}
	holds with constants $c_1, c_2 \in (0, \infty)$ independent of $f \in {\rm BMO}(\XX)$, $B \subset X$, and $\lambda \in (0, \infty)$.
\end{proposition}

\begin{proof}
	Let $f \in {\rm BMO}(\XX)$ be such that $\|f\|_\ast = 1$. First, observe that the main difficulty in proving \eqref{b14} is related to the situation described in \ref{5b}, where $B$ is of the form $T_{n}$ or $T_{n} \cup \{x_{n+1, 1, 0}\}$. Indeed, if the case \ref{5a} happens instead, then we have $\max \{ |f(x) - f(y)| : x, y \in B\} \leq 4$ and hence \eqref{b14} holds for any $\lambda \in ( 0, \infty)$ if we choose $c_1$ and $c_2$ such that $c_1 \exp(-4c_2) \geq 1$. Therefore, fix $n \in \NN \setminus \{1\}$ and consider $B$ of the form specified in \ref{5b}. Note that $2^{k} \leq |B| \leq 2^{k+1}$, where $k = \frac{n(n+1)}{2}$. Once again we will take advantage of the useful property that $|f(x) - f(y)| \leq 2$ holds for neighboring points $x$ and $y$. Proceeding just like we did before to get \eqref{b11}, we can estimate the value $|f_B - f(x_{n,n,0})|$ by some even integer $N$ which does not depend on $f$, $n$, and the fact whether $B$ includes $\{x_{n+1, 1, 0}\}$ or not. Then for any integer $l \in \NN \setminus [N-1]$ we have
	\begin{align*}
	|\{x \in B : |f(x) - f_B| > 2 l \}| \leq |\{x \in B : |f(x) - f(x_{n,n,0})| > 2(l - N/2) \}| & \leq 2^{k-l+N/2+1} \\
	&\leq 2^{N/2+1} 2^{-l} |B|,
	\end{align*}
	and now it is routine to choose $c_1$ and $c_2$ (independent of significant parameters) such that (\ref{b14}) holds for all $\lambda \in (0,\infty)$ and $B \subset X$ of an arbitrary form.
\end{proof}

For the presentation of the remaining two results we return to the matrix description of the space $\XX$. The following construction of $M$ is very similar to the one described in Section~\ref{S5.4}, where the conditions \eqref{b1} and \eqref{b2} were used, but this time we choose the parameter $p_0$ separately in each step of induction. Namely, let $P = (p_1, p_2, \dots)$ be a sequence of numbers strictly greater than $1$. We define $m_{n, i}$ for $n \in \NN \setminus \{1\}$ and $i \in [n]$ by

\begin{equation} \tag{{\rm C1*}} \label{b1*}
m_{n, i} = \Big\lfloor \frac{b_{n} }{(n - i + 1)^{p_{n}}} - \frac{b_{n} }{(n - i + 2)^{p_{n}}} \Big\rfloor, 
\end{equation}
where $b_{n}$ is an even integer so large that
\begin{equation} \tag{{\rm C2*}} \label{b2*}
|T_{n-1}| \leq \min \Big( \Big\lfloor \frac{b_{n} }{(n + 1)^{p_{n}}} - \frac{b_{n} }{(n + 2)^{p_{n}}} \Big\rfloor, \frac{b_{n}}{n^{2p_{n}}}, \frac{b_{n}}{n^{n}}\Big). 
\end{equation}

Our next purpose will be to show that by a suitable choice of $P$ it is possible to obtain a~space $\XX$ for which the associated spaces ${\rm BMO}^p(\XX)$ are all different. Although this result can also be deduced from Theorem~\ref{T5.1.1}, the advantage of the current approach lies in the fact that the proof presented below, contrary to the proof of Theorem~\ref{T5.1.1}, is constructive. Namely, for each $p, p' \in [1, \infty)$ with $p < p'$ we explicitly construct $f \in {\rm BMO}^{p}(\XX) \setminus {\rm BMO}^{p'}(\XX)$. In what follows we take $P$ formed by writing the elements of some countable dense subset of $(1, \infty)$ in an arbitrary order. We can use the set $\mathbb{Q} \cap (1, \infty)$, for example.

\begin{proposition}
	Let $P = (p_1, p_2, \dots)$ be as above and consider the test space $\XX = (X, \rho, | \, \cdot \, |)$ with $M$ defined by using \eqref{b1*} and \eqref{b2*}. Then for each $p, p' \in [1, \infty)$ with $p < p'$ there exists $g \in {\rm BMO}^{p}(\XX)$ such that $g \notin {\rm BMO}^{p'}(\XX)$.
\end{proposition}

\begin{proof}
	Fix $p,p'$ as above and let $J \coloneqq J(p,p') \coloneqq [\frac{p+p'}{2}, p']$. We take $g$ defined by the formula
	\begin{equation*}
	g(x_{n,i,j}) \coloneqq i \cdot \mathbf{1}_J(p_n) + \sum_{k=2}^{n-1} k \cdot \mathbf{1}_J(p_k), \qquad j \in \{0\} \cup [m_{n,i}], \ i \in [n], \ n \in \mathbb{N}.
	\end{equation*}	
	Note that $g$ is similar to the analogous function considered in the proof of Proposition~\ref{P5.2.2}, but this time it grows only in those $S_n$ for which the corresponding values $p_n$ belong to $J$. It is now a~standard procedure to show that $g \in {\rm BMO}^p(\XX) \setminus {\rm BMO}^{p'}(\XX)$ and most of the work consists of proving the appropriate variants of the estimates \eqref{b10}, \eqref{b11}, and \eqref{b13}.
\end{proof}

We conclude our studies with an example of a test space $\XX$ for which the associated spaces ${\rm BMO}^p(\XX)$ coincide for the full range of the parameter $p$, but the John--Nirenberg inequality does not hold. Namely, we have the following.

\begin{proposition}
	There exists a (test) space $\XX$ with the following properties:
	\begin{enumerate}[label=\rm(\roman*)]
		\item \label{5i} For each $p \in (1, \infty)$ there exists $C_p \in (0, \infty)$ such that the inequality $\|f\|_{\ast, p} \leq C_p \|f\|_\ast$ holds for all $f \in {\rm BMO}(\XX)$. 
		\item \label{5ii} There exists $g \in {\rm BMO}(\XX)$ such that for each $l \in \mathbb{N}$ we can find $B_l \subset X$ and $\lambda_l \in (0,\infty)$ satisfying
		\begin{equation*}
		\frac{|\{x \in B_l : |g(x) - g_{B_l}| > \lambda_l \}|}{|B_l|} > l \exp(- \lambda_l / l).
		\end{equation*}
		In particular, there is no inequality of the form \eqref{b14} satisfied by all $f \in {\rm BMO}(\XX)$.
	\end{enumerate}
\end{proposition}

\begin{proof}
	The space $\XX$ will be built by using $M$ constructed with the aid of \eqref{b1*} and \eqref{b2*} for some suitable sequence $P$ of positive integers. The key idea is to choose $P$ such that $p_n$ tends to infinity very slowly. 
	
	First, notice that the sole assumption $\lim_{n \to \infty} p_n = \infty$ implies \ref{5i}. Indeed, let $f$ be such that $\|f\|_\ast = 1$. Observe that for each $p \in (1, \infty)$ there exists $N(p) \in \NN \setminus \{1\}$ such that $p_{n} \geq p + 1$ for all $n \in \NN \setminus [N(p)]$. Therefore, (\ref{b10}) holds with $p + 1$ instead of $p_0$ for each such $n$ and $B$ of the form $T_{n}$ or $T_{n} \cup \{x_{n+1, 1, 0}\}$. Since there exists a~numerical constant $K \in (0, \infty)$ depending only on $p$ such that for any other choices of $B$ the inequality $\max\{|f(x) - f(y)| : x, y \in B\} \leq K$ holds, we see that \ref{5i} is satisfied. 
	
	It remains to show that with additional assumptions imposed on $P$ also \ref{5ii} holds true. To be more specific, the slow growth of the elements of $P$ will suffice. Suppose for convenience that $p_2 = 2$ and $P$ is nondecreasing. We claim that there exists $N \in \mathbb{N}$ such that the inequality $|f_B - f(x_{n,n,0})| \leq N$ holds for $B = T_{n}$ with $n \in \NN \setminus \{1\}$ and any $f$ such that $\|f\|_\ast=1$. Indeed, we see that now \eqref{b11} holds with $p_0$ replaced by $2$. We are ready to define $P$ inductively. Set $p_2 \coloneqq 2$. Assuming that $p_n = k$ for some $n, k \in \NN \setminus \{1\}$, we define $p_{n+1}$ by the formula
	\begin{equation}\label{b15}
	p_{n+1} \coloneqq \left\{ \begin{array}{rl}
	k & \textrm{if }  \frac{1}{4} \big( n^{-k} - (n+1)^{-k}\big) \leq k  \exp(-(n-N-1)/k), \\
	
	k+1 & \textrm{otherwise. }  \end{array} \right.
	\end{equation}
	Clearly, we have $\lim_{n \to \infty} p_n = \infty$ and the sequence $(p_n)_{n=2}^\infty$ is nondecreasing as planned. 
	
	Finally, let $g$ be as in the proof of Proposition~\ref{P5.2.2}. Of course, we have $g \in {\rm BMO}(\XX)$. Fix $l \in \mathbb{N} \setminus\{1\}$ such that $L \coloneqq L(l) \coloneqq \max\{k \in \NN \setminus \{1\} : p_k = l\}$ is strictly greater than $N+1$. Then 
	\begin{align*}
	\frac{|\{ x \in T_{L} : |g(x)-g_{T_{L}}| \geq L-N-1 \}|}{|T_{L}|} \geq \frac{|\{ x \in T_{L} : |g(x)-g(x_{L, L, 0})| \geq L-1 \}|}{|T_{L}|} 
	\geq \frac{|S_{L,1}|}{|T_{L}|}
	\end{align*}
	and, by using \eqref{b4} and \eqref{b15},
	\[
	\frac{|S_{L,1}|}{|T_{L}|} \geq \frac{1}{4} \big( L^{-l} - (L+1)^{-l}\big) \geq l  \exp(-(L-N-1)/l).
	\]
	Thus, if $l \in \NN$ is large, then the inequality stated in \ref{5ii} holds with $B_l = T_{L}$ and $\lambda_l = L - N - 1$.
\end{proof}
\chapter{Dichotomy property}\label{chap6}
\setstretch{1.05}
A dichotomy regarding the finiteness of the Hardy--Littlewood maximal functions was noticed for the first time by Bennett, DeVore and Sharpley \cite{BDVS} in the context of functions of bounded mean oscillation. Namely, the authors discovered the principle that for each $f \in {\rm BMO}(\mathbb{R}^d)$ the maximal function $\MM f$ (or $\MN f$) either is finite almost everywhere or equals $\infty$ on the whole $\mathbb{R}^d$. Later on, however, it turned out that this property is not directly related to the ${\rm BMO}$ concept. Fiorenza and Krbec \cite{FK} proved that for any $f \in L^1_{\rm loc}(\RR^d)$ the following holds: if $\MM f(x_0) < \infty$ for some $x_0 \in \mathbb{R}^d$, then $\MM f$ is finite almost everywhere. Finally, Aalto and Kinnunen \cite{AK} have shown in a very elegant way that this implication remains true if one replaces the Euclidean space by any doubling space. On the other hand, some negative results in similar contexts also appeared in the literature. For example, C.-C.~Lin, Stempak and Y.-S. Wang \cite{LSW} observed that such a principle does not take place for local maximal operators. In the following chapter we shed more light on the aforementioned issue by examining the occurrences of the dichotomy property for maximal operators associated with metric measure spaces $\XX$ for which the doubling condition fails to hold. We focus on the two most common maximal operators, centered $\MM$ and noncentered $\MN$.

Given a metric measure space $\XX = (X, \rho, \mu)$, we always have $\{ x \in \XX : \MM f(x) = \infty\} \subset \{ x \in \XX : \MN f(x) = \infty\}$ for any $f \in L^1_{\rm loc}(\mu)$. Moreover, if $\XX$ is doubling, then the reverse inclusion follows as well. Thus, the sentences ``$\MM$ possesses the dichotomy property'' and ``$\MN$ possesses the dichotomy property'' are equivalent as long as the doubling condition is satisfied. In nondoubling setting the situation is different. First of all, we have no assurance that the dichotomy property for $\MM$ or $\MN$ still occurs. Moreover, since $\MM f$ and $\MN f$ may be incomparable, there are no obvious indications that the existence or absence of the dichotomy property for one operator implies its existence or absence for another one. Therefore, a natural problems arise: 
\[
\emph{Can each of the four possibilities actually take place for some (nondoubling) space?}
\]
The aim of this chapter is to answer this question in the affirmative.

In Section~\ref{S6.1} we formulate the main problem. In Sections~\ref{S6.2}~and~\ref{S6.3} the appropriate examples of nondoubling spaces are provided. Finally, in Section~\ref{S6.4} we characterize (in terms of $\mu$) all situations in which $\MM$ possesses the dichotomy property in the case $\XX = (\RR^d, \rho, \mu)$, $d \in \NN$, where $\rho$ is the metric induced by any norm on $\mathbb{R}^d$. Throughout this chapter we assume that $\mu$ is such that $|B| \in (0,\infty)$ holds for each open ball $B$ determined by $\rho$.

\setstretch{1.1}
\section{Preliminaries and results}\label{S6.1}

Consider a metric measure space $\XX = (X, \rho, \mu)$. Throughout this chapter $B_r(x)$ stands for the open ball centered at $x \in X$ with radius $r \in (0,\infty)$. By $L^1_{\rm loc}(\mu)$ we denote the space consisting of functions $f$ which are integrable (with respect to $\mu$) on every ball $B \subset X$. Note that this definition is slightly different from the standard one, where the integrability of $f$ on every compact subset of $X$ is assumed. Nevertheless, if $\XX$ is such that every ball has compact closure, then the two definitions are equivalent (this is the case for the standard Euclidean space, for example).  

We say that the associated noncentered maximal operator $\MN$ possesses the {\it dichotomy property} if for any $f \in L^1_{\rm loc}(\mu)$ exactly one of the following cases holds: either $|E_\infty(f)| = 0$ or $E_\infty(f) = X$, where $E_\infty(f) \coloneqq \{x \in X : \MN f(x) = \infty\}$. Similarly, the associated centered maximal operator $\MM$ possesses the dichotomy property if for any $f \in L^1_{\rm loc}(\mu)$ we have either $|E^{\rm c}_\infty(f)| = 0$ or $E^{\rm c}_\infty(f) = X$, where $E^{\rm c}_\infty(f) \coloneqq \{x \in X : \MM f(x) = \infty\}$. Notice that, equivalently, the dichotomy property can be formulated in the following way: if $\MN f(x_0) < \infty$ (respectively, $\MM f(x_0) < \infty$) for some $f \in L^1_{\rm loc}(\mu)$ and $x_0 \in X$, then $\MN f$ (respectively, $\MM f$) is finite $\mu$-almost everywhere.

The following theorem is the main result of this chapter.

\begin{theorem}\label{T6.1.1}
	For each of the four possibilities regarding whether $\MN$ and $\MM$ possess the dichotomy property or not, there exists a nondoubling metric measure space for which the associated maximal operators behave just the way we demand.
\end{theorem}

\begin{proof}
	Examples~\hyperref[6AA]{A},~\hyperref[6BB]{B},~\hyperref[6CC]{C},~and~\hyperref[6DD]{D} in Sections~\ref{S6.2}~and~\ref{S6.3} together constitute the proof of this theorem, illustrating all the desired situations.
\end{proof}

It is worth noting at this point that, in addition to indicating appropriate examples, our goal is also to ensure that they are constructed as simply as possible. Thus, in all examples presented later on $X$ is either $\mathbb{R}^d$ or $\mathbb{Z}^d$, while $\rho$ is the standard Euclidean metric $d_{\rm e}$ or the supremum metric $d_\infty$. Finally, in the discrete setting $\mu$ is defined by letting $\mu(\{x\}) \in (0, \infty)$ for each point $x \in X$, while in the continuous situation $\mu$ is determined by a suitable strictly positive and locally integrable weight $w$.

For the convenience of the reader, the results obtained in Examples~\hyperref[6AA]{A},~\hyperref[6BB]{B},~\hyperref[6CC]{C},~and~\hyperref[6DD]{D} have been summarized in Table~\ref{T6.1} below.

\begin {table}[H]
\vspace*{0.3cm}
\caption {Occurrences of the dichotomy property (DP) for $\MN$ and $\MM$ associated with spaces described in Examples~\hyperref[6AA]{A},~\hyperref[6BB]{B},~\hyperref[6CC]{C},~and~\hyperref[6DD]{D}.} \label{T}

\begin{center}
	\begin{tabular}{ | c | c | c | c | c | c |}
		
		\hline
		& $X$ & $\rho$ & $\mu$ & DP for $\MN$ & DP for $\MM$ \\ \hline
		
		Ex.$\rm~\hyperref[6AA]{A}$ & $\mathbb{R}$ & $d_{\rm e}$ & ${\rm d}\mu(x) = {\rm exp}(x^2) {\rm d}x$ & \cmark &  \xmark \\ \hline
		
		Ex.$\rm~\hyperref[6BB]{B}$ & $\mathbb{R}$ & $d_{\rm e}$ & ${\rm d}\mu(x) = {\rm exp}(-x^2) {\rm d}x$ & \cmark & \cmark \\ \hline
		
		Ex.$\rm~\hyperref[6CC]{C}$ & $\mathbb{Z}^2$ & $d_\infty$ & 
		
		$\mu(\{(n,m)\}) = \left\{ \begin{array}{rl}
		4^{|m|} & \textrm{if } n = 0,  \\
		1 & \textrm{otherwise. }  \end{array} \right.$ & 
		
		\xmark &  \cmark \\ \hline
		
		Ex.$\rm~\hyperref[6DD]{D}$ & $\mathbb{Z}^2$ & $d_\infty$ & 
		
		$\mu(\{(n,m)\}) = \left\{ \begin{array}{rl}
		4^{|m|} & \textrm{if } n = 0,  \\
		2^{n^2} & \textrm{if } n < 0 \textrm{ and } m = 0,  \\
		1 & \textrm{otherwise. }  \end{array} \right. $ & 
		
		\xmark & \xmark \\ \hline
	\end{tabular}
\label{T6.1}
\end{center}
\end{table}

One more comment is in order. While the doubling condition for measures is often assumed in the literature to provide that most of the classical theory works, some statements can be verified under the less strict condition that the space is geometrically doubling or satisfies both geometric doubling and upper doubling properties (see \cite{H} for the details). In our case, although the metric measure spaces appearing in Table~\ref{T6.1} are nondoubling, the corresponding metric spaces are geometrically doubling. This means that the general result for the class of doubling spaces, concerning the existence of the dichotomy property for maximal operators, cannot be repeated in the context of geometrically doubling spaces. Finally, Example~\ref{6EE} in Section~\ref{S6.4} illustrates the situation where the space is geometrically doubling and upper doubling at the same time, while the associated operator $\MN$ does not possess the dichotomy property. 

\section{Real line case}\label{S6.2}

In this section we study the dichotomy property for $\MN$ and $\MM$ associated with the one-dimensional space $(\mathbb{R}, d_{\rm e}, \mu)$ with arbitrary Borel measure $\mu$. We consider spaces from this class separately, since they share certain specific properties, mainly due to their linear order (for example, in this context $\MN$ always satisfies the weak type $(1,1)$ inequality with constant $2$). Our first task is to prove the following result. 

\begin{proposition}\label{P6.2.1}
	Let $\XX = (\mathbb{R}, d_{\rm e}, \mu)$ with $\mu$ such that $|B|  \in (0, \infty)$ for each $B \subset \RR$. Then $\MN$ possesses the dichotomy property.
\end{proposition}
\noindent The proof of Proposition~\ref{P6.2.1} is preceded by some additional considerations.

Let $r(B)$ be the radius of a given ball $B$. For $f \in L^1_{\rm loc}(\mu)$ we denote
\begin{displaymath}
L_f \coloneqq L_f(\mu) \coloneqq \Big\{ x \in \mathbb{R} : \lim_{r \rightarrow 0} \sup_{B \ni x : r(B)=r} \frac{1}{|B|} \int_{B} | f(y) - f(x) | \, {\rm d}\mu(y) = 0 \Big\} 
\end{displaymath}
and
\begin{displaymath}
L^{\rm c}_f \coloneqq L^{\rm c}_f(\mu) \coloneqq \Big\{ x \in \mathbb{R} : \lim_{r \rightarrow 0} \frac{1}{|B_{r}(x)|} \int_{B_{r}(x)} | f(y) - f(x) | \, {\rm d}\mu(y) = 0 \Big\}. 
\end{displaymath}
Notice that there is a small nuisance here, because $f$ is actually an equivalence class of functions, while $L_f $ and $L^{\rm c}_f$ clearly depend on the choice of its representative. Nevertheless, for any two representatives $f_1$ and $f_2$ of a fixed equivalence class we have $|L_{f_1} \triangle L_{f_2}|=0$ and $|L^{\rm c}_{f_1} \triangle L^{\rm c}_{f_2}|=0$ (where $\triangle$ denotes the symmetric difference) and this circumstance is sufficient for our purposes. 

The conclusion of the following lemma is a simple modification of the well known fact about the set of Lebesgue points of a given function. Although the proof is rather standard, we present it for the sake of completeness (cf.~\cite[Theorem~3.20]{Fo}). 

\begin{lemma}
	Let $\XX = (\mathbb{R}, d_{\rm e}, \mu)$ with $\mu$ such that $|B|  \in (0, \infty)$ for each $B \subset \RR$. If $f \in L^1_{\rm loc}(\mu)$, then $|\mathbb{R} \setminus L_f| = 0$. 
\end{lemma}

\begin{proof}
	Let us introduce the sets $L_{f,N}$ with $N \in \mathbb{N}$ by
	\begin{displaymath}
	L_{f,N} \coloneqq \Big\{ x \in \mathbb{R} : \limsup_{r \rightarrow 0} \sup_{B \ni x : r(B)=r} \frac{1}{|B|} \int_{B} | f(y) - f(x) | \, {\rm d}\mu(y) \leq \frac{1}{N} \Big\}.
	\end{displaymath}
	Note that $L_f = \bigcap_{N = 1}^{\infty} L_{f,N}$. Therefore, it suffices to prove that for each $N \in \mathbb{N}$ there exists a~Borel set $A_N$ such that $(-N, N) \setminus L_{f,N} \subset A_N$ and $|A_N| \leq 1/N$. 
	
	Fix $N$ and consider $f_N \coloneqq f \cdot \mathbf{1}_{(-N-1, N+1)}$. Thus, $f_N \in L^1(\mu)$ and $L_{f_N,N}$ coincides with $L_{f,N}$ on $(-N, N)$. Take a continuous function $g_N$ satisfying $\|f_N - g_N \|_{L^1(\mu)} \leq 1 / (9N^2)$ (notice that continuous functions are dense in $L^1(\mu)$ by \cite[Proposition 7.9]{Fo}) and define two auxiliary sets
	\begin{displaymath}
	E_N^1 \coloneqq \big\{x \in \mathbb{R} : |(f_N - g_N)(x)| > 1 / (3N) \big\}, \qquad E_N^2 \coloneqq \big\{x \in \mathbb{R} : \MN (f_N - g_N)(x) > 1 / (3N) \big\}.
	\end{displaymath} 
	Observe that $|E^1_N| \leq 1 / (3N)$ and $|E^2_N| \leq 2 / (3N)$. Now we fix $x_0 \in (-N, N) \setminus (E_N^1 \cup E_N^2)$ and take $\epsilon \in (0,1)$ such that $|g_N(y) - g_N(x_0)| \leq 1 / (3N)$ whenever $|y - x_0| < \epsilon$. If $B$ contains $x_0$ and satisfies $r(B)< \epsilon / 2$, then by using the estimate
	\begin{displaymath}
	| f(y) - f(x_0) | \leq | f_N(y) - g_N(y) | + |g_N(y) - g_N(x_0)| + |(g_N(x_0) - f_N(x_0)|,
	\end{displaymath}
	which is valid for all $y \in B$, we obtain 
	\begin{displaymath}
	\frac{1}{|B|} \int_{B} | f(y) - f(x_0) | \, {\rm d}\mu(y)
	\leq \MN (f_N - g_N)(x_0) + \frac{1}{3N} + |f_N(x_0) - g_N(x_0)| \leq \frac{1}{N},
	\end{displaymath}
	and therefore $A_N = E_N^1 \cup E_N^2$ satisfies the desired conditions. 
\end{proof}

\begin{remark}
	The definitions of $L_f$ and $L_f^{\rm c}$ can also be adapted to the case of arbitrary space $(X, \rho, \mu)$. Then we have $|X \setminus L_f| = 0$ (respectively, $|X \setminus L_f^{\rm c}| = 0$) for all $f \in L^1_{\rm loc}(\mu)$ if only the associated maximal operator $\MN$ (respectively, $\MM$) is of weak type $(1,1)$ and continuous functions are dense in $L^1(\mu)$. This is the case, for example, when dealing with $L_f^{\rm c}$ and $(\mathbb{R}^d, \rho, \mu)$ with $d \in \NN$, where $\rho$ is induced by a fixed norm (in particular, $\rho = d_{\rm e}$ and $\rho = d_\infty$ are included) and $\mu$ is a~Borel measure. We explain some details more precisely in Section~\ref{S6.4}.
\end{remark}

We are now ready to prove Proposition~\ref{P6.2.1}.

\begin{proof}[Proof of Proposition~\ref{P6.2.1}]
	Assume that $|E_\infty(f)| > 0$. Then we can take $x \in L_f$ such that $\MN f(x) = \infty$. For each $n \in \mathbb{N}$ there exist a ball $B_n$ containing $x$ and satisfying
	\begin{displaymath}
	\frac{1}{|B_n|} \int_{B_n} |f(y)| \, {\rm d}\mu(y) > n.
	\end{displaymath}
	Fix $\epsilon \in (0, \infty)$ such that
	\begin{displaymath}
	\frac{1}{|B|} \int_{B} | f(y) - f(x) | \, {\rm d}\mu(y) < 1
	\end{displaymath}
	holds whenever $r(B) \leq \epsilon$, and denote $\delta \coloneqq \min\{ \mu((x-\epsilon/2, x]), \mu([x, x + \epsilon/2))\} \in (0,\infty)$. We obtain that $B_n \subsetneq (x-\epsilon/2, x + \epsilon/2)$ if $n \geq |f(x)| + 1$ and, as a result, $|B_n| \geq \delta$ for each such $n$.
	
	Next we fix an arbitrary point $x' \in \RR$ with  $x' > x$ (the case $x < x'$ can be considered analogously). We denote $\gamma = |(x, x'+1)|$ and notice that $\gamma < \infty$. Moreover, let $B_n' \coloneqq B_n \cup (x, x'+1)$ for $n \in \mathbb{N}$. If $n \geq |f(x)| + 1$, then the set $B_n'$ forms a ball containing $x'$ and therefore
	\begin{displaymath}
	\MN f(x') \geq \frac{1}{|B_n'|} \int_{B_n'} |f(y)| \, {\rm d}\mu(y) \geq \frac{|B_n|}{|B_n'|} \cdot \frac{1}{|B_n|} \int_{B_n} |f(y)| \, {\rm d}\mu(y) \geq \frac{\delta n}{\delta + \gamma}.
	\end{displaymath}
	This, in turn, implies that $\MN f(x') = \infty$, since $n$ can be arbitrarily large. 
\end{proof} 

At the end of this section we give an example of a space of the form $(\mathbb{R}, d_{\rm e}, w(x) {\rm d}x)$, where $w$ is a suitable weight (and $w(x) {\rm d}x$ is nondoubling), for which the centered Hardy--Littlewood maximal operator does not possess the dichotomy property.

\begin{example}[Example~A]\label{6AA}
	Consider the space $(\mathbb{R}, d_{\rm e}, \mu)$ with ${\rm d} \mu(x) \coloneqq  {\rm exp}(x^2) {\rm d}x$. Then $\MN$ possesses the dichotomy property, while $\MM$ does not.   
\end{example}

\noindent Indeed, $\MN$ possesses the dichotomy property by Proposition~\ref{P6.2.1}, and for the second part we argue as follows. Let $f(x) \coloneqq  x \cdot \mathbf{1}_{(0, \infty)} (x)$. We shall show that $\MM f(x) = \infty$ if and only if $x \in [0, \infty)$. For $x \in \mathbb{R}$ and $r\in (0,\infty)$ let us introduce the quantity
\begin{displaymath}
A_rf(x) \coloneqq \frac{1}{|B_r(x)|} \int_{B_r(x)} |f(y)| \, e^{y^2} {\rm d}y.
\end{displaymath}
First, observe that $\lim_{r \rightarrow \infty} A_rf(0) = \infty$. Indeed, fix $N \in \mathbb{N}$ and take $r_0 \in (N, \infty)$ such that
\begin{displaymath}
\int_{(N, r)} e^{x^2} {\rm d}x \geq \frac{1}{3} 	\int_{(-r, r)} e^{x^2} {\rm d}x 
\end{displaymath}
holds whenever $r \geq r_0$. Therefore, for each such $r$ we obtain
\begin{displaymath}
A_rf(0) = \frac{1}{|B_r(0)|} \int_{B_r(0)} f(x) \, e^{x^2} {\rm d}x \geq \frac{N}{|B_r(0)|} \int_{(N, r)} e^{x^2} {\rm d}x \geq \frac{N}{3},
\end{displaymath}
and thus $\MM f(0) = \infty$. If $x \in ( 0, \infty)$, then $A_rf(x) \geq A_{r+x}f(0)$ holds whenever $r \geq x$. This fact gives $\MM f(x) = \infty$ for each such $x$. It remains to show that $\MM f(x) < \infty$ if $x$ is strictly negative. 
Observe that it is possible to choose $r_0 \in (|x|,\infty)$ such that 
$
e^{(x+r)^2} \leq 2 |x| e^{r^2}
$
holds whenever $r \geq r_0$. If $r < r_0$, then $A_rf(x) \leq f(x + r_0)$. On the other hand, if $r \geq r_0$, then
\begin{displaymath}
A_rf(x) \leq \frac{1}{|B_r(x)|} \int_{B_r(x)} f(x) \, e^{x^2} {\rm d}x \leq \frac{e^{(x+r)^2}}{2  |(x-r, -r)|} \leq \frac{e^{(x+r)^2}}{2 |x| e^{r^2}} \leq 1.
\end{displaymath}
Consequently, we obtain $\MM f(x) < \infty$. 

\section{Multidimensional case}\label{S6.3}

Throughout this section we work with spaces that do not necessarily have a linear structure. In the first place we would like to get that in certain circumstances $\MM$ must possess the dichotomy property. Of course, for our purpose, we should ensure that the introduced criterion is relatively easy to apply and returns positive results also for some nondoubling spaces. Fortunately, it turns out that it is possible to find a condition that successfully meets all these requirements. 

The following result is embedded in the context of Euclidean spaces but it is worth keeping in mind that, in fact, it concerns all spaces for which $|X \setminus L_f^{\rm c}| = 0$ holds whenever $f \in L^1_{\rm loc}(\mu)$.

\begin{proposition}\label{P6.3.1}
	Let $\XX = (\mathbb{R}^d, d_{\rm e}, \mu)$ with $d \in \NN$ and $\mu$ such that $|B| \in (0, \infty)$ for each $B \subset \RR$. Assume that
	\begin{equation}\label{6C}
	\widetilde{C} \coloneqq \widetilde{C}(y_0) \coloneqq \limsup_{r \rightarrow \infty} \frac{|B_{r+1}(y_0)|}{|B_{r}(y_0)|} < \infty \textrm{ for some }  y_0 \in \mathbb{R}^d.
	\end{equation}
	Then the associated maximal operator $\MM$ possesses the dichotomy property. 	
\end{proposition}

Observe that the condition \eqref{6C} is related to certain global properties of a given metric measure space $\XX$ and thus its occurrence (or not) should be independent of the choice of $y_0$. Indeed, it can easily be shown that if the inequality in \eqref{6C} holds for some $y_0$, then it is also true for any $y \in X$ in place of $y_0$ (possibly with another constant $\widetilde{C}$ depending on $y$). 

Secondly, as it turns out according to Theorem~\ref{T6.4.1} in Section~\ref{S6.4}, the converse also holds in the case $\XX = (\mathbb{R}^d, d_{\rm e}, \mu)$. Namely, we shall prove that if $\MM$ possesses the dichotomy property, then (\ref{6C}) holds for some $y_0 \in \mathbb{R}^d$. In Proposition~\ref{P6.3.1} above we state only one of the implications, since this is enough to prove Theorem~\ref{T6.1.1}. On the other hand, the opposite implication allows us to say that the formulated condition is sufficient and necessary at the same time and, since looking for such conditions is interesting itself, we discuss it in a separate section.

\begin{proof}
	Let $f \in L^1_{\rm loc}(\mu)$ and assume that $|E_\infty^{\rm c}(f)| > 0$. We take $x_0 \in L_f^{\rm c}$ such that $\MM f(x_0) = \infty$. Hence for each $n \in \mathbb{N}$ we have a ball $B_n$ centered at $x_0$ and satisfying 
	\begin{displaymath}
	\frac{1}{|B_n|} \int_{B_n} |f(y)| \, {\rm d}\mu(y) > n.
	\end{displaymath}
	Fix $\epsilon \in (0, \infty)$ such that
	\begin{displaymath}
	\frac{1}{|B_r(x_0)|} \int_{B_r(x_0)} |f(y)-f(x_0)| \, {\rm d}\mu(y) \leq 1
	\end{displaymath} 
	holds whenever $r \leq \epsilon$ and denote $\delta \coloneqq |B_\epsilon(x_0)| \in (0,\infty)$. If $n \geq |f(x_0)| + 1$, then $B_n \subsetneq B_\epsilon(x_0)$ and, as a result, we have $|B_n| \geq \delta$.	This fact easily implies that $\lim_{n \rightarrow \infty} r_n = \infty$, since $f$ is locally integrable.
	
	Let us now fix $x \in \mathbb{R}^d$. There exists $r_0\in (0, \infty)$ such that
	\begin{displaymath}
	|B_{r+1}(y_0)| \leq 2 \widetilde{C} |B_{r}(y_0)| 
	\end{displaymath}
	if $r \geq r_0$. We take $n_0 \geq |f(x_0)| + 1$ large enough to ensure that if $n \geq n_0$, then $r_n - |y_0 - x_0| \geq r_0$. For each $n \in \mathbb{N}$ consider the ball $B_n' \coloneqq B_{r_n + |x_0-x|}(x)$. If $n \geq n_0$, then
	\begin{displaymath}
	|B_n'| \leq |B_{r_n + |x_0-x| + |y_0 - x|}(y_0)| \leq (2\widetilde{C})^m |B_{r_n - |x_0 - y_0|}(x_0)| \leq (2\widetilde{C})^m |B_n|,
	\end{displaymath}
	where $m$ is a positive integer independent of $n$ and such that $m > |x_0-x| + |y_0 - x| + |x_0 - y_0|$. Finally, by using the fact that $B_n \subset B_n'$ we obtain
	\begin{displaymath}
	\MM f(x) \geq \frac{1}{|B_n'|} \int_{B_n'} |f(y)| \, {\rm d}\mu(y) \geq \frac{|B_n|}{|B_n'|} \cdot \frac{1}{|B_n|} \int_{B_n} |f(y)| \, {\rm d}\mu(y) \geq \frac{n}{(2\widetilde{C})^m}
	\end{displaymath}
	which gives $\MM f(x) = \infty$, since $n$ can be arbitrarily large. 
\end{proof}

\begin{remark}\label{R6.3.2}
	The conclusion of Proposition~\ref{P6.3.1} remains true if we take $d_\infty$ instead of $d_{\rm e}$ provided that this time the balls determined by $d_\infty$ are used in \eqref{6C}. There are also no obstacles to getting discrete counterparts of these statements. Namely, one can replace $\mathbb{R}^d$ by $\mathbb{Z}^d$, and obtain the desired result for the space $(\mathbb{Z}^d, \rho, \mu)$, where $\rho = d_{\rm e}$ or $\rho = d_\infty$ and $\mu$ is arbitrary. 
\end{remark}

Now, with Propositions~\ref{P6.2.1}~and~\ref{P6.3.1} in hand, we can easily give an example of a nondoubling space, for which both $\MN$ and $\MM$ possess the dichotomy property. 

\begin{example}[Example~B]\label{6BB}
	Consider the space $(\mathbb{R}, d_{\rm e}, \mu)$ with ${\rm d}\mu(x) \coloneqq  {\rm exp}(-x^2) {\rm d}x$. Then both $\MN$ and $\MM$ possess the dichotomy property.
\end{example}

\noindent Indeed, $\MN$ possesses the dichotomy property by Proposition~\ref{P6.2.1}, while $\MM$ possesses the dichotomy property by Proposition~\ref{P6.3.1}, since $\lim_{r \rightarrow \infty} |B_{r+1}(0)| / |B_r(0)| = 1$. \newline

At this point a natural question arises: will we get the same result for Gaussian measures in higher dimensions? The next proposition settles this in the affirmative. 

\begin{proposition}
	Let $\XX = (\mathbb{R}^d, d_{\rm e}, \mu)$ with $d \in \NN$ and $\mu$ such that $|\RR^d| < \infty$. Assume that $\mu$ is determined by a strictly positive weight $w$ such that, for each $n \in \mathbb{N}$,
	\begin{equation}\label{6D}
	w(x) \in [c_n, C_n], \qquad x \in B_n(0),
	\end{equation}
	for some numerical constants $c_n,C_n \in (0, \infty)$ with $c_n < C_n$. Then the associated maximal operators, $\MN$ and $\MM$, both possess the dichotomy property.
\end{proposition}
	
\begin{proof}
	It suffices to prove that $\MN$ possesses the dichotomy property, since $|\mathbb{R}^d| < \infty$ implies that \eqref{6C} is satisfied with $\widetilde{C} = 1$ (regardless of which point $y_0 \in \mathbb{R}^d$ we choose).
		
	Take $f \in L_{\rm loc}^1(\mu)$. We shall show that $|\mathbb{R}^d \setminus L_f| = 0$. For each $n \in \mathbb{N}$ consider $\mu_n$ determined by the weight $w_n$ given by
	\begin{displaymath}
	w_n(x) \coloneqq \left\{ \begin{array}{rl}
	w(x) & \textrm{if } x \in B_n(0),  \\
	1 & \textrm{otherwise. }  \end{array} \right. 
	\end{displaymath}
	Observe that the condition \eqref{6D} implies that $\mu_n$ is doubling. Let $f_n \coloneqq  f \mathbf{1}_{B_n(0)}$. We have
	\begin{displaymath}
	|B_n(0) \setminus L_f| = |B_n(0) \setminus L_{f_n}(\mu_n)| \leq |\mathbb{R}^d \setminus L_{f_n}(\mu_n)| = 0,
	\end{displaymath}
	because $f_n \in L^1_{\rm loc}(\mu_n)$. This gives $|\mathbb{R}^d \setminus L_f| = 0$, since $n$ can be arbitrarily large.
		
	Assume that $|E_\infty(f)| > 0$ and take $x_0 \in L_f$ such that $\MN f(x_0) = \infty$. For each $n \in \mathbb{N}$ there exists a ball $B_n$ containing $x_0$ and such that
	\begin{displaymath}
	\frac{1}{|B_n|} \int_{B_n} |f(y)| \, {\rm d}\mu(y) > n.
	\end{displaymath}
	Fix $\epsilon \in (0, \infty)$ such that
	\begin{displaymath}
	\frac{1}{|B|} \int_{B} |f(y)-f(x_0)| \, {\rm d}\mu(y) \leq 1
	\end{displaymath} 
	holds whenever $B \subset B_\epsilon(x_0)$. If $n \geq |f(x_0)| + 1$, then $B_n \subsetneq B_\epsilon(x_0)$. Thus, combining the condition \eqref{6D} with the fact that $r(B_n) \geq \epsilon / 2$ for each $n$ as before, we conclude that $|B_n| \geq \delta$ for some $\delta \in (0,\infty)$ depending on $x_0$ and $\epsilon$ but independent of $n$.
		
	Let us now fix $x \in \mathbb{R}^d$ and take $n \in \NN$ such that $n \geq |f(x_0)| + 1$. Let $B_n'$ be any ball containing $x$ and $B_n$. Then we obtain
	\begin{displaymath}
	\MN f(x) \geq \frac{1}{|B_n'|} \int_{B_n'} |f(y)| \, {\rm d}\mu(y) \geq \frac{1}{|\mathbb{R}^d|} \int_{B_n} |f(y)| \, {\rm d}\mu(y) \geq \frac{\delta n}{|\mathbb{R}^d|}
	\end{displaymath}
	which gives $\MN f(x) = \infty$, since $n$ can be arbitrarily large. 
\end{proof}
	
Until now we furnished examples illustrating two of the four possibilities related to our initial problem. In both cases the specified space was $\mathbb{R}$ with the usual metric and a Borel measure determined by a suitable weight. Unfortunately, as was indicated in Proposition~\ref{P6.2.1}, such examples cannot be used to cover the remaining two cases, since this time we want $\MN$ to not possess the dichotomy property. Therefore, a natural step is to try to use $\mathbb{R}^2$ instead of $\mathbb{R}$. This idea turns out to be right. However, for simplicity, the other two examples will be initially constructed in the discrete setting $\mathbb{Z}^2$. Also, for purely technical reasons, the metric $d_{\rm e}$ is replaced by $d_\infty$. Nevertheless, after presenting Examples~\hyperref[6CC]{C} and~\hyperref[6DD]{D}, we include some additional comments in order to convince the reader that it is also possible to obtain the desired examples using metric measure spaces of the form $(\mathbb{R}^2, d_{\rm e}, \mu)$.
	
While dealing with $\mathbb{Z}^2$, for the sake of brevity, we will write shortly $B_r(n,m), \mu(n,m), |(n,m)|$ instead of $B_r((n,m)), \mu(\{(n,m)\}), |\{(n,m)\}|$, respectively.
	
\begin{example}[Example~C] \label{6CC}
	Consider the space $(\mathbb{Z}^2, d_\infty, \mu)$, where $\mu$ is defined by 
	\begin{displaymath}
	\mu(n,m) \coloneqq \left\{ \begin{array}{rl}
	4^{|m|} & \textrm{if } n = 0,  \\		
	1 & \textrm{otherwise. }  \end{array} \right. 
	\end{displaymath}
	Then $\MM$ possesses the dichotomy property, while $\MN$ does not.
\end{example}
		
\noindent Indeed, observe that $\MM$ possesses the dichotomy property by Proposition~\ref{P6.3.1} (or, more precisely, by Remark~\ref{R6.3.2}), since 
\begin{displaymath}
\lim_{r \rightarrow \infty} \frac{|B_{r+1}(0,0)|}{|B_{r}(0,0)|} = 4.
\end{displaymath}
To verify the second part of the conclusion consider the function $f$ defined by
\begin{displaymath}
f(n,m) \coloneqq \left\{ \begin{array}{rl}
2^n & \textrm{if } n > 0 \textrm{ and } m=0,  \\		
0 & \textrm{otherwise. }  \end{array} \right. 
\end{displaymath}
We will show that $\MN f(1,0) = \infty$ and $\MN f(-1,0) < \infty$ (in fact, it should be clear to the reader that $(1,0)$ and $(-1, 0)$ may be replaced by any other points $(n_1,m_1)$ and $(n_2,m_2)$ such that $n_1$ is strictly positive and $n_2$ is strictly negative).
For each $N \in \mathbb{N}$ consider the ball $B_N \coloneqq B_N(N,0)$. Observe that
\begin{displaymath}
\MN f(1,0) \geq \frac{1}{|B_N|} \sum_{(n,m) \in B_N} f(n,m) \cdot |(n,m)| \geq \frac{f(N,0) \cdot |(N,0)| }{(2N-1)^2} =   \frac{2^N}{(2N-1)^2}	
\end{displaymath}
which implies that $\MN f(1,0) = \infty$, since $N$ can be arbitrarily large.
On the other hand, consider any ball $B$ containing $(-1,0)$ and denote
\begin{displaymath}
K \coloneqq K(B) \coloneqq \max\{n \in \mathbb{Z} : (n,0) \in B\}.
\end{displaymath}
If $K \leq 0$, then clearly $\sum_{(n,m) \in B} f(n,m) \cdot |(n,m)| = 0$. In turn, if $K > 0$, then $B$ must contain at least one of the points $(0,-\lfloor K/2 \rfloor )$ and $(0,\lfloor K/2 \rfloor)$. Consequently, we have
\begin{displaymath}
\frac{1}{|B|} \sum_{(n,m) \in B} f(n,m) \cdot |(n,m)| \leq \frac{2 f(K,0)}{4^{\lfloor K/2 \rfloor}} \leq 4
\end{displaymath}
which implies that $\MN f(-1,0) < \infty$.
	
\begin{example}[Example~D] \label{6DD}
	Consider the space $(\mathbb{Z}^2, d_\infty, \mu)$, where $\mu$ is defined by
	\begin{displaymath}
	\mu(n,m) \coloneqq \left\{ \begin{array}{rl}
	4^{|m|} & \textrm{if } n = 0,  \\
	2^{n^2} & \textrm{if } n < 0 \textrm{ and } m = 0,  \\
	1 & \textrm{otherwise. }  \end{array} \right. 
	\end{displaymath}
	Then both $\MN$ and $\MM$ do not possess the dichotomy property.
\end{example} 
	
\noindent Indeed, to verify that $\MN$ does not possess the dichotomy property we can use exactly the same function $f$ as in Example~\hyperref[6CC]{C}. It is easy to see that $\MN f(1,0) = \infty$ and $\MN f(-1,0) < \infty$ hold as before. In order to show that $\MM$ does not possess the dichotomy property consider $g$ defined by
	
\begin{displaymath}
g(n,m) \coloneqq \left\{ \begin{array}{rl}
2^{n^2} & \textrm{if } n > 0 \textrm{ and } m = 0,  \\
0 & \textrm{otherwise. }  \end{array} \right. 
\end{displaymath}
For each $N \in \mathbb{N}$ consider the balls $B_N^+ \coloneqq B_N(1,0)$ and $B_N^- \coloneqq B_N(-1,0)$. If $N$ is large, then
\begin{displaymath}
\frac{1}{|B_N^+|} \sum_{(n,m) \in B_N^+} g(n,m) \cdot |(n,m)| \geq \frac{g(N,0)}{2 |(-N+2, 0)|} = 2^{N^2 - (N-2)^2 - 1}
\end{displaymath}
and
\begin{displaymath}
\frac{1}{|B_N^-|} \sum_{(n,m) \in B_N^-} g(n,m) \cdot |(n,m)| \leq \frac{2 g(N-2,0)}{|(-N, 0)|} = 2^{-N^2 + (N-2)^2 + 1}.
\end{displaymath}
This, in turn, easily leads to the conclusion that $\MM g(1,0) = \infty$ and $\MM g(-1,0) < \infty$. \newline
	
Finally, as we mentioned earlier, we outline a sketch of how to adapt Examples~\hyperref[6CC]{C}~and~\hyperref[6DD]{D} to the situation of $\mathbb{R}^2$ with the Euclidean metric. First, note that the key idea of Example~\hyperref[6CC]{C} was to construct a measure which creates a kind of barrier separating (in the proper meaning) the points $(n,m)$ with positive and negative values of $n$, respectively. Exactly the same effect can be achieved if we define $w$ so that it behaves like $e^{|y|}$ in the strip $x \in (-\frac{1}{2}, \frac{1}{2})$ and like $1$ outside of it. However, because of some significant differences between the shapes of the balls determined by $d_{\rm e}$ and $d_\infty$, respectively, one should be a bit more careful when looking for the proper function $f$ such that $\MN f(x,y) = \infty$ if $x > 1$ and $\MN f(x,y) < \infty$ if $x < -1$. Observe that any ball $B$ such that $(-1,0) \in B$ and $(N,0) \in B$ must contain at least one of the points $(0, - \sqrt{N})$ and $(0, \sqrt{N})$. Therefore, if $B_N$ is such that $N$ is the largest positive integer $n$ satisfying $(n,0) \in B_N$, then one should ensure that the integral $\int_{B_N} f(x,y) \, w(x,y) \, {\rm d}x \, {\rm d}y$ does not exceed $C e^{\sqrt{N}}$, where $C \in (0,\infty)$ is some numerical constant. On the other hand, we want this quantity to tend to infinity with $N$ faster than $N^2$. This two conditions are fulfilled simultaneously if, for example, $f(x,y)$ behaves like $x^2$ in the region $\{(x,y) \in \mathbb{R}^2 : x > 0, \, |y| < \frac{1}{2}\}$, and equals $0$ outside of it. 
	
Finally, to arrange the situation of Example~\hyperref[6DD]{D}, it suffices to define $w$ in such a way that it is comparable to $e^{|y|}$ if $|x| < \frac{1}{2}$, to $e^{x^2}$ if $x < 0$ and $|y| < \frac{1}{2}$, and to $1$ elsewhere. Also, apart from those described above, there are no further difficulties in finding the appropriate functions $f$ and $g$ that break the dichotomy condition for $\MN$ and $\MM$, respectively.

\section{Necessary and sufficient condition}\label{S6.4}

The last section is mainly devoted to describing the exact characterization of situations in which $\MM$ possesses the dichotomy property, for metric measure spaces of the form $(\mathbb{R}^d, d_{\rm e}, \mu)$ with $d \in \NN$ and $\mu$ such that $|B| \in (0, \infty)$ for each $B \subset \RR$. Namely, we have the following theorem.

\begin{theorem} \label{T6.4.1}
	Let $\XX = (\mathbb{R}^d, d_{\rm e}, \mu)$ with $d \in \NN$ and $\mu$ such that $|B| \in (0, \infty)$ for each $B \subset \RR$. Then $\MM$ possesses the dichotomy property if and only if \eqref{6C} holds.
\end{theorem}

We show the proof only for $d=2$, since in this case all the significant difficulties are well exposed and, at the same time, we omit a few additional technical details that arise when $d \geq 3$. In turn, the case $d=1$ is much simpler than the others, so we do not focus on it. When dealing with $\mathbb{R}^2$, we will write shortly $B_r(x,y)$ instead of $B_r((x,y))$, just like we did in the previous section in the context of $\mathbb{Z}^2$. 

\begin{proof}
	Let us first recall that one of the implications has already been proven in Proposition~\ref{P6.3.1}. Thus, it is enough to show that \eqref{6C} is necessary for $\MM$ to possess the dichotomy property. 
	
	Take $(\mathbb{R}^2, d_{\rm e}, \mu)$ and assume that (\ref{6C}) fails to occur. Thus, for the point $(0,0)$ there exists a strictly increasing sequence of positive numbers $(a_k)_{k \in \mathbb{N}}$ such that
	\begin{displaymath}
	|B_{a_k+1}(0,0)| \geq 2^{2k} |B_{a_k}(0,0)|
	\end{displaymath}
	holds for each $k \in \mathbb{N}$. In addition, we can force that $a_1 \geq 8$ and $a_{k+1} \geq a_k + 2$. For each $n \in \mathbb{N}$ and $j \in [2^n]$ we define
	\begin{displaymath}
	S_{k+, j}^{(n)} \coloneqq \Big \{(x,y) \in B_{a_k+1}(0,0) : \phi(x,y) \in \Big[ \frac{2 \pi (j-1)}{2^n}, \frac{2 \pi j}{2^n} \Big) \Big\},
	\end{displaymath}
	where $\phi(x,y) \in [0, 2\pi)$ is the angle that $(x,y)$ takes in polar coordinates.
	
	Take $n = 1$ and choose $j_1 \in [2]$ such that the set
	\begin{displaymath}
	\Lambda_1 \coloneqq \Big\{k \in \mathbb{N} : |S_{k+, j_1}^{(1)}| \geq  |B_{a_{k}}(0,0)| / 2 \Big\}
	\end{displaymath}
	is infinite. Next, take $n = 2$ and choose $j_2 \in [4]$ satisfying $\lceil j_2/2 \rceil = j_1$ and such that
	\begin{displaymath}
	\Lambda_2 \coloneqq \Big\{ k \in \Lambda_1 : |B_{k+, j_2}^{(2)}| \geq |B_{a_{k}}(0,0)| / 4 \Big\}
	\end{displaymath}
	is infinite. Continuing this process inductively we get a sequence $(j_n)_{n \in \mathbb{N}}$ satisfying $\lceil j_{n+1} / 2 \rceil = j_n$ for each $n \in \mathbb{N}$ and, by invoking a suitable diagonal argument, a strictly increasing subsequence $(a_{k_n})_{n \in \mathbb{N}}$ such that for each $n \in \mathbb{N}$ we have
	\begin{displaymath}
	|S_{k_n+, j_n}^{(n)}| \geq |B_{a_{k_n}}(0,0)| / 2^n, \qquad n \in \mathbb{N}.
	\end{displaymath}
	
	From now on, for simplicity, we will write $B_n$ 
	and $S_{n+, j_n}$ instead of $B_{a_{k_n}}(0,0)$ 
	and $S_{k_n+, j_n}^{(n)}$, respectively. Observe that the obtained sequence $(j_n)_{n \in \mathbb{N}}$ determines a unique angle $\phi_0 \in [0, 2\pi)$ which indicates a ray around which, loosely speaking, a significant part of $\mu$ is concentrated. For the sake of clarity we assume that $\phi_0 = 0$ (thus, $(j_n)_{n \in \mathbb{N}}$ equals either $(1, 1, 1, \dots)$ or $(2, 4, 8, \dots)$).
	
	For each $n \in \mathbb{N}$ denote $B_{n-} \coloneqq B_{1/2}(-a_{k_n}+2, 0)$ and consider the function $f$ defined by
	\begin{displaymath}
	f \coloneqq \sum_{n=1}^\infty \frac{2^n |B_n|}{|B_{n-}|} \mathbf{1}_{B_{n-}}.
	\end{displaymath} 
	Of course, $f \in L^1_{\rm loc}(\mu)$. We will show that $\MM f(x,y) = \infty$ for $(x,y) \in B_{1/2}(0,0)$ and $\MM f(x,y) < \infty$ for $(x,y) \in B_{1/2}(3,0)$.
	
	Fix $(x,y) \in B_{1/2}(0,0)$ and observe that $B_{n-} \subset B_{a_{k_n}-1}(x,y) \subset B_n$ and therefore
	\begin{displaymath}
	\frac{1}{|B_{a_{k_n}-1}(x,y)|} \int_{B_{a_{k_n}-1}(x,y)} f \, {\rm d}\mu \geq \frac{1}{|B_n|} \int_{B_{n-}} f \, {\rm d}\mu = 2^n
	\end{displaymath}
	which gives $\MM f(x,y) = \infty$.
	
	In turn, fix $(x,y) \in B_{1/2}(3,0)$ and consider $r \in (0,\infty)$ such that $B_r(x,y)$ intersects at least one of the sets $B_{n-}$. Notice that this requirement forces $r>2$. We denote
	\begin{displaymath}
	N \coloneqq N(x,y,r) \coloneqq \max \big\{ n \in \mathbb{N} : B_r(x,y) \cap B_{n-} \neq \emptyset \big\}.
	\end{displaymath}
	One can easily see that $r > a_{k_n}$ and hence $(a_{k_n},0) \in B_{r-2}(x,y)$. Also, it is possible to choose $N_0 \coloneqq  N_0(x,y) \in \NN \setminus \{1\}$ such that if $N \geq N_0$ and $(a_{k_N},0) \in B_{r-2}(x,y)$, then $S_{N+,j_N} \subset B_r(x,y)$. Define
	 \[
	 \widetilde{N} \coloneqq \widetilde{N}(x,y) \coloneqq \max \big\{r \in (2, \infty) : N(x,y,r)<N_0 \big\} \in (2, \infty).
	 \] 
	 If $r \in (2, \widetilde{N}]$, then
	\begin{displaymath}
	\frac{1}{|B_r(x,y)|} \int_{B_r(x,y)} f \, {\rm d}\mu \leq \frac{1}{|B_2(x,y)|} \int_{B_{\widetilde{N}+1}(x,y)} f \, {\rm d}\mu = C, 
	\end{displaymath}
	where $C \in (0,\infty)$ is a numerical constant depending on $(x,y)$ but independent of $r$ as above. On the other hand, if $r > \widetilde{N}$, then
	\begin{displaymath}
	\frac{1}{|B_r(x,y)|} \int_{B_r(x,y)} f \, {\rm d}\mu \leq \frac{2^{N+1} |B_N|}{|S_{N+,j_N}|} \leq 2.
	\end{displaymath}
	Consequently, we have $\MM f(x,y) < \infty$.
\end{proof}

\begin{remark}
	The proof presented above relies on certain Euclidean geometry properties and therefore it cannot be repeated in a more general context. However, one can replace $d_{\rm e}$ with $\rho$ induced by any norm on $\mathbb{R}^d$ and get the same result by using similar arguments with only minor modifications. In this case, of course, the balls in \eqref{6C} are taken with respect to $\rho$. Thus, among other things, we must take into account how the shape of these balls is related to the direction determined by the angle $\phi_0$ specified in the proof. Finally, the weak type $(1,1)$ inequality of $\MM$ associated with $(\mathbb{R}^d, \rho, \mu)$, which is needed to provide $|\mathbb{R}^d \setminus L^{\rm c}_f| = 0$ in Proposition~\ref{P6.3.1}, can be deduced from a stronger version of the Besicovitch covering lemma (see \cite[Theorem~2.8.14]{Fed}).
\end{remark}

We conclude our studies with an example which indicates that a potential necessary and sufficient condition for $\MN$ must be significantly different from the one stated for $\MM$. Namely, while \eqref{6C} concerns only the growth at infinity of a given measure, the parallel condition for $\MN$ should deal with both global and local aspects of the considered spaces. Thus, this problem, probably more difficult, is an interesting starting point for further investigation.

\begin{example} \label{6EE}
	Consider the space $(\mathbb{R}^2, d_{\rm e}, \mu)$ with $\mu \coloneqq \lambda_1 + \lambda_2$, where $\lambda_1$ is one-dimensional Lebesgue measure on $A \coloneqq [0,1] \times \{0\}$ and $\lambda_2$ is two-dimensional Lebesgue measure on the whole plane. Then there exists $f \in L^1(\mu)$ with compact support such that $E_\infty(f) = A$. 
\end{example}

\noindent Indeed, for each $n \in \NN$ denote $S_n \coloneqq [0,1] \times (2^{-n^2}, 2^{-n^2+1})$ and consider the function
\begin{displaymath}
f \coloneqq \sum_{n=1}^{\infty} 2^n \mathbf{1}_{S_n}.
\end{displaymath}
Observe that $f$ equals $0$ outside the square $[0,1] \times [0,1]$ and $\|f\|_{1} = \sum_{n=1}^{\infty} 2^n \cdot 2^{-n^2} \leq 2$. 
Fix $x \in [0,1]$ and for each $n \in \mathbb{N}$ consider the ball $B_n \coloneqq B_{2^{-n^2 + \epsilon_n}}(x, 2^{-n^2})$, where $\epsilon_n \in (0, \infty)$ is such that $|B_n| \leq 2^{-2n^2+2}$. Notice that $(x,0) \in B_n$ for each $n$. If $n \geq 2$, then $|B_n \cap S_n| \geq 2^{-2n^2 - 1}$ which gives
\begin{displaymath}
\frac{1}{|B_n|} \int_{B_n} f \, {\rm d}\mu \geq \frac{2^n \cdot 2^{-2n^2 - 1}}{2^{-2n^2+2}} = 2^{n-3}. 
\end{displaymath}
Consequently, we have $\MN f(x,0) = \infty$. On the other hand, consider $(x, y) \notin A$. Now there exist $\epsilon, L \in (0,\infty)$ such that if $d_{\rm e}((x,y),(x',y')) < 2\epsilon$, then $f(x',y') \leq L$. Consequently, we obtain $\MN f(x,y) \leq \max\{L, \|f\|_1/(\pi \epsilon^2)\} < \infty$.
\backmatter
\chapter*{Appendix. Interpolation theorem}\label{Appendix}
\markboth{Appendix. Interpolation theorem}{Appendix. Interpolation theorem}
\addcontentsline{toc}{chapter}{Appendix. Interpolation theorem}
\setstretch{1.1}
Here we give an elementary proof of Theorem~\ref{44-T2} stated in Subsection~\ref{44-S5}. In what follows the operator is specified to be $\MM_\XX$, but one can also replace it with, for example, any operator $\HH$ satisfying the following assertions:
\begin{enumerate}[label=(\Alph*)]\setlength\itemsep{0em}
	\item $0 \leq f_1 \leq f_2 \implies 0 \leq \HH f_1 \leq \HH f_2$,
	\item $|\HH f| \leq \HH |f|$,
	\item $\HH(|f_1| + |f_2|) \leq C_\HH (\HH |f_1| + \HH |f_2|)$.
\end{enumerate}   

\begin{proof}[Proof of Theorem~\ref{44-T2}]
First, notice that we can assume that $q_0 < q_1$ and $r_0 < r_1$. Indeed, in each of the remaining cases the thesis follows easily from Fact~\ref{F4.1.3}.

Fix $\theta \in (0,1)$ and let $\CC_\rightarrow$ be such that
\begin{displaymath}
\|\MM_\XX g \|_{p, r_i} \leq \CC_\rightarrow \| g \|_{p, q_i}, \qquad g \in L^{p, q_i}(\XX), \ i \in \{0,1\}.
\end{displaymath}
Our aim is to obtain the inequality
\begin{equation}\label{44-A1}\tag{A1}
\|\MM_\XX g \|_{p, r_\theta} \leq C \| g \|_{p, q_\theta},
\end{equation}
for each $g \in L^{p, q_\theta}(\XX)$ with some $C$ independent of $g$.
For any measurable function $g \colon X \rightarrow \CCC$ we introduce $\SSSS g, \TTTT g \colon \ZZ \rightarrow [0, \infty]$ by
\begin{displaymath}
\SSSS g(n) \coloneqq 2^n d_g(2^n)^{1/p}, \qquad n \in \ZZ, 
\end{displaymath}
and 
\begin{displaymath}
\TTTT g(n) \coloneqq  \SSSS\MM_\XX g(n) = 2^n d_{\MM_\XX g}(2^n)^{1/p}, \qquad n \in \ZZ.
\end{displaymath}
We observe that for each $q \in [1, \infty]$ there is a numerical constant $\CC_\square(p,q)$ such that
\begin{displaymath}
\frac{1}{\CC_\square(p,q)} \, \| \SSSS g \|_{q} \leq \|g\|_{p,q} \leq \CC_\square(p,q) \, \| \SSSS g \|_{q}, \qquad g \in L^{p, q}(\XX),
\end{displaymath}
where $\|  \cdot  \|_q$ denotes the standard norm on $\ell^q(\ZZ)$. Let
\begin{displaymath}
\CC_\square \coloneqq  \max \{\CC_\square(p,q_0), \CC_\square(p,q_\theta), \CC_\square(p,q_1), \CC_\square(p,r_0), \CC_\square(p,r_\theta), \CC_\square(p,r_1)\}.
\end{displaymath}
Thus for each $i \in \{ 0, 1\}$ we have
\begin{equation}\label{44-A2}\tag{A2}
\| \TTTT g \|_{r_i} \leq \CC_\square^2 \, \CC_\rightarrow \, \|\SSSS g \|_{q_i}
\end{equation}
and our aim is to obtain the inequality
\begin{equation}\label{44-A3}\tag{A3}
\| \TTTT g \|_{r_\theta} \leq \widetilde{C}  \|\SSSS g \|_{q_\theta},
\end{equation}
which would imply \eqref{44-A1} with $C = \widetilde{C} \CC_\square^2$. 

In order to deduce \eqref{44-A3} from \eqref{44-A2} we follow the classical proof of the Marcinkiewicz interpolation theorem for operators acting on Lebesgue spaces (see \cite[Theorem 1]{Z}). It turns out that this strategy can be successfully applied but we must take into account certain additional difficulties. Namely, our ``map'' is given by $\SSSS g \mapsto \TTTT g$ and this operation cannot be understood as a~well defined operator, since there are usually many different functions with the same distribution function. Thus, we do not apply the Marcinkiewicz interpolation theorem directly but rather repeat its proof in our context. We proceed with the details below.

Assume that $r_1 < \infty$ and fix $f \in L^{p,q_\theta}(\XX)$ satisfying $f \geq 0$. For each $\lambda \in (0,\infty)$ we introduce the set $N_\lambda \coloneqq  \{n \in \ZZ : \SSSS f > \lambda \}$. Observe that either $N_\lambda = \emptyset$ or $N_\lambda$ consists of finitely many elements $n_1 > \dots > n_m$ for some $m \in \NN$. For each $j \in \ZZ$ let $E_j \coloneqq  \{x \in X  : f(x) \geq 2^j \}$. If $N_\lambda = \emptyset$, then we let $f_0^\lambda \coloneqq  0$ and $f_1^\lambda \coloneqq  f$. Otherwise, if $N_\lambda \neq \emptyset$, then we define
\begin{displaymath}
f_0^\lambda \coloneqq  f \cdot \big( \mathbf{1}_{E_{n_1}} + \sum_{k=2}^{m} \mathbf{1}_{E_{n_k} \setminus E_{n_{k-1}}} \big), \qquad f_1^\lambda \coloneqq  f \cdot \sum_{j \in \ZZ \setminus N_\lambda} \mathbf{1}_{E_{n_j} \setminus E_{n_{j-1}}}.
\end{displaymath}
Notice that $f \leq f_0^\lambda + f_1^\lambda$ and hence $\MM_\XX f \leq \MM_\XX f_0^\lambda + \MM_\XX f_1^\lambda$. Moreover, we have
\begin{displaymath}
\SSSS f_0^\lambda(n) = \SSSS f(n) > \lambda, \qquad n \in N_\lambda,
\end{displaymath}
and
\begin{displaymath}
\SSSS f_0^\lambda(n) \leq \min\{ \lambda, \SSSS f(n)\}, \qquad n \in \ZZ.
\end{displaymath}
Let $(\SSSS f)^\lambda_0 \coloneqq  \SSSS f \cdot \mathbf{1}_{N_\lambda}$ and $(\SSSS f)^\lambda_1 \coloneqq  \SSSS f \cdot \mathbf{1}_{\ZZ \setminus N_\lambda}$. Then it is not hard to check that
\begin{equation}\label{44-geom}\tag{A4}
\| \SSSS f^\lambda_i \|_{q_i} \leq \big( 1 + 2^{-q_i/p} + 4^{-q_i/p} + \dots \big)^{1/q_i} \, \|(\SSSS f)^\lambda_i \|_{q_i}.
\end{equation} 

Next we study the distribution functions of $(\SSSS f)^\lambda_i$, $i \in \{0,1\}$, more carefully. Observe that we have $d_{(\SSSS f)^\lambda_0}(y) \leq d_{\SSSS f}(\lambda)$ for $y \in (0,\lambda)$ and $d_{(\SSSS f)^\lambda_0}(y) \leq d_{\SSSS f}(y)$ for $y \in [\lambda, \infty)$. Hence, combining these estimates, the fact that $d_{(\SSSS f)^\lambda_0}$ is nonincreasing, and the equality
\begin{displaymath}
2^{q_0} \int_{0}^{\lambda/2} y^{q_0-1} \, {\rm d}y = \int_{0}^{\lambda} y^{q_0-1} \, {\rm d}y,
\end{displaymath}
we conclude that
\begin{equation}\label{44-d0}\tag{A5}
\int_{0}^\infty y^{q_0-1} d_{(\SSSS f)^\lambda_0}(y) \, {\rm d}y \leq \frac{2^{q_0}}{2^{q_0}-1} \int_{\lambda/2}^\infty y^{q_0-1} d_{\SSSS f}(y) \, {\rm d}y \leq 2^{q_0} \int_{\lambda/4}^\infty (y - \lambda/4)^{q_0-1} d_{\SSSS f}(y) \, {\rm d}y.
\end{equation}
Similarly, since $d_{(\SSSS f)^\lambda_1}(y) \leq d_{\SSSS f}(y)$ for $y \in (0,\lambda)$ and $d_{(\SSSS f)^\lambda_0}(y) = 0$ for $y \in [\lambda, \infty)$, we have
\begin{equation}\label{44-d1}\tag{A6}
\int_{0}^\infty y^{q_1-1} d_{(\SSSS f)^\lambda_1}(y) \, {\rm d}y \leq \int_{0}^\lambda y^{q_0-1} d_{\SSSS f}(y) \, {\rm d}y \leq 2^{2q_0} \int_{0}^{\lambda/4} y^{q_0-1} d_{\SSSS f}(y) \, {\rm d}y.
\end{equation}

Now we turn our attention to $\TTTT f$. Fix $y \in (0, \infty)$ and $\lambda = \lambda(y) \in (0, \infty)$ (which will be specified later on). Since $\MM_\XX f \leq \MM_\XX f_0^\lambda + \MM_\XX f_1^\lambda$, we have $\TTTT f(n) \leq 2^{1/p} (\TTTT f_0^\lambda(n-1) + \TTTT f_1^\lambda(n-1))$ for each $n \in \NN$. Hence
\begin{equation}\label{44-d2}\tag{A7}
d_{\TTTT f}(y) \leq d_{\TTTT f_0^\lambda}(y/2^{1/p}) + d_{\TTTT f_1^\lambda}(y/2^{1/p}).
\end{equation}
By the hypothesis we have
\begin{equation}\label{44-d3}\tag{A8}
d_{\TTTT f_i^\lambda}(y/2^{1/p}) \leq 2^{r_i/p} \, \frac{\|\TTTT f_i^\lambda \|_{r_i}^{r_i}}{y^{r_i}} \leq \big( 2^{1/p} \CC_\square^2 \CC_\rightarrow \big)^{r_i} \, \frac{\|\SSSS f_i^\lambda \|_{q_i}^{r_i}}{y^{r_i}}. 
\end{equation}
Therefore, combining \eqref{44-geom}, \eqref{44-d0}, \eqref{44-d1}, \eqref{44-d2}, and \eqref{44-d3} gives
\begin{align*}
\| \TTTT f \|_{r_\theta}^{r_\theta} = r_\theta \int_{0}^\infty y^{r_\theta - 1} d_{\TTTT f}(y) \, {\rm d}y & \leq C' \Big( \int_0^\infty y^{r_\theta-r_0-1} \, \Big( \int_{\lambda(y)/4}^\infty (t-\lambda(y)/4)^{q_0 - 1} d_{\SSSS f}(t) \, {\rm d}t \Big)^{r_0/q_0} \, {\rm d}y \\ &
\qquad  + \int_0^\infty y^{r_\theta-r_1-1} \, \Big( \int_0^{\lambda(y)/4} t^{q_1 - 1} d_{\SSSS f}(t) \, {\rm d}t \Big)^{r_1/q_1} \, {\rm d}y
\Big)
\end{align*}
with some constant $C'$ which may depend on $p, q_0, q_1, r_0, r_1, \theta$, and $\CC_\rightarrow$ but is independent of $f$ and the choice of $\lambda(y)$.

It is worth noting here that the inequality above reduces the problem to estimating the expression of the form very similar to that appearing in \cite[(3.7)]{Z} (here $d_{\SSSS f}$, $\lambda/4$, $q_0$, $q_1$, $r_0$, $r_1$, and $r_\theta$ play the roles of $m$, $z$, $a_2$, $a_1$, $b_2$, $b_1$, and $b$, respectively). Thus, in order to obtain \eqref{44-A3}, we may repeat the remaining calculations without any further changes. We briefly sketch the rest of the proof for the sake of completeness.

Denote by $P$ and $Q$ the two double integrals appearing in the last estimate. Then
\begin{displaymath}
P^{q_0/r_0} = \sup_{\omega_0} \, \int_0^\infty y^{r_\theta - r_0 - 1} \, \int_{\lambda(y) / 4}^{\infty} (t - \lambda(y)/4)^{q_0-1} \, d_{\SSSS f}(t) \, {\rm d}t \, \omega_0(y) \, {\rm d}y
\end{displaymath}
and
\begin{displaymath}
Q^{q_1/r_1} = \sup_{\omega_1} \, \int_0^\infty y^{r_\theta - r_1 - 1} \, \int_{0}^{\lambda(y)/4} t^{q_1-1} \, d_{\SSSS f}(t) \, {\rm d}t \, \omega_1(y) \, {\rm d}y,
\end{displaymath}
where the functions $\omega_i$ are nonnegative and satisfy
\begin{displaymath}
\int_0^\infty y^{r_\theta-r_i-1} \omega_i^{\frac{r_i}{r_i-q_i}}(y) \, {\rm d}y \leq 1
\end{displaymath} 
(note that $(\frac{r_i}{q_i})^{-1} + (\frac{r_i}{r_i-q_i})^{-1} = 1$). We set $\lambda(y) \coloneqq  4 \|\SSSS f\|_{q_\theta}^{- \tau \xi} y^\xi$, where $\tau, \xi \in \RR$ will be determined later on (of course, we can assume that $\|\SSSS f\|_{q_\theta} > 0$). Now, by using H\"older's inequality, we obtain
\begin{align*}
& \int_0^\infty y^{r_\theta-r_0-1} \int_{\|\SSSS f\|_{q_\theta}^{- \tau \xi} y^\xi}^{\infty} (t - \|\SSSS f\|_{q_\theta}^{- \tau \xi} y^\xi)^{q_0-1} \, d_{\SSSS f}(t) \, {\rm d}t \, \omega_0(y) \, {\rm d}y \\
& \qquad \leq \int_0^\infty t^{q_0-1} \, d_{\SSSS f}(t) \int_{0}^{\| \SSSS f \|_{q_\theta}^{\tau} t^{1/\xi}} y^{r_\theta-r_0-1} \omega_0(y) \, {\rm d}y \, {\rm d}t \\
& \qquad \leq \int_0^\infty t^{q_0-1} \, d_{\SSSS f}(t) \Big( \int_{0}^{\| \SSSS f \|_{q_\theta}^{\tau} t^{1/\xi}} y^{r_\theta-r_0-1} {\rm d}y \Big)^\frac{q_0}{r_0} \, \Big( \int_{0}^{\| \SSSS f \|_{q_\theta}^{\tau} t^{1/\xi}} y^{r_\theta-r_0-1} \omega_0^{\frac{r_0}{r_0-q_0}}(y) {\rm d}y \Big)^{\frac{r_0-q_0}{r_0}} {\rm d}t
\end{align*}
and the last expression does not exceed
\[
(r_\theta-r_0)^{-\frac{q_0}{r_0}} \| \SSSS f \|_{q_\theta}^{\frac{(r_\theta-r_0) q_0 \tau}{r_0}} \int_0^\infty t^{q_0-1 + \frac{(r_\theta-r_0) q_0}{\xi r_0}} \, d_{\SSSS f}(t) \, {\rm d}t.
\]
Similarly, we obtain
\begin{align*}
& \int_0^\infty y^{r_\theta-r_1-1} \int_0^{\|\SSSS f\|_{q_\theta}^{- \tau \xi} y^\xi} t^{q_1-1} \, d_{\SSSS f}(t) \, {\rm d}t \, \omega_1(y) \, {\rm d}y \\
& \qquad \leq (r_1-r_\theta)^{-\frac{q_1}{r_1}} \| \SSSS f \|_{q_\theta}^{\frac{(r_\theta-r_1) q_1 \tau}{r_1}} \int_0^\infty t^{q_1-1 + \frac{(r_\theta-r_1) q_1}{\xi r_1}} \, d_{\SSSS f}(t) \, {\rm d}t.
\end{align*}
Collecting these estimates we conclude that
\begin{displaymath}
\| \TTTT f \|_{r_\theta}^{r_\theta} \leq C'' \sum_{i=0}^1 \| \SSSS f \|_{q_\theta}^{(r_\theta-r_i) \tau} \Big( \int_0^\infty t^{q_i - 1 + \frac{(r_\theta - r_i) q_i}{r_i \xi}} d_{\SSSS f}(t) \, {\rm d}t \Big)^{r_i/q_i},
\end{displaymath}
for some $C''$ independent of $f$. Choosing
\begin{equation}\label{44-tau}\tag{A9}
\tau \coloneqq \frac{q_\theta (r_1 q_1^{-1} - r_0 q_0^{-1}) }{r_1 - r_0}, \qquad \xi \coloneqq \frac{q_\theta^{-1} (r_1^{-1} - r_\theta^{-1})}{r_\theta^{-1} (q_1^{-1} - q_\theta^{-1})},
\end{equation}
gives that both terms in the sum above equal $\| \SSSS f \|_{q_\theta}^{r_\theta}$. Thus \eqref{44-A3} holds with $\widetilde{C} = (2 C'')^{1/r_\theta}$, which completes the proof in the case $r_1 < \infty$.

Finally, let us assume that $r_1 = \infty$. If $q_1 = \infty$, then the formulas in \eqref{44-tau} reduce to $\tau = 0$ and $\xi = 1$. We choose $\lambda(y) \coloneqq cy$ for some sufficiently small constant $c \in (0,\infty)$. In fact, if $c < \CC_\rightarrow^{-1} \CC_\square^{-2} 2^{-1/p}$, then we have $d_{\TTTT f^\lambda_1}(y/2^{1/p}) = 0$, while $d_{\TTTT f^\lambda_0}(y/2^{1/p})$ may be estimated in the same way as it was done before. On the other hand, if $q_1 < \infty$, then the formulas in \eqref{44-tau} reduce to $\tau = q_\theta / q_1$ and $\xi = q_1 / (q_1 - q_\theta)$. Again, it can be shown that if $\lambda(y) \coloneqq c' \|f\|_{q_\theta}^{-q_\theta / (q_1 - q_\theta)} y^{q_1/(q_1- q_\theta)}$, where $c' \in (0,\infty)$ is sufficiently small (but independent of $f$ and $y$), then $d_{\TTTT f^\lambda_1}(y/2^{1/p}) = 0$ and $d_{\TTTT f^\lambda_0}(y/2^{1/p})$ may be estimated as before. This completes the proof in the case $r_1 = \infty$.
\end{proof} 

\setstretch{0.98}

\end{document}